\documentclass[11pt,reqno]{amsart}

 \usepackage{amsmath,amsthm,amsfonts}
\usepackage{latexsym,mathabx,txfonts}
\usepackage{amssymb}

\newcommand{\M}{\mathcal{M}}

\newcommand{\R}{\mathbb{R}}

\newcommand{\N}{\mathbb{N}}
\newcommand{\Z}{\mathbb{Z}}

\def\calL{{\mathcal L}}

\numberwithin{equation}{section}

\newtheorem{thm}{Theorem}[section]
\newtheorem{cor}[thm]{Corollary}
\newtheorem{lem}[thm]{Lemma}
\newtheorem{prop}[thm]{Proposition}

\theoremstyle{remark}
\newtheorem{rem}{Remark}[section]
\newtheorem{defn}{Definition}[section]

\newcommand{\Del}[1]{}

\newcommand{\x}{\xi}

\def\calR{\mathcal{R}}

\def\eps{\varepsilon}
\def\sign{\mathrm{sign}}

\def\til{\tilde}

\begin{document}

\title{On stability of type II blow up for the critical NLW on $\R^{3+1}$}
\author{Joachim Krieger}



\thanks{Support of the Swiss National Fund is gratefully acknowledged. The author would also like to thank Prof. W. Schlag of the University of Chicago as well as IHES for their hospitality in August 2015, resp. June/July 2016}
\begin{abstract}
We show that the finite time type II blow up solutions for the energy critical nonlinear wave equation 
\[
\Box u = -u^5
\]
on $\R^{3+1}$ constructed in \cite{KST}, \cite{KS1} are stable along a co-dimension three manifold of radial data perturbations in a suitable topology, provided the scaling parameter $\lambda(t) = t^{-1-\nu}$ is sufficiently close to the self-similar rate, i. e. $\nu>0$ is sufficiently small. Our method is based on Fourier techniques adapted to time dependent wave operators of the form 
\[
-\partial_t^2 + \partial_r^2 + \frac2r\partial_r +V(\lambda(t)r)
\]
for suitable monotone scaling parameters $\lambda(t)$ and potentials $V(r)$ with a resonance at zero. 
\end{abstract}

\maketitle

\section{Introduction}

The focussing energy critical nonlinear wave equation 
\begin{equation}\label{eq:Main}
\Box u = (-\partial_t^2 + \triangle)u = -u^5,\,u = u(t, x),\,(t,x)\in \R^{3+1},
\end{equation}
has recently attracted a lot of attention, as it has become clear that some of its dynamic features are characteristic of more geometric and physical models such as critical Wave Maps or Yang-Mills, yet the simplicity of the model 
allows to avoid many technical issues. In particular, it is believed that techniques developed for characterising its so-called type II dynamics, such as in the seminal series of works \cite{DKM1} - \cite{DKM4}, ought to be adaptable to attack the characterisation of solutions for more physical models. 
\\

The model \eqref{eq:Main} comes with a conserved energy 
\begin{equation}\label{eq:energy}
E(u) = \int_{\R^3}\big[\frac{1}{2}\big|\nabla_{t,x}u\big|^2 - \frac{1}{6}u^6\big]\,dx,
\end{equation}
of ambiguous sign. In particular, energy class Shatah-Struwe solutions (see \cite{SStruwe}) are not necessarily a priori bounded in $H^1(\R^3)$. This leads one to distinguish between type I and type II solutions, the latter being characterised by the property that 
\begin{equation}\label{eq:typeIIbound}
\sup_{t\in I}\big\|\nabla_{t,x}u(t, \cdot)\big\|_{L_x^2}<+\infty,
\end{equation}
where the open interval $I$ is the maximal interval of existence for a Shatah-Struwe solution $u$. On the other hand, type I solutions satisfy 
\begin{equation}\label{eq:typeI}
\sup_{t\in I}\big\|\nabla_{t,x}u(t, \cdot)\big\|_{L_x^2}=+\infty,
\end{equation}

Restricting to radial solution $u(t, x) = \tilde{u}(t, r)$, $r = |x|$, the works \cite{DKM1} - \cite{DKM4} provide an abstract classification of all singular type II solutions in terms of the so-called {\it{ground state}} $W(x)$, given by the explicit static solution
\begin{equation}\label{eq:groundstate}
u(t, x) = W(x): = \big(1+\frac{|x|^2}{3}\big)^{-\frac{1}{2}},
\end{equation} 
also known as {\it{Aubin-Talenti solution}} from its geometric origins. Observe that the {\it{scaling transformation}} associated with \eqref{eq:Main},
given by 
\[
u(t, x)\longrightarrow \lambda^{\frac12}u(\lambda t, \lambda x),\,\lambda>0,
\]
and carrying solutions into solutions, means that we obtain the one parameter family of scaled ground state static solutions
\[
W_{\lambda}(x) = \lambda^{\frac12} \big(1+\frac{\lambda^2|x|^2}{3}\big)^{-\frac{1}{2}}
\]
Then the Duyckaerts-Kenig-Merle classification of radial singular type II solutions for \eqref{eq:Main} is given by 
\begin{thm}(Duyckaerts-Kenig-Merle '09) Let $u(t, x)$, $t\in I$, be a type II radial Shatah-Struwe solution of \eqref{eq:Main} with maximal interval of existence $I$, and suppose $I = (T_0, T)$ with $T_0<T\in \R$. Then there exists $N\in \mathbf{N}$ as well as continuous functions $\lambda_j: I\rightarrow \R_+$, $j = 1,2,\ldots, N$, with 
\[
\lim_{t\rightarrow T}(T-t)\lambda_j(t) = +\infty,\,\lim_{t\rightarrow T}\big|\log\big(\frac{\lambda_j(t)}{\lambda_{j'}(t)}\big)\big| = +\infty,\,j\neq j', 
\]
and such that we can write 
\begin{equation}\label{eq:profiledecomp1}
u(t, x) = \sum_{j=1}^N\kappa_j W_{\lambda_j(t)}(x) + v(t, x),\,\kappa_j\in \{\pm 1\},
\end{equation}
and where $v(t, x)\in C^0(I\cup \{T\}; H^1),\,v_t\in C^0(I\cup\{T\}; L^2)$. 
\end{thm}
In particular, the limit $\lim_{t\rightarrow T}\nabla_{t,x}v(t, x)$ exists in the $L^2(\R^3)$-topology and allows one to prolong the solution $u(t,x)$ beyond the endpoint $t = T$ in the distributional sense, see \cite{KW}. 
\\

The result by Duyckaerts-Kenig-Merle does not make any assertion about the {\it{existence}} of solutions of the form \eqref{eq:profiledecomp1}, nor does it assert anything about their stability. In fact, at this point in time, it appears that only finite or infinite time blow up solutions with one bulk term $W_{\lambda_j(t)}$ are known. Furthermore, assuming the error term $v(t, x)$ to be of small critical norm, such solutions are known to be unstable in a quite precise way \cite{CNLW4}, as follows: 
\begin{thm}\label{thm:codim1stable}(\cite{CNLW4}) Let 
\begin{equation}\label{eq:oneprofile}
u(t, x) = W_{\lambda(t)}(x) + v(t, x)
\end{equation}
be a type II blow up solution for \eqref{eq:Main}, with $\lim_{t\rightarrow T}\lambda(t) = +\infty$, and such that 
\[
\sup_{t\in I}\big\|\nabla_{t,x}v(t, \cdot)\big\|_{L_x^2}\leq \delta\ll 1
\]
for some sufficiently small $\delta>0$, where as usual $I$ denotes the maximal life span of the Shatah-Struwe solution $u$. Also, assume that $t_0\in I$. Then there exists a co-dimension one Lipschitz manifold $\Sigma$ in a small neighbourhood of the data $\big(u(t_0, \cdot), u_t(t_0, \cdot)\big)\in \Sigma$ in the energy topology $\dot{H}^1(\R^3)\times L^2(\R^3)$ and such that initial data $\big(u_0, u_1\big)\in \Sigma$ result in a type II solution, while initial data 
\[
\big(u_0, u_1\big)\in B_{\delta}\backslash \Sigma,
\]
where $B_{\delta}\subset \dot{H}^1(\R^3)\times L^2(\R^3)$ is a sufficiently small ball centred at $\big(u(t_0, \cdot), u_t(t_0, \cdot)\big)$, either lead to blow up in finite time, or solutions scattering to zero, depending on the 'side of $\Sigma$' these data are chosen from. 
\end{thm}

In fact, the work \cite{CNLW4} gives much more, but the preceding result is all that is relevant to the discussion of the present paper. 
\\

The preceding theorem reveals that understanding the optimal stability of type II solutions of the form \eqref{eq:oneprofile} reduces\footnote{Note that it appears natural to conjecture that working in a suitable topology, the blow up solutions 'above' and sufficiently close to $\Sigma$ are all of type I.} to understanding the stability of such solutions under perturbations {\it{along the hyper surface $\Sigma$}}, and this irrespective of the precise scaling law $\lambda(t)$. This is then precisely the kind of question we aim to address in this paper, for a specific class of rather explicit blow up solutions constructed in \cite{KST}, \cite{KS1}. Before explaining this, we remark that the co-dimension one condition for $\Sigma$ simply reflects the fact that the linear operator $\mathcal{L}: = -\triangle - 5W^4$ has a unique negative eigenvalue $\xi_d<0$ which generically causes exponential growth for solutions of the corresponding linear wave equation 
\[
\big(\partial_t^2-\triangle - 5W^4\big)v(t, x) = 0. 
\]
For the nonlinear problem, this implies that perturbing initial data $u[0]\in \Sigma$ by a positive multiple of the corresponding positive eigenfunction $(\phi_d,0)$ leads to solutions which escape from a suitable neighbourhood of the family $\{\pm W_{\lambda}\}_{\lambda>0}$, and blow up in finite time(in fact, in a certain sense this blow up is of type I, see \cite{KW}), while adding a negative multiple of $(\phi_d, 0)$ leads to a solution existing globally in forward time and scattering toward zero. 
\\

{\it{From now on, we shall work exclusively in the context of radial solutions without further mention}}.
\\

Now pick $\nu>0$, set $\lambda(t) = t^{-1-\nu}$, and consider one of the blow up solutions $u_{\nu}(t, x) =: W_{\lambda(t)}(x) + v_{\nu}(t, x)$ constructed in \cite{KST}, \cite{KS1} and existing on an interval $I = (0, t_0]$ for $t_0>0$ sufficiently small. 
Note that $\sup_{t\in I}\|\nabla_{t,x}v(t, \cdot)\|_{L_x^2}\leq \delta\ll1 $ may be assumed for these solutions, and so the result from \cite{CNLW4} implies the existence of a co-dimension one Lipschitz manifold $\Sigma$ passing through the data $u_{\nu}[t_0]$ such that 
data $u[t_0]\in \Sigma$ result (forward in time) in type II solutions. 
\\

The question we want to address here is whether the blow up for $u_{\nu}$ is preserved generically under such perturbations along $\Sigma$. In fact, for small enough $\nu>0$, we can establish such a result under a suitable co-dimension two condition {\it{for perturbations along $\Sigma$}}, which in turn corresponds to a co-dimension three condition on these perturbations amongst all possible perturbations: 
\begin{thm}\label{thm:Main} There is $\nu_0>0$ sufficiently small, such that the following holds: Let $u_{\nu}$, $0<\nu<\nu_0$ be one of the solutions constructed in \cite{KST}, \cite{KS1}, on a time slice $(0, t_0]\times \R^3$, with $0<t_0\ll 1$ sufficiently small. Then there exists a co-dimension two Lipschitz hyper surface $\Sigma_0$ in a Hilbert space $\tilde{S}\times \R$ where $\tilde{S}$ is essentially $\big(H^{\frac{3}{2}+}_{rad}(\R^3)\cap\{\phi_d\}^{\perp}\big)\times \big(H^{\frac{1}{2}+}_{rad}(\R^3)\cap\{\phi_d\}^{\perp}\big)$, and a positive $\delta_1\ll 1$, such that for any $(u_0, u_1,\gamma)\in \Sigma_0\cap  \big(B_{\delta_1, \tilde{S}}(0)\times (-\delta_1, \delta_1)\big)$ and suitable Lipschitz functions 
\[
\gamma_{1,2}:  \Sigma_0\cap  \big(B_{\delta_1, \tilde{S}}(0)\times (-\delta_1, \delta_1)\big)\longrightarrow \R, 
\]
the solution of \eqref{eq:Main} with data 
\begin{align*}
u[t_0]: &=u_{\nu}[t_0] + \big(u_0, u_1\big) + \big(\gamma\phi_d +\gamma_1(u_0, u_1,\gamma)\phi_d, \gamma_2(u_0, u_1,\gamma)\phi_d\big)\\&\in \big(H^{1+}_{rad}(\R^3)\times H^{0+}_{rad}(\R^3)\big)\cap \Sigma
\end{align*}
exists on $I = (0, t_0]$ and can be written in the form 
\[
u(t, x) = W_{\lambda(t)} + v_1(t, x),\,\lambda(t) = t^{-1-\nu}
\]
with $(v_1, v_{1,t})\in H^{1+\frac{\nu}{2}-}\times H^{\frac{\nu}{2}-}$ on each time slice $t = t_1\in I$, and furthermore
\[
\big(E_{loc}(v)\big)(t): = \int_{|x|\leq t}\frac{1}{2}\big|\nabla_{t,x}v_1\big|^2\,dx\longrightarrow 0
\]
as $t\rightarrow 0$. Thus for small enough $\nu>0$, the solutions constructed in \cite{KST}, \cite{KS1} are stable under perturbations along a co-dimension three manifold in a suitable topology. 
\end{thm}

\begin{rem}We note that two of the co-dimension conditions producing $\Sigma_0$ may be thought of as being forced by the requirement of fixing the exponent $\nu$ in $\lambda(t) = Ct^{-1-\nu}$, as well as the scaling factor $C = 1$. We cannot 'modulate' in $\nu$ or $C$ to force these orthogonality conditions, since the solutions $u_{\nu}$ from \cite{KST}, \cite{KS1} corresponding to different $\nu$, $C$ are infinitely far apart with respect to $\big\|\cdot\big\|_{\tilde{S}}$ for small $\nu>0$. This then suggests that in order to gain two co-dimensions (and hence produce an optimal stability result in light of Theorem~\ref{thm:codim1stable}), one needs to work with a more flexible blow up scaling law $\lambda(t)$, depending on two additional parameters. 
\end{rem}
\begin{rem} The space $\tilde{S}$ will be specified after a recall of the distorted Fourier transform associated with $\mathcal{L}$ below. 

\end{rem}

The proof of the preceding theorem will proceed via Fourier techniques in suitable coordinate systems adapted to the blow up solutions $u_{\nu}(t, x)$. In particular, no use of virial type identities is being made. It is to be expected that such techniques extend to significantly more general contexts, for example the stability of blow up solutions for the critical Wave Maps equation. 

\subsection{Relation to other works} There is by now a very sizeable literature dealing with the construction as well as stability analysis of type II and other blow up solutions. In particular, in the context of type II dynamics, a remarkably robust method has been developed in the 
pioneering works by Merle, Martel, Raphael and co-authors in the last decade. In relation to the equation at hand, we mention for example the important works \cite{HR}, \cite{RaRod}, \cite{MaMeRa}. At the most basic level, the strategy in these works is to construct solutions in a two step process, first building an approximate solution, and then completing it to an exact solution, by controlling the remaining error via Morawetz and virial type identities and exploiting subtle monotonicity properties. 
It is in regard to the second step, the control of the remaining error, where the present work is striving to achieve a different approach, based on a constructive parametrix approach to the linear operator arising upon linearisation around the bulk part of the blow up solution. Our approach may be seen as somewhat in the spirit of the recent remarkable work by Donninger and Donninger and Schorkhuber on the stability of self-similar blow up solutions, for example in \cite{Don}, \cite{DoSch}, which also completely avoids Lyapounov/Morawetz type estimates. It may be hoped that constructive methods like those employed in this paper may shed further light on the precise features of the solutions obtained. \\
We also point out an exciting novel point of view espoused in works by J. Jendrej, \cite{Jen1}, \cite{Jen2}, where an analysis of type II blow up solutions of the natural five dimensional analogue of \eqref{eq:Main} in terms of the leftover radiation at blow up time is performed. This author has also succeeded in constructing the first examples of {\it{two bubble solutions}}, \cite{Jen3} on semi-infinite time intervals. The issue of constructing finite time radial two bubble blow ups appears still open at this time.

\subsection{Overview of the paper}

In section 2 we quickly recall the properties of the approximate solutions $u^{(k, \nu)}_{approx}$ which are used in \cite{KST}, \cite{KS1} to construct the blow up solutions $u_{\nu}(t, x)$. We also quickly gather the basic facts about the distorted Fourier transform associated to the operator 
\[
-\triangle - 5W^4
\]
in this section, as well as the translation of the problem in 'physical space' to a problem in 'Fourier space', given by \eqref{eq:transport}. In section 3, we analyse the growth of the forward flow associated with the transport operator on the left hand side of 
\eqref{eq:transport}, and in particular arrive at the two crucial vanishing conditions which guarantee that the growth is limited to linear in a certain sense. All of this is of course restricted to the linear flow associated with the transport operator. In section 4, we provide a precise formulation of the stability theorem we are striving to prove in terms of the Fourier variables, see Theorem~\ref{thm:MainTechnical}. In section 5 we show that the better growth estimates deduced under the key vanishing properties in section 3 suffice to control the nonlinear terms on the right hand side of \eqref{eq:transport}. In section 6, we show the strategy for constructing the blow up solutions by means of a suitable iterative scheme, leaving the precise choice of the iterates $\triangle x^{(j)}$ to the following sections. 
\\

Section 7 in a sense forms the heart of the paper. It is here where we analyse the {\it{first iterate}}, and more precisely the contribution to the first iterate of the delicate {\it{non-local linear source terms}} on the right hand side of \eqref{eq:transport}. The contributions of these terms to the first correction do not a priori gain smallness. We briefly explain here the main idea which allows us to make the correct choice for their contribution to the first iterate, by modifying the data slightly: thus consider an equation of the form 
\[
(\mathcal{D}_{\tau}^2 + \beta_{\nu}(\tau)\mathcal{D}_{\tau} + \xi)x(\tau, \xi) = f(\tau, \xi),
\]
where the right hand side denotes (say) one of the non-local linear source terms on the right hand side of \eqref{eq:transport}. Here we cannot proceed in analogy to \cite{KST}, \cite{KS1} and solve this by imposing vanishing of $x(\tau, \xi)$ toward $\tau = +\infty$, since we would then generate errors at time $\tau = \tau_0$ (which corresponds to the initial time $t = t_0$) which would be of the same order of magnitude as the original perturbation $(x_0, x_1)$. On the other hand, using a Duhamel parametrix at time $\tau = \tau_0$ forward in time will lead to functions $x(\tau, \xi)$ which grow much too rapidly toward $\tau = +\infty$ in the small frequency regime $\xi<1$, which is of course related to the issue forcing the vanishing conditions in section 3 for the free transport equation. The way out of this impasse is to {\it{add a small correction to the data}}, given by the pair\footnote{In fact, the precise choice of the corrections $(\triangle\tilde{\tilde{x}}_0^{(1)}, \triangle\tilde{\tilde{x}}_1^{(1)})$ shows that this correction is the analogue of 'modulating' on the inherent scaling parameter as was for example done in \cite{KS}.}
\[
(\triangle\tilde{\tilde{x}}_0^{(1)}, \triangle\tilde{\tilde{x}}_1^{(1)}).
\]
Then the free transport of the evolution of this pair will in some sense cancel the growth of the inhomogeneous Duhamel term for small frequencies, while for large frequencies its contribution will be small. Moreover, crucially, the correction is smaller by a factor $\tau_0^{-1+}$ than the data, and 'essentially preserves the initial perturbation'. We then carefully check in section 7 that the contribution to the first iterate by the linear non-local source terms thus defined itself admits in the low frequency regime a decomposition into a free transport term (given in terms of data $(\triangle\tilde{x}_0^{(1)}, \triangle\tilde{x}_1^{(1)})$ satisfying the crucial vanishing conditions from section 3) plus a term satisfying a better estimate (the term $\triangle_{>\tau}x^{(1)}$), and we spell out the precise bounds for the functions $\triangle_{>\tau}x^{(1)}$, $\mathcal{D}_{\tau}\triangle_{>\tau}x^{(1)}$, with a view toward the later stages of the iterative scheme. In particular, this involves a somewhat delicate refinement of the bound obtained for $\mathcal{D}_{\tau}\triangle_{>\tau}x^{(1)}$ in Proposition~\ref{prop:firstiteratelincont} to a square-sum over dyadic time scales type bound in Proposition~\ref{prop:firstiteratelincontimprov}. 
\\

We then repeat the same kind of construction for the contribution to the first iterate by the nonlinear source terms in \eqref{eq:transport} in section 8. 
\\
With these stages in place, it is not so hard anymore to formulate the iterative step in section 9, where we show that the kind of decomposition obtained for the first iterate can be perpetuated across an iterative scheme, resulting in a priori bounds, which may however grow at each stage. In particular, at this point, it is by no means clear that the iterative scheme thus arrived at will converge. 
\\

The ground work for convergence will be laid in sections 10, 11, where we show that certain expressions arrived at in the iterative scheme do enjoy a gain of smallness under certain frequency restrictions, or else under a {\it{re-iteration}} of the iterative step. This will then allow us in section 12 to reduce the convergence of the scheme to Proposition~\ref{prop:mainestimate}, which derives smallness of a {\it{many times re-iterated wave type propagator}} by a combinatorial type argument, reducing to integration over high dimensional simplices. This is in fact very closely related to arguments used in \cite{KST1}, \cite{DoMHKS}, and indeed it appears that it is in the present context that this technique comes to bear in the most natural and powerful way. 
\\

The proof of Theorem~\ref{thm:MainTechnical} will then finally be accomplished in section 13. 

\subsection{Further remarks} It appears that the techniques developed in this paper ought to have much wider applicability. In fact, the method expounded in sections 3 - 12 appears to apply to much more general wave operators of the form 
\[
-\partial_t^2 + \partial_r^2 + \frac2r\partial_r +V(\lambda(t)r),
\]
providing an approach to study their evolution based on a completely explicit iterative parametrix construction, and avoiding any kind of Lyapounov/Morawetz type estimates and implicit reasonings. 
\section*{Acknowledgement}

The author would like to thank Stefano Burzio for numerous corrections.

\section{Recall of the properties of the solutions $u_{\nu}(t, x)$}

The solutions $u_{\nu}$ in \cite{KST}, \cite{KS1}, are constructed via a two stage process, first generating an approximate solution via an ad hoc iterative procedure, and completion thereof to an exact solution. While the first stage consists in the solution of suitable elliptic problems, the second stage (i. e. completion to an exact solution) relies on a parametrix construction for a wave equation with time dependent potential. It is the techniques for the second step which play a prominent role in this paper, while the first step is mostly used as a black box here. 
\\

To be more precise, the approximate solution in \cite{KST}, \cite{KS1} is obtained in the form 
\[
u_{approx}^{(k, \nu)}(t, x) = W_{\lambda(t)}(x) + u_{2k-1}^{(k, \nu)}(t, x),
\]
where in turn $u_{2k-1}^{(k, \nu)}(t, x) = \sum_{l=1}^{2k-1}v_l(t, x)$ is a sum of corrections. Using the variables $R = \lambda(t)\cdot r$, $\tau = \nu^{-1}t^{-\nu}$, we have 
\begin{equation}\label{eq:u2k-1bound}
\big|u_{2k-1}^{(k, \nu)}(t, R)\big|\lesssim\frac{\lambda^{\frac12}}{(\lambda t)^2}R
\end{equation}
as well as 
\[
\big\|\nabla_R^{1+\frac{\nu}{2}-}u_{2k-1}\big\|_{L^2_{dR}(R\lesssim \tau)}\lesssim \frac{\lambda^{\frac12}}{\tau^{\frac{3}{2}+\frac{\nu}{2}-}}.
\]
Importantly, the error generated by the approximation $u_{approx}^{(k, \nu)}(t, x)$ can be made arbitrarily small in a suitable sense by picking $k$ suitably large. More precisely, we have 
\[
\big|t^2 e_{2k-1}(t, r)\big|\lesssim \frac{\lambda^{\frac12}}{(\lambda t)^{2k-1}}\big(1-a\big)^{-\frac{1}{2}+\frac{\nu}{2}-},\,a = \frac{r}{t}
\]
where 
\[
e_{2k-1} = \Box u_{approx}^{(k, \nu)} + \big(u_{approx}^{(k, \nu)}\big)^5. 
\]
To pass from the approximate solution $u_{approx}^{(k, \nu)}$ to the exact solution $u_{\nu}(t, x)$, we set $u_{\nu} = u_{approx}^{(k, \nu)} + \epsilon$, where $\epsilon$ is most easily controlled by passing to the new variable 
\[
\til\eps(\tau, R): = R\epsilon(t(\tau), r(\tau, R)). 
\]
In fact, see (3.1) in \cite{KS1}, we obtain the equation 
\begin{equation}\label{eq:epseqn}\begin{split}
&(\partial_{\tau} + \dot{\lambda}\lambda^{-1}R\partial_R)^2 \tilde{\eps} - \beta_{\nu}(\tau)(\partial_{\tau} + \dot{\lambda}\lambda^{-1}R\partial_R)\tilde{\eps} + \calL\tilde{\eps}\\
&=\lambda^{-2}(\tau)R[N_{2k-1}(\eps) + e_{2k-1}]+\partial_\tau(\dot{\lambda}\lambda^{-1})\tilde{\eps};\,\beta_{\nu}(\tau) = \dot{\lambda}(\tau)\lambda^{-1}(\tau), 
\end{split}\end{equation}
where the operator $\calL$ is given by 
\[
\calL = -\partial_R^2 - 5W^4(R)
\]
and we have 
\[
RN_{2k-1}(\eps) = 5(u_{2k-1}^4 - u_0^4)\tilde{\eps} + RN(u_{2k-1}, \tilde{\eps}),
\]
\[
RN(u_{2k-1}, \tilde{\eps}) = R(u_{2k-1}+\frac{\tilde{\eps}}{R})^5 - R u_{2k-1}^5 - 5u_{2k-1}^4\tilde{\eps}
\]
Introducing the operator, with $\beta_{\nu}(\tau) = \frac{\dot{\lambda}(\tau)}{\lambda(\tau)}$,
\[
\mathcal{D} = \partial_{\tau} + \beta_{\nu}(\tau)(R\partial_R - 1),
\]
we can also write the above equation as 
\begin{equation}\label{eq:Rtauwave1}
\mathcal{D}^2\tilde{\eps} + \beta_{\nu}(\tau)\mathcal{D}\tilde{\eps} + \calL\tilde{\eps} = \lambda^{-2}(\tau)\big[5(u_{2k-1}^4 - u_0^4)\tilde{\eps} + RN(u_{2k-1}, \tilde{\eps}) + R e_{2k-1}\big]
\end{equation}
To deal with this equation, we translate things in terms of the Fourier transform associated to the operator $\mathcal{L}$, just as in \cite{KST}, \cite{KS1}. In fact, we quote the following paragraph directly from \cite{KS1}: 
\\

Recall from \cite{KST} that there exists a Fourier basis $\phi(R, \xi)$ and associated spectral measure $\rho(\xi)$ satisfying the asymptotic expansions and growth conditions explained in \cite[Section 4]{KST} such that 
\[
\tilde{\eps}(\tau, R) = x_d(\tau)\phi_d(R) + \int_0^\infty x(\tau, \xi)\phi(R, \xi)\rho(\xi)\,d\xi
\]
For the asymptotic behavior of $\phi(R,\xi)$ in various regimes we shall rely on results from \cite{KST}. 
Here the functions $x(\tau, \xi)$ are the (distorted) Fourier coefficients associated with $\tilde{\eps}$, and $\phi_d(R)$ is the unique ground state with associated negative eigenvalue for the operator $\calL$. We also note the important asymptotic estimates 
\begin{equation}\label{eq:rhoasympto}
\rho(\xi)\simeq \xi^{-\frac{1}{2}},\,\xi\ll 1,\,\rho(\xi)\simeq \xi^{\frac{1}{2}},\,\xi\gg1.
\end{equation}
as well as the fact that near $\xi = 0$ as well as $\xi = \infty$ the spectral measure behaves like a symbol upon differentiation. 
We shall henceforth write 
\[
\underline{x}(\tau, \xi): = \left(\begin{array}{c}x_d(\tau)\\ x(\tau, \xi)\end{array}\right),\qquad \underline{\xi} = \binom{\xi_d}{\xi}
\]
Then proceeding as in \cite{DoKr}, in particular section 3.5 in loc. cit.     
which uses a variation on the procedure in \cite{KST}, we derive the following transport equation for $x(\tau, \xi)$: 
\begin{equation}\label{eq:transport}
\big(\mathcal{D}_{\tau}^2 + \beta_{\nu}(\tau)\mathcal{D}_{\tau} + \underline{\xi}\big)\underline{x}(\tau, \xi) = \calR(\tau, \underline{x}) + \underline{f}(\tau, \underline{\xi}),
\end{equation}
where we have
\begin{equation}\label{eq:Rterms}
\calR(\tau, \underline{x})(\xi) = \Big(-4\beta_{\nu}(\tau)\mathcal{K}\mathcal{D}_{\tau}\underline{x} - \beta_{\nu}^2(\tau)(\mathcal{K}^2 + [\mathcal{A}, \mathcal{K}] + \mathcal{K} +  \beta_{\nu}' \beta_{\nu}^{-2}\mathcal{K})\underline{x}\Big)(\xi)
\end{equation}
 with $\beta_{\nu}(\tau) = \frac{\dot{\lambda}(\tau)}{\lambda(\tau)}$, and we set $\underline{f} = \left(\begin{array}{c}f_d\\ f\end{array}\right)$ where
 \begin{equation}\label{eq:fterms}\begin{split}
 &f(\tau, \xi) = \mathcal{F}\big( \lambda^{-2}(\tau)\big[5(u_{2k-1}^4 - u_0^4)\tilde{\eps} + RN(u_{2k-1}, \tilde{\eps}) + R e_{2k-1}\big]\big)\big(\xi\big)\\
 &f_d(\tau) = \langle \lambda^{-2}(\tau)\big[5(u_{2k-1}^4 - u_0^4)\tilde{\eps} + RN(u_{2k-1}, \tilde{\eps}) + R e_{2k-1}\big], \phi_d(R)\rangle.
 \end{split}\end{equation}
Also the key operator 
 \[
 \mathcal{D}_{\tau} = \partial_{\tau} + \beta_{\nu}(\tau)\mathcal{A},\quad \mathcal{A} = \left(\begin{array}{cc}0&0\\0&\mathcal{A}_c\end{array}\right)
 \]
 and we have 
 \[
 \mathcal{A}_c = -2\xi\partial_{\xi} - \Big (\frac{5}{2}  + \frac{\rho'(\xi)\xi}{\rho(\xi)} \Big)
 \]
 Finally, we observe that the ``transference operator'' $\mathcal{K}$ is given by the following type of expression 
 \begin{equation}\label{eq:Kstructure}
 \mathcal{K} = \left(\begin{array}{cc}\mathcal{K}_{dd}&\mathcal{K}_{dc}\\
 \mathcal{K}_{cd}&\mathcal{K}_{cc}\end{array}\right)
 \end{equation}
 with the following description (\cite{KST}): 
 \begin{itemize}
 \item The operators $\mathcal{K}_{dd}$, $ \mathcal{K}_{cd}$ are given by $\mathcal{K}_{dd} = -\frac12$, $\big(\mathcal{K}_{cd}(1)\big)(\xi) = K_d(\xi)$, where $K_d$ is smooth and rapidly decaying toward $\xi = +\infty$. 
 \item We have $\mathcal{K}_{dc}(f) = -\int_0^\infty f(\xi)K_d(\xi)\rho(\xi)\,d\xi$, with $K_d$ as above. 
 \item We have $\big(\mathcal{K}_{cc}(f)\big)(\xi) = \int_0^\infty K_0(\xi, \eta)f(\eta)\,d\eta$, where $K_0(\xi, \eta) = \frac{\rho(\eta)}{\xi - \eta}F(\xi, \eta)$, and $F$ is a $C^2$-function with the further smoothness and decay properties listed in Theorem 5.1 in \cite{KST}, see also the remarks in \cite{KS1}, proof of Lemma 5.1. 
 \end{itemize}
 
The construction in \cite{KS1} then relies crucially on the observation that the equation
\[
\big(\mathcal{D}_{\tau}^2 + \beta_{\nu}(\tau)\mathcal{D}_{\tau} + \underline{\xi}\big)\underline{x}(\tau, \xi) = \underline{f}(\tau, \xi) = \left(\begin{array}{c}f_d(\tau)\\ f(\tau, \xi)\end{array}\right)
\]
can be solved completely explicitly; in particular, imposing vanishing boundary data at $\tau = \infty$, we obtain the following expression for the continuous part $x(\tau, \xi)$: 
\begin{equation}\label{eq:para1}
x(\tau, \xi) = \xi^{-\frac{1}{2}}\int_{\tau}^\infty \frac{\lambda^{\frac{3}{2}}(\tau)}{\lambda^{\frac{3}{2}}(\sigma)}\frac{\rho^{\frac{1}{2}}(\frac{\lambda^{2}(\tau)}{\lambda^{2}(\sigma)}\xi)}{\rho^{\frac{1}{2}}(\xi)}\sin\Big[\lambda(\tau)\xi^{\frac{1}{2}}\int_{\tau}^{\sigma}\lambda^{-1}(u)\,du\Big]f\big(\sigma, \frac{\lambda^{2}(\tau)}{\lambda^{2}(\sigma)}\xi\big)\,d\sigma
\end{equation}
On the other hand, one immediately obtains the elementary implicit relation 
\begin{equation}\begin{split}\label{eq:x_d}
&x_d(\tau) = \int_{\tau_0}^\infty H_d(\tau, \sigma)\tilde{f}_d(\sigma)\,d\sigma,\; \; H_d(\tau, \sigma) = -\frac{1}{2}|\xi_d|^{-\frac{1}{2}}e^{-|\xi_d|^{\frac{1}{2}}|\tau-\sigma|}\\
&\tilde{f}_d(\sigma) = f_d(\sigma) -  \beta_{\nu}(\sigma)\partial_{\tau}x_d(\sigma)
\end{split}\end{equation}
See \cite{KS1} for justification of these facts. 
\\

In the present work, we shall want to perturb the data at time $t = t_0$, which means that we shall no longer enforce the condition of vanishing at infinity (in terms of $\tau$). For this, it shall be important to understand the properties of solutions of the {\it{homogeneous linear problem}} 
\begin{equation}\label{eq:linhom1}
\big(\mathcal{D}_{\tau}^2 + \beta_{\nu}(\tau)\mathcal{D}_{\tau} + \underline{\xi}\big)\underline{x}(\tau, \xi) = 0,\,\big(\underline{x}(\tau_0,\cdot), \mathcal{D}_{\tau}\underline{x}(\tau_0,\cdot)\big) = \big(\underline{x}_0, \underline{x}_1\big).
\end{equation}
Here we use the convention that $\mathcal{D}_{\tau}$ acts on the first component $x_d(\tau)$ in the standard way $x_d(\tau)\longrightarrow \partial_{\tau}x_d(\tau)$. Then the following result follows from the arguments in \cite{KS1}: 
\begin{lem}\label{lem:linhom} The equation \eqref{eq:linhom1} is solved for the continuous spectral part $x(\tau, \xi)$ via the following parametrix: 
\begin{equation}\label{eq:linhomparam1}\begin{split}
x(\tau, \xi) = &\frac{\lambda^{\frac{5}{2}}(\tau)}{\lambda^{\frac{5}{2}}(\tau_0)}\frac{\rho^{\frac{1}{2}}(\frac{\lambda^{2}(\tau)}{\lambda^{2}(\tau_0)}\xi)}{\rho^{\frac{1}{2}}(\xi)}\cos\Big[\lambda(\tau)\xi^{\frac{1}{2}}\int_{\tau_0}^{\tau}\lambda^{-1}(u)\,du\Big]x_0\big(\frac{\lambda^{2}(\tau)}{\lambda^{2}(\tau_0)}\xi\big)\\
& + \frac{\lambda^{\frac{3}{2}}(\tau)}{\lambda^{\frac{3}{2}}(\tau_0)}\frac{\rho^{\frac{1}{2}}(\frac{\lambda^{2}(\tau)}{\lambda^{2}(\tau_0)}\xi)}{\rho^{\frac{1}{2}}(\xi)}\frac{\sin\Big[\lambda(\tau)\xi^{\frac{1}{2}}\int_{\tau_0}^{\tau}\lambda^{-1}(u)\,du\Big]}{\xi^{\frac12}}x_1\big(\frac{\lambda^{2}(\tau)}{\lambda^{2}(\tau_0)}\xi\big)\\
\end{split}\end{equation}
Moreover, writing $\underline{x}_0 = \left(\begin{array}{c}x_{0d}\\ x_0(\xi)\end{array}\right)$, $\underline{x}_1 = \left(\begin{array}{c}x_{1d}\\ x_1(\xi)\end{array}\right)$ and picking $\tau_0\gg 1$ sufficiently large, there is $c_d = 1 + O(\tau_0^{-1})$ as well as $\gamma_d = -|\xi_d|^{\frac12} + O(\tau_0^{-1})$ such that if we impose the co-dimension one condition 
\begin{equation}\label{eq:cond1}
x_{1d} = \gamma_d x_{0d},
\end{equation}
then the discrete spectral part of $\underline{x}(\tau,\xi)$ admits for any $\kappa>0$ the representation
\[
x_d(\tau) = \big(1+O_{\kappa}(\tau^{-1}e^{\kappa(\tau-\tau_0)})\big)e^{-|\xi_d|^{\frac12}(\tau - \tau_0)}c_d x_{0d}
\]
One also has for $i\geq 1$
\[
(-\partial_{\tau})^ix_d(\tau) = \big(1+O_{\kappa}(\tau^{-1}e^{\kappa(\tau-\tau_0)})\big)|\xi_d|^{\frac{i}{2}}e^{-|\xi_d|^{\frac12}(\tau - \tau_0)}c_d x_{0d}
\]
\end{lem}

The proof of this lemma is essentially contained in \cite{KS1}, except for the last part. To see this, we use \eqref{eq:x_d} with $f_d = 0$ but adding an exponentially decaying free term to match the part of the data at time $\tau = \tau_0$, which gives the implicit relation
\begin{align*}
x_d(\tau) = &\int_{\tau_0}^\infty\beta_{\nu}(\sigma)[\partial_{\sigma}x_d(\sigma) - \tilde{x}_d|\xi_d|^{\frac12}e^{-|\xi_d|^{\frac12}(\sigma - \tau_0)}]\cdot\frac{1}{2}|\xi_d|^{-\frac12}e^{-|\xi_d|^{\frac12}|\tau - \sigma|}\,d\sigma\\& + e^{-|\xi_d|^{\frac12}(\tau - \tau_0)}\cdot \tilde{x}_d.
\end{align*}
Carrying out an integration by parts, this gives 
\begin{align*}
x_d(\tau) &=  (\tilde{x}_d - \beta_{\nu}(\tau_0)x_d(\tau_0))e^{-|\xi_d|^{\frac12}(\tau - \tau_0)}\\&   -\int_{\tau_0}^\infty \partial_{\sigma}\beta_{\nu}(\sigma)x_d(\sigma)\frac{1}{2}|\xi_d|^{-\frac12}e^{-|\xi_d|^{\frac12}|\tau - \sigma|}\,d\sigma\\
& - \int_{\tau_0}^\infty \beta_{\nu}(\sigma)x_d(\sigma)\frac{1}{2}\sign(\tau-\sigma)e^{-|\xi_d|^{\frac12}|\tau - \sigma|}\,d\sigma\\
& - \int_{\tau_0}^\infty\tilde{x}_d\beta_{\nu}(\sigma)e^{-|\xi_d|^{\frac12}(\sigma - \tau_0)}\cdot\frac{1}{2}e^{-|\xi_d|^{\frac12}|\tau - \sigma|}\,d\sigma.
\end{align*}
The conclusion of the lemma then follows from a simple fixed point argument. 

\section{Growth properties of the forward linear parametrix \eqref{eq:linhomparam1}}

Recall that $t = (\nu\tau)^{-\nu^{-1}}$, which implies\footnote{We abuse notation and write $\lambda(\tau)$ instead of $\lambda(t(\tau))$} $\lambda(\tau) \sim \tau^{1+\nu^{-1}}$ and hence rapidly polynomially growing for $\nu\ll 1$. However, what decides whether an iteration scheme for the full nonlinear problem starting at the linear approximation \eqref{eq:linhomparam1} will converge are not the growth properties of the Fourier coefficients $x(\tau, \xi)$, but rather those of the function $\til\eps(\tau, R)$ they represent, for it is the latter which gets inserted in the nonlinear expressions, and it is here where fast growth toward $\tau = +\infty$ or alternatively $t\rightarrow 0$ would doom any iteration. 
\\

It turns out that there is a natural co-dimension two condition which while not altogether eliminating growth at infinity (for $\til\eps(\tau, R)$!) reduces it to only linear growth in a suitable sense, and it turns out that this suffices to still run an iteration scheme, thanks to damping factors attached to the nonlinear terms. The following proposition is the key for the sequel: 
\begin{prop}\label{prop:lingrowthcond}Assume the data $(x_0, x_1)\in \langle\xi\rangle^{-1-}\xi^{0+}L^2_{d\xi}\times  \langle\xi\rangle^{-\frac{1}{2}-}\xi^{0+}L^2_{d\xi}$. Furthermore, assume that we have the vanishing relations 
 \begin{equation}\label{eq:vanishing}
 \int_0^\infty\frac{\rho^{\frac{1}{2}}(\xi)x_0(\xi)}{\xi^{\frac{1}{4}}}\cos[\nu\tau_0\xi^{\frac{1}{2}}]\,d\xi = 0,\,\int_0^\infty\frac{\rho^{\frac{1}{2}}(\xi)x_1(\xi)}{\xi^{\frac{3}{4}}}\sin[\nu\tau_0\xi^{\frac{1}{2}}]\,d\xi = 0. 
 \end{equation}
 at time $\tau = \tau_0$. Assume that $x(\tau, \xi)$ is given by \eqref{eq:linhomparam1}. Then the function $P_c\tilde{\epsilon}(\tau, R)$ represented by the Fourier coefficients $x(\tau, \xi)$ via 
 \[
 P_c\tilde{\epsilon}(\tau, R) = \int_0^\infty \phi(R, \xi)x(\tau, \xi)\rho(\xi)\,d\xi
 \]
 satisfies 
 \[
 P_c\tilde{\epsilon}(\tau, R) = \tilde{\epsilon}_1(\tau, R) + \tilde{\epsilon}_2(\tau, R),
 \]
 where we have 
 \[
 \big\|\frac{\tilde{\epsilon}_1(\tau, R)}{R}\big\|_{L^\infty_{dR}}\lesssim \big\|(\langle\xi\rangle^{\frac12++}\xi^{\frac12-}x_0, \langle\xi\rangle^{\frac{1}{2}++}\xi^{0-}x_1)\big\|_{L^2_{d\xi}}
 \]
 \[
 \big\| \tilde{\epsilon}_2(\tau, R)\big\|_{L^\infty_{dR}}\lesssim \tau \big\|(\langle\xi\rangle^{\frac12++}\xi^{\frac12-}x_0, \langle\xi\rangle^{\frac{1}{2}++}\xi^{0-}x_1)\big\|_{L^2_{d\xi}}
 \]
 Here we use for a small constant $\delta_0>0$ the notation $\frac12++ : = \frac{1}{2}+2\delta_0,\,\frac12+ = \frac12 + \delta_0,\,\frac12-: = \frac12-\delta_0,\,0-: = -\delta_0$. For later reference, we shall also use the notation $1+ = 1+\delta_0, 0+ = \delta_0$.
  \end{prop}
  We note that the integrals in \eqref{eq:vanishing} converge  under the hypothesis on $x_{0,1}$ in the proposition. 
  We also observe that the norm displayed on the right corresponds roughly to $H^{\frac{3}{2}+}_{dR}\times H^{\frac{1}{2}+}_{dR}$ on the physical side. 
\\
The preceding proposition motivates introduction of the following norms: we put 
\begin{equation}\label{eq:tildeSnorm}\begin{split}
\big\|(x_0, x_1)\big\|_{\tilde{S}}&: = \big\|x_0\big\|_{\tilde{S}_1} + \big\|x_1\big\|_{S_2}\\&: = \big\|\langle\xi\rangle^{\frac12++}\min\{\tau_0\xi^{\frac12},1\}^{-1}\xi^{\frac12-}x_0\big\|_{L^2_{d\xi}} + \big\| \langle\xi\rangle^{\frac{1}{2}++}\xi^{0-}x_1\big\|_{L^2_{d\xi}}.  
\end{split}\end{equation}
Note that for very small frequencies $\xi\ll \tau_0^{-2}$, we have built in control over $\big\|\xi^{0-}x_0\big\|_{L^2_{d\xi}}$, whence control over $\big\|\cdot\big\|_{\tilde{S}}$ suffices to ensure the quantities in \eqref{eq:vanishing} are well-defined. 
For later reference, we also introduce the following stronger norm: 
\begin{equation}\label{eq:Snorm}\begin{split}
\big\|(x_0, x_1)\big\|_{S}&: = \big\|x_0\big\|_{S_1} + \big\|x_1\big\|_{S_2}\\&: = \big\|\langle\xi\rangle^{1++}\xi^{0-}x_0\big\|_{L^2_{d\xi}} + \big\| \langle\xi\rangle^{\frac{1}{2}++}\xi^{0-}x_1\big\|_{L^2_{d\xi}},\,1++ = 1+2\delta_0.
\end{split}\end{equation}

In the following proof, we shall heavily rely on the spectral theory developed in \cite{KST}. 
\begin{proof}
We shall constantly refer to the formula \eqref{eq:linhomparam1}. In particular, we shall call the first term on the right the contribution of $x_0$, while the second term is the contribution of $x_1$. 
Here we start with the contribution of $x_0$, which we treat by distinguishing between different regimes for the frequency $\xi$. Throughout we take advantage of the fact (see \cite{KST}) that $\frac{\phi(R, \xi)}{R}$ is uniformly bounded in $R, \xi>0$. 
\\

{\it(1). Large frequencies: $\xi>1$. } In this regime we have 
\[
\frac{\rho^{\frac{1}{2}}(\frac{\lambda^{2}(\tau)}{\lambda^{2}(\tau_0)}\xi)}{\rho^{\frac{1}{2}}(\xi)}\sim \big(\frac{\lambda(\tau)}{\lambda(\tau_0)}\big)^{\frac{1}{2}},\,\rho(\xi)\sim \xi^{\frac{1}{2}},
\]
Hence  the corresponding contribution to $P_c\tilde{\epsilon}(\tau, R)$ is bounded by 
\begin{align*}
&\lesssim \int_1^\infty\frac{\lambda^3(\tau)}{\lambda^3(\tau_0)}|\frac{\phi(R, \xi)}{R}||\cos[\lambda(\tau)\xi^{\frac{1}{2}}\int_{\tau_0}^{\tau}\lambda^{-1}(u)\,du]||x_0|(\xi\frac{\lambda^2(\tau)}{\lambda^{2}(\tau_0)})\xi^{\frac{1}{2}}\,d\xi\\
&\lesssim \big\|\langle\xi\rangle^{1+}x_0\big\|_{L^2_{d\xi}},
\end{align*}
where we have used the Cauchy-Schwarz inequality as well as a simple change of variables. Note that in fact we also get $\big\|P_c\tilde{\epsilon}(\tau, R)\big\|_{L^\infty}\lesssim \big\|\langle\xi\rangle^{1+}x_0\big\|_{L^2_{d\xi}}$ for this contributions (division by $R$ for large $R$ is unnecessary). 
\\

{\it(2). Small frequencies $\xi<1$.} Here we use a decomposition 
\begin{equation}\label{eq:Rxisplit}\begin{split}
&\int_0^1\phi(R, \xi)y(\xi)\,d\xi\\&= \int_0^{\min\{R^{-2},1\}}\phi(R, \xi)y(\xi)\,d\xi + \int_{\min\{R^{-2},1\}}^1\phi(R, \xi)y(\xi)\,d\xi\\
& = \int_0^{\min\{R^{-2},1\}}[\phi(R, 0) + O(R^2\xi)]y(\xi)\,d\xi + \int_{\min\{R^{-2},1\}}^1\phi(R, \xi)y(\xi)\,d\xi\\
& = \phi(R, 0)\int_0^{\min\{R^{-2},1\}} y(\xi)\,d\xi + \int_0^{\min\{R^{-2},1\}}O(R^2\xi)y(\xi)\,d\xi\\& + \int_{\min\{R^{-2},1\}}^1\phi(R, \xi)y(\xi)\,d\xi.\\
\end{split}\end{equation}
Then we substitute
\[
y(\xi) =  \rho(\xi)\frac{\lambda^{\frac{5}{2}}(\tau)}{\lambda^{\frac{5}{2}}(\tau_0)}\frac{\rho^{\frac{1}{2}}(\frac{\lambda^{2}(\tau)}{\lambda^{2}(\tau_0)}\xi)}{\rho^{\frac{1}{2}}(\xi)}\cos[\lambda(\tau)\xi^{\frac{1}{2}}\int_{\tau_0}^{\tau}\lambda^{-1}(u)\,du]x_0(\frac{\lambda^2(\tau)}{\lambda^{2}(\tau_0)}\xi)
\]
We get 
\begin{align*}
&\big|R^{-1}\int_0^{\min\{R^{-2},1\}}O(R^2\xi)y(\xi)\,d\xi\big|\lesssim \int_0^{\min\{R^{-2},1\}}\xi^{\frac{1}{2}}|y(\xi)|\,d\xi,
\end{align*}
and we can further bound this by 
\begin{align*}
&\int_0^{\min\{R^{-2},1\}}\xi^{\frac{1}{2}}|y(\xi)|\,d\xi\\&\lesssim \int_0^{\min\{R^{-2},1\}}\frac{\lambda^{\frac{5}{2}}(\tau)}{\lambda^{\frac{5}{2}}(\tau_0)}\frac{\rho^{\frac{1}{2}}(\frac{\lambda^{2}(\tau)}{\lambda^{2}(\tau_0)}\xi)}{\rho^{\frac{1}{2}}(\xi)}\xi^{\frac12}\big|\cos[\lambda(\tau)\xi^{\frac{1}{2}}\int_{\tau_0}^{\tau}\lambda^{-1}(u)\,du]\big|\big|x_0(\frac{\lambda^2(\tau)}{\lambda^{2}(\tau_0)}\xi)\big|\rho(\xi)\,d\xi,
\end{align*}
Further restricting to the range $\xi\frac{\lambda^2(\tau)}{\lambda^2(\tau_0)}<1$, this is bounded by 
\begin{align*}
&\int_0^{\min\{R^{-2},1\}}\chi_{\xi\frac{\lambda^2(\tau)}{\lambda^2(\tau_0)}<1}\frac{\lambda^{\frac{5}{2}}(\tau)}{\lambda^{\frac{5}{2}}(\tau_0)}\frac{\rho^{\frac{1}{2}}(\frac{\lambda^{2}(\tau)}{\lambda^{2}(\tau_0)}\xi)}{\rho^{\frac{1}{2}}(\xi)}\xi^{\frac12}\big|\\&\hspace{4cm}\cdot\cos[\lambda(\tau)\xi^{\frac{1}{2}}\int_{\tau_0}^{\tau}\lambda^{-1}(u)\,du]\big|\big|x_0(\frac{\lambda^2(\tau)}{\lambda^{2}(\tau_0)}\xi)\big|\rho(\xi)\,d\xi\\
&\lesssim \int_0^\infty\chi_{\xi\frac{\lambda^2(\tau)}{\lambda^2(\tau_0)}<1}\frac{\lambda^{2}(\tau)}{\lambda^{2}(\tau_0)}\big|x_0(\frac{\lambda^2(\tau)}{\lambda^{2}(\tau_0)}\xi)\big|\,d\xi\lesssim \big\|\xi^{\frac12-}x_0\big\|_{L^2_{d\xi}(\xi<1)},
\end{align*}
where we have again used the Cauchy-Schwarz inequality in the last step. 
\\
Similarly, restricting to $\xi\frac{\lambda^2(\tau)}{\lambda^2(\tau_0)}>1$, we have 
\begin{align*}
&\int_0^{\min\{R^{-2},1\}}\chi_{\xi\frac{\lambda^2(\tau)}{\lambda^2(\tau_0)}>1}\frac{\lambda^{\frac{5}{2}}(\tau)}{\lambda^{\frac{5}{2}}(\tau_0)}\frac{\rho^{\frac{1}{2}}(\frac{\lambda^{2}(\tau)}{\lambda^{2}(\tau_0)}\xi)}{\rho^{\frac{1}{2}}(\xi)}\xi^{\frac12}\big|\\&\hspace{4cm}\cdot\cos[\lambda(\tau)\xi^{\frac{1}{2}}\int_{\tau_0}^{\tau}\lambda^{-1}(u)\,du]\big|\big|x_0(\frac{\lambda^2(\tau)}{\lambda^{2}(\tau_0)}\xi)\big|\rho(\xi)\,d\xi\\
&\lesssim \int_0^1\chi_{\xi\frac{\lambda^2(\tau)}{\lambda^2(\tau_0)}>1}\frac{\lambda^{3}(\tau)}{\lambda^{3}(\tau_0)}\xi^{\frac12}\big|x_0(\frac{\lambda^2(\tau)}{\lambda^{2}(\tau_0)}\xi)\big|\,d\xi\lesssim \big\|\langle \xi\rangle^{1+}x_0\big\|_{L^2_{d\xi}(\xi>1)}
\end{align*}
after change of variables and application of Cauchy-Schwarz. 
\\

Similarly, in the range $1>\xi>R^{-2}$, we get 
\begin{align*}
\big|R^{-1}\int_{\min\{R^{-2},1\}}^1\phi(R, \xi) y(\xi)\,d\xi\big|\lesssim \int_{\min\{R^{-2},1\}}^1\xi^{\frac{1}{2}}|\phi(R, \xi)||y(\xi)|\,d\xi,
\end{align*}
and this expression is bounded just as before by splitting into the cases $\xi\frac{\lambda^2(\tau)}{\lambda^2(\tau_0)}\lessgtr 1$. 
\\

It then remains to deal with the most delicate case, which is the multiple of the resonance 
\begin{align*}
&\phi(R, 0)\int_0^{\min\{R^{-2},1\}} y(\xi)\,d\xi\\&  = \phi(R, 0)\int_0^{\min\{R^{-2},1\}} \frac{\lambda^{\frac{5}{2}}(\tau)}{\lambda^{\frac{5}{2}}(\tau_0)}\frac{\rho^{\frac{1}{2}}(\frac{\lambda^{2}(\tau)}{\lambda^{2}(\tau_0)}\xi)}{\rho^{\frac{1}{2}}(\xi)}\\&\hspace{4cm}\cdot\big|\cos[\lambda(\tau)\xi^{\frac{1}{2}}\int_{\tau_0}^{\tau}\lambda^{-1}(u)\,du]\big|\big|x_0(\frac{\lambda^2(\tau)}{\lambda^{2}(\tau_0)}\xi)\big|\rho(\xi)\,d\xi\ \,d\xi 
\end{align*}
For small frequencies $\xi<1$, we have 
\begin{equation}\label{eq:rhoasymptotics}
\rho(\xi) = \frac{1}{3\pi}\xi^{-\frac{1}{2}}(1+O(\xi^{\frac{1}{2}})). 
\end{equation}
Here, the contribution of the error term $ \frac{1}{3\pi}\xi^{-\frac{1}{2}}\cdot O(\xi^{\frac{1}{2}})$ is bounded by 
\begin{align*}
&\lesssim\int_0^1 \frac{\lambda^{\frac{5}{2}}(\tau)}{\lambda^{\frac{5}{2}}(\tau_0)}\frac{\rho^{\frac{1}{2}}(\frac{\lambda^{2}(\tau)}{\lambda^{2}(\tau_0)}\xi)}{\rho^{\frac{1}{2}}(\xi)}\big|\cos[\lambda(\tau)\xi^{\frac{1}{2}}\int_{\tau_0}^{\tau}\lambda^{-1}(u)\,du]\big|\big|x_0(\frac{\lambda^2(\tau)}{\lambda^{2}(\tau_0)}\xi)\big|\,d\xi\\
&\lesssim \big\|\xi^{\frac{1}{2}-}\langle\xi\rangle^{\frac12++}x_0\big\|_{L^2_{d\xi}}, 
\end{align*}
as one sees by splitting into the cases $\xi\frac{\lambda^2(\tau)}{\lambda^2(\tau_0)}\lessgtr 1$ and using a change of variables as well as Cauchy-Schwarz. 
 We can thus replace the factors $\rho^{-\frac{1}{2}}(\xi), \rho(\xi)$ by $\xi^{\frac{1}{4}}, \xi^{-\frac{1}{2}}$, respectively, which reduces things to the expression
\[
\phi(R, 0)\int_0^{\min\{R^{-2},1\}} \frac{\lambda^{\frac{5}{2}}(\tau)}{\lambda^{\frac{5}{2}}(\tau_0)}\rho^{\frac{1}{2}}(\frac{\lambda^{2}(\tau)}{\lambda^{2}(\tau_0)}\xi)\frac{\cos[\lambda(\tau)\xi^{\frac{1}{2}}\int_{\tau_0}^{\tau}\lambda^{-1}(u)\,du]}{\xi^{\frac{1}{4}}}x_0(\frac{\lambda^2(\tau)}{\lambda^{2}(\tau_0)}\xi)\,d\xi
\]
Note that on account of $\frac{\phi(R, 0)}{R}\xi^{-\frac12}\lesssim \min\{1, \xi^{-\frac12}\}$ for $\xi\geq \min\{R^{-2},1\}$, we get  
\begin{align*}
&\big|\frac{\phi(R, 0)}{R}\int_{\min\{R^{-2},1\}}^\infty \frac{\lambda^{\frac{5}{2}}(\tau)}{\lambda^{\frac{5}{2}}(\tau_0)}\rho^{\frac{1}{2}}(\frac{\lambda^{2}(\tau)}{\lambda^{2}(\tau_0)}\xi)\frac{\cos[\lambda(\tau)\xi^{\frac{1}{2}}\int_{\tau_0}^{\tau}\lambda^{-1}(u)\,du]}{\xi^{\frac{1}{4}}}x_0(\frac{\lambda^2(\tau)}{\lambda^{2}(\tau_0)}\xi)\,d\xi\big|\\
&\lesssim \big|\int_{\min\{R^{-2},1\}}^\infty \frac{\lambda^{\frac{5}{2}}(\tau)}{\lambda^{\frac{5}{2}}(\tau_0)}\rho^{\frac{1}{2}}(\frac{\lambda^{2}(\tau)}{\lambda^{2}(\tau_0)}\xi)\xi^{\frac14}\cos[\lambda(\tau)\xi^{\frac{1}{2}}\int_{\tau_0}^{\tau}\lambda^{-1}(u)\,du]x_0(\frac{\lambda^2(\tau)}{\lambda^{2}(\tau_0)}\xi)\,d\xi\big|,\\
\end{align*}
and the preceding expression is again easily seen to be bounded by 
\[
\lesssim \big\|\xi^{\frac12-}\langle\xi\rangle^{\frac12++}x_0\big\|_{L^2_{d\xi}}.
\]
It follows that it suffices to bound 
\begin{align*}
\phi(R, 0)\int_0^{\infty} \frac{\lambda^{\frac{5}{2}}(\tau)}{\lambda^{\frac{5}{2}}(\tau_0)}\rho^{\frac{1}{2}}(\frac{\lambda^{2}(\tau)}{\lambda^{2}(\tau_0)}\xi)\frac{\cos[\lambda(\tau)\xi^{\frac{1}{2}}\int_{\tau_0}^{\tau}\lambda^{-1}(u)\,du]}{\xi^{\frac{1}{4}}}x_0(\frac{\lambda^2(\tau)}{\lambda^{2}(\tau_0)}\xi)\,d\xi
\end{align*}
Now introduce the new variable 
\[
\tilde{\xi}: = \frac{\lambda^2(\tau)}{\lambda^{2}(\tau_0)}\xi,
\]
in terms of which the integral becomes 
\[
\int_0^\infty \frac{\lambda(\tau)}{\lambda(\tau_0)}\rho^{\frac{1}{2}}(\tilde{\xi})\frac{\cos[\lambda(\tau_0)\tilde{\xi}^{\frac{1}{2}}\int_{\tau_0}^{\tau}\lambda^{-1}(u)\,du]}{\tilde{\xi}^{\frac{1}{4}}}x_0(\tilde{\xi})\,d\xi
\]
Then write
\begin{equation}\label{eq:simplebutimportant}
\lambda(\tau_0)\int_{\tau_0}^{\tau}\lambda^{-1}(u)\,du = \nu\tau_0 - \nu\tau_0^{1+\nu^{-1}}\tau^{-\nu^{-1}},
\end{equation}
which gives
\begin{align*}
&\int_0^\infty \frac{\lambda(\tau)}{\lambda(\tau_0)}\rho^{\frac{1}{2}}(\tilde{\xi})\frac{\cos[\lambda(\tau_0)\tilde{\xi}^{\frac{1}{2}}\int_{\tau_0}^{\tau}\lambda^{-1}(u)\,du]}{\tilde{\xi}^{\frac{1}{4}}}x_0(\tilde{\xi})\,d\xi\\
& = \frac{\lambda(\tau)}{\lambda(\tau_0)}\int_0^\infty\frac{\rho^{\frac{1}{2}}(\tilde{\xi})x_0(\tilde{\xi})}{\tilde{\xi}^{\frac{1}{4}}}\cos[\nu\tau_0\tilde{\xi}^{\frac{1}{2}}]\cos[\nu\tau_0^{1+\nu^{-1}}\tau^{-\nu^{-1}}\tilde{\xi}^{\frac{1}{2}}]\,d\tilde{\xi}\\
& +  \frac{\lambda(\tau)}{\lambda(\tau_0)}\int_0^\infty\frac{\rho^{\frac{1}{2}}(\tilde{\xi})x_0(\tilde{\xi})}{\tilde{\xi}^{\frac{1}{4}}}\sin[\nu\tau_0\tilde{\xi}^{\frac{1}{2}}]\sin[\nu\tau_0^{1+\nu^{-1}}\tau^{-\nu^{-1}}\tilde{\xi}^{\frac{1}{2}}]\,d\tilde{\xi}\\
\end{align*}
Under the first of the assumed vanishing conditions \eqref{eq:vanishing}, we can bound the sum of these terms in absolute value by 
\begin{align*}
&\frac{\lambda(\tau)}{\lambda(\tau_0)}\big|\int_0^\infty\frac{\rho^{\frac{1}{2}}(\tilde{\xi})x_0(\tilde{\xi})}{\tilde{\xi}^{\frac{1}{4}}}\cos[\nu\tau_0\tilde{\xi}^{\frac{1}{2}}]\big(\cos[\nu\tau_0^{1+\nu^{-1}}\tau^{-\nu^{-1}}\tilde{\xi}^{\frac{1}{2}}] - 1\big)\,d\tilde{\xi}\big|\\
& +  \frac{\lambda(\tau)}{\lambda(\tau_0)}\big|\int_0^\infty\frac{\rho^{\frac{1}{2}}(\tilde{\xi})x_0(\tilde{\xi})}{\tilde{\xi}^{\frac{1}{4}}}\sin[\nu\tau_0\tilde{\xi}^{\frac{1}{2}}]\sin[\nu\tau_0^{1+\nu^{-1}}\tau^{-\nu^{-1}}\tilde{\xi}^{\frac{1}{2}}]\,d\tilde{\xi}\big|\\
&\lesssim 
\tau\cdot\big\|\xi^{\frac12-}\langle\xi\rangle^{\frac{1}{2}++}x_0\big\|_{L^2_{d\xi}}
\end{align*}
This establishes the desired bound for the contribution of the first term in \eqref{eq:linhomparam1}.
\\

Next, we turn to the contribution from $x_1$, i. e. the second term in \eqref{eq:linhomparam1}, which follows exactly the same scheme: 
\\

{\it(3). Contribution of $x_1$. Large frequencies: $\xi>1$. } In fact, arguing as in {\it{(1)}}, we get 
\begin{align*}
&\lesssim \int_1^\infty\frac{\lambda^2(\tau)}{\lambda^2(\tau_0)}|\frac{\phi(R, \xi)}{R}||\frac{\sin[\lambda(\tau)\xi^{\frac{1}{2}}\int_{\tau_0}^{\tau}\lambda^{-1}(u)\,du]}{\xi^{\frac12}}||x_1|(\xi\frac{\lambda^2(\tau)}{\lambda^{2}(\tau_0)})\xi^{\frac{1}{2}}\,d\xi\\
&\lesssim \big\|\langle\xi\rangle^{\frac12+}x_1(\cdot)\big\|_{L^2_{d\xi}}.
\end{align*}
We can then reduce to the contribution of $x_1$ in the low- or intermediate frequency regime: 
\\

{\it(4). Contribution of $x_1$. Frequencies: $\xi<1$. } Use \eqref{eq:Rxisplit}, where we substitute 
\[
y(\xi) =  \rho(\xi)\frac{\lambda^{\frac{3}{2}}(\tau)}{\lambda^{\frac{3}{2}}(\tau_0)}\frac{\rho^{\frac{1}{2}}(\frac{\lambda^{2}(\tau)}{\lambda^{2}(\tau_0)}\xi)}{\rho^{\frac{1}{2}}(\xi)}\frac{\sin[\lambda(\tau)\xi^{\frac{1}{2}}\int_{\tau_0}^{\tau}\lambda^{-1}(u)\,du]}{\xi^{\frac{1}{2}}}x_1(\frac{\lambda^2(\tau)}{\lambda^{2}(\tau_0)}\xi)
\]
We get 
\begin{align*}
&\big|R^{-1}\int_0^{\min\{R^{-2},1\}}O(R^2\xi)y(\xi)\,d\xi\big|\lesssim \int_0^{\min\{R^{-2},1\}}\xi^{\frac{1}{2}}|y(\xi)|\,d\xi,
\end{align*}
and we can further bound this by 
\begin{align*}
&\int_0^{\min\{R^{-2},1\}}\xi^{\frac{1}{2}}|y(\xi)|\,d\xi\\&\lesssim \int_0^{\min\{R^{-2},1\}}\frac{\lambda^{\frac{3}{2}}(\tau)}{\lambda^{\frac{3}{2}}(\tau_0)}\frac{\rho^{\frac{1}{2}}(\frac{\lambda^{2}(\tau)}{\lambda^{2}(\tau_0)}\xi)}{\rho^{\frac{1}{2}}(\xi)}\big|\sin[\lambda(\tau)\xi^{\frac{1}{2}}\int_{\tau_0}^{\tau}\lambda^{-1}(u)\,du]\big|\big|x_1(\frac{\lambda^2(\tau)}{\lambda^{2}(\tau_0)}\xi)\big|\rho(\xi)\,d\xi,
\end{align*}
and further splitting into the cases $\xi\frac{\lambda^2(\tau)}{\lambda^2(\tau_0)}\lessgtr 1$ as in {\it{(2)}}, we can bound this term by 
\[
\lesssim \big\|\langle\xi\rangle^{\frac12++}\xi^{0-}x_1\big\|_{L^2_{d\xi}}.
\]

Similarly, in the range $1>\xi>R^{-2}$, we get 
\begin{align*}
\big|R^{-1}\int_{\min\{R^{-2},1\}}^1\phi(R, \xi) y(\xi)\,d\xi\big|\lesssim \int_{\min\{R^{-2},1\}}^1\xi^{\frac{1}{2}}|\phi(R, \xi)||y(\xi)|\,d\xi,
\end{align*}
and we recover the same bound as for the immediately preceding term.
\\
It then remains to deal with the most delicate term, which is the multiple of the zero mode, given above by 
\begin{align*}
&\phi(R, 0)\int_0^{\min\{R^{-2},1\}} y(\xi)\,d\xi\\&  = \phi(R, 0)\int_0^{\min\{R^{-2},1\}} \frac{\lambda^{\frac{3}{2}}(\tau)}{\lambda^{\frac{3}{2}}(\tau_0)}\frac{\rho^{\frac{1}{2}}(\frac{\lambda^{2}(\tau)}{\lambda^{2}(\tau_0)}\xi)}{\rho^{\frac{1}{2}}(\xi)}\big|\frac{\sin[\lambda(\tau)\xi^{\frac{1}{2}}\int_{\tau_0}^{\tau}\lambda^{-1}(u)\,du]}{\xi^{\frac{1}{2}}}\big|\\&\hspace{7.5cm}\cdot\big|x_1(\frac{\lambda^2(\tau)}{\lambda^{2}(\tau_0)}\xi)\big|\rho(\xi)\,d\xi.
\end{align*}
Recalling \eqref{eq:rhoasymptotics}, the contribution of the error term $ \frac{1}{3\pi}\xi^{-\frac{1}{2}}\cdot O(\xi^{\frac{1}{2}})$ is bounded by 
\begin{align*}
&\int_0^1 \frac{\lambda^{\frac{3}{2}}(\tau)}{\lambda^{\frac{3}{2}}(\tau_0)}\frac{\rho^{\frac{1}{2}}(\frac{\lambda^{2}(\tau)}{\lambda^{2}(\tau_0)}\xi)}{\rho^{\frac{1}{2}}(\xi)}\big|\sin[\lambda(\tau)\xi^{\frac{1}{2}}\int_{\tau_0}^{\tau}\lambda^{-1}(u)\,du]\big|\big|x_1(\frac{\lambda^2(\tau)}{\lambda^{2}(\tau_0)}\xi)\big|\big|\rho(\xi)\big|\,d\xi\\
&\lesssim \big\|\langle\xi\rangle^{\frac{1}{2}+}x_1\big\|_{L^2_{d\xi}}, 
\end{align*}
similarly to case {\it{(2)}} for the contribution by $x_0$. We can thus replace the factors $\rho^{-\frac{1}{2}}(\xi), \rho(\xi)$ by $\xi^{\frac{1}{4}}, \xi^{-\frac{1}{2}}$, respectively, which reduces things to the expression
\[
\phi(R, 0)\int_0^{\min\{R^{-2},1\}} \frac{\lambda^{\frac{3}{2}}(\tau)}{\lambda^{\frac{3}{2}}(\tau_0)}\rho^{\frac{1}{2}}(\frac{\lambda^{2}(\tau)}{\lambda^{2}(\tau_0)}\xi)\frac{\sin[\lambda(\tau)\xi^{\frac{1}{2}}\int_{\tau_0}^{\tau}\lambda^{-1}(u)\,du]}{\xi^{\frac{3}{4}}}x_1(\frac{\lambda^2(\tau)}{\lambda^{2}(\tau_0)}\xi)\,d\xi
\]
Then again using $\frac{\phi(R, 0)}{R}\xi^{-\frac12}\lesssim \min\{1, \xi^{-\frac12}\}$ for $\xi\geq \min\{R^{-2},1\}$, we get  
\begin{align*}
&\big|\frac{\phi(R, 0)}{R}\int_{\min\{R^{-2},1\}}^\infty \frac{\lambda^{\frac{3}{2}}(\tau)}{\lambda^{\frac{3}{2}}(\tau_0)}\rho^{\frac{1}{2}}(\frac{\lambda^{2}(\tau)}{\lambda^{2}(\tau_0)}\xi)\frac{\sin[\lambda(\tau)\xi^{\frac{1}{2}}\int_{\tau_0}^{\tau}\lambda^{-1}(u)\,du]}{\xi^{\frac{3}{4}}}x_1(\frac{\lambda^2(\tau)}{\lambda^{2}(\tau_0)}\xi)\,d\xi\big|\\
&\lesssim \big|\int_{\min\{R^{-2},1\}}^\infty \frac{\lambda^{\frac{3}{2}}(\tau)}{\lambda^{\frac{3}{2}}(\tau_0)}\rho^{\frac{1}{2}}(\frac{\lambda^{2}(\tau)}{\lambda^{2}(\tau_0)}\xi)\frac{\sin[\lambda(\tau)\xi^{\frac{1}{2}}\int_{\tau_0}^{\tau}\lambda^{-1}(u)\,du]}{\xi^{\frac{1}{4}}}x_1(\frac{\lambda^2(\tau)}{\lambda^{2}(\tau_0)}\xi)\,d\xi\big|,\\
\end{align*}
which in turn is bounded by $\lesssim \big\|\langle\xi\rangle^{\frac12+}\xi^{0-}x_1\big\|_{L^2_{d\xi}}$. It follows that it suffices to bound 
\begin{align*}
\phi(R, 0)\int_0^{\infty} \frac{\lambda^{\frac{3}{2}}(\tau)}{\lambda^{\frac{3}{2}}(\tau_0)}\rho^{\frac{1}{2}}(\frac{\lambda^{2}(\tau)}{\lambda^{2}(\tau_0)}\xi)\frac{\sin[\lambda(\tau)\xi^{\frac{1}{2}}\int_{\tau_0}^{\tau}\lambda^{-1}(u)\,du]}{\xi^{\frac{3}{4}}}x_1(\frac{\lambda^2(\tau)}{\lambda^{2}(\tau_0)}\xi)\,d\xi
\end{align*}
Now as before using the variable 
\[
\tilde{\xi}: = \frac{\lambda^2(\tau)}{\lambda^{2}(\tau_0)}\xi,
\]
the integral becomes 
\[
\int_0^\infty \frac{\lambda(\tau)}{\lambda(\tau_0)}\rho^{\frac{1}{2}}(\tilde{\xi})\frac{\sin[\lambda(\tau_0)\tilde{\xi}^{\frac{1}{2}}\int_{\tau_0}^{\tau}\lambda^{-1}(u)\,du]}{\tilde{\xi}^{\frac{3}{4}}}x_1(\tilde{\xi})\,d\xi,
\]
and using \eqref{eq:simplebutimportant} we obtain 
\begin{align*}
&\int_0^\infty \frac{\lambda(\tau)}{\lambda(\tau_0)}\rho^{\frac{1}{2}}(\tilde{\xi})\frac{\sin[\lambda(\tau_0)\tilde{\xi}^{\frac{1}{2}}\int_{\tau_0}^{\tau}\lambda^{-1}(u)\,du]}{\tilde{\xi}^{\frac{3}{4}}}x_1(\tilde{\xi})\,d\xi\\
& = \frac{\lambda(\tau)}{\lambda(\tau_0)}\int_0^\infty\frac{\rho^{\frac{1}{2}}(\tilde{\xi})x_1(\tilde{\xi})}{\tilde{\xi}^{\frac{3}{4}}}\sin[\nu\tau_0\tilde{\xi}^{\frac{1}{2}}]\cos[\nu\tau_0^{1+\nu^{-1}}\tau^{-\nu^{-1}}\tilde{\xi}^{\frac{1}{2}}]\,d\tilde{\xi}\\
& -  \frac{\lambda(\tau)}{\lambda(\tau_0)}\int_0^\infty\frac{\rho^{\frac{1}{2}}(\tilde{\xi})x_1(\tilde{\xi})}{\tilde{\xi}^{\frac{3}{4}}}\cos[\nu\tau_0\tilde{\xi}^{\frac{1}{2}}]\sin[\nu\tau_0^{1+\nu^{-1}}\tau^{-\nu^{-1}}\tilde{\xi}^{\frac{1}{2}}]\,d\tilde{\xi}\\
\end{align*}
Under the second of the assumed vanishing conditions \eqref{eq:vanishing}, we can bound these terms in absolute value by 
\begin{align*}
&\frac{\lambda(\tau)}{\lambda(\tau_0)}\big|\int_0^\infty\frac{\rho^{\frac{1}{2}}(\tilde{\xi})x_1(\tilde{\xi})}{\tilde{\xi}^{\frac{3}{4}}}\sin[\nu\tau_0\tilde{\xi}^{\frac{1}{2}}]\big(\cos[\nu\tau_0^{1+\nu^{-1}}\tau^{-\nu^{-1}}\tilde{\xi}^{\frac{1}{2}}] - 1\big)\,d\tilde{\xi}\big|\\
& +  \frac{\lambda(\tau)}{\lambda(\tau_0)}\big|\int_0^\infty\frac{\rho^{\frac{1}{2}}(\tilde{\xi})x_1(\tilde{\xi})}{\tilde{\xi}^{\frac{3}{4}}}\cos[\nu\tau_0\tilde{\xi}^{\frac{1}{2}}]\sin[\nu\tau_0^{1+\nu^{-1}}\tau^{-\nu^{-1}}\tilde{\xi}^{\frac{1}{2}}]\,d\tilde{\xi}\big|\\
&\lesssim 
\tau\cdot\big\|\langle\xi\rangle^{\frac{1}{2}+}x_1\big\|_{L^2_{d\xi}}
\end{align*}
This completes the proof of the proposition. 

\end{proof}

The preceding $L^\infty$-type bound needs to be complemented with an energy type bound for the parametrix \eqref{eq:linhomparam1} in order to be able to recover the $S$-norm bounds later on. We shall rely on the following simple 
\begin{prop}\label{prop:energybound1} For $x(\tau, \xi)$ defined as before, we have 
\[
\big\|\xi^{\frac{3}{4}+}x(\tau, \xi)\big\|_{L^2_{d\xi}}\lesssim \big\|(\xi^{\frac12-}\langle\xi\rangle^{\frac12++}x_0, \langle\xi\rangle^{\frac{1}{2}+}x_1)\big\|_{L^2_{d\xi}} 
\]
\end{prop}
\begin{proof} 
For the contribution of $x_0$ in \eqref{eq:linhomparam1}, using the terminology of the preceding proof, we have in case {\it{(2)}} under the additional restriction $\frac{\lambda^2(\tau)}{\lambda^2(\tau_0)}\xi<1$ the bound
\begin{align*}
&\big\|\xi^{\frac{3}{4}+}\frac{\lambda^{\frac{5}{2}}(\tau)}{\lambda^{\frac{5}{2}}(\tau_0)}\frac{\rho^{\frac{1}{2}}(\frac{\lambda^{2}(\tau)}{\lambda^{2}(\tau_0)}\xi)}{\rho^{\frac{1}{2}}(\xi)}\cos[\lambda(\tau)\xi^{\frac{1}{2}}\int_{\tau_0}^{\tau}\lambda^{-1}(u)\,du]x_0(\frac{\lambda^2(\tau)}{\lambda^{2}(\tau_0)}\xi)\big\|_{L^2_{d\xi}(\xi<\frac{\lambda^2(\tau_0)}{\lambda^2(\tau)})}\\
&\lesssim \big\|\xi^{\frac{3}{4}+}\frac{\lambda^2(\tau)}{\lambda^2(\tau_0)}\cos[\lambda(\tau)\xi^{\frac{1}{2}}\int_{\tau_0}^{\tau}\lambda^{-1}(u)\,du]x_0(\frac{\lambda^2(\tau)}{\lambda^{2}(\tau_0)}\xi)\big\|_{L^2_{d\xi}(\xi<\frac{\lambda^2(\tau_0)}{\lambda^2(\tau)})}\\
&\lesssim \big\|\xi^{\frac12}x_0\big\|_{L^2_{d\xi}}. 
\end{align*}
Still in case {\it{(2)}} and assuming $\frac{\lambda^2(\tau)}{\lambda^2(\tau_0)}\xi>1$, we find 
\begin{align*}
&\big\|\xi^{\frac{3}{4}+}\frac{\lambda^{\frac{5}{2}}(\tau)}{\lambda^{\frac{5}{2}}(\tau_0)}\frac{\rho^{\frac{1}{2}}(\frac{\lambda^{2}(\tau)}{\lambda^{2}(\tau_0)}\xi)}{\rho^{\frac{1}{2}}(\xi)}\cos[\lambda(\tau)\xi^{\frac{1}{2}}\int_{\tau_0}^{\tau}\lambda^{-1}(u)\,du]x_0(\frac{\lambda^2(\tau)}{\lambda^{2}(\tau_0)}\xi)\big\|_{L^2_{d\xi}(1>\xi>\frac{\lambda^2(\tau_0)}{\lambda^2(\tau)})}\\
&\lesssim \big\|\xi^{\frac{5}{4}+}\frac{\lambda^3(\tau)}{\lambda^3(\tau_0)}\cos[\lambda(\tau)\xi^{\frac{1}{2}}\int_{\tau_0}^{\tau}\lambda^{-1}(u)\,du]x_0(\frac{\lambda^2(\tau)}{\lambda^{2}(\tau_0)}\xi)\big\|_{L^2_{d\xi}(1>\xi>\frac{\lambda^2(\tau_0)}{\lambda^2(\tau)})}\\
&\lesssim \big\|\langle\xi\rangle x_0\big\|_{L^2_{d\xi}}
\end{align*}
Finally, in case {\it{(1)}}, we have 
\begin{align*}
&\big\|\xi^{\frac{3}{4}+}\frac{\lambda^{\frac{5}{2}}(\tau)}{\lambda^{\frac{5}{2}}(\tau_0)}\frac{\rho^{\frac{1}{2}}(\frac{\lambda^{2}(\tau)}{\lambda^{2}(\tau_0)}\xi)}{\rho^{\frac{1}{2}}(\xi)}\cos[\lambda(\tau)\xi^{\frac{1}{2}}\int_{\tau_0}^{\tau}\lambda^{-1}(u)\,du]x_0(\frac{\lambda^2(\tau)}{\lambda^{2}(\tau_0)}\xi)\big\|_{L^2_{d\xi}(\xi>1)}\\
&\lesssim \big\|\xi^{\frac{3}{4}+}\frac{\lambda^{3}(\tau)}{\lambda^{3}(\tau_0)}\cos[\lambda(\tau)\xi^{\frac{1}{2}}\int_{\tau_0}^{\tau}\lambda^{-1}(u)\,du]x_0(\frac{\lambda^2(\tau)}{\lambda^{2}(\tau_0)}\xi)\big\|_{L^2_{d\xi}(\xi>1)}\\
&\lesssim \big\|\xi x_0\big\|_{L^2_{d\xi}(\xi>1)}
\end{align*}
The contribution of $x_1$ is handled analogously. 
\end{proof}

\section{Setting up the perturbative problem}

Fix  $\nu>0$ which shall be held fixed throughout and will be assumed sufficiently small later on.
We shall now seek to construct solutions of \eqref{eq:Main} of the form 
\[
u(t, x) = u_{\nu}(t, x) + \epsilon(t, x),
\]
where $\epsilon$ is supposed to match initial data at time $t = t_0>0$ on a co-dimension three hyper surface. Our point of departure shall be the following equation for $\til\eps = R\eps$ in terms of the renormalised coordinates $(\tau, R)$ introduced in the preceding section: 
\begin{equation}\label{eq:epseqn1}\begin{split}
&(\partial_{\tau} + \dot{\lambda}\lambda^{-1}R\partial_R)^2 \tilde{\eps} - \beta_{\nu}(\tau)(\partial_{\tau} + \dot{\lambda}\lambda^{-1}R\partial_R)\tilde{\eps} + \calL\tilde{\eps}\\
&=\lambda^{-2}(\tau)RN_{\nu}(\til\eps)+\partial_\tau(\dot{\lambda}\lambda^{-1})\tilde{\eps};
\end{split}\end{equation}
where we have 
\[
RN_{\nu}(\til\eps) = 5(u_{\nu}^4 - u_0^4)\tilde{\eps} + RN(u_{\nu}, \tilde{\eps}),
\]
and as before
\[
RN(u_{\nu}, \tilde{\eps}) = R(u_{\nu}+\frac{\tilde{\eps}}{R})^5 - R u_{\nu}^5 - 5u_{\nu}^4\tilde{\eps}
\]
We shall solve \eqref{eq:epseqn1} on a sufficiently small interval $(0, t_0]$, which in terms of $\tau$ becomes $[\tau_0, \infty)$, say, with $\tau_0\gg 1$. 
As hinted at previously we shall employ Fourier methods to control the solution, and so we shall start with explaining how to pass from the initial data $(\epsilon, \partial_t\epsilon)|_{t = t_0}$ to the Fourier data $(\underline{x}_0, \underline{x}_1)|_{\tau = \tau_0}$ 
which will be used to build the zeroth iterate given in terms of the linear evolution \eqref{eq:linhomparam1}. 
\\

To understand this mapping between data, assume that $\til\eps(\tau, R)$ is of the form 
\begin{equation}\label{eq:Fourierrepr}
 \til\eps(\tau, R) = x_d(\tau)\phi_d(R) + \int_0^\infty x(\tau, \xi)\phi(R, \xi)\rho(\xi)\,d\xi.
 \end{equation}
 with $x(\tau, \xi), x_d(\tau)$ given as in the statement of Lemma~\ref{lem:linhom}. Then arguing as in \cite{KST}, \cite{KS1}, we find 
 \[
 -\frac{R}{\lambda}\epsilon_t = \big(\partial_{\tau} + \beta_{\nu}(\tau)(R\partial_R - 1)\big)\til\eps,\,\beta_{\nu}(\tau) = \frac{\lambda_{\tau}}{\lambda}, 
 \]
 which in terms of the Fourier transform can be expressed as 
\[
-\left(\begin{array}{c}\mathcal{F}\big(\frac{R}{\lambda}\epsilon_t\big)\\ \langle\phi_d, \frac{R}{\lambda}\epsilon_t\rangle\end{array}\right) = \mathcal{D}_{\tau}\underline{x}(\tau,\cdot) + \beta_{\nu}(\tau)\mathcal{K}\underline{x}(\tau, \cdot)
\]
where $\mathcal{F}$ denotes the Fourier transform corresponding to the continuous spectrum, and the operators $\mathcal{D}, \mathcal{K}$ are specified as in \cite{KS1}, see the paragraph after \eqref{eq:Kstructure}. Evaluating the preceding relation at initial time $\tau = \tau_0$ (or equivalently $t = t_0$) and using \eqref{eq:linhomparam1}, we find the first half of transference principle for the initial data 
\begin{equation}\label{eq:datatransference1}
-\mathcal{F}\big(\frac{R}{\lambda}\epsilon_t\big)\big|_{t = t_0} = x_1 + \beta_{\nu}(\tau_0)\mathcal{K}_{cc}x_0 + \beta_{\nu}(\tau_0)\mathcal{K}_{cd}x_{0d}
\end{equation}
\begin{equation}\label{eq:datatransference2}
-\langle\phi_d, \frac{R}{\lambda}\epsilon_t\rangle\big|_{t = t_0} = x_{1d} + \beta_{\nu}(\tau_0)\mathcal{K}_{dd}x_{0d} + \beta_{\nu}(\tau_0)\mathcal{K}_{dc}x_{0}
\end{equation}
These relations get complemented with the immediate relations
\begin{equation}\label{eq:datatransference3}
\mathcal{F}\big(\til\eps\big)\big|_{t = t_0} = x_0,\,\langle\phi_d, \til\eps\rangle\big|_{t = t_0} = x_{0d}. 
\end{equation}
Conversely, if the relations \eqref{eq:datatransference1} - \eqref{eq:datatransference3} are satisfied, where it is assumed that the co-dimension one condition \eqref{eq:cond1}
is satisfied so that we can construct the evolution of $\underline{x}(\tau, \xi)$ as in Lemma~\ref{lem:linhom}, and if we define $\til\eps$ via \eqref{eq:Fourierrepr},then $\epsilon = \frac{\til\eps}{R}$ will have initial data $(\epsilon, \epsilon_t)\big|_{t = t_0}$ in accordance with the left hand side of \eqref{eq:datatransference1} - \eqref{eq:datatransference3}. As these latter relations can be easily solved for the functions $\big(x_{0,1}, x_{0d,1d}\big)$, we shall from now on identify a data set $\big(\epsilon(t_0, \cdot), \epsilon_t(t_0, \cdot)\big)$ with the set of Fourier coefficients $\big(x_{0,1}, x_{0d,1d}\big)$. The co-dimension one condition \eqref{eq:cond1} then implicitly gives a co-dimension one condition for $\big(\epsilon(t_0, \cdot), \epsilon_t(t_0, \cdot)\big)$. The space $\tilde{S}$ is then specified in terms of the variables $\big(x_{0,1}, x_{0d,1d}\big)$ via the norm \eqref{eq:tildeSnorm}. 
\\

Having specified the transition between the physical data and the data on the Fourier side at the linear level, we shall now set up the equation system we attempt to solve in terms of the Fourier variables, in exact analogy to \cite{KS1}. Specifically, we get upon letting 
\begin{equation}\label{eq:Fourierrepr2}
 \til\eps(\tau, R) = x_d(\tau)\phi_d(R) + \int_0^\infty x(\tau, \xi)\phi(R, \xi)\rho(\xi)\,d\xi
 \end{equation}
 and now $x(\tau, \xi)$ is of course no longer defined in terms of the linear evolution \eqref{eq:linhomparam1}
\begin{equation}\label{eq:Fouriersystem}
 (\mathcal{D}_{\tau}^2 + \beta_{\nu}(\tau)\mathcal{D}_{\tau} + \underline{\xi})\underline{x}(\tau, \xi) = \mathcal{R}(\tau, \underline{x}) + \underline{f}(\tau,\xi),
 \end{equation}
 where we have 
 \[
 \underline{f}(\tau, \xi) = \left(\begin{array}{c}\langle \phi_d, \lambda^{-2}(\tau)RN_{\nu}(\epsilon)\rangle\\ \mathcal{F}\big(\lambda^{-2}(\tau)RN_{\nu}(\epsilon) \big)(\xi)\end{array}\right)
 \]
 and the additional linear error term is defined by 
 \begin{equation}\label{eq:Rterms}
\mathcal{R}(\tau, \underline{x})(\xi) = \Big(-4\beta_{\nu}(\tau)\mathcal{K}\mathcal{D}_{\tau}\underline{x} - \beta_{\nu}^2(\tau)(\mathcal{K}^2 + [\mathcal{A}, \mathcal{K}] + \mathcal{K} +  \beta_{\nu}' \beta_{\nu}^{-2}\mathcal{K})\underline{x}\Big)(\xi)
\end{equation}
 with $\beta_{\nu}(\tau) = \frac{\dot{\lambda}(\tau)}{\lambda(\tau)}$,
 \[
 \mathcal{D}_{\tau} = \partial_{\tau} + \beta_{\nu}(\tau)\mathcal{A},\quad \mathcal{A} = \left(\begin{array}{cc}0&0\\0&\mathcal{A}_c\end{array}\right)
 \]
 and we have 
 \[
 \mathcal{A}_c = -2\xi\partial_{\xi} - \Big (\frac{5}{2}  + \frac{\rho'(\xi)\xi}{\rho(\xi)} \Big)
 \]
 Finally, we observe that the ``transference operator'' $\mathcal{K}$ is given by the following type of expression 
 \begin{equation}\label{eq:Kstructure}
 \mathcal{K} = \left(\begin{array}{cc}\mathcal{K}_{dd}&\mathcal{K}_{dc}\\
 \mathcal{K}_{cd}&\mathcal{K}_{cc}\end{array}\right)
 \end{equation}
We shall from now on work in terms of the Fourier coefficients $\underline{x}(\tau, \xi)$, and use the relation \eqref{eq:Fourierrepr2} to define the nonlinear expressions involving $\epsilon$. 
\\

The proof of Theorem~\ref{thm:Main} shall then follow from the following more technical and precise version 
\begin{thm}\label{thm:MainTechnical} Let $\big(\underline{x}_0, \underline{x}_1\big) = \big(\left(\begin{array}{c}x_{0d}\\x_0\end{array}\right), \left(\begin{array}{c}x_{1d}\\x_1\end{array}\right)\big)$ with $\big(x_0, x_1\big)\in \tilde{S}$, see \eqref{eq:tildeSnorm}, and moreover 
\[
\big\|\big(x_0, x_1\big)\big\|_{\tilde{S}} + |x_{0d}|<\delta_2
\]
sufficiently small. Also, assume the compatibility relation \eqref{eq:cond1} and the basic vanishing conditions \eqref{eq:vanishing} hold. Then picking $\tau_0$ sufficiently large (or equivalently $t_0$ sufficiently small), there exist corrections $\big(\underline{\triangle x}_0, \underline{\triangle x}_1\big) = \big(\left(\begin{array}{c}\triangle x_{0d}\\ \triangle x_{0}\end{array}\right), \left(\begin{array}{c}\triangle x_{1d}\\ \triangle x_{1}\end{array}\right)\big)$ with 
\[
\big\|\big(\triangle x_0, \triangle x_1\big)\big\|_{\tilde{S}} + |\triangle x_{0d}| +  |\triangle x_{1d}|\ll \big\|\big(x_0, x_1\big)\big\|_{\tilde{S}} + |x_{0d}|
\]
and such that $\big(\underline{\triangle x}_0, \underline{\triangle x}_1\big)$ depend in Lipschitz continuous fashion on $\big(\underline{x}_0, \underline{x}_1\big) $ with respect to the $\tilde{S}$-norm, with Lipschitz constant $\ll 1$, such that determining the initial data $\big(\epsilon, \epsilon_t\big)\big|_{t = t_0}$ via \eqref{eq:datatransference1} - \eqref{eq:datatransference3} with $\big(\underline{x}_0, \underline{x}_1\big)$ replaced by $\big(\underline{x}_0 + \underline{\triangle x}_0, \underline{x}_1 + \underline{\triangle x}_1\big)$, the solution $u(t, x)$ with initial data 
\[
u_{\nu}[t_0] + \big(\epsilon, \epsilon_t\big)\big|_{t = t_0}
\]
exists on $(0, t_0]$ and is a type II blow up solution satisfying the conclusion of Theorem~\ref{thm:Main}.  
\end{thm}

In the rest of the paper, we shall prove Theorem~\ref{thm:MainTechnical}. 

\section{Nonlinear estimates}

The proof of Theorem~\ref{thm:MainTechnical} proceeds via a somewhat involved iterative procedure, constructing the corrections $\big(\underline{\triangle x}_0, \underline{\triangle x}_1\big)$ via an infinite sequence of iterative steps. The main complication shall be caused by the linear source terms $\mathcal{R}(\tau, \underline{x}) $ given by \eqref{eq:Rterms}, as it is not a priori clear how one should gain smallness for these. Moreover, at each iterative step we will have to adjust the initial data a bit, basically in order to keep satisfying analogues of the basic vanishing conditions \eqref{eq:vanishing}. Controlling the nonlinear source terms $\underline{f}(\tau, \xi)$ shall be relatively straightforward, by contrast, and this is what we do in the current section. 
\\

In the sequel, it will become clear that the corrections we add in the iterative steps will roughly have the structure of the sum of two terms, whose continuous part of the Fourier transform are either of the form \eqref{eq:linhomparam1} or else given by 
$x(\tau, \xi)$ satisfying bounds of the form (here recall the definition \eqref{eq:Snorm}; thus the norm $\big\|\cdot\big\|_{S}$ is stronger than the norm $\big\|\cdot\big\|_{\tilde{S}}$)
\begin{equation}\label{eq:crudeSbound1}
\big\|x(\tau, \cdot)\big\|_{S_1}\lesssim \big(\frac{\tau}{\tau_0}\big)^{\kappa}\big[ \big\|\big(x_0, x_1\big)\big\|_{\tilde{S}} + |x_{0d}|\big]
\end{equation}
where $0<\kappa\lesssim \delta_0$, whence can be made arbitrarily small by picking $\delta_0$ small enough. Moreover, the discrete part $x_d(\tau)$ shall be bounded by $ \big\|\big(x_0, x_1\big)\big\|_{\tilde{S}} + |x_{0d}|$. 
Accordingly we formulate 
\begin{prop}\label{prop:nonlinbounds} Assume that function $\til\eps(\tau, R)$ is given by the Fourier representation 
\[
\til\eps(\tau, R) = x_d(\tau)\phi_d(R) + \int_0^\infty x(\tau, \xi)\phi(R, \xi)\rho(\xi)\,d\xi,
\]
where $|x_d(\tau)|\lesssim \alpha\ll 1$ and either (i) $x(\tau, \xi)$ is given by \eqref{eq:linhomparam1} with data $\big(x_0, x_1\big)$ satisfying the vanishing condition \eqref{eq:vanishing}, and with 
\[
\big\|\big(x_0, x_1\big)\big\|_{\tilde{S}}\leq \alpha,
\]
or else (ii) we have the bound 
\[
\big\|x(\tau, \cdot)\big\|_{S_1}\lesssim \big(\frac{\tau}{\tau_0}\big)^{\kappa}\alpha
\]
for $0<\kappa\ll1 $ sufficiently small. Then, if $0<\nu\ll 1$ is sufficiently small, in situation (ii) we have 
\[
\big\|\lambda^{-2}(\tau)RN_{\nu}(\til\eps)(\tau, \cdot)\big\|_{(H^{1+}_{dR}\cap L^1_{dR})(R\lesssim \tau)}\lesssim \alpha\cdot \tau^{-(2-)}
\]
Here we use the convention introduced after \eqref{eq:Snorm} that $1+ = 1+\delta_0$ with $\delta_0>0$ sufficiently small fixed throughout, while $2- = 2-O(\kappa)$. On the other hand, in situation (i), there is a splitting 
\[
\lambda^{-2}(\tau)RN_{\nu}(\til\eps)(\tau, \cdot) = E_1(\tau, \cdot) + E_2(\tau, \cdot), 
\]
such that
\[
\big\| E_1(\tau, \cdot)\big\|_{(H^{1+}_{dR}\cap L^1_{dR})(R\lesssim \tau)}\lesssim \alpha\cdot \tau^{-(2-)},
\]
while the term $E_2$ has worse decay but a more precise structure: 
\[
E_2(\tau, R) = \tau^{-2}\tilde{g}(\tau)\cdot \langle R\rangle^{-3}g(R,a)
\]
where the function $\tilde{g}(R, a)$ is in the space $IS^2(R, \mathcal{Q})$ defined precisely in \cite{KST}, and the function $g(\tau)$ (which implicitly also depends on $\tau_0$, $x_{0,1}$ and $\nu$) satisfies 
\[
\tau^{-1}|\tilde{g}(\tau)| + |\partial_{\tau}\tilde{g}(\tau)|\lesssim \alpha. 
\]

\end{prop}
\begin{proof} We shall treat the situations (i) and (ii) separately. Also, we refer to \eqref{eq:epseqn1} and sequel for the precise structure of $RN_{\nu}(\til\eps)$. 
\\

{\it{Proof in case (i)}}. We need to bound the following terms 
\begin{equation}\label{eq:list1}
 \lambda^{-2}(\tau)5(u_{\nu}^4 - u_0^4)\tilde{\eps},\, \lambda^{-2}(\tau)Ru_{\nu}^{5-j}\big(\frac{\til\eps}{R}\big)^j,\,2\leq j\leq 5. 
 \end{equation}
 Recall from \eqref{eq:u2k-1bound} as well as the bounds established on the specific solution of \eqref{eq:epseqn} vanishing at $\tau = +\infty$ in a suitable sense in \cite{KST}, \cite{KS1} we get 
 \begin{equation}\label{eq:unu-u0}
 u_{\nu} - u_0 = \frac{\lambda^{\frac12}}{(\lambda t)^2}g(R, a) +  O(\frac{\lambda^{\frac12}}{(\lambda t)^4}R\log R),\,g(R, a)\in IS^2(R, \mathcal{Q}). 
\end{equation}
Further, referring to the proof of Proposition~\ref{prop:lingrowthcond} and in particular cases {\it{(2)}} and {\it{(4)}} there, we can write 
\[
\til\eps(\tau, R) = \til\eps_1(\tau, R) + \til\eps_2(\tau, R), 
\]
where we have 
\begin{align*}
\til\eps_1(\tau, R) &= \phi(R, 0)\int_0^\infty \frac{\lambda(\tau)}{\lambda(\tau_0)}\rho^{\frac{1}{2}}(\tilde{\xi})\frac{\sin[\lambda(\tau_0)\tilde{\xi}^{\frac{1}{2}}\int_{\tau_0}^{\tau}\lambda^{-1}(u)\,du]}{\tilde{\xi}^{\frac{3}{4}}}x_1(\tilde{\xi})\,d\xi\\
& + \phi(R, 0)\int_0^\infty \frac{\lambda(\tau)}{\lambda(\tau_0)}\rho^{\frac{1}{2}}(\tilde{\xi})\frac{\cos[\lambda(\tau_0)\tilde{\xi}^{\frac{1}{2}}\int_{\tau_0}^{\tau}\lambda^{-1}(u)\,du]}{\tilde{\xi}^{\frac{1}{4}}}x_0(\tilde{\xi})\,d\xi\\
& =:\phi(R,0)\tilde{g}(\tau).
\end{align*}
Then using the observations in {\it{(2)}}, {\it{(4)}} of the proof of Proposition~\ref{prop:lingrowthcond}, we see that 
\[
\tau^{-1}|\tilde{g}(\tau)| + |\partial_{\tau}\tilde{g}(\tau)|\lesssim \big\|x_0\big\|_{\tilde{S}_1} + \big\|x_1\big\|_{S_2}\leq \alpha. 
\]
Defining $\til\eps_2(\tau, R) = \til\eps(\tau, R) -\til\eps_1(\tau, R)$, we see that the proof of Proposition~\ref{prop:lingrowthcond} gives 
\[
\big\|R^{-1}\til\eps_2\big\|_{L^\infty}\lesssim \alpha.
\]
Then write 
\begin{align*}
 \lambda^{-2}(\tau)5(u_{\nu}^4 - u_0^4)\til\eps = 20 \lambda^{-2}(\tau)(u_{\nu} - u_0)u_0^3\til\eps_1 + F_1(\tau, R), 
\end{align*}
where the second term on the right is defined via this relation. Then using \eqref{eq:unu-u0} as well as the preceding observations we get the schematic relation
\[
20 \lambda^{-2}(\tau)(u_{\nu} - u_0)u_0^3\til\eps_1 = \langle R\rangle^{-3} g(R, a)\tau^{-2}\tilde{g}(\tau) + F_2(\tau, R)
\]
where the first term on the right has the properties of the term $E_2$ of the proposition, while the second term satisfies 
\[
\big\|F_2(\tau, R)\big\|_{(H^{1+}_{dR}\cap L^1)(R\lesssim \tau)}\lesssim \alpha\tau^{-3}. 
\]
As for the term $F_1(\tau, R)$ further above, it is schematically of the form 
\[
 F_1(\tau, R) = 20 \lambda^{-2}(\tau)(u_{\nu} - u_0)u_0^3\til\eps_2 + \sum_{2\leq j\leq 4}\lambda^{-2}(\tau)(u_{\nu} - u_0)^ju_0^{4-j}\til\eps
 \]
 Then using the fractional Leibniz rule as well as Proposition~\ref{prop:energybound1}, we get 
 \begin{align*}
&\big\|20 \lambda^{-2}(\tau)(u_{\nu} - u_0)u_0^3\til\eps_2\big\|_{(H^{1+}_{dR}\cap L^1)(R\lesssim \tau)}\\&\lesssim\big\|\lambda^{-\frac32}(\tau)(u_{\nu} - u_0)u_0^2\big\|_{(H^{1+}_{dR}\cap L^1)(R\lesssim \tau)}\big\|R^{-1}\til\eps_2(\tau, \cdot)\big\|_{L^\infty}\\
 & + \big\|\lambda^{-\frac32}(\tau)(u_{\nu} - u_0)u_0^2\big\|_{L^\infty(R\lesssim \tau)}\big\|R^{-1}\til\eps_2(\tau, \cdot)\big\|_{(\dot{H}^{1+}_{dR})(R\lesssim \tau)}\\
 &\lesssim \alpha\log\tau\tau^{-2} \lesssim  \alpha\tau^{-(2-)}. 
 \end{align*}
 Further, again taking advantage of \eqref{eq:unu-u0} as well as the fractional Leibniz rule, we get 
 \begin{align*}
&\big\|\sum_{2\leq j\leq 4}\lambda^{-2}(\tau)(u_{\nu} - u_0)^ju_0^{4-j}\til\eps\big\|_{(H^{1+}_{dR}\cap L^1)(R\lesssim \tau)}\\
&\lesssim \sum_{2\leq j\leq 4}\big\|\lambda^{-2}(\tau)(u_{\nu} - u_0)^ju_0^{4-j}\big\|_{(H^{1+}_{dR}\cap L^1)(R\lesssim \tau)}\big\|\til\eps(\tau, \cdot)\big\|_{L^\infty}\\
& + \sum_{2\leq j\leq 4}\big\|\lambda^{-2}(\tau)(u_{\nu} - u_0)^ju_0^{4-j}\big\|_{L^\infty}\big\|\til\eps(\tau, \cdot)\big\|_{(H^{1+}_{dR}\cap L^1)(R\lesssim \tau)}\\
&\lesssim \alpha(\log\tau)^2\tau^{-2}\lesssim \alpha\tau^{-(2-)}. 
 \end{align*}
 This completes bounding the term $ \lambda^{-2}(\tau)5(u_{\nu}^4 - u_0^4)\til\eps$ in case (i). 
 \\
 
 Next, consider the remaining terms $\lambda^{-2}(\tau)Ru_{\nu}^{5-j}\big(\frac{\til\eps}{R}\big)^j,\,2\leq j\leq 5$. We shall treat the case $j = 5$ here, the other cases being more of the same. Thus we get by the fractional Leibniz rule
 \begin{align*}
 \big\|\lambda^{-2}(\tau)\big(\frac{\til\eps}{R}\big)^4\til\eps\big\|_{(H^{1+}_{dR}\cap L^1_{dR})(R\lesssim \tau)}\lesssim &\lambda^{-2}(\tau)\big\|R^{0-}\big(\frac{\til\eps}{R}\big)^5\big\|_{L^2_{dR}(R\lesssim \tau)}\\
 & + \lambda^{-2}(\tau)\big\|\big(\frac{\til\eps}{R}\big)^4\big\|_{L^\infty\cap L^2}\big\|\langle\nabla\rangle^{1+}\til\eps\big\|_{L^2_{dR}(R\lesssim \tau)}\\
 &\lesssim \lambda^{-2}(\tau)\tau^6\cdot\alpha\lesssim \tau^{-2}\alpha
 \end{align*}
 provided $2(1+\nu^{-1})-6\geq 2$, or equivalently $\nu\leq \frac{1}{3}$.

 {\it{Proof in case (ii)}}.\\ 
 The estimate in the case (ii) is quite similar, in that one simply uses 
 \begin{align*}
 \big\|x_d(\tau)\phi_d(R) + \int_0^\infty x(\tau, \xi)\phi(R, \xi)\rho(\xi)\,d\xi\big\|_{L^\infty_{dR}}
 \lesssim |x_d(\tau)| + \big\|x(\tau, \cdot)\big\|_{S_1}, 
 \end{align*}
as well as the energy type bound 
\begin{align*} 
 \big\|x_d(\tau)\phi_d(R) + \int_0^\infty x(\tau, \xi)\phi(R, \xi)\rho(\xi)\,d\xi\big\|_{H^{1+}_{dR}}\lesssim |x_d(\tau)| + \big\|x(\tau, \cdot)\big\|_{S_1}.
\end{align*}
We omit the simple details. 
\end{proof}

\section{Outline of the iterative scheme}

Our point of departure shall be a data quadruple $\big(\underline{x}_0, \underline{x}_1\big)$ as specified in the statement of Theorem~\ref{thm:MainTechnical}. In particular, we introduce the {\it{zeroth iterate}} 
\[
\underline{x}^{(0)}(\tau, \xi) = \left(\begin{array}{c} x_d^{(0)}(\tau)\\x^{(0)}(\tau, \xi)\end{array}\right),
\]
where the right hand components are defined as the linear evolution detailed in Lemma~\ref{lem:linhom}. On the other hand, we define the first and higher iterates $\underline{x}^{(j)}(\tau, \xi),\,j\geq 1$, by using the iterative step
\begin{equation}\label{eq:iterationstep}
 (\mathcal{D}_{\tau}^2 + \beta_{\nu}(\tau)\mathcal{D}_{\tau} + \underline{\xi})\underline{x}^{(j)}(\tau, \xi) = \mathcal{R}(\tau, \underline{x}^{(j-1)}) + \underline{f}^{(j-1)}(\tau,\xi),\,j\geq 1, 
 \end{equation}
 where we set 
  \[
 \underline{f}^{(j-1)}(\tau, \xi) = \left(\begin{array}{c}\langle \phi_d, \lambda^{-2}(\tau)RN_{\nu}(\epsilon^{(j-1)})\rangle\\ \mathcal{F}\big(\lambda^{-2}(\tau)RN_{\nu}(\epsilon^{(j-1)}) \big)(\xi)\end{array}\right)
 \]
 and the additional linear error term is defined by 
 \begin{equation}\label{eq:Rterms}
\mathcal{R}(\tau, \underline{x}^{(j-1)})(\xi) = \Big(-4\beta_{\nu}(\tau)\mathcal{K}\mathcal{D}_{\tau}\underline{x}^{(j-1)} - \beta_{\nu}^2(\tau)(\mathcal{K}^2 + [\mathcal{A}, \mathcal{K}] + \mathcal{K} +  \beta_{\nu}' \beta_{\nu}^{-2}\mathcal{K})\underline{x}^{(j-1)}\Big)(\xi)
\end{equation}
Moreover, we write 
\[
\epsilon^{(j-1)}(\tau, R) = R^{-1}\til\eps^{(j-1)}(\tau, R) = R^{-1}[\phi_d(R)x^{(j-1)}_d(\tau) + \int_0^\infty \phi(R, \xi)x^{(j-1)}(\tau, \xi)\rho(\xi)\,d\xi]. 
\]
Thinking of the iterates $\underline{x}^{(j)}(\tau, \xi)$ as obtained by adding a sequence of 'increments' we can alternatively set $\underline{x}^{(j)} = \underline{x}^{(0)} + \sum_{k=1}^j\underline{\triangle x}^{(k)}(\tau, \xi)$, where we put 
\begin{equation}\label{eq:increment1}
 (\mathcal{D}_{\tau}^2 + \beta_{\nu}(\tau)\mathcal{D}_{\tau} + \underline{\xi})\underline{\triangle x}^{(1)}(\tau, \xi) = \mathcal{R}(\tau, \underline{x}^{(0)}) + \underline{f}^{(0)}(\tau,\xi), 
 \end{equation}
as well as 
\begin{equation}\label{eq:incrementj}
 (\mathcal{D}_{\tau}^2 + \beta_{\nu}(\tau)\mathcal{D}_{\tau} + \underline{\xi})\underline{\triangle x}^{(j)}(\tau, \xi) = \mathcal{R}(\tau, \underline{\triangle x}^{(j-1)}) + \underline{\triangle f}^{(j-1)}(\tau,\xi),\,j\geq 2,  
 \end{equation}
where we evidently define 
\begin{align*}
\underline{\triangle f}^{(j-1)}(\tau,\xi): = \left(\begin{array}{c}\langle \phi_d, \lambda^{-2}(\tau)RN_{\nu}(\epsilon^{(j-1)})\rangle - \lambda^{-2}(\tau)RN_{\nu}(\epsilon^{(j-2)})\rangle\\ \mathcal{F}\big(\lambda^{-2}(\tau)RN_{\nu}(\epsilon^{(j-1)}) \big)(\xi) - \lambda^{-2}(\tau)RN_{\nu}(\epsilon^{(j-2)}) \big)(\xi)\end{array}\right)
\end{align*}
Ideally, we would pick the $\underline{\triangle x}^{(j)}$ with vanishing data at initial time $\tau = \tau_0$, which, however, does not work for both the continuous as well as the discrete part. Instead, we shall have to add small additional incremental corrections to the data of the form
\[
\big(\left(\begin{array}{c}\triangle x_{0d}^{(j)}\\ \triangle\tilde{\tilde{x}}^{(j)}_0\end{array}\right), \left(\begin{array}{c}\triangle x_{1d}^{(j)}\\ \triangle\tilde{\tilde{x}}^{(j)}_1\end{array}\right)
\]
at the $j$-th iterate, which is the reason for the nonlinear nature of the set $\Sigma_0$ described in Theorem~\ref{thm:Main}. 
\\

In the next sections, we shall carefully develop this iteration scheme, starting with a description of the first iterate $\underline{x}^{(1)}(\tau, \xi)$, and then followed by the higher iterates. It is when treating the latter that we introduce a stable functional framework which shall be preserved under all the increments. 
Finally, we shall have to prove convergence of the series $\sum_{j=1}^\infty  \underline{\triangle x}^{(j)}$, which will be the culmination of our efforts, and establish the Theorem~\ref{thm:MainTechnical}.

\section{Control of the first iterate; contribution of the linear terms $ \mathcal{R}(\tau, \underline{x}^{(0)})$}

Note that the solution $x(\tau, \xi)$ of the inhomogeneous equation 
\[
 (\mathcal{D}_{\tau}^2 + \beta_{\nu}(\tau)\mathcal{D}_{\tau} + \underline{\xi})\underline{x}(\tau, \xi) = \underline{f}(\tau, \xi)
 \]
 upon imposition of trivial data $\big(x(\tau_0, \cdot), \mathcal{D}_{\tau}x(\tau_0, \cdot)\big) = \big(0, 0\big)$ is solved in terms of its continuous part $x(\tau, \xi)$ by the formula 
 \[
 x(\tau, \xi) = \int_{\tau_0}^\tau\frac{\lambda^{\frac{3}{2}}(\tau)}{\lambda^{\frac{3}{2}}(\sigma)}\frac{\rho^{\frac{1}{2}}(\frac{\lambda^2(\tau)}{\lambda^2(\sigma)}\xi)}{\rho^{\frac12}(\xi)}\frac{\sin[\lambda(\tau)\xi^{\frac12}\int_{\tau}^{\sigma}\lambda^{-1}(u)\,du]}{\xi^{\frac12}}f(\sigma, \frac{\lambda^2(\tau)}{\lambda^2(\sigma)}\xi)\,d\sigma
 \]
 This expression of course grows in some sense rapidly as $\tau\rightarrow +\infty$, and we are interested in a suitable vanishing condition which ensures that it in fact does not 'grow too much'. The following proposition explains how this can be achieved by adding on a suitable 'free wave', provided we set\footnote{Here, the expression $\mathcal{R}(\tau, \underline{x}^{(0)})(\xi)$ refers to the continuous spectral part of the vectorial quantity $\mathcal{R}(\tau, \underline{x}^{(0)})$. We shall deal with the much simpler discrete spectral part later.}  $f(\tau, \xi) = \mathcal{R}(\tau, \underline{x}^{(0)})(\xi)$: denote by $S(\tau)(x_0, x_1)$ the evolution of the continuous part by the free flow 
 \[
  (\mathcal{D}_{\tau}^2 + \beta_{\nu}(\tau)\mathcal{D}_{\tau} + \underline{\xi})\underline{x}(\tau, \xi) = 0
  \]
  with data $(x_0, x_1)$ at time $\tau = \tau_0$. Thus this evolution is given by the formula \eqref{eq:linhomparam1} in Lemma~\ref{lem:linhom}. {\it{Throughout we assume the vanishing conditions \eqref{eq:vanishing}} for the data $(x_0, x_1)$}. 
  
  \begin{prop}\label{prop:firstiteratelincont} Defining (see \eqref{eq:tildeSnorm})
  \[
 \big\|\cdot,\cdot\big\|_{\tilde{S}}: = \big\|\cdot, \cdot\big\|_{\langle\xi\rangle^{-(\frac12++)}\xi^{-(\frac12-)}\min\{\tau_0^2\xi,1\}^{\frac12}L^2_{d\xi}\times  \langle\xi\rangle^{-(\frac{1}{2}++)}\xi^{0+}L^2_{d\xi}}, 
 \]
there is a choice of $(\triangle\tilde{\tilde{x}}_0^{(1)}, \triangle\tilde{\tilde{x}}_1^{(1)})\in \tilde{S}$ with 
\begin{equation}\label{eq:keysmall}
\big\|(\triangle\tilde{\tilde{x}}_0^{(1)}, \triangle\tilde{\tilde{x}}_1^{(1)})\big\|_{\tilde{S}}\lesssim \tau_0^{-(1-)}[\big\|(x_0, x_1)\big\|_{\tilde{S}} + |x_{0d}|]
\end{equation}
such that setting 
 \begin{align*}
 &\triangle x^{(1)}(\tau, \xi): \\
 &=\int_{\tau_0}^\tau\frac{\lambda^{\frac{3}{2}}(\tau)}{\lambda^{\frac{3}{2}}(\sigma)}\frac{\rho^{\frac{1}{2}}(\frac{\lambda^2(\tau)}{\lambda^2(\sigma)}\xi)}{\rho^{\frac{1}{2}}(\xi)}\frac{\sin[\lambda(\tau)\xi^{\frac{1}{2}}\int_{\tau}^\sigma\lambda^{-1}(u)\,du]}{\xi^{\frac{1}{2}}}\mathcal{R}(\sigma, \underline{x}^{(0)})(\frac{\lambda^2(\tau)}{\lambda^2(\sigma)}\xi)\,d\sigma\\
 & + S(\tau)(\triangle\tilde{\tilde{x}}_0^{(1)}, \triangle\tilde{\tilde{x}}_1^{(1)}) 
 \end{align*}
 we have the high frequency bound (recall \eqref{eq:Snorm} for the definition of $\big\|(\cdot, \cdot)\big\|_{S}$, $\big\|\cdot\big\|_{S_{1,2}}$)
 \begin{equation}\label{eq:firstiterhigh}
  \big\|\chi_{\xi>1}\big(\triangle x^{(1)}(\tau, \xi), \mathcal{D}_{\tau}\triangle x^{(1)}(\tau, \xi)\big)\big\|_{S}\lesssim \big\|(x_0, x_1)\big\|_{\tilde{S}} + |x_{0d}|
  \end{equation}
  uniformly in $\tau\geq \tau_0$. Moreover, there is a splitting 
  \begin{equation}\label{eq:splitting1}
  \triangle x^{(1)}(\tau, \xi) =  \triangle x^{(1)}_{>\tau}(\tau, \xi) + S(\tau)\big(\triangle \tilde{x}^{(1)}_0(\xi), \triangle \tilde{x}^{(1)}_1(\xi)\big)
  \end{equation}
  such that we have 
  \begin{equation}\label{eq:firstiterlow}\begin{split}
 &\sup_{\tau\geq \tau_0}(\frac{\tau}{\tau_0})^{-\kappa}\big\|\chi_{\xi<1}\triangle x^{(1)}_{>\tau}(\tau, \xi)\big\|_{S_1} + \sup_{\tau\geq \tau_0}\big\|\chi_{\xi<1}\mathcal{D}_{\tau}\triangle x^{(1)}_{>\tau}(\tau, \xi)\big\|_{S_2}\\&
 + \big\|\big(\triangle \tilde{x}^{(1)}_0(\xi), \triangle \tilde{x}^{(1)}_1(\xi)\big)\big\|_{S}\lesssim  \big\|(x_0, x_1)\big\|_{\tilde{S}} + |x_{0d}|, 
  \end{split}\end{equation}
  and such that $\triangle \tilde{x}^{(1)}_{0,1}(\xi)$ in turn satisfy the vanishing conditions \eqref{eq:vanishing}. 
  Here we shall let $\kappa = 2(1+\nu^{-1})\delta_0$ throughout, with $\delta_0>0$ as in the definition of $\tilde{S}, S$, see \eqref{eq:tildeSnorm}, \eqref{eq:Snorm}. 
 \end{prop}
 \begin{rem} We note that for the most part, the factor $\min\{\tau_0^2\xi, 1\}^{-\frac12}$ in the definition of $\big\|\cdot\big\|_{\tilde{S}}$ could have been omitted, except that it ensures both the fact that the first vanishing condition in \eqref{eq:vanishing} is defined, as well as the fact that the correction $\triangle\tilde{\tilde{x}}_0^{(1)}$ will be much smaller than the data. 
\end{rem}
\begin{proof}
 {\bf{Step 0}}: To begin with, we observe the following simple 
\begin{lem}\label{lem:freeSbound} Let $(x_0, x_1)\in \tilde{S}$. Then we have 
\[
\big\|\big(\chi_{\xi>1}S(\tau)(x_0, x_1), \chi_{\xi>1}\mathcal{D}_{\tau}S(\tau)(x_0, x_1)\big\|_{S}\lesssim \big\|(x_0, x_1)\big\|_{\tilde{S}}
\]
uniformly in $\tau\geq \tau_0$. 

\end{lem}
\begin{proof} This follows from direct inspection of \eqref{eq:linhomparam1} and the asymptotics of $\rho(\xi)$. 
\end{proof}
This lemma takes care of the high frequency contribution of $S(\tau)(\triangle\tilde{\tilde{x}}_0^{(1)}, \triangle\tilde{\tilde{x}}_1^{(1)})$, provided we have picked  suitable $\triangle\tilde{\tilde{x}}_{0,1}^{(1)}$ satisfying the bound alleged in the proposition to be proved. This we shall do later on in the proof. Next, we shall deal with the contribution of of $\mathcal{R}(\sigma, \underline{x}^{(0)})$, which in turn is defined in terms of the operator $\mathcal{K}$. Recall from \eqref{eq:Kstructure}, that we have 
\[
\mathcal{K}(\underline{x}^{(0)}) = \mathcal{K}_{cc}(x^{(0)}) + \mathcal{K}_{cd}(x^{(0)}_{d})
\]
We shall then write 
\[
\mathcal{R}(\tau, \underline{x}^{(0)})(\xi) = \mathcal{R}_1(\tau, x^{(0)})(\xi) + \mathcal{R}_2(\tau, x^{(0)}_d)(\xi)
\]
and consider these contributions separately. Again the exponential decay of $x^{(0)}_d$ (see Lemma~\ref{lem:linhom}) ensures that the second term on the right will be much simpler to treat than the first, which will be our main concern for the proof. Thus relegating the contribution of $\mathcal{R}_2(\tau, x^{(0)}_d)(\xi)$ to the end, we now focus on the main term $\mathcal{R}_1(\tau, x^{(0)})(\xi)$, which in light of the considerations in \cite{KST}, \cite{KS1} can be schematically written as the sum of three terms 
\begin{equation}\label{eq:R_1structure}
\mathcal{R}_1(\tau, x^{(0)})(\xi) = \beta_{\nu}^2(\tau)\mathcal{K}_{cc}x^{(0)} + \beta_{\nu}(\tau)\mathcal{K}_{cc}\mathcal{D}_{\tau}x^{(0)} + \beta_{\nu}^2(\tau)\mathcal{K}_{cc}^2x^{(0)},
\end{equation}
 where we keep in mind that $x^{(0)}$ represents the continuous part of the zeroth iterate $\underline{x}^{(0)}$, see \eqref{eq:Rterms}. In fact, the commutator term $[\mathcal{A}, \mathcal{K}]$ is seen to have a structure similar to the one of $\mathcal{K}$. Moreover, the operator $\mathcal{K}_{cc}$ is given by integration against a kernel of the form 
 \[
 K_{cc}(\xi, \eta) = \frac{F(\xi, \eta)\rho(\eta)}{\xi - \eta},
 \]
with $F$ having the decay and vanishing properties specified in \cite{KST}. In the sequel, we treat the contributions of the first two terms on the right in \eqref{eq:R_1structure}, dealing with the last term in the appendix. 
\\

{\bf{Step 1}}: {\it{Proof of the high frequency estimate \eqref{eq:firstiterhigh} for the contribution from $\mathcal{R}_1(\tau, x^{(0)})$.}} We establish the bound for $\triangle x^{(1)}$, the bound for $\mathcal{D}_{\tau}\triangle x^{(1)}$ being similar. We shall have to take into account the precise form of $\mathcal{K}_{cc}x^{(0)}$, which in light of the definition of the zeroth iterate $x^{(0)}$ is in fact given by the expression 
 \begin{equation}\label{eq:Kx0formula}\begin{split}
 &\mathcal{K}_{cc}x^{(0)}(\sigma, \xi)\\
 &  = \int_0^\infty \frac{F(\xi, \eta)\rho(\eta)}{\xi- \eta}\frac{\lambda^{\frac{5}{2}}(\sigma)}{\lambda^{\frac{5}{2}}(\tau_0)}\frac{\rho^{\frac{1}{2}}(\frac{\lambda^2(\sigma)}{\lambda^2(\tau_0)}\eta)}{\rho^{\frac{1}{2}}(\eta)}\cos[\lambda(\sigma)\eta^{\frac{1}{2}}\int_{\tau_0}^{\sigma}\lambda^{-1}(u)\,du]x_0(\frac{\lambda^2(\sigma)}{\lambda^2(\tau_0)}\eta)\,d\eta\\
 & + \int_0^\infty \frac{F(\xi, \eta)\rho(\eta)}{\xi- \eta}\frac{\lambda^{\frac{3}{2}}(\sigma)}{\lambda^{\frac{3}{2}}(\tau_0)}\frac{\rho^{\frac{1}{2}}(\frac{\lambda^2(\sigma)}{\lambda^2(\tau_0)}\eta)}{\rho^{\frac{1}{2}}(\eta)}\frac{\sin[\lambda(\sigma)\eta^{\frac{1}{2}}\int_{\tau_0}^{\sigma}\lambda^{-1}(u)\,du]}{\eta^{\frac{1}{2}}}x_1(\frac{\lambda^2(\sigma)}{\lambda^2(\tau_0)}\eta)\,d\eta,\\
 \end{split}\end{equation}
Both terms on the right are handled along the same lines, and so we shall restrict below to the contribution of the second term, treating the contribution of the first term in the appendix. Recalling the above schematic identity \eqref{eq:R_1structure}, we further divide this step into two tasks: 
\\

{\bf{Step 1 (i)}}: {\it{Contribution of $ \beta_{\nu}^2(\tau)\mathcal{K}_{cc}x^{(0)}$.}}  Substituting the second term on the right in \eqref{eq:Kx0formula}
 for $\mathcal{K}x^{(0)}(\sigma, \xi)$ into the formula for $\triangle x^{(1)}(\tau, \xi)$, we arrive at the following expression: 
 \begin{align*}
 \int_{\tau_0}^\tau\frac{\lambda^{\frac{3}{2}}(\tau)}{\lambda^{\frac{3}{2}}(\sigma)}\frac{\rho^{\frac{1}{2}}(\frac{\lambda^2(\tau)}{\lambda^2(\sigma)}\xi)}{\rho^{\frac{1}{2}}(\xi)}\frac{\sin[\lambda(\tau)\xi^{\frac{1}{2}}\int_{\tau}^\sigma\lambda^{-1}(u)\,du]}{\xi^{\frac{1}{2}}}\beta_{\nu}^2(\sigma)A(\xi, \tau, \tau_0, \sigma)\,d\sigma,
 \end{align*}
 where we set 
 \begin{align*}
 &A(\xi, \tau, \tau_0, \sigma):\\
 & =  \int_0^\infty\frac{\lambda^{\frac32}(\sigma)}{\lambda^{\frac32}(\tau_0)}\frac{F(\frac{\lambda^2(\tau)}{\lambda^2(\sigma)}\xi, \frac{\lambda^2(\tau)}{\lambda^2(\sigma)}\tilde{\eta})\rho^{\frac{1}{2}}(\frac{\lambda^2(\tau)}{\lambda^2(\sigma)}\tilde{\eta})}{\xi- \tilde{\eta}}\rho^{\frac{1}{2}}(\frac{\lambda^2(\tau)}{\lambda^2(\tau_0)}\tilde{\eta})\frac{\sin[\lambda(\tau)\tilde{\eta}^{\frac{1}{2}}\int_{\tau_0}^{\sigma}\lambda^{-1}(u)\,du]}{(\frac{\lambda^2(\tau)}{\lambda^2(\sigma)}\tilde{\eta})^{\frac{1}{2}}}\\
 &\hspace{9.5cm}\cdot x_1(\frac{\lambda^2(\tau)}{\lambda^2(\tau_0)}\tilde{\eta})\,d\tilde{\eta}. 
 \end{align*}
To dispose of the contribution when $\tilde{\eta}\geq 1$, we simply observe that when in addition $\xi>1$ we have that 
\begin{align*}
B(\tau, \sigma, \xi, \tilde{\eta}): = \frac{\lambda^{\frac{3}{2}}(\tau)}{\lambda^{\frac{3}{2}}(\sigma)}\frac{\xi^{\frac14+}\rho^{\frac{1}{2}}(\frac{\lambda^2(\tau)}{\lambda^2(\sigma)}\xi)F(\frac{\lambda^2(\tau)}{\lambda^2(\sigma)}\xi, \frac{\lambda^2(\tau)}{\lambda^2(\sigma)}\tilde{\eta})\rho^{\frac{1}{2}}(\frac{\lambda^2(\tau)}{\lambda^2(\sigma)}\tilde{\eta})}{(\frac{\lambda^2(\tau)}{\lambda^2(\sigma)}\tilde{\eta})^{\frac{1}{2}}}
\end{align*}
satisfies $\big|B(\tau, \sigma, \xi, \tilde{\eta})\big|\lesssim \big(\frac{\lambda(\tau)}{\lambda(\sigma)}\big)^{-\frac12}$; indeed, this follows from the bounds established in \cite{KST}. Furthermore, we have (for $\tilde{\eta}\geq 1$)
\[
\big|\rho^{\frac{1}{2}}(\frac{\lambda^2(\tau)}{\lambda^2(\tau_0)}\tilde{\eta}) x_1(\frac{\lambda^2(\tau)}{\lambda^2(\tau_0)}\tilde{\eta})\big|\lesssim \big(\frac{\lambda(\tau)}{\lambda(\tau_0)}\big)^{-\frac12}\big|(\cdot)^{\frac{1}{2}}x_1\big|(\frac{\lambda^2(\tau)}{\lambda^2(\tau_0)}\tilde{\eta})
\]
Furthermore, the operator given by integration with respect to $\tilde{\eta}$ and kernel $\frac{B(\tau, \sigma, \xi, \tilde{\eta})}{\xi - \tilde{\eta}}$ is seen to be bounded as an operator on $L^2_{d\xi}$ by arguments as in \cite{KST} with norm bound $\big(\frac{\lambda(\tau)}{\lambda(\sigma)}\big)^{-\frac12 +}$. We then easily infer the desired bound 
\begin{align*}
&\big\| \chi_{\xi>1}\langle\xi\rangle^{1+}\int_{\tau_0}^\tau\frac{\lambda^{\frac{3}{2}}(\tau)}{\lambda^{\frac{3}{2}}(\sigma)}\frac{\rho^{\frac{1}{2}}(\frac{\lambda^2(\tau)}{\lambda^2(\sigma)}\xi)}{\rho^{\frac{1}{2}}(\xi)}\frac{\sin[\lambda(\tau)\xi^{\frac{1}{2}}\int_{\tau}^\sigma\lambda^{-1}(u)\,du]}{\xi^{\frac{1}{2}}}\beta_{\nu}^2(\sigma)A_{\tilde{\eta}\geq 1}(\xi, \tau, \tau_0, \sigma)\,d\sigma\big\|_{L^2_{d\xi}}\\
&\lesssim \big\|x_1\big\|_{S_2}\int_{\tau_0}^\tau\frac{\lambda^{2-}(\sigma)}{\lambda^{(2-)}(\tau)}\beta_{\nu}^2(\sigma)\,d\sigma\lesssim \tau_0^{-1}\big\|(x_0, x_1)\big\|_{S}. 
\end{align*}
where the subscript in $A_{\tilde{\eta}\geq 1}$ indicates the restriction of the integrated variable we additionally impose in the definition of $A(\xi, \tau, \tau_0, \sigma)$. 
\\
Continuing with the contribution of the second term in \eqref{eq:Kx0formula}, the case $\tilde{\eta}<1$ is much more delicate, and we have to exploit the fine structure of $A(\xi, \tau, \tau_0, \sigma)$, as well as the vanishing condition of $x_1$, to treat it. Write (\cite{KST}) 
 \[
 F(\frac{\lambda^2(\tau)}{\lambda^2(\sigma)}\xi, \frac{\lambda^2(\tau)}{\lambda^2(\sigma)}\tilde{\eta}) = F(\frac{\lambda^2(\tau)}{\lambda^2(\sigma)}\xi, 0) + O([\frac{\lambda^2(\tau)}{\lambda^2(\sigma)}\xi]^{-N} \frac{\lambda^2(\tau)}{\lambda^2(\sigma)}\tilde{\eta})
 \]
 Since $O([\frac{\lambda^2(\tau)}{\lambda^2(\sigma)}\xi]^{-N} \frac{\lambda^2(\tau)}{\lambda^2(\sigma)}\tilde{\eta}) = O([\frac{\lambda^2(\tau)}{\lambda^2(\sigma)}\xi]^{-(N-1)}\tilde{\eta})$ in light of $\xi>1$, we get 
 \begin{align*}
&\big|O([\frac{\lambda^2(\tau)}{\lambda^2(\sigma)}\xi]^{-N} \frac{\lambda^2(\tau)}{\lambda^2(\sigma)}\tilde{\eta})\rho^{\frac{1}{2}}(\frac{\lambda^2(\tau)}{\lambda^2(\sigma)}\tilde{\eta})\rho^{\frac{1}{2}}(\frac{\lambda^2(\tau)}{\lambda^2(\tau_0)}\tilde{\eta})\frac{\sin[\lambda(\tau)\tilde{\eta}^{\frac{1}{2}}\int_{\tau_0}^{\sigma}\lambda^{-1}(u)\,du]}{(\frac{\lambda^2(\tau)}{\lambda^2(\sigma)}\tilde{\eta})^{\frac{1}{2}}}\big|\\
&\lesssim \big[\chi_{\frac{\lambda^2(\tau)}{\lambda^2(\tau_0)}\tilde{\eta}\geq 1}\tilde{\eta}^{\frac{1}{4}}[\frac{\lambda^2(\tau)}{\lambda^2(\tau_0)}\tilde{\eta}]^{\frac{1}{4}} + \chi_{\frac{\lambda^2(\tau)}{\lambda^2(\tau_0)}\tilde{\eta}<1}\big][\frac{\lambda^2(\tau)}{\lambda^2(\sigma)}\xi]^{-(N-1)} 
\end{align*}
where we keep the assumption $\sigma\leq \tau, \tilde{\eta}<1$, $\xi>1$ in mind. It follows that substituting $O([\frac{\lambda^2(\tau)}{\lambda^2(\sigma)}\xi]^{-N} \frac{\lambda^2(\tau)}{\lambda^2(\sigma)}\tilde{\eta})
$ instead of $F(\frac{\lambda^2(\tau)}{\lambda^2(\sigma)}\xi, \frac{\lambda^2(\tau)}{\lambda^2(\sigma)}\tilde{\eta})$ in $A_{\tilde{\eta}<1}(\xi, \tau, \tau_0, \sigma)$ and calling the corresponding term $\tilde{A}_{\tilde{\eta}<1}(\xi, \tau, \tau_0, \sigma)$, we get 
\begin{align*}
 &\big\|\int_{\tau_0}^\tau\frac{\lambda^{\frac{3}{2}}(\tau)}{\lambda^{\frac{3}{2}}(\sigma)}\frac{\rho^{\frac{1}{2}}(\frac{\lambda^2(\tau)}{\lambda^2(\sigma)}\xi)}{\rho^{\frac{1}{2}}(\xi)}\frac{\sin[\lambda(\tau)\xi^{\frac{1}{2}}\int_{\tau}^\sigma\lambda^{-1}(u)\,du]}{\xi^{\frac{1}{2}}}\beta_{\nu}^2(\sigma)\tilde{A}_{\tilde{\eta}<1}(\xi, \tau, \tau_0, \sigma)\,d\sigma\big\|_{\langle\xi\rangle^{-1-}L^2_{d\xi}(\xi>1)}\\
 &\lesssim \tau_0^{-1}\big\|x_1\big\|_{\langle\xi\rangle^{-\frac{1}{2}-}L^2_{d\xi}},\\
\end{align*}
Similarly, using 
\[
\rho^{\frac{1}{2}}(\frac{\lambda^2(\tau)}{\lambda^2(\sigma)}\tilde{\eta}) = \chi_{\frac{\lambda^2(\tau)}{\lambda^2(\sigma)}\tilde{\eta}>1}\rho^{\frac{1}{2}}(\frac{\lambda^2(\tau)}{\lambda^2(\sigma)}\tilde{\eta})  + c\chi_{\frac{\lambda^2(\tau)}{\lambda^2(\sigma)}\tilde{\eta}<1}[\frac{\lambda^2(\tau)}{\lambda^2(\sigma)}\tilde{\eta}]^{-\frac{1}{4}}(1+O([\frac{\lambda^2(\tau)}{\lambda^2(\sigma)}\tilde{\eta})]^{\frac{1}{2}})
\]
for a suitable constant $c>0$, and using that 
\begin{align*}
&\big\|\tilde{\eta}^{-\frac{1}{4}}\rho^{\frac{1}{2}}(\frac{\lambda^2(\tau)}{\lambda^2(\tau_0)}\tilde{\eta})x_1(\frac{\lambda^2(\tau)}{\lambda^2(\tau_0)}\tilde{\eta})\big\|_{L^1_{d\tilde{\eta}}}\lesssim \big(\frac{\lambda(\tau_0)}{\lambda(\tau)}\big)^{\frac{3}{2}}\big\|x_1\big\|_{\langle\xi\rangle^{-\frac{1}{2}-}\xi^{0+}L^2_{d\xi}},
\end{align*}
we see that we may replace the factor $\rho^{\frac{1}{2}}(\frac{\lambda^2(\tau)}{\lambda^2(\sigma)}\tilde{\eta}) $ by $\chi_{\frac{\lambda^2(\tau)}{\lambda^2(\sigma)}\tilde{\eta}<1}[\frac{\lambda^2(\tau)}{\lambda^2(\sigma)}\tilde{\eta}]^{-\frac{1}{4}}$. Further, one may replace $\frac{1}{\xi-\tilde{\eta}}$ by $\frac{1}{\xi}$, contribution bounded by $\lesssim \tau_0^{-1}\big\|x_1\big\|_{\langle\xi\rangle^{-\frac{1}{2}-}L^2_{d\xi}}$. 
\\
It now follows that we may replace $A_{\tilde{\eta}<1}(\xi, \tau, \tau_0, \sigma)$ by a term of the form 
\[
\frac{\lambda^{\frac32}(\sigma)}{\lambda^{\frac32}(\tau_0)}\frac{F(\frac{\lambda^2(\tau)}{\lambda^2(\sigma)}\xi, 0)}{\xi}\int_0^\infty\chi_{\frac{\lambda^2(\tau)}{\lambda^2(\sigma)}\tilde{\eta}<1} \frac{\sin[\lambda(\tau)\tilde{\eta}^{\frac{1}{2}}\int_{\tau_0}^{\sigma}\lambda^{-1}(u)\,du]}{(\frac{\lambda^2(\tau)}{\lambda^2(\sigma)}\tilde{\eta})^{\frac{3}{4}}}\cdot \rho^{\frac{1}{2}}(\frac{\lambda^2(\tau)}{\lambda^2(\tau_0)}\tilde{\eta})x_1(\frac{\lambda^2(\tau)}{\lambda^2(\tau_0)}\tilde{\eta})\,d\tilde{\eta}
\]
Here one may easily get rid of the cutoff $\chi_{\frac{\lambda^2(\tau)}{\lambda^2(\sigma)}\tilde{\eta}<1} $, as the contribution of $\chi_{\frac{\lambda^2(\tau)}{\lambda^2(\sigma)}\tilde{\eta}\geq1}$ is treated like before on account of 
\begin{align*}
&\big|\rho^{\frac{1}{2}}(\frac{\lambda^2(\tau)}{\lambda^2(\tau_0)}\tilde{\eta})x_1(\frac{\lambda^2(\tau)}{\lambda^2(\tau_0)}\tilde{\eta})\big|\\
& \lesssim \big|\chi_{\frac{\lambda^2(\tau)}{\lambda^2(\tau_0)}\tilde{\eta}>1}\big(\frac{\lambda(\tau_0)}{\lambda(\tau)}\big)^{\frac{1}{2}}\tilde{\eta}^{-\frac{1}{4}}\rho(\frac{\lambda^2(\tau)}{\lambda^2(\tau_0)}\tilde{\eta})x_1(\frac{\lambda^2(\tau)}{\lambda^2(\tau_0)}\tilde{\eta})\big| + \big|\tilde{\eta}^{-\frac{1}{4}}\frac{\lambda^{\frac{1}{2}}(\tau_0)}{\lambda^{\frac{1}{2}}(\tau)}x_1(\frac{\lambda^2(\tau)}{\lambda^2(\tau_0)}\tilde{\eta})\big| 
\end{align*}
For the remaining expression, if one  enacts a change of variables, it is transformed into 
\[
\frac{F(\frac{\lambda^2(\tau)}{\lambda^2(\sigma)}\xi, 0)}{\frac{\lambda^{\frac{3}{2}}(\tau)}{\lambda^{\frac{3}{2}}(\sigma)}\xi}\frac{\lambda(\tau)}{\lambda(\tau_0)}\int_0^\infty\frac{\sin[\lambda(\tau_0)\tilde{\eta}^{\frac{1}{2}}\int_{\tau_0}^{\sigma}\lambda^{-1}(u)\,du]}{\tilde{\eta}^{\frac{3}{4}}}\cdot \rho^{\frac{1}{2}}(\tilde{\eta})x_1(\tilde{\eta})\,d\tilde{\eta}
\]
But this is precisely the expression which came up at the end of the proof of Proposition~\ref{prop:lingrowthcond} and so, in light of the vanishing condition on $x_1$, we can bound it by 
\begin{align*}
&\big|\frac{F(\frac{\lambda^2(\tau)}{\lambda^2(\sigma)}\xi, 0)}{\frac{\lambda^{\frac{3}{2}}(\tau)}{\lambda^{\frac{3}{2}}(\sigma)}\xi}\frac{\lambda(\tau)}{\lambda(\tau_0)}\int_0^\infty\frac{\sin[\lambda(\tau_0)\tilde{\eta}^{\frac{1}{2}}\int_{\tau_0}^{\sigma}\lambda^{-1}(u)\,du]}{\tilde{\eta}^{\frac{3}{4}}}\cdot \rho^{\frac{1}{2}}(\tilde{\eta})x_1(\tilde{\eta})\,d\tilde{\eta}\big|\\
&\lesssim \frac{F(\frac{\lambda^2(\tau)}{\lambda^2(\sigma)}\xi, 0)}{\frac{\lambda^{\frac{1}{2}}(\tau)}{\lambda^{\frac{1}{2}}(\sigma)}\xi}\sigma\big\|\langle\xi\rangle^{\frac{1}{2}+}\xi^{0-}x_1\big\|_{L^2_{d\xi}}. 
\end{align*}
Finally, substituting this for $A_{\tilde{\eta}<1}(\xi, \tau, \tau_0, \sigma)$ we arrive at the bound 
\begin{align*}
 &\big\|\int_{\tau_0}^\tau\frac{\lambda^{\frac{3}{2}}(\tau)}{\lambda^{\frac{3}{2}}(\sigma)}\frac{\rho^{\frac{1}{2}}(\frac{\lambda^2(\tau)}{\lambda^2(\sigma)}\xi)}{\rho^{\frac{1}{2}}(\xi)}\frac{\sin[\lambda(\tau)\xi^{\frac{1}{2}}\int_{\tau}^\sigma\lambda^{-1}(u)\,du]}{\xi^{\frac{1}{2}}}\beta_{\nu}^2(\sigma)A_{\tilde{\eta}<1}(\xi, \tau, \tau_0, \sigma)\,d\sigma\big\|_{\langle\xi\rangle^{-1-}L^2_{d\xi}(\xi>1)}\\
 &\lesssim \big\|x_1\big\|_{\langle\xi\rangle^{-\frac{1}{2}-}\xi^{0+}L^2_{d\xi}},\\
\end{align*}
which is also acceptable. 
\begin{rem}\label{rem:prop71highfreqimprov} We note that in the last step we could have performed an integration by parts with respect to $\sigma$, and gained an extra power $\tau_0^{-(1-)}$. 

\end{rem}

{\bf{Step 1 (ii)}}: {\it{Contribution of $ \beta_{\nu}(\tau)\mathcal{K}_{cc}\mathcal{D}_{\tau}x^{(0)}$.}} This follows the exact same pattern as before, except that now we have only one power of decay $\sigma^{-1}$ but in exchange we replace $A(\xi, \tau, \tau_0, \sigma)$ (corresponding to the contribution by the second term in \eqref{eq:Kx0formula}) by 
 \begin{align*}
 &A(\xi, \tau, \tau_0, \sigma):\\
 & =  \int_0^\infty\frac{\lambda^{\frac32}(\sigma)}{\lambda^{\frac32}(\tau_0)}\frac{F(\frac{\lambda^2(\tau)}{\lambda^2(\sigma)}\xi, \frac{\lambda^2(\tau)}{\lambda^2(\sigma)}\tilde{\eta})\rho^{\frac{1}{2}}(\frac{\lambda^2(\tau)}{\lambda^2(\sigma)}\tilde{\eta})}{\xi- \tilde{\eta}}\rho^{\frac{1}{2}}(\frac{\lambda^2(\tau)}{\lambda^2(\tau_0)}\tilde{\eta})\cos[\lambda(\tau)\tilde{\eta}^{\frac{1}{2}}\int_{\tau_0}^{\sigma}\lambda^{-1}(u)\,du]\\
 &\hspace{10cm}\cdot x_1(\frac{\lambda^2(\tau)}{\lambda^2(\tau_0)}\tilde{\eta})\,d\tilde{\eta},
\end{align*}
and here we no longer lose a factor $\sigma$ as above (which necessitated the decay factor $\beta_{\nu}^2(\sigma)$ to close things). Note that one doesn't lose logarithmically here for the time integral (in $\sigma$), as one gains a factor of the form $(\frac{\lambda(\tau)}{\lambda(\sigma)})^{-\alpha}$ for $\alpha>0$.
\\

{\bf{Step 2}}: {\it{Proof of the high frequency estimate \eqref{eq:firstiterhigh} for the contribution from $\mathcal{R}_2(\tau, x^{(0)}_d)$.}} Here we exploit the fact that the operator $\mathcal{K}_{cd}$ is given by a kernel $K_{cd}(\xi)$ which is smooth and rapidly decaying, see \cite{KST}. In particular, we infer 
\begin{align*}
\big\|\chi_{\xi>1}\langle\xi\rangle^{1+}\frac{\lambda^{\frac{3}{2}}(\tau)}{\lambda^{\frac{3}{2}}(\sigma)}\frac{\rho^{\frac{1}{2}}(\frac{\lambda^2(\tau)}{\lambda^2(\sigma)}\xi)}{\rho^{\frac{1}{2}}(\xi)}\frac{\sin[\lambda(\tau)\xi^{\frac{1}{2}}\int_{\tau}^\sigma\lambda^{-1}(u)\,du]}{\xi^{\frac{1}{2}}}K_{cd}(\frac{\lambda^2(\tau)}{\lambda^2(\sigma)}\xi)\big\|_{L^2_{d\xi}}\lesssim_N \big(\frac{\lambda^2(\sigma)}{\lambda^2(\tau)}\big)^N,
\end{align*}
which entails the bound 
\begin{align*}
\big\|\chi_{\xi>1}\langle\xi\rangle^{1+}\int_{\tau_0}^{\tau}\frac{\lambda^{\frac{3}{2}}(\tau)}{\lambda^{\frac{3}{2}}(\sigma)}\frac{\rho^{\frac{1}{2}}(\frac{\lambda^2(\tau)}{\lambda^2(\sigma)}\xi)}{\rho^{\frac{1}{2}}(\xi)}\frac{\sin[\lambda(\tau)\xi^{\frac{1}{2}}\int_{\tau}^\sigma\lambda^{-1}(u)\,du]}{\xi^{\frac{1}{2}}}\mathcal{R}_2(\sigma, x^{(0)}_d)\,d\sigma\big\|_{L^2_{d\xi}}\lesssim \big|x_{0d}\big|,
\end{align*}
in light of the estimate satisfied by $x_d^{(0)}$ via Lemma~\ref{lem:linhom}. 
\\

{\bf{Step 3}}: {\it{Towards the low frequency bound; choosing the corrections $(\triangle \tilde{\tilde{x}}_0^{(1)}, \triangle \tilde{\tilde{x}}_1^{(1)})$.}} We next restrict to $\xi<1$. Here the factor 
\[
\frac{\sin[\lambda(\tau)\xi^{\frac{1}{2}}\int_{\tau}^\sigma\lambda^{-1}(u)\,du]}{\xi^{\frac{1}{2}}}
\]
becomes potentially dangerous, as we cannot absorb the $\xi^{-\frac{1}{2}}$ into the term 
\[
\sin[\lambda(\tau)\xi^{\frac{1}{2}}\int_{\tau}^\sigma\lambda^{-1}(u)\,du]
\]
without losing a very large factor $\frac{\lambda(\tau)}{\lambda(\sigma)}\sigma$. On the other hand, such a factor is essentially harmless when $\sigma\geq \tau$, and so reformulate the expression accordingly. Specifically, we now write 
\begin{align*}
 &\triangle x^{(1)}(\tau, \xi): \\
 &=\int_{\tau_0}^\tau\frac{\lambda^{\frac{3}{2}}(\tau)}{\lambda^{\frac{3}{2}}(\sigma)}\frac{\rho^{\frac{1}{2}}(\frac{\lambda^2(\tau)}{\lambda^2(\sigma)}\xi)}{\rho^{\frac{1}{2}}(\xi)}\frac{\sin[\lambda(\tau)\xi^{\frac{1}{2}}\int_{\tau}^\sigma\lambda^{-1}(u)\,du]}{\xi^{\frac{1}{2}}}\mathcal{R}(\sigma, \underline{x}^{(0)})(\frac{\lambda^2(\tau)}{\lambda^2(\sigma)}\xi)\,d\sigma\\
 & + S(\tau)(\triangle \tilde{\tilde{x}}_0^{(1)}, \triangle\tilde{\tilde{x}}_1^{(1)})\\
 & = -\int_{\tau}^\infty\frac{\lambda^{\frac{3}{2}}(\tau)}{\lambda^{\frac{3}{2}}(\sigma)}\frac{\rho^{\frac{1}{2}}(\frac{\lambda^2(\tau)}{\lambda^2(\sigma)}\xi)}{\rho^{\frac{1}{2}}(\xi)}\frac{\sin[\lambda(\tau)\xi^{\frac{1}{2}}\int_{\tau}^\sigma\lambda^{-1}(u)\,du]}{\xi^{\frac{1}{2}}}\mathcal{R}(\sigma, \underline{x}^{(0)})(\frac{\lambda^2(\tau)}{\lambda^2(\sigma)}\xi)\,d\sigma\\
 &+ \int_{\tau_0}^\infty\frac{\lambda^{\frac{3}{2}}(\tau)}{\lambda^{\frac{3}{2}}(\sigma)}\frac{\rho^{\frac{1}{2}}(\frac{\lambda^2(\tau)}{\lambda^2(\sigma)}\xi)}{\rho^{\frac{1}{2}}(\xi)}\frac{\sin[\lambda(\tau)\xi^{\frac{1}{2}}\int_{\tau}^\sigma\lambda^{-1}(u)\,du]}{\xi^{\frac{1}{2}}}\mathcal{R}(\sigma, \underline{x}^{(0)})(\frac{\lambda^2(\tau)}{\lambda^2(\sigma)}\xi)\,d\sigma\\
 & + S(\tau)(\triangle \tilde{\tilde{x}}_0^{(1)}, \triangle\tilde{\tilde{x}}_1^{(1)})(\xi) =: A + B + C. 
 \end{align*}
Observe that now both $B$, $C$ solve the homogeneous equation (recall that here we are analysing the continuous part of the correction $\underline{\triangle x}^{(1)}$)
\[
(\mathcal{D}_{\tau}^2x + \beta_{\nu}(\tau)\mathcal{D}_{\tau}x + \xi x) = 0. 
\]
Also, we shall re-denote 
\[
A = : \triangle x^{(1)}_{>\tau}(\tau, \xi)
\]
Moreover, as both $B, C$ are solutions of the free problem, we can write 
\[
(B+C)(\tau, \xi) = S(\tau)\big(\triangle \tilde{x}^{(1)}_0(\xi), \triangle \tilde{x}^{(1)}_1(\xi)\big)
\]
where we have 
\begin{align*}
\triangle \tilde{x}^{(1)}_0(\xi) = &\int_{\tau_0}^\infty\frac{\lambda^{\frac{3}{2}}(\tau_0)}{\lambda^{\frac{3}{2}}(\sigma)}\frac{\rho^{\frac{1}{2}}(\frac{\lambda^2(\tau_0)}{\lambda^2(\sigma)}\xi)}{\rho^{\frac{1}{2}}(\xi)}\frac{\sin[\lambda(\tau_0)\xi^{\frac{1}{2}}\int_{\tau_0}^\sigma\lambda^{-1}(u)\,du]}{\xi^{\frac{1}{2}}}\mathcal{R}(\sigma, \underline{x}^{(0)})(\frac{\lambda^2(\tau_0)}{\lambda^2(\sigma)}\xi)\,d\sigma\\
& + \triangle \tilde{\tilde{x}}_0^{(1)}\\
&=:\widetilde{\triangle} \tilde{x}^{(1)}_0(\xi) + \triangle\tilde{\tilde{x}}_0^{(1)}(\xi).
\end{align*}
as well as 
\begin{align*}
\triangle \tilde{x}^{(1)}_1(\xi) =&\int_{\tau_0}^\infty\frac{\lambda^{\frac{3}{2}}(\tau_0)}{\lambda^{\frac{3}{2}}(\sigma)}\frac{\rho^{\frac{1}{2}}(\frac{\lambda^2(\tau_0)}{\lambda^2(\sigma)}\xi)}{\rho^{\frac{1}{2}}(\xi)}\cos[\lambda(\tau_0)\xi^{\frac{1}{2}}\int_{\tau_0}^\sigma\lambda^{-1}(u)\,du]\mathcal{R}(\sigma, \underline{x}^{(0)})(\frac{\lambda^2(\tau_0)}{\lambda^2(\sigma)}\xi)\,d\sigma\\& + \triangle\tilde{\tilde{x}}_1^{(1)}(\xi)\\
&=:\widetilde{\triangle} \tilde{x}^{(1)}_1(\xi) + \triangle\tilde{\tilde{x}}_1^{(1)}(\xi).
\end{align*}

This is now the point where we can choose $\triangle\tilde{\tilde{x}}_{0,1}^{(1)}(\xi)$.  In fact, we claim

\begin{lem}\label{lem:choiceofcorrection}There exist $(\triangle\tilde{\tilde{x}}_0^{(1)}(\xi), \triangle\tilde{\tilde{x}}_1^{(1)}(\xi))$ satisfying \eqref{eq:keysmall} and such that we have 
\begin{equation}\label{eq:vanishing2}
\int_0^\infty \frac{(\rho^{\frac{1}{2}}\triangle\tilde{x}^{(1)}_0)(\xi)}{\xi^{\frac{1}{4}}}\cos[\nu\tau_0\xi^{\frac{1}{2}}]\,d\xi = 0, \int_0^\infty \frac{(\rho^{\frac{1}{2}}\triangle\tilde{x}^{(1)}_1)(\xi)}{\xi^{\frac{3}{4}}}\sin[\nu\tau_0\xi^{\frac{1}{2}}]\,d\xi = 0.
\end{equation}
\end{lem}

A key point here is to satisfy the bound \eqref{eq:keysmall}, since {\it{this is the first correction to the data}}, which has to be of smaller size than the original chosen data to obtain an initial data set $\Sigma_0$ which is identifiable as a co-dimension one Lipschitz surface in $S$. The fact that we can achieve such a gain comes from the fact that the terms in $\mathcal{R}$ enjoy a special vanishing property on account of their symbols. 
\begin{proof}(lemma) The two bounds being quite similar, we prove the second one. We proceed in two steps:
\\

{\bf{(1)}}: {\it{Establishing the bound 
\[
\big|\int_0^\infty \frac{(\rho^{\frac{1}{2}}\widetilde{\triangle}\tilde{x}^{(1)}_1)(\xi)}{\xi^{\frac{3}{4}}}\sin[\nu\tau_0\xi^{\frac{1}{2}}]\,d\xi \big|\lesssim \tau_0^{0+}\big\|(x_0, x_1)\big\|_{\tilde{S}} + \tau_0^{-(1-)}|x_{0d}|.
\]
}}
We immediately reduce this to the contribution of the continuous part $x^{(0)}(\tau, \xi)$ to $\mathcal{R}(\sigma, \underline{x}^{(0)})$, as it is straightforward to check that substituting $\mathcal{R}_2(\sigma, x_d^{(0)})$ for $\mathcal{R}(\sigma, \underline{x}^{(0)})$ in the formula for $\widetilde{\triangle}\tilde{x}^{(1)}_1$ results in the bound 
\[
\big|\int_0^\infty \frac{(\rho^{\frac{1}{2}}\widetilde{\triangle}\tilde{x}^{(1)}_1)(\xi)}{\xi^{\frac{3}{4}}}\sin[\nu\tau_0\xi^{\frac{1}{2}}]\,d\xi \big|\lesssim \tau_0^{-(1-)}|x_{0d}|.
\]
We shall henceforth replace $\mathcal{R}(\sigma, \underline{x}^{(0)})$ by $\mathcal{R}_1(\sigma, x^{(0)})$ and use the schematic relation \eqref{eq:R_1structure} to further divide into two key contributions: 
\\

{\bf{(1i)}}: {\it{Contribution of  $\beta_{\nu}^2(\sigma)\mathcal{K}_{cc}x^{(0)}(\sigma, \xi)$, low frequency $\xi<1$}}.
Throughout we only consider the contribution of the second term in \eqref{eq:Kx0formula} to $x^{(0)}$, leaving the contribution of the first term to the appendix. Substituting this formula into the expression for $\widetilde{\triangle}\tilde{x}^{(1)}_1$ and then into the second expression in \eqref{eq:vanishing2}, we obtain the following kind of integral expression 
\begin{align*}
&\int_{\tau_0}^\infty \frac{\lambda^{\frac{3}{2}}(\tau_0)}{\lambda^{\frac{3}{2}}(\sigma)}\beta_{\nu}^2(\sigma)\int_0^\infty\chi_{\xi<1}\frac{\rho^{\frac{1}{2}}(\frac{\lambda^2(\tau_0)}{\lambda^2(\sigma)}\xi)}{\rho^{\frac{1}{2}}(\xi)}\frac{\rho^{\frac{1}{2}}(\xi)\sin[\nu\tau_0\xi^{\frac{1}{2}}]\cos[\lambda(\tau_0)\xi^{\frac{1}{2}}\int_{\tau_0}^\sigma\lambda^{-1}(u)\,du]}{\xi^{\frac{3}{4}}}\\
&\hspace{9cm}\cdot B(\tau_0, \sigma, \xi)\,d\xi d\sigma,
\end{align*}
where we have set 
\begin{align*}
&B(\tau_0, \sigma, \xi):\\& = \int_0^\infty \frac{F(\frac{\lambda^2(\tau_0)}{\lambda^2(\sigma)}\xi, \eta)\rho(\eta)}{\frac{\lambda^2(\tau_0)}{\lambda^2(\sigma)}\xi - \eta}\frac{\lambda^{\frac{3}{2}}(\sigma)}{\lambda^{\frac{3}{2}}(\tau_0)}\frac{\rho^{\frac{1}{2}}(\frac{\lambda^2(\sigma)}{\lambda^2(\tau_0)}\eta)}{\rho^{\frac{1}{2}}(\eta)}\frac{\sin[\lambda(\sigma)\eta^{\frac{1}{2}}\int_{\tau_0}^\sigma\lambda^{-1}(u)\,du]}{\eta^{\frac{1}{2}}}x_1(\frac{\lambda^2(\sigma)}{\lambda^2(\tau_0)}\eta)\,d\eta
\end{align*}
Observe that under our frequency restriction $\xi<1$ we get 
\[
\big|\frac{\rho^{\frac{1}{2}}(\frac{\lambda^2(\tau_0)}{\lambda^2(\sigma)}\xi)}{\rho^{\frac{1}{2}}(\xi)}\big|\lesssim (\frac{\lambda^2(\tau_0)}{\lambda^2(\sigma)})^{-\frac{1}{4}}. 
\]
Note that we can afford to absorb one singular factor $\xi^{-\frac{1}{2}}$ into $\sin[\nu\tau_0\xi^{\frac{1}{2}}]$ at the expense of a factor $\nu\tau_0$ which we can afford on account of the weight $\beta_{\nu}^2(\sigma)$ and $\sigma\geq \tau_0$. 
Then we further divide into the following cases: 
\\

{\bf{(1i.a)}}: {\it{Imbalanced frequencies $\frac{\lambda^2(\tau_0)}{\lambda^2(\sigma)}\xi\ll \eta$.}} Here we absorb one factor $\xi^{-\frac{1}{2}}$ at the expense of $\nu\tau_0$ and by switching the orders of integration we perform the $\xi$-integral first, which leads to a factor
\[
\nu\tau_0\frac{\lambda(\tau_0)}{\lambda(\sigma)}\beta_{\nu}^2(\sigma)\xi^{\frac{1}{2}}|_{\frac{\lambda^2(\tau_0)}{\lambda^2(\sigma)}\xi = c\eta}\sim \tau_0\beta_{\nu}^2(\sigma)\eta^{\frac{1}{2}}
\]
Including the extra factor $\eta^{\frac{1}{2}}$ here into the $B$-integral, we then reduce at fixed time $\sigma\geq \tau_0$ to bounding 
\begin{align*}
&\tau_0\beta_{\nu}^2(\sigma)\int_0^\infty \sup_{\frac{\lambda^2(\tau_0)}{\lambda^2(\sigma)}\xi\ll \eta}\big|\eta^{\frac12}\frac{F(\frac{\lambda^2(\tau_0)}{\lambda^2(\sigma)}\xi, \eta)\rho(\eta)}{\frac{\lambda^2(\tau_0)}{\lambda^2(\sigma)}\xi - \eta}\big|\frac{\lambda^{\frac{3}{2}}(\sigma)}{\lambda^{\frac{3}{2}}(\tau_0)}\big|\frac{\rho^{\frac{1}{2}}(\frac{\lambda^2(\sigma)}{\lambda^2(\tau_0)}\eta)}{\rho^{\frac{1}{2}}(\eta)}\\&\hspace{5cm}\cdot\eta^{-\frac12}\sin[\lambda(\sigma)\eta^{\frac{1}{2}}\int_{\tau_0}^\sigma\lambda^{-1}(u)\,du]x_1(\frac{\lambda^2(\sigma)}{\lambda^2(\tau_0)}\eta)\big|\,d\eta\\
&\lesssim \tau_0\beta_{\nu}^2(\sigma)\int_0^\infty\frac{\lambda(\sigma)}{\lambda(\tau_0)} \big(\eta^{-\frac{1}{2}}\chi_{\frac{\lambda^2(\sigma)}{\lambda^2(\tau_0)}\eta<1}\big|x_1(\frac{\lambda^2(\sigma)}{\lambda^2(\tau_0)}\eta)\big|\\&\hspace{4cm} + \eta^{-\frac{1}{2}}\chi_{\frac{\lambda^2(\sigma)}{\lambda^2(\tau_0)}\eta\geq 1}[\frac{\lambda^2(\sigma)}{\lambda^2(\tau_0)}\eta]^{\frac{1}{2}}\big|x_1(\frac{\lambda^2(\sigma)}{\lambda^2(\tau_0)}\eta)\big|\big)\,d\eta
\end{align*}
Using the Cauchy-Schwarz inequality, this last expression is then bounded by 
\begin{align*}
&\tau_0\beta_{\nu}^2(\sigma)\int_0^\infty\frac{\lambda(\sigma)}{\lambda(\tau_0)} \big(\eta^{-\frac{1}{2}}\chi_{\frac{\lambda^2(\sigma)}{\lambda^2(\tau_0)}\eta<1}\big|x_1(\frac{\lambda^2(\sigma)}{\lambda^2(\tau_0)}\eta)\big|\\&\hspace{4cm} + \eta^{-\frac{1}{2}}\chi_{\frac{\lambda^2(\sigma)}{\lambda^2(\tau_0)}\eta\geq 1}[\frac{\lambda^2(\sigma)}{\lambda^2(\tau_0)}\eta]^{\frac{1}{2}}\big|x_1(\frac{\lambda^2(\sigma)}{\lambda^2(\tau_0)}\eta)\big|\big)\,d\eta\\
&\lesssim \tau_0\beta_{\nu}^2(\sigma)\big(\frac{\lambda^2(\sigma)}{\lambda^2(\tau_0)}\big)^{0+}\big\|x_1\big\|_{S_2}
\end{align*} 
Inserting this bound into the preceding, we finally arrive at the $\sigma$-integral 
\[
\big\|x_1\big\|_{S_2}\int_{\tau_0}^\infty\tau_0\beta_{\nu}^2(\sigma)\big(\frac{\lambda^2(\sigma)}{\lambda^2(\tau_0)}\big)^{0+}\,d\sigma\lesssim \big\|x_1\big\|_{S_2}. 
\]

{\bf{(1i.b)}}: {\it{Balanced frequencies $\frac{\lambda^2(\tau_0)}{\lambda^2(\sigma)}\xi\sim\eta$.}} Using the assumption $\xi<1$ underlying case {\bf{(1i)}} we get here 
\[
\frac{\rho^{\frac{1}{2}}(\frac{\lambda^2(\sigma)}{\lambda^2(\tau_0)}\eta)}{\rho^{\frac{1}{2}}(\eta)}\lesssim \frac{\lambda^{\frac12}(\tau_0)}{\lambda^{\frac12}(\sigma)}
\]
and further due to \cite{KST}
\[
\big|F(\frac{\lambda^2(\tau_0)}{\lambda^2(\sigma)}\xi, \eta)\rho(\eta)\eta^{-\frac12}\big|\lesssim 1
\]
and the integral operator with kernel 
\[
\frac{F(\frac{\lambda^2(\tau_0)}{\lambda^2(\sigma)}\xi, \eta)\rho(\eta)\eta^{-\frac12}}{\frac{\lambda^2(\tau_0)}{\lambda^2(\sigma)}\xi - \eta}
\]
is $L^2$-bounded up to a factor $\frac{\lambda(\sigma)}{\lambda(\tau_0)}$. Furthermore, we bound 
\begin{equation}\label{eq:anotherbound1}
\big|\chi_{\xi<1}\frac{\rho^{\frac{1}{2}}(\frac{\lambda^2(\tau_0)}{\lambda^2(\sigma)}\xi)}{\rho^{\frac{1}{2}}(\xi)}\frac{\rho^{\frac{1}{2}}(\xi)\sin[\nu\tau_0\xi^{\frac{1}{2}}]\cos[\lambda(\tau_0)\xi^{\frac{1}{2}}\int_{\tau_0}^\sigma\lambda^{-1}(u)\,du]}{\xi^{\frac{3}{4}}}\big|\lesssim \nu\tau_0\frac{\lambda^{\frac12}(\sigma)}{\lambda^{\frac12}(\tau_0)}\xi^{-\frac{1}{2}}
\end{equation}
This is almost $L^2_{d\xi}$-integrable over $\xi<1$, except for including an extra factor $\xi^{0+}\sim \big(\frac{\lambda^2(\sigma)}{\lambda^2(\tau_0)}\eta\big)^{0+}$ which we can absorb into $x_1(\frac{\lambda^2(\sigma)}{\lambda^2(\tau_0)}\eta)$ due to the definition of $S_2$.  Using the Cauchy-Schwarz inequality to handle the integral over $\xi$, we obtain the same bound in situation {\bf{(1i.b)}} as in situation {\bf{(1i.a)}}. 
\\

{\bf{(1i.c)}}: {\it{Imbalanced frequencies $\frac{\lambda^2(\tau_0)}{\lambda^2(\sigma)}\xi\gg\eta$.}}  Here we can no longer absorb the singular factor $\eta^{-\frac{1}{2}}$ into $F(\frac{\lambda^2(\tau_0)}{\lambda^2(\sigma)}\xi, \eta)$, and so we have again to take advantage of the cancellation condition satisfied by $x_1$ in order to not lose a factor $\frac{\lambda(\sigma)}{\lambda(\tau_0)}$. Specifically, we write 
\begin{equation}\label{eq:ugly0}\begin{split}
\frac{F(\frac{\lambda^2(\tau_0)}{\lambda^2(\sigma)}\xi, \eta)\rho^{\frac{1}{2}}(\eta)}{\frac{\lambda^2(\tau_0)}{\lambda^2(\sigma)}\xi - \eta} &=\frac{F(\frac{\lambda^2(\tau_0)}{\lambda^2(\sigma)}\xi, 0)\rho^{\frac{1}{2}}(\eta)}{\frac{\lambda^2(\tau_0)}{\lambda^2(\sigma)}\xi} + O\big(\frac{\eta}{\frac{\lambda^2(\tau_0)}{\lambda^2(\sigma)}\xi}\big)\rho^{\frac{1}{2}}(\eta)\\
& =  c\frac{F(\frac{\lambda^2(\tau_0)}{\lambda^2(\sigma)}\xi, 0)\eta^{-\frac{1}{4}}}{\frac{\lambda^2(\tau_0)}{\lambda^2(\sigma)}\xi} + O\big(\frac{\eta}{\frac{\lambda^2(\tau_0)}{\lambda^2(\sigma)}\xi}\big)\rho^{\frac{1}{2}}(\eta)+O(\eta^{\frac{1}{4}}).
\end{split}\end{equation}
for a suitable constant $c$. 
\\
Substituting the first expression on the right into $B(\tau_0, \sigma, \xi)$ leads to the expression 
\begin{align*}
c\frac{F(\frac{\lambda^2(\tau_0)}{\lambda^2(\sigma)}\xi, 0)}{\frac{\lambda^2(\tau_0)}{\lambda^2(\sigma)}\xi}\int_0^{c\frac{\lambda^2(\tau_0)}{\lambda^2(\sigma)}\xi}\frac{\lambda^{\frac{3}{2}}(\sigma)}{\lambda^{\frac{3}{2}}(\tau_0)}\rho^{\frac{1}{2}}(\frac{\lambda^2(\sigma)}{\lambda^2(\tau_0)}\eta)\frac{\sin[\lambda(\sigma)\eta^{\frac{1}{2}}\int_{\tau_0}^\sigma\lambda^{-1}(u)\,du]}{\eta^{\frac{3}{4}}}x_1(\frac{\lambda^2(\sigma)}{\lambda^2(\tau_0)}\eta)\,d\eta
\end{align*}
and we split the integral into 
\begin{align}
&\int_0^{c\frac{\lambda^2(\tau_0)}{\lambda^2(\sigma)}\xi}\frac{\lambda^{\frac{3}{2}}(\sigma)}{\lambda^{\frac{3}{2}}(\tau_0)}\rho^{\frac{1}{2}}(\frac{\lambda^2(\sigma)}{\lambda^2(\tau_0)}\eta)\frac{\sin[\lambda(\sigma)\eta^{\frac{1}{2}}\int_{\tau_0}^\sigma\lambda^{-1}(u)\,du]}{\eta^{\frac{3}{4}}}x_1(\frac{\lambda^2(\sigma)}{\lambda^2(\tau_0)}\eta)\,d\eta\nonumber\\
& = \int_0^{\infty}\frac{\lambda^{\frac{3}{2}}(\sigma)}{\lambda^{\frac{3}{2}}(\tau_0)}\rho^{\frac{1}{2}}(\frac{\lambda^2(\sigma)}{\lambda^2(\tau_0)}\eta)\frac{\sin[\lambda(\sigma)\eta^{\frac{1}{2}}\int_{\tau_0}^\sigma\lambda^{-1}(u)\,du]}{\eta^{\frac{3}{4}}}x_1(\frac{\lambda^2(\sigma)}{\lambda^2(\tau_0)}\eta)\,d\eta\label{eq:ugly1}\\
&-\int_{c\frac{\lambda^2(\tau_0)}{\lambda^2(\sigma)}\xi}^\infty\frac{\lambda^{\frac{3}{2}}(\sigma)}{\lambda^{\frac{3}{2}}(\tau_0)}\rho^{\frac{1}{2}}(\frac{\lambda^2(\sigma)}{\lambda^2(\tau_0)}\eta)\frac{\sin[\lambda(\sigma)\eta^{\frac{1}{2}}\int_{\tau_0}^\sigma\lambda^{-1}(u)\,du]}{\eta^{\frac{3}{4}}}x_1(\frac{\lambda^2(\sigma)}{\lambda^2(\tau_0)}\eta)\,d\eta\label{eq:ugly2}
\end{align}
Then proceeding as at the end of the proof of Proposition~\ref{prop:lingrowthcond}, we have 
\begin{align*}
&\big| \int_0^{\infty}\frac{\lambda^{\frac{3}{2}}(\sigma)}{\lambda^{\frac{3}{2}}(\tau_0)}\rho^{\frac{1}{2}}(\frac{\lambda^2(\sigma)}{\lambda^2(\tau_0)}\eta)\frac{\sin[\lambda(\sigma)\eta^{\frac{1}{2}}\int_{\tau_0}^\sigma\lambda^{-1}(u)\,du]}{\eta^{\frac{3}{4}}}x_1(\frac{\lambda^2(\sigma)}{\lambda^2(\tau_0)}\eta)\,d\eta\big|\\
&\lesssim \sigma \big\|x_1\big\|_{S_2}. 
\end{align*}
For the second integral expression above, we obtain using the Cauchy-Schwarz inequality
\begin{align*}
&\big|\int_{c\frac{\lambda^2(\tau_0)}{\lambda^2(\sigma)}\xi}^\infty\frac{\lambda^{\frac{3}{2}}(\sigma)}{\lambda^{\frac{3}{2}}(\tau_0)}\rho^{\frac{1}{2}}(\frac{\lambda^2(\sigma)}{\lambda^2(\tau_0)}\eta)\frac{\sin[\lambda(\sigma)\eta^{\frac{1}{2}}\int_{\tau_0}^\sigma\lambda^{-1}(u)\,du]}{\eta^{\frac{3}{4}}}x_1(\frac{\lambda^2(\sigma)}{\lambda^2(\tau_0)}\eta)\,d\eta\big|\\
&\lesssim (\frac{\lambda(\sigma)}{\lambda(\tau_0)})^{1+}\xi^{-(\frac{1}{2}-)}\cdot  \big\|x_1\big\|_{S_2}.
\end{align*}
It follows that substituting either one of 
\[
c\frac{F(\frac{\lambda^2(\tau_0)}{\lambda^2(\sigma)}\xi, 0)}{\frac{\lambda^2(\tau_0)}{\lambda^2(\sigma)}\xi}\cdot\eqref{eq:ugly1},\,c\frac{F(\frac{\lambda^2(\tau_0)}{\lambda^2(\sigma)}\xi, 0)}{\frac{\lambda^2(\tau_0)}{\lambda^2(\sigma)}\xi}\cdot\eqref{eq:ugly2}
\]
for $B(\tau_0, \sigma, \xi)$ and using \eqref{eq:anotherbound1}, we deduce the bound 
\begin{align*}
&\int_{\tau_0}^\infty \frac{\lambda^{\frac{3}{2}}(\tau_0)}{\lambda^{\frac{3}{2}}(\sigma)}\beta_{\nu}^2(\sigma)\int_0^\infty\chi_{\xi<1}\frac{\rho^{\frac{1}{2}}(\frac{\lambda^2(\tau_0)}{\lambda^2(\sigma)}\xi)}{\rho^{\frac{1}{2}}(\xi)}\frac{\rho^{\frac{1}{2}}(\xi)\sin[\nu\tau_0\xi^{\frac{1}{2}}]\cos[\lambda(\tau_0)\xi^{\frac{1}{2}}\int_{\tau_0}^\sigma\lambda^{-1}(u)\,du]}{\xi^{\frac{3}{4}}}\\
&\hspace{9cm}\cdot B(\tau_0, \sigma, \xi)\,d\xi d\sigma\\
&\lesssim \big\|x_1\big\|_{S_2}\int_{\tau_0}^\infty[\log\tau_0\frac{\lambda(\tau_0)}{\lambda(\sigma)}\sigma^{-1} + \nu\tau_0\sigma^{-2}(\frac{\lambda(\sigma)}{\lambda(\tau_0)})^{0+}]\,d\sigma\lesssim \tau_0^{0+}\big\|x_1\big\|_{S_2}. 
\end{align*}
Further, substituting the error terms 
\[
O\big(\frac{\eta}{\frac{\lambda^2(\tau_0)}{\lambda^2(\sigma)}\xi}\big)\rho^{\frac{1}{2}}(\eta)+O(\eta^{\frac{1}{4}})
\]
in \eqref{eq:ugly0} for the factor $\frac{F(\frac{\lambda^2(\tau_0)}{\lambda^2(\sigma)}\xi, \eta)\rho^{\frac{1}{2}}(\eta)}{\frac{\lambda^2(\tau_0)}{\lambda^2(\sigma)}\xi - \eta} $ inside $B(\tau_0, \sigma, \xi)$ leads to terms which are handled just like case {\bf{(1i.a)}}, and we omit the details. This concludes case {\bf{(1i)}}.
\\

{\bf{(1ii)}}: {\it{Contribution of  $\beta_{\nu}^2(\sigma)\mathcal{K}_{cc}x^{(0)}(\sigma, \xi)$, high frequency $\xi\geq 1$}}. Proceeding in analogy to case {\it{(1i)}} and only considering the contribution of $x_1$, i. e. the second term in \eqref{eq:Kx0formula}, we need to estimate the expression 
 \begin{align*}
&\int_{\tau_0}^\infty \frac{\lambda^{\frac{3}{2}}(\tau_0)}{\lambda^{\frac{3}{2}}(\sigma)}\beta_{\nu}^2(\sigma)\int_0^\infty\chi_{\xi\geq 1}\frac{\rho^{\frac{1}{2}}(\frac{\lambda^2(\tau_0)}{\lambda^2(\sigma)}\xi)}{\rho^{\frac{1}{2}}(\xi)}\frac{\rho^{\frac{1}{2}}(\xi)\sin[\nu\tau_0\xi^{\frac{1}{2}}]\cos[\lambda(\tau_0)\xi^{\frac{1}{2}}\int_{\tau_0}^\sigma\lambda^{-1}(u)\,du]}{\xi^{\frac{3}{4}}}\\
&\hspace{9cm}\cdot B(\tau_0, \sigma, \xi)\,d\xi d\sigma,
\end{align*}
where again
\begin{align*}
&B(\tau_0, \sigma, \xi):\\& = \int_0^\infty \frac{F(\frac{\lambda^2(\tau_0)}{\lambda^2(\sigma)}\xi, \eta)\rho(\eta)}{\frac{\lambda^2(\tau_0)}{\lambda^2(\sigma)}\xi - \eta}\frac{\lambda^{\frac{3}{2}}(\sigma)}{\lambda^{\frac{3}{2}}(\tau_0)}\frac{\rho^{\frac{1}{2}}(\frac{\lambda^2(\sigma)}{\lambda^2(\tau_0)}\eta)}{\rho^{\frac{1}{2}}(\eta)}\frac{\sin[\lambda(\sigma)\eta^{\frac{1}{2}}\int_{\tau_0}^\sigma\lambda^{-1}(u)\,du]}{\eta^{\frac{1}{2}}}x_1(\frac{\lambda^2(\sigma)}{\lambda^2(\tau_0)}\eta)\,d\eta
\end{align*}

Here we split into the following cases: 
\\

{\bf{(1ii.a)}}:  {\it{intermediate output frequencies $\frac{\lambda^2(\tau_0)}{\lambda^2(\sigma)}\xi<1$.}} In this situation, we have 
\[
\frac{\rho^{\frac{1}{2}}(\frac{\lambda^2(\tau_0)}{\lambda^2(\sigma)}\xi)}{\rho^{\frac{1}{2}}(\xi)}\sim \frac{1}{(\frac{\lambda^2(\tau_0)}{\lambda^2(\sigma)}\xi))^{\frac{1}{4}}\xi^{\frac{1}{4}}},
\]
and so we reduce to estimating 
\begin{align*}
&\int_{\tau_0}^\infty \frac{\lambda(\tau_0)}{\lambda(\sigma)}\beta_{\nu}^2(\sigma)\int_0^\infty\chi_{\frac{\lambda^2(\sigma)}{\lambda^2(\tau_0)}>\xi\geq 1}\frac{\sin[\nu\tau_0\xi^{\frac{1}{2}}]\cos[\lambda(\tau_0)\xi^{\frac{1}{2}}\int_{\tau_0}^\sigma\lambda^{-1}(u)\,du]}{\xi}\\
&\hspace{9cm}\cdot B(\tau_0, \sigma, \xi)\,d\xi d\sigma
\end{align*}
To control the inner integral, we divide into cases ($\alpha$) - ($\gamma$) concerning the integration variable $\eta$ in the formula for $ B(\tau_0, \sigma, \xi)$: 
\\

{\bf{(1ii.a.$\alpha$)}}:  $\frac{\lambda^2(\tau_0)}{\lambda^2(\sigma)}\xi\ll \eta$. We pass to the new variable $\eta = \frac{\lambda^2(\tau_0)}{\lambda^2(\sigma)}\tilde{\eta}$, so here we have $\xi\ll\tilde{\eta}$. Then we get 
\begin{align*}
&B_{\frac{\lambda^2(\tau_0)}{\lambda^2(\sigma)}\xi\ll \eta}(\tau_0, \sigma, \xi):\\& = \int_{c\xi}^\infty \frac{F(\frac{\lambda^2(\tau_0)}{\lambda^2(\sigma)}\xi, \frac{\lambda^2(\tau_0)}{\lambda^2(\sigma)}\tilde{\eta})\rho^{\frac{1}{2}}(\frac{\lambda^2(\tau_0)}{\lambda^2(\sigma)}\tilde{\eta})}{[\xi - \tilde{\eta}][\frac{\lambda^2(\tau_0)}{\lambda^2(\sigma)}\tilde{\eta}]^{\frac{1}{2}}}\frac{\lambda^{\frac{3}{2}}(\sigma)}{\lambda^{\frac{3}{2}}(\tau_0)}\sin[\lambda(\tau_0)\tilde{\eta}^{\frac{1}{2}}\int_{\tau_0}^\sigma\lambda^{-1}(u)\,du]\rho^{\frac{1}{2}}(\tilde{\eta})x_1(\tilde{\eta})\,d\tilde{\eta}
\end{align*}
Then we get the bound 
\begin{align*}
\big|B_{\frac{\lambda^2(\tau_0)}{\lambda^2(\sigma)}\xi\ll \eta}(\tau_0, \sigma, \xi)\big|\lesssim &\int_{c\xi}^\infty\chi_{ \frac{\lambda^2(\tau_0)}{\lambda^2(\sigma)}\tilde{\eta}<1}\frac{\lambda(\sigma)}{\lambda(\tau_0)}\tilde{\eta}^{-\frac{3}{4}}|\rho^{\frac{1}{2}}(\tilde{\eta})x_1(\tilde{\eta})|\,d\tilde{\eta}\\
& + \int_{c\xi}^\infty\chi_{ \frac{\lambda^2(\tau_0)}{\lambda^2(\sigma)}\tilde{\eta}\geq 1}\frac{\lambda^{\frac{3}{2}}(\sigma)}{\lambda^{\frac{3}{2}}(\tau_0)}\tilde{\eta}^{-1}|\rho^{\frac{1}{2}}(\tilde{\eta})x_1(\tilde{\eta})|\,d\tilde{\eta}\\
\end{align*}
Note that for the first expression on the right we used that for $1\leq \xi\ll\tilde{\eta}$, $\sigma\geq \tau_0$,
\[
\chi_{ \frac{\lambda^2(\tau_0)}{\lambda^2(\sigma)}\tilde{\eta}<1}\big|\frac{F(\frac{\lambda^2(\tau_0)}{\lambda^2(\sigma)}\xi, \frac{\lambda^2(\tau_0)}{\lambda^2(\sigma)}\tilde{\eta})\rho^{\frac{1}{2}}(\frac{\lambda^2(\tau_0)}{\lambda^2(\sigma)}\tilde{\eta})}{[\xi - \tilde{\eta}][\frac{\lambda^2(\tau_0)}{\lambda^2(\sigma)}\tilde{\eta}]^{\frac{1}{2}}}\big|\lesssim \frac{\lambda^{\frac{1}{2}}(\tau_0)}{\lambda^{\frac{1}{2}}(\sigma)}\tilde{\eta}^{-\frac{3}{4}}. 
\]
It then follows that we have the bound 
\[
\big|B_{\frac{\lambda^2(\tau_0)}{\lambda^2(\sigma)}\xi\ll \eta}(\tau_0, \sigma, \xi)\big|\lesssim \frac{\lambda(\sigma)}{\lambda(\tau_0)}\xi^{-\frac{1}{2}}\big\|x_1\big\|_{S_2}. 
\]
Inserting this back into the full expression furnishes the bound 
\begin{align*}
&\big|\int_{\tau_0}^\infty \frac{\lambda(\tau_0)}{\lambda(\sigma)}\beta_{\nu}^2(\sigma)\int_0^\infty\chi_{\frac{\lambda^2(\sigma)}{\lambda^2(\tau_0)}>\xi\geq 1}\frac{\sin[\nu\tau_0\xi^{\frac{1}{2}}]\cos[\lambda(\tau_0)\xi^{\frac{1}{2}}\int_{\tau_0}^\sigma\lambda^{-1}(u)\,du]}{\xi}\\
&\hspace{8cm}\cdot B_{\frac{\lambda^2(\tau_0)}{\lambda^2(\sigma)}\xi\ll \eta}(\tau_0, \sigma, \xi)\,d\xi d\sigma\big|\\
&\lesssim \tau_0^{-1}\big\|x_1\big\|_{S_2}. 
\end{align*}

{\bf{(1ii.a.$\beta$)}}: $\frac{\lambda^2(\tau_0)}{\lambda^2(\sigma)}\xi\sim\eta$. Here we bound (note that here we have $\eta\lesssim 1$)
 \begin{align*}
 \big|F(\frac{\lambda^2(\tau_0)}{\lambda^2(\sigma)}\xi, \eta)\rho^{\frac{1}{2}}(\eta)\eta^{-\frac{1}{2}}\big|\lesssim \eta^{\frac{1}{4}}, 
 \end{align*}
 as well as 
 \[
 \big|\eta^{\frac{1}{4}}\rho^{\frac12}(\frac{\lambda^2(\sigma)}{\lambda^2(\tau_0)}\eta)x_1(\frac{\lambda^2(\sigma)}{\lambda^2(\tau_0)}\eta)\big|\lesssim \frac{\lambda^{\frac{1}{2}}(\tau_0)}{\lambda^{\frac{1}{2}}(\sigma)}[\frac{\lambda^2(\sigma)}{\lambda^2(\tau_0)}\eta]^{\frac{1}{2}}\big|x_1(\frac{\lambda^2(\sigma)}{\lambda^2(\tau_0)}\eta)\big|.
 \]
 It follows that 
 \begin{align*}
\frac{\lambda(\tau_0)}{\lambda(\sigma)}\big\|B_{\frac{\lambda^2(\tau_0)}{\lambda^2(\sigma)}\xi\sim\eta}(\tau_0, \sigma, \xi)\big\|_{L_{\xi}^2}\lesssim \big\|x_1\big\|_{S_2},
 \end{align*}
 and then using Cauchy-Schwarz we infer that 
 \begin{align*}
&\big|\int_{\tau_0}^\infty \frac{\lambda(\tau_0)}{\lambda(\sigma)}\beta_{\nu}^2(\sigma)\int_0^\infty\chi_{\frac{\lambda^2(\sigma)}{\lambda^2(\tau_0)}>\xi\geq 1}\frac{\sin[\nu\tau_0\xi^{\frac{1}{2}}]\cos[\lambda(\tau_0)\xi^{\frac{1}{2}}\int_{\tau_0}^\sigma\lambda^{-1}(u)\,du]}{\xi}\\
&\hspace{8cm}\cdot B_{\frac{\lambda^2(\tau_0)}{\lambda^2(\sigma)}\xi\sim\eta}(\tau_0, \sigma, \xi)\,d\xi d\sigma\big|\\
&\lesssim \tau_0^{-1}\big\|x_1\big\|_{S_2}. 
\end{align*}

{\bf{(1ii.a.$\gamma$)}}: $\frac{\lambda^2(\tau_0)}{\lambda^2(\sigma)}\xi\gg\eta$. Here $\eta<1$ is again in the singular range, and one concludes this case just as in {\it{(1i.c)}} earlier. 
\\

{\bf{(1ii.b)}}:  {\it{large  output frequencies $\frac{\lambda^2(\tau_0)}{\lambda^2(\sigma)}\xi\geq 1$.}} In this situation, we have 
\[
\frac{\rho^{\frac{1}{2}}(\frac{\lambda^2(\tau_0)}{\lambda^2(\sigma)}\xi)}{\rho^{\frac{1}{2}}(\xi)}\sim \frac{\lambda^{\frac12}(\tau_0)}{\lambda^{\frac12}(\sigma)}
\]
and so we reduce to estimating 
\begin{align*}
&\int_{\tau_0}^\infty \frac{\lambda^2(\tau_0)}{\lambda^2(\sigma)}\beta_{\nu}^2(\sigma)\int_0^\infty\chi_{\frac{\lambda^2(\sigma)}{\lambda^2(\tau_0)}\leq \xi}\frac{\sin[\nu\tau_0\xi^{\frac{1}{2}}]\cos[\lambda(\tau_0)\xi^{\frac{1}{2}}\int_{\tau_0}^\sigma\lambda^{-1}(u)\,du]}{\xi^{\frac12}}\\
&\hspace{9cm}\cdot B(\tau_0, \sigma, \xi)\,d\xi d\sigma
\end{align*}
Again one needs to split into the sub-cases {\it{(1ii.b.$\alpha$)}} - {\it{(1ii.b.$\gamma$)}} in analogy to the preceding case {\it{(1ii.a)}}, and since these are completely analogous, we only treat the case {\it{(1ii.b.$\alpha$)}} here: 
\\

{\bf{(1ii.b.$\alpha$)}}:  $1\leq \frac{\lambda^2(\tau_0)}{\lambda^2(\sigma)}\xi\ll \eta$. Note that in this situation we automatically have $\frac{\lambda^2(\tau_0)}{\lambda^2(\sigma)}\tilde{\eta}\geq 1$ (enacting the same change of variables as in the preceding case), and then 
\[
\big|B_{ \frac{\lambda^2(\tau_0)}{\lambda^2(\sigma)}\xi\ll \eta}(\tau_0, \sigma, \xi)\big|\lesssim \frac{\lambda^{\frac32}(\sigma)}{\lambda^{\frac32}(\tau_0)}\xi^{-\frac34} \big\|x_1\big\|_{S_2}, 
\]
which implies that 
\begin{align*}
&\int_{\tau_0}^\infty \frac{\lambda^2(\tau_0)}{\lambda^2(\sigma)}\beta_{\nu}^2(\sigma)\int_0^\infty\chi_{\frac{\lambda^2(\sigma)}{\lambda^2(\tau_0)}\leq \xi}\frac{\sin[\nu\tau_0\xi^{\frac{1}{2}}]\cos[\lambda(\tau_0)\xi^{\frac{1}{2}}\int_{\tau_0}^\sigma\lambda^{-1}(u)\,du]}{\xi^{\frac12}}\\
&\hspace{9cm}\cdot B(\tau_0, \sigma, \xi)\,d\xi d\sigma\\
&\lesssim \tau_0^{-1} \big\|x_1\big\|_{S_2}. 
\end{align*}

We have at this point finished controlling the cases {\it{(1i), (1ii)}}, as far as the contribution from $x_1$ was concerned, i. e. the contribution to $\beta_{\nu}^2(\sigma)\mathcal{K}_{cc}x^{(0)}(\sigma, \xi)$ by the second term in \eqref{eq:Kx0formula}.
\\

{\bf{(1iii)}}:  {\it{Contribution of  $\beta_{\nu}(\sigma)\mathcal{K}_{cc}\mathcal{D}_{\sigma}x^{(0)}(\sigma, \xi)$ to 
\[
\big|\int_0^\infty \frac{(\rho^{\frac{1}{2}}\widetilde{\triangle}\tilde{x}^{(1)}_1)(\xi)}{\xi^{\frac{3}{4}}}\sin[\nu\tau_0\xi^{\frac{1}{2}}]\,d\xi \big|.
\]
}}
Here we trade in better regularity of the expression defining $\mathcal{K}_{cc}\mathcal{D}_{\sigma}x^{(0)}$ in terms of $\eta$ for the weaker weight function $\beta_{\nu}(\sigma)$ which only decays like $\sigma^{-1}$, thus barely non-integrable as $\sigma\rightarrow\infty$. However, it is straightforward to check that in the analogues of situation {\it{(1ii)}}, i. e. the high frequency case $\xi\geq 1$, one always obtains bounds involving $[\frac{\lambda(\tau_0)}{\lambda(\sigma)}]^{\kappa_0}\sigma^{-1}$ for some $\kappa_0>0$, and these are integrable over $\sigma$. \\
In fact, this is also the case in the analogue of case {\it{(1i.c)}}, i. e. the case $\xi<1$, $\frac{\lambda^2(\tau_0)}{\lambda^2(\sigma)}\xi\gg \eta$, where this time we obtain (up to better error terms) the expression 
\begin{align*}
c\frac{F(\frac{\lambda^2(\tau_0)}{\lambda^2(\sigma)}\xi, 0)}{\frac{\lambda^2(\tau_0)}{\lambda^2(\sigma)}\xi}\int_0^{c\frac{\lambda^2(\tau_0)}{\lambda^2(\sigma)}\xi}\frac{\lambda^{\frac{3}{2}}(\sigma)}{\lambda^{\frac{3}{2}}(\tau_0)}\rho^{\frac{1}{2}}(\frac{\lambda^2(\sigma)}{\lambda^2(\tau_0)}\eta)\frac{\cos[\lambda(\sigma)\eta^{\frac{1}{2}}\int_{\tau_0}^\sigma\lambda^{-1}(u)\,du]}{\eta^{\frac{1}{4}}}x_1(\frac{\lambda^2(\sigma)}{\lambda^2(\tau_0)}\eta)\,d\eta
\end{align*}
Note that here we have 
\[
\big|\rho^{\frac{1}{2}}(\frac{\lambda^2(\sigma)}{\lambda^2(\tau_0)}\eta)\eta^{-\frac14}\big|\lesssim \frac{\lambda^{\frac12}(\tau_0)}{\lambda^{\frac12}(\sigma)}\eta^{-\frac12}, 
\]
and therefore (recalling $\xi<1$ in case {\it{(1i)}})
\begin{align*}
&\big|\int_0^{c\frac{\lambda^2(\tau_0)}{\lambda^2(\sigma)}\xi}\frac{\lambda^{\frac{3}{2}}(\sigma)}{\lambda^{\frac{3}{2}}(\tau_0)}\rho^{\frac{1}{2}}(\frac{\lambda^2(\sigma)}{\lambda^2(\tau_0)}\eta)\frac{\cos[\lambda(\sigma)\eta^{\frac{1}{2}}\int_{\tau_0}^\sigma\lambda^{-1}(u)\,du]}{\eta^{\frac{1}{4}}}x_1(\frac{\lambda^2(\sigma)}{\lambda^2(\tau_0)}\eta)\,d\eta\big|\\
&\lesssim \xi^{0+}\big\|x_1\big\|_{S_2}, 
\end{align*}
and the corresponding contribution to $\big|\int_0^\infty \frac{(\rho^{\frac{1}{2}}\widetilde{\triangle}\tilde{x}^{(1)}_1)(\xi)}{\xi^{\frac{3}{4}}}\sin[\nu\tau_0\xi^{\frac{1}{2}}]\,d\xi \big|$ is seen to be bounded by 
\[
\lesssim \big\|x_1\big\|_{S_2}\cdot\int_{\tau_0}^\infty \frac{\lambda(\tau_0)}{\lambda(\sigma)}\beta_{\nu}(\sigma)\,d\sigma\lesssim \big\|x_1\big\|_{S_2}. 
\]
It then remains to treat the analogues of situations  {\it{(1i.a)}},  {\it{(1i.b)}} from before, where we do have to argue slightly differently, and in fact use an integration by parts trick to reduce to expressions essentially as before. Thus re-defining for our purposes 
 \begin{align*}
 B(\tau_0, \sigma, \xi) &= \int_0^\infty \frac{F(\frac{\lambda^2(\tau_0)}{\lambda^2(\sigma)}\xi, \frac{\lambda^2(\tau_0)}{\lambda^2(\sigma)}\tilde{\eta})\rho(\frac{\lambda^2(\tau_0)}{\lambda^2(\sigma)}\tilde{\eta})}{\xi - \tilde{\eta}}(\mathcal{D}_{\sigma}x^{(0)})(\sigma, \frac{\lambda^2(\tau_0)}{\lambda^2(\sigma)}\tilde{\eta})\,d\tilde{\eta}\\
 & =  \int_0^\infty \frac{F(\frac{\lambda^2(\tau_0)}{\lambda^2(\sigma)}\xi, \frac{\lambda^2(\tau_0)}{\lambda^2(\sigma)}\tilde{\eta})\rho(\frac{\lambda^2(\tau_0)}{\lambda^2(\sigma)}\tilde{\eta})}{\xi - \tilde{\eta}}\partial_{\sigma}\big(x^{(0)}(\sigma, \frac{\lambda^2(\tau_0)}{\lambda^2(\sigma)}\tilde{\eta})\big)\,d\tilde{\eta},\\
 & + \text{error}
 \end{align*}
which corresponds of course to $\big(\mathcal{K}_{cc}\mathcal{D}_{\sigma}x^{(0)}\big)(\frac{\lambda^2(\tau_0)}{\lambda^2(\sigma)}\xi)$, and where the error term accounts for the discrepancy between $\mathcal{D}_{\sigma}$ and $\partial_{\sigma} - 2\beta_{\nu}(\sigma)\xi\partial_{\xi}$ and can be handled like the previous terms, we can then bound the expression 
 \begin{align*}
&\int_{\tau_0}^\infty \frac{\lambda^{\frac{3}{2}}(\tau_0)}{\lambda^{\frac{3}{2}}(\sigma)}\beta_{\nu}(\sigma)\int_0^\infty\chi_{\xi<1}\frac{\rho^{\frac{1}{2}}(\frac{\lambda^2(\tau_0)}{\lambda^2(\sigma)}\xi)}{\rho^{\frac{1}{2}}(\xi)}\frac{\rho^{\frac{1}{2}}(\xi)\sin[\nu\tau_0\xi^{\frac{1}{2}}]\cos[\lambda(\tau_0)\xi^{\frac{1}{2}}\int_{\tau_0}^\sigma\lambda^{-1}(u)\,du]}{\xi^{\frac{3}{4}}}\\
&\hspace{9cm}\cdot B_{\frac{\lambda^2(\tau_0)}{\lambda^2(\sigma)}\xi\lesssim \eta}(\tau_0, \sigma, \xi)\,d\xi d\sigma,
\end{align*}
via integration by parts in $\sigma$. This in fact either produces an additional factor $\sigma^{-1}$ or else (when $\partial_{\sigma}$ hits $\cos[\lambda(\tau_0)\xi^{\frac{1}{2}}\int_{\tau_0}^\sigma\lambda^{-1}(u)\,du]$) an additional factor 
\[
\frac{\lambda(\tau_0)}{\lambda(\sigma)}\xi^{\frac12}.
\]
In this latter case, we thus gain an additional term $\xi^{\frac12}$, whence we needn't absorb any factor into $\sin[\nu\tau_0\xi^{\frac12}]$ as before, and the additional factor $\frac{\lambda(\tau_0)}{\lambda(\sigma)}$ when combined with the weight $\beta_{\nu}(\sigma)$ is now integrable. 
This completes {\it{(1)}} of the proof of Lemma~\ref{lem:choiceofcorrection}. 
\\

{\bf{(2)}} {\it{Choice of $\triangle\tilde{\tilde{x}}_1^{(1)}$.}} Here we shall pick $\triangle\tilde{\tilde{x}}_1^{(1)} = \alpha \mathcal{F}(\chi_{R\leq C\tau_0}\phi(R, 0))$ for suitable $\alpha\in \R$. We easily get 
\[
\int_0^\infty\frac{(\rho^{\frac{1}{2}}\triangle\tilde{\tilde{x}}_1^{(1)})(\xi)}{\xi^{\frac{3}{4}}}\sin[\nu\tau_0\xi^{\frac{1}{2}}]\,d\xi\sim \alpha\tau_0,
\]
while we also have 
\begin{align*}
\big\|(0, \triangle\tilde{\tilde{x}}_1^{(1)})\big\|_{\tilde{S}}\lesssim \alpha\tau_0^{0+}. 
\end{align*}
In light of {\bf{(1)}}, there is a choice of $\alpha$ with $|\alpha|\lesssim \tau_0^{-1+}$, and such that the conclusion of the lemma is satisfied.
\end{proof}

{\bf{Step 4}}: {\it{Establishing the bound \eqref{eq:firstiterlow} for $\triangle x^{(1)}_{>\tau}$.}} Here we prove the bound 
\begin{equation}\label{eq:Step4bound}
\sup_{\tau\geq \tau_0}(\frac{\tau}{\tau_0})^{-\kappa}\big\|\chi_{\xi<1}\triangle x^{(1)}_{>\tau}(\tau, \xi)\big\|_{S_1} + \sup_{\tau\geq \tau_0}\big\|\chi_{\xi<1}\mathcal{D}_{\tau}\triangle x^{(1)}_{>\tau}(\tau, \xi)\big\|_{S_2}\lesssim  \big\|(x_0, x_1)\big\|_{\tilde{S}} + \big|x_{0d}\big|.
\end{equation}
with $\triangle x^{(1)}_{>\tau}$ defined as in the preceding step. Thus 
\begin{align*}
\triangle x^{(1)}_{>\tau} =  -\int_{\tau}^\infty\frac{\lambda^{\frac{3}{2}}(\tau)}{\lambda^{\frac{3}{2}}(\sigma)}\frac{\rho^{\frac{1}{2}}(\frac{\lambda^2(\tau)}{\lambda^2(\sigma)}\xi)}{\rho^{\frac{1}{2}}(\xi)}\frac{\sin[\lambda(\tau)\xi^{\frac{1}{2}}\int_{\tau}^\sigma\lambda^{-1}(u)\,du]}{\xi^{\frac{1}{2}}}\mathcal{R}(\sigma, \underline{x}^{(0)})(\frac{\lambda^2(\tau)}{\lambda^2(\sigma)}\xi)\,d\sigma
\end{align*}

Writing as before $\mathcal{R}(\sigma, \underline{x}^{(0)}) =  \mathcal{R}_1(\sigma, x^{(0)})(\xi) + \mathcal{R}_2(\sigma, x^{(0)}_d)(\xi)$, it is then straightforward to check, exploiting the exponential decay of $x^{(0)}_d$ (Lemma~\ref{lem:linhom}) as well as the fact that $\mathcal{R}_2(\sigma, x^{(0)}_d)(\xi)$ is smooth and rapidly decaying with respect to $\xi$ (\cite{KST}), that the contribution of $\mathcal{R}_2(\sigma, x^{(0)}_d)$ to \eqref{eq:Step4bound} is in effect bounded by $\lesssim \tau^{-(1-)}x_{0d}$, which is better than what we need. It thus suffices to consider the contribution of $\mathcal{R}_1(\sigma, x^{(0)})(\xi)$, which we again split into two main contributions: 
\\

{\bf{Step 4(i)}}: {\it{Contribution of  $\beta_{\nu}^2(\sigma)\mathcal{K}_{cc}x^{(0)}(\sigma, \xi)$ to $ \triangle x^{(1)}_{>\tau}(\tau, \xi)$}}. In light of the definition of the norm $\|\cdot\|_{S}$ (see \eqref{eq:Snorm}), we have to estimate the following expression 
\begin{align*}
\big\|\xi^{-0+}\int_{\tau}^\infty \frac{\lambda^{\frac{3}{2}}(\tau)}{\lambda^{\frac{3}{2}}(\sigma)}\frac{\rho^{\frac{1}{2}}(\frac{\lambda^2(\tau)}{\lambda^2(\sigma)}\xi)}{\rho^{\frac{1}{2}}(\xi)}\frac{\sin[\lambda(\tau)\xi^{\frac{1}{2}}\int_{\tau}^\sigma\lambda^{-1}(u)\,du]}{\xi^{\frac{1}{2}}}\beta_{\nu}^2(\sigma)\mathcal{K}_{cc}x^{(0)}(\sigma, \frac{\lambda^2(\tau)}{\lambda^2(\sigma)}\xi)\,d\sigma\big\|_{L^2_{d\xi}(\xi<1)}
\end{align*}

We follow roughly the steps in the proof of the preceding lemma, dividing into the cases {\it{(4(i).a)}} - {\it{(4(i).c)}}. We use 
\[
 \frac{\lambda^{\frac{3}{2}}(\tau)}{\lambda^{\frac{3}{2}}(\sigma)}\frac{\rho^{\frac{1}{2}}(\frac{\lambda^2(\tau)}{\lambda^2(\sigma)}\xi)}{\rho^{\frac{1}{2}}(\xi)}\sim  \frac{\lambda(\tau)}{\lambda(\sigma)}
\]
for $\xi<1$, $\sigma\geq \tau$. Then expanding the term $\mathcal{K}_{cc}x^{(0)}(\sigma, \xi)$ as before and only retaining the contribution of $x_1$ to $x^{(0)}$, i. e. the second term on the right of \eqref{eq:linhomparam1}, we quickly consider the different cases, referring to the algebra in the proof of the preceding lemma: 
\\

 {\bf{(4(i).a)}}: $\xi\cdot\frac{\lambda^2(\tau)}{\lambda^2(\sigma)}\ll \eta$. Localise $\xi\cdot\frac{\lambda^2(\tau)}{\lambda^2(\sigma)}\sim 2^J, \eta\sim 2^K$, so that $2^J\ll 2^K$. Then we bound 
 \begin{align*}
 &A_{J,K}: =\int_{\tau}^\infty \frac{\lambda^{\frac{3}{2}}(\tau)}{\lambda^{\frac{3}{2}}(\sigma)}\beta_{\nu}^2(\sigma)\chi_{\frac{\lambda^2(\tau)}{\lambda^2(\sigma)}\xi\sim 2^J}\frac{\rho^{\frac{1}{2}}(\frac{\lambda^2(\tau)}{\lambda^2(\sigma)}\xi)}{\rho^{\frac{1}{2}}(\xi)}\frac{\sin[\lambda(\tau)\xi^{\frac{1}{2}}\int_{\tau}^\sigma\lambda^{-1}(u)\,du]}{\xi^{\frac{1}{2}}}\\
&\hspace{9cm}\cdot B_{\eta\sim 2^K}(\tau, \tau_0, \sigma, \xi)\,d\sigma,
 \end{align*}
 where we have 
 \begin{align*}
& B_{\eta\sim 2^K}(\tau, \tau_0, \sigma, \xi):  \\
&=\int_0^\infty\chi_{\eta\sim 2^K} \frac{F(\frac{\lambda^2(\tau)}{\lambda^2(\sigma)}\xi, \eta)\rho(\eta)}{\frac{\lambda^2(\tau)}{\lambda^2(\sigma)}\xi - \eta}\frac{\lambda^{\frac{3}{2}}(\sigma)}{\lambda^{\frac{3}{2}}(\tau_0)}\frac{\rho^{\frac{1}{2}}(\frac{\lambda^2(\sigma)}{\lambda^2(\tau_0)}\eta)}{\rho^{\frac{1}{2}}(\eta)}\frac{\sin[\lambda(\sigma)\eta^{\frac{1}{2}}\int_{\tau_0}^\sigma\lambda^{-1}(u)\,du]}{\eta^{\frac{1}{2}}}x_1(\frac{\lambda^2(\sigma)}{\lambda^2(\tau_0)}\eta)\,d\eta
  \end{align*}
 Estimate (for $\sigma\geq \tau), \xi<1$)
 \[
 \big|\frac{\lambda^{\frac{3}{2}}(\tau)}{\lambda^{\frac{3}{2}}(\sigma)}\frac{\rho^{\frac{1}{2}}(\frac{\lambda^2(\tau)}{\lambda^2(\sigma)}\xi)}{\rho^{\frac{1}{2}}(\xi)}\frac{\sin[\lambda(\tau)\xi^{\frac{1}{2}}\int_{\tau}^\sigma\lambda^{-1}(u)\,du]}{\xi^{\frac{1}{2}}}\big|\lesssim \tau\frac{\lambda(\tau)}{\lambda(\sigma)}.
 \]
 Then fixing $\tau, \sigma$ for now, we get 
 \begin{align*}
 &\big\|\xi^{-0+}\frac{\lambda^{\frac{3}{2}}(\tau)}{\lambda^{\frac{3}{2}}(\sigma)}\beta_{\nu}^2(\sigma)\chi_{\frac{\lambda^2(\tau)}{\lambda^2(\sigma)}\xi\sim 2^J}\frac{\rho^{\frac{1}{2}}(\frac{\lambda^2(\tau)}{\lambda^2(\sigma)}\xi)}{\rho^{\frac{1}{2}}(\xi)}\frac{\sin[\lambda(\tau)\xi^{\frac{1}{2}}\int_{\tau}^\sigma\lambda^{-1}(u)\,du]}{\xi^{\frac{1}{2}}}\\
&\hspace{9cm}\cdot B_{\eta\sim 2^K}(\tau, \tau_0, \sigma, \xi)\big\|_{L^2_{d\xi}(\xi<1)}\\
&\lesssim \tau\beta_{\nu}^2(\sigma)2^{(\frac12-)(J-K)}\big|\int_0^\infty\chi_{\eta\sim 2^K} \eta^{-\frac{1}{2}}\frac{\lambda^{\frac{3}{2}}(\sigma)}{\lambda^{\frac{3}{2}}(\tau_0)}\frac{\rho^{\frac{1}{2}}(\frac{\lambda^2(\sigma)}{\lambda^2(\tau_0)}\eta)}{\rho^{\frac{1}{2}}(\eta)}\frac{\sin[\lambda(\sigma)\eta^{\frac{1}{2}}\int_{\tau_0}^\sigma\lambda^{-1}(u)\,du]}{\eta^{0+}}x_1(\frac{\lambda^2(\sigma)}{\lambda^2(\tau_0)}\eta)\,d\eta\big|\\
&\lesssim \tau\beta_{\nu}^2(\sigma)2^{(\frac12-)(J-K)}(\frac{\lambda(\sigma)}{\lambda(\tau_0)})^{2\delta_0}\big\|\chi_{\eta\sim \frac{\lambda^2(\sigma)}{\lambda^2(\tau_0)}2^K}x_1\big\|_{S_2}
 \end{align*}
It follows using a simple orthogonality argument that 
\begin{align*}
&\big\|\sum_{J\ll K}\xi^{-0+}\frac{\lambda^{\frac{3}{2}}(\tau)}{\lambda^{\frac{3}{2}}(\sigma)}\beta_{\nu}^2(\sigma)\chi_{\frac{\lambda^2(\tau)}{\lambda^2(\sigma)}\xi\sim 2^J}\frac{\rho^{\frac{1}{2}}(\frac{\lambda^2(\tau)}{\lambda^2(\sigma)}\xi)}{\rho^{\frac{1}{2}}(\xi)}\frac{\sin[\lambda(\tau)\xi^{\frac{1}{2}}\int_{\tau}^\sigma\lambda^{-1}(u)\,du]}{\xi^{\frac{1}{2}}}\\
&\hspace{9cm}\cdot B_{\eta\sim 2^K}(\tau, \tau_0, \sigma, \xi)\big\|_{L^2_{d\xi}(\xi<1)}\\
&\lesssim \tau\beta_{\nu}^2(\sigma)(\frac{\lambda(\sigma)}{\lambda(\tau_0)})^{2\delta_0}\big\|x_1\big\|_{S_2}, 
\end{align*}
 We conclude that 
 \begin{align*}
 \big\|\xi^{-0+}\sum_{J\ll K}A_{J,K}\big\|_{L^2_{d\xi}(\xi<1)}&\lesssim \big(\int_{\tau}^{\infty}\tau\beta_{\nu}^2(\sigma)(\frac{\lambda(\sigma)}{\lambda(\tau_0)})^{2\delta_0}\,d\sigma\big)\big\|x_1\big\|_{S_2}\\
 &\lesssim (\frac{\lambda(\tau)}{\lambda(\tau_0)})^{2\delta_0}\big\|x_1\big\|_{S_2}. 
 \end{align*}
 
  {\bf{(4(i).b)}}: {\it{diagonal case $\xi\cdot\frac{\lambda^2(\tau)}{\lambda^2(\sigma)}\sim \eta$}}. Here we bound 
  \begin{align*}
 \big| F(\frac{\lambda^2(\tau)}{\lambda^2(\sigma)}\xi, \eta)\rho(\eta)\eta^{-\frac{1}{2}}\big|\lesssim 1, 
  \end{align*}
  and furthermore we have 
  \begin{align*}
 &\big|\frac{\rho^{\frac{1}{2}}(\frac{\lambda^2(\sigma)}{\lambda^2(\tau_0)}\eta)}{\rho^{\frac{1}{2}}(\eta)}x_1(\frac{\lambda^2(\sigma)}{\lambda^2(\tau_0)}\eta)|\\&\lesssim \chi_{\frac{\lambda^2(\sigma)}{\lambda^2(\tau_0)}\eta<1}\frac{\lambda^{\frac{1}{2}}(\tau_0)}{\lambda^{\frac{1}{2}}(\sigma)}\big|x_1(\frac{\lambda^2(\sigma)}{\lambda^2(\tau_0)}\eta)\big| + \chi_{\frac{\lambda^2(\sigma)}{\lambda^2(\tau_0)}\eta\geq 1}\frac{\lambda^{\frac{1}{2}}(\sigma)}{\lambda^{\frac{1}{2}}(\tau_0)}\eta^{\frac{1}{2}}\big|x_1(\frac{\lambda^2(\sigma)}{\lambda^2(\tau_0)}\eta)\big|
  \end{align*}
  and then we get 
  \begin{align*}
  &\big\|\xi^{-0+}\int_0^\infty\frac{F(\frac{\lambda^2(\tau)}{\lambda^2(\sigma)}\xi, \eta)\rho(\eta)}{\frac{\lambda^2(\tau)}{\lambda^2(\sigma)}\xi - \eta}\frac{\lambda^{\frac{3}{2}}(\sigma)}{\lambda^{\frac{3}{2}}(\tau_0)}\frac{\rho^{\frac{1}{2}}(\frac{\lambda^2(\sigma)}{\lambda^2(\tau_0)}\eta)}{\rho^{\frac{1}{2}}(\eta)}\frac{\sin[\lambda(\sigma)\eta^{\frac{1}{2}}\int_{\tau_0}^\sigma\lambda^{-1}(u)\,du]}{\eta^{\frac{1}{2}}}x_1(\frac{\lambda^2(\sigma)}{\lambda^2(\tau_0)}\eta)\,d\eta\big\|_{L^2_{d\xi}}\\
  &\lesssim [\frac{\lambda(\tau)}{\lambda(\sigma)}]^{-1}(\frac{\lambda(\sigma)}{\lambda(\tau_0)})^{2\delta_0}\big\|x_1\big\|_{S_2}. 
  \end{align*}
This in turn furnishes the bound 
\begin{align*}
&\big\|\xi^{-0+}\int_{\tau}^\infty \frac{\lambda^{\frac{3}{2}}(\tau)}{\lambda^{\frac{3}{2}}(\sigma)}\beta_{\nu}^2(\sigma)\frac{\rho^{\frac{1}{2}}(\frac{\lambda^2(\tau)}{\lambda^2(\sigma)}\xi)}{\rho^{\frac{1}{2}}(\xi)}\frac{\sin[\lambda(\tau)\xi^{\frac{1}{2}}\int_{\tau}^\sigma\lambda^{-1}(u)\,du]}{\xi^{\frac{1}{2}}}\\
&\hspace{9cm}\cdot B_{\eta\sim \frac{\lambda^2(\tau)}{\lambda^2(\sigma)}\xi}(\tau, \tau_0, \sigma, \xi)\,d\sigma\big\|_{L^2_{d\xi}(\xi<1)}\\
&\lesssim (\frac{\lambda(\tau)}{\lambda(\tau_0)})^{2\delta_0}\big\|x_1\big\|_{S_2}. 
\end{align*}
 
  {\bf{(4(i).c)}}: {\it{off-diagonal case $\xi\cdot\frac{\lambda^2(\tau)}{\lambda^2(\sigma)}\gg\eta$}}. Here it is again the cancellation property of $x_1$ which plays in prominently. As this is the direct analogue of case {\it{(1i.c)}} in the proof of the preceding lemma, we shall be correspondingly brief: first, we may replace the kernel function 
 \[
 \frac{F(\frac{\lambda^2(\tau)}{\lambda^2(\sigma)}\xi, \eta)\rho^{\frac12}(\eta)}{\frac{\lambda^2(\tau)}{\lambda^2(\sigma)}\xi - \eta}
 \]
 by $\frac{F(\frac{\lambda^2(\tau)}{\lambda^2(\sigma)}\xi, 0)}{\frac{\lambda^2(\tau)}{\lambda^2(\sigma)}\xi}\eta^{-\frac{1}{4}}$ up to better errors which can be handled as in the preceding two cases.  Then we reduce to bounding the expression 
 \begin{align*}
 &\big\|\xi^{-0+}\int_{\tau}^\infty  \frac{\lambda(\tau)}{\lambda(\sigma)}\beta_{\nu}^2(\sigma)\frac{\sin[\lambda(\tau)\xi^{\frac{1}{2}}\int_{\tau}^\sigma\lambda^{-1}(u)\,du]}{\xi^{\frac{1}{2}}}\frac{F(\frac{\lambda^2(\tau)}{\lambda^2(\sigma)}\xi, 0)}{\frac{\lambda^2(\tau)}{\lambda^2(\sigma)}\xi}\\&\hspace{3cm}\cdot\int_0^\infty \frac{\lambda(\sigma)}{\lambda(\tau_0)}\frac{\sin[\lambda(\tau_0)\tilde{\eta}\int_{\tau_0}^\sigma\lambda^{-1}(u)\,du]}{\tilde{\eta}^{\frac{3}{4}}}(\rho^{\frac{1}{2}}x_1)(\tilde{\eta})\,d\tilde{\eta}\big\|_{L^2_{d\xi}(\xi<1)}
 \end{align*}
 Then note that since 
 \[
 \tilde{\eta} = \frac{\lambda^2(\sigma)}{\lambda^2(\tau_0)}\eta\ll \frac{\lambda^2(\sigma)}{\lambda^2(\tau_0)}\frac{\lambda^2(\tau)}{\lambda^2(\sigma)}\xi = \frac{\lambda^2(\tau)}{\lambda^2(\tau_0)}\xi, 
\]
we can absorb the extra weight $\xi^{-0+}$ into $x_1$ at the cost of a factor $(\frac{\lambda(\tau)}{\lambda(\tau_0)})^{2\delta_0}$, and then the argument used to prove Proposition~\ref{prop:lingrowthcond} in conjunction with a simple orthogonality argument yields 
\begin{align*}
&\big\|\xi^{-0+}\int_{\tau}^\infty  \frac{\lambda(\tau)}{\lambda(\sigma)}\beta_{\nu}^2(\sigma)\frac{\sin[\lambda(\tau)\xi^{\frac{1}{2}}\int_{\tau}^\sigma\lambda^{-1}(u)\,du]}{\xi^{\frac{1}{2}}}\frac{F(\frac{\lambda^2(\tau)}{\lambda^2(\sigma)}\xi, 0)}{\frac{\lambda^2(\tau)}{\lambda^2(\sigma)}\xi}\\&\hspace{3cm}\cdot\int_0^\infty \frac{\lambda(\sigma)}{\lambda(\tau_0)}\frac{\sin[\lambda(\tau_0)\tilde{\eta}\int_{\tau_0}^\sigma\lambda^{-1}(u)\,du]}{\tilde{\eta}^{\frac{3}{4}}}(\rho^{\frac{1}{2}}x_1)(\tilde{\eta})\,d\tilde{\eta}\big\|_{L^2_{d\xi}(\xi<1)}\\
&\lesssim (\frac{\lambda(\tau)}{\lambda(\tau_0)})^{2\delta_0}\big(\int_{\tau}^\infty \frac{\lambda(\tau)}{\lambda(\sigma)}\beta_{\nu}^2\cdot\sigma\,d\sigma\big)\big\|x_1\big\|_{S_2}\lesssim (\frac{\lambda(\tau)}{\lambda(\tau_0)})^{2\delta_0}\big\|x_1\big\|_{S_2}. 
\end{align*}
This concludes {\it{Step 4(i)}}. 
\\

{\bf{Step 4(ii)}}: {\it{Contribution of  $\beta_{\nu}(\sigma)\mathcal{K}_{cc}\mathcal{D}_{\sigma}x^{(0)}(\sigma, \xi)$ to $ \triangle x^{(1)}_{>\tau}(\tau, \xi)$}}. We again limit the estimates here to the contribution to $x^{(0)}$ by $x_1$. In light of the definition of the norm $\|\cdot\|_{S}$, we have to estimate the following expression 
\begin{align*}
\big\|\xi^{-0+}\int_{\tau}^\infty \frac{\lambda^{\frac{3}{2}}(\tau)}{\lambda^{\frac{3}{2}}(\sigma)}\frac{\rho^{\frac{1}{2}}(\frac{\lambda^2(\tau)}{\lambda^2(\sigma)}\xi)}{\rho^{\frac{1}{2}}(\xi)}\frac{\sin[\lambda(\tau)\xi^{\frac{1}{2}}\int_{\tau}^\sigma\lambda^{-1}(u)\,du]}{\xi^{\frac{1}{2}}}\beta_{\nu}(\sigma)\mathcal{K}_{cc}\mathcal{D}_{\sigma}x^{(0)}(\sigma, \frac{\lambda^2(\tau)}{\lambda^2(\sigma)}\xi)\,d\sigma\big\|_{L^2_{d\xi}(\xi<1)}
\end{align*}
We divide into cases (a) - (c) as in the preceding. 
\\

 {\bf{(4(ii).a)}}: $\xi\cdot\frac{\lambda^2(\tau)}{\lambda^2(\sigma)}\ll \eta$. We have to perform an integration by parts in $\sigma$. Use 
\begin{align*}
\mathcal{K}_{cc}\mathcal{D}_{\sigma}x^{(0)}(\sigma, \frac{\lambda^2(\tau)}{\lambda^2(\sigma)}\xi) &= \int_0^\infty \frac{F(\frac{\lambda^2(\tau)}{\lambda^2(\sigma)}\xi, \frac{\lambda^2(\tau)}{\lambda^2(\sigma)}\tilde{\eta})\rho(\frac{\lambda^2(\tau)}{\lambda^2(\sigma)}\tilde{\eta})}{\xi - \tilde{\eta}}(\mathcal{D}_\sigma x^{(0)})(\sigma, \frac{\lambda^2(\tau)}{\lambda^2(\sigma)}\tilde{\eta})\,d\tilde{\eta}\\
& = \int_0^\infty \frac{F(\frac{\lambda^2(\tau)}{\lambda^2(\sigma)}\xi, \frac{\lambda^2(\tau)}{\lambda^2(\sigma)}\tilde{\eta})\rho(\frac{\lambda^2(\tau)}{\lambda^2(\sigma)}\tilde{\eta})}{\xi - \tilde{\eta}}\partial_{\sigma}\big(x^{(0)}(\sigma, \frac{\lambda^2(\tau)}{\lambda^2(\sigma)}\tilde{\eta})\big)\,d\tilde{\eta}\\
& + \text{error},
\end{align*}
with the term 'error' again coming from the difference $\mathcal{D}_{\sigma} - (\partial_{\sigma} - 2\beta_{\nu}(\sigma)\xi\partial_{\xi})$ and hence being treatable like the contribution by $\beta_{\nu}^2(\sigma)\mathcal{K}_cc x^{(0)}$. 
 Then integration by parts either produces an extra factor $\sigma^{-1}$ which transforms the term into one of the type considered in the preceding (in the situation {\it{(4(i).a)}}), or else one replaces 
 \[
 \frac{\sin[\lambda(\tau)\xi^{\frac{1}{2}}\int_{\tau}^\sigma\lambda^{-1}(u)\,du]}{\xi^{\frac{1}{2}}}
 \]
 by 
 \[
\frac{\lambda(\tau)}{\lambda(\sigma)}\cos[\lambda(\tau)\xi^{\frac{1}{2}}\int_{\tau}^\sigma\lambda^{-1}(u)\,du]
 \]
 But then if one proceeds as in case {\it{(4(i).a)}} above one obtains the bound 
 \begin{align*}
&\big\|\xi^{-0+}\int_{\tau}^\infty \frac{\lambda^{\frac{5}{2}}(\tau)}{\lambda^{\frac{5}{2}}(\sigma)}\frac{\rho^{\frac{1}{2}}(\frac{\lambda^2(\tau)}{\lambda^2(\sigma)}\xi)}{\rho^{\frac{1}{2}}(\xi)}\cos[\lambda(\tau)\xi^{\frac{1}{2}}\int_{\tau}^\sigma\lambda^{-1}(u)\,du]\beta_{\nu}(\sigma)\mathcal{K}_{cc}x^{(0)}(\sigma, \frac{\lambda^2(\tau)}{\lambda^2(\sigma)}\xi)\,d\sigma\big\|_{L^2_{d\xi}(\xi<1)}\\
&\lesssim \big\|x_1\big\|_{S_2}\int_\tau^\infty \frac{\lambda(\tau)}{\lambda(\sigma)}\beta_{\nu}(\sigma)(\frac{\lambda(\sigma)}{\lambda(\tau_0)})^{2\delta_0}\,d\sigma\lesssim (\frac{\lambda(\tau)}{\lambda(\tau_0)})^{2\delta_0}\big\|x_1\big\|_{S_2}. 
 \end{align*}

 {\bf{(4(ii).b)}}: $\xi\cdot\frac{\lambda^2(\tau)}{\lambda^2(\sigma)}\sim\eta$. Here we simply trade a factor $\eta^{\frac12}$ for a factor $\frac{\lambda(\tau)}{\lambda(\sigma)}\xi^{\frac12}$ and then proceed as in  {\it{(4(i).b)}}. 
\\

 {\bf{(4(ii).c)}}: $\xi\cdot\frac{\lambda^2(\tau)}{\lambda^2(\sigma)}\gg\eta$. Here we still integrate by parts with respect to $\sigma$ but we need to combine the oscillatory factors correctly. Specifically, write
 \[
 \mathcal{K}\mathcal{D}_{\sigma}x^{(0)}(\sigma, \frac{\lambda^2(\tau)}{\lambda^2(\sigma)}\xi) 
 \]
 in the form 
 \begin{align*}
 \int_0^\infty\frac{F(\frac{\lambda^2(\tau)}{\lambda^2(\sigma)}\xi, \eta)\rho(\eta)}{\frac{\lambda^2(\tau)}{\lambda^2(\sigma)}\xi - \eta}\frac{\lambda^{\frac{3}{2}}(\sigma)}{\lambda^{\frac{3}{2}}(\tau_0)}\frac{\rho^{\frac{1}{2}}(\frac{\lambda^2(\sigma)}{\lambda^2(\tau_0)}\eta)}{\rho^{\frac{1}{2}}(\eta)}\cos[\lambda(\sigma)\eta^{\frac{1}{2}}\int_{\tau_0}^\sigma\lambda^{-1}(u)\,du]x_1(\frac{\lambda^2(\sigma)}{\lambda^2(\tau_0)}\eta)\,d\eta,
 \end{align*}
 which then gets localised to (with a smooth cutoff $\chi_{\xi\cdot\frac{\lambda^2(\tau)}{\lambda^2(\sigma)}\gg\eta}$, say)
  \begin{align*}
 &\int_0^\infty\chi_{\xi\cdot\frac{\lambda^2(\tau)}{\lambda^2(\sigma)}\gg\eta}\frac{F(\frac{\lambda^2(\tau)}{\lambda^2(\sigma)}\xi, \eta)\rho(\eta)}{\frac{\lambda^2(\tau)}{\lambda^2(\sigma)}\xi - \eta}\frac{\lambda^{\frac{3}{2}}(\sigma)}{\lambda^{\frac{3}{2}}(\tau_0)}\frac{\rho^{\frac{1}{2}}(\frac{\lambda^2(\sigma)}{\lambda^2(\tau_0)}\eta)}{\rho^{\frac{1}{2}}(\eta)}\\&\hspace{4cm}\cdot\cos[\lambda(\sigma)\eta^{\frac{1}{2}}\int_{\tau_0}^\sigma\lambda^{-1}(u)\,du]x_1(\frac{\lambda^2(\sigma)}{\lambda^2(\tau_0)}\eta)\,d\eta,
 \end{align*}
and re-write the integral in terms of the variable $\tilde{\eta}: = \frac{\lambda^2(\sigma)}{\lambda^2(\tau_0)}\eta$. Then combine the oscillatory factor 
\[
\cos[\lambda(\tau_0)\tilde{\eta}^{\frac{1}{2}}\int_{\tau_0}^\sigma\lambda^{-1}(u)\,du]
\]
with the outer oscillatory factor $\sin[\lambda(\tau)\xi^{\frac{1}{2}}\int_{\tau}^\sigma\lambda^{-1}(u)\,du]$, resulting in an expression which is a linear combination of exponentials with phase 
\[
\lambda(\tau)\xi^{\frac{1}{2}}\int_{\tau}^\sigma\lambda^{-1}(u)\,du\pm \lambda(\tau_0)\tilde{\eta}^{\frac{1}{2}}\int_{\tau_0}^\sigma\lambda^{-1}(u)\,du, 
\]
and perform an integration by parts with respect to $\sigma$. We gain an additional $\sigma^{-1}$ at the expense of an extra $[\xi\cdot\frac{\lambda^2(\tau)}{\lambda^2(\sigma)}]^{-1}$, which however gets compensated for by a gain from the $\tilde{\eta}$-integral. The details here are then just as in case {\it{(4(i).a)}}. This concludes case {\it{4(ii)}} as far as the contribution to $x^{(0)}$ of the second term on the right in \eqref{eq:linhomparam1} is concerned. 
\\

The preceding concludes the estimate for $\sup_{\tau\geq \tau_0}(\frac{\tau}{\tau_0})^{-0+}\big\|\chi_{\xi<1}\triangle \tilde{x}^{(1)}_{>\tau}(\tau, \cdot)\big\|_{S_1}$. One argues similarly to get the desired bound for $\sup_{\tau\geq \tau_0}\big\|\chi_{\xi<1}\mathcal{D}_{\tau}\triangle \tilde{x}^{(1)}_{>\tau}(\tau, \cdot)\big\|_{S_2}$. Note that the loss of $(\frac{\tau}{\tau_0})^{2\delta_0(1+\nu^{-1})}$ for the former expression comes from the singular weight $\xi^{-0+}$ in conjunction with the singular factor 
\[
 \frac{\sin[\lambda(\tau)\xi^{\frac{1}{2}}\int_{\tau}^\sigma\lambda^{-1}(u)\,du]}{\xi^{\frac{1}{2}}}
 \]
which just barely fails to be in $L^2_{d\xi}$. Applying $\mathcal{D}_{\tau}$ de-singularizes the preceding term, whence no loss $(\frac{\tau}{\tau_0})^{2\delta_0(1+\nu^{-1})}$ results for $\big\|\chi_{\xi<1}\mathcal{D}_{\tau}\triangle \tilde{x}^{(1)}_{>\tau}(\tau, \cdot)\big\|_{S_2}$. 
\\

{\bf{Step 5}}: {\it{Control over the initial data $\big(\triangle \tilde{x}^{(1)}_0(\xi), \triangle \tilde{x}^{(1)}_1(\xi)\big)$.}} Here we prove the estimate 
\[
\big\|\big(\triangle \tilde{x}^{(1)}_0(\xi), \triangle \tilde{x}^{(1)}_1(\xi)\big)\big\|_{S}\lesssim  \big\|(x_0, x_1)\big\|_{\tilde{S}} + \big|x_{0d}\big|. 
\]
Observe that in light of the definition of $\big(\triangle \tilde{x}^{(1)}_0(\xi), \triangle \tilde{x}^{(1)}_1(\xi)\big)$ preceding the statement of Lemma~\ref{lem:choiceofcorrection} and the preceding {\it{Step 4}}, it suffices to prove the high frequency bound
\[
\big\|\chi_{\xi>1}\big(\triangle \tilde{x}^{(1)}_0(\xi), \triangle \tilde{x}^{(1)}_1(\xi)\big)\big\|_{S}\lesssim  \big\|(x_0, x_1)\big\|_{\tilde{S}} + \big|x_{0d}\big|. 
\]
Also, we may immediately reduce to the contribution of the continuous spectral part of the data $(x_0, x_1)$, the discrete part again being straightforward on account if its exponential decay. 
We consider here the term $\triangle \tilde{x}^{(1)}_0$, and more specifically the contribution of the source term $x_1$ via \eqref{eq:linhomparam1} to it, the other cases being treated similarly. As before, we split the term $\mathcal{R}_1(\sigma, \underline{x}^{(0)})(\frac{\lambda^2(\tau_0)}{\lambda^2(\sigma)}\xi)$ involved in the definition of $\triangle\tilde{x}^{(1)}_0(\xi)$ into two parts, by invoking the schematic splitting \eqref{eq:R_1structure}:
\\

{\bf{Step 5(i)}}: {\it{Contribution of $ \beta_{\nu}^2(\tau)\mathcal{K}_{cc}x^{(0)}$}}. Thus, to begin with, we need to bound the expression 
\begin{align*}
\big\|\langle\xi\rangle^{1+}\int_{\tau_0}^\infty\frac{\lambda^{\frac{3}{2}}(\tau_0)}{\lambda^{\frac{3}{2}}(\sigma)}\frac{\rho^{\frac{1}{2}}(\frac{\lambda^2(\tau_0)}{\lambda^2(\sigma)}\xi)}{\rho^{\frac{1}{2}}(\xi)}\frac{\sin[\lambda(\tau_0)\xi^{\frac{1}{2}}\int_{\tau_0}^\sigma\lambda^{-1}(u)\,du]}{\xi^{\frac{1}{2}}}\beta_{\nu}^2(\sigma)B(\sigma, \tau_0, \xi)\,d\sigma\big\|_{L^2_{d\xi}(\xi>1)}, 
\end{align*}
where we put 
\begin{align*}
&B(\sigma, \tau_0, \xi)\\&
 = \int_0^\infty \frac{F(\frac{\lambda^2(\tau_0)}{\lambda^2(\sigma)}\xi, \eta)\rho(\eta)}{\frac{\lambda^2(\tau_0)}{\lambda^2(\sigma)}\xi - \eta}\frac{\lambda^{\frac{3}{2}}(\sigma)}{\lambda^{\frac{3}{2}}(\tau_0)}\frac{\rho^{\frac{1}{2}}(\frac{\lambda^2(\sigma)}{\lambda^2(\tau_0)}\eta)}{\rho^{\frac{1}{2}}(\eta)}\frac{\sin[\lambda(\sigma)\eta^{\frac{1}{2}}\int_{\tau_0}^\sigma\lambda^{-1}(u)\,du]}{\eta^{\frac{1}{2}}}x_1(\frac{\lambda^2(\sigma)}{\lambda^2(\tau_0)}\eta). 
\end{align*}
{\bf{(5(i).a)}}: $\frac{\lambda^2(\tau_0)}{\lambda^2(\sigma)}\xi\ll \eta<1$. Here we get (in light of $\xi>1$)
\[
\big|\langle\xi\rangle^{1+}\frac{\rho^{\frac{1}{2}}(\frac{\lambda^2(\tau_0)}{\lambda^2(\sigma)}\xi)}{\rho^{\frac{1}{2}}(\xi)}\frac{\sin[\lambda(\tau_0)\xi^{\frac{1}{2}}\int_{\tau_0}^\sigma\lambda^{-1}(u)\,du]}{\xi^{\frac{1}{2}}}\big|
\lesssim [\frac{\lambda^2(\tau_0)}{\lambda^2(\sigma)}]^{-\frac{1}{4}}\xi^{0+}
\]
Furthermore, we have in our current situation
\begin{align*}
\big|\frac{F(\frac{\lambda^2(\tau_0)}{\lambda^2(\sigma)}\xi, \eta)\rho(\eta)}{\frac{\lambda^2(\tau_0)}{\lambda^2(\sigma)}\xi - \eta}\frac{\rho^{\frac{1}{2}}(\frac{\lambda^2(\sigma)}{\lambda^2(\tau_0)}\eta)}{\rho^{\frac{1}{2}}(\eta)}\frac{\sin[\lambda(\sigma)\eta^{\frac{1}{2}}\int_{\tau_0}^\sigma\lambda^{-1}(u)\,du]}{\eta^{\frac{1}{2}}}\big|\lesssim [\frac{\lambda^2(\sigma)}{\lambda^2(\tau_0)}]^{\frac{1}{4}}\eta^{-\frac{1}{2}}. 
\end{align*}
Then taking the factor $\frac{\lambda^{\frac{3}{2}}(\tau_0)}{\lambda^{\frac{3}{2}}(\sigma)}$ into account, the $L^2_{d\xi}$ integral leads to a gain of 
\[
\frac{\lambda(\tau_0)}{\lambda(\sigma)}\xi^{\frac{1}{2}+}|_{\frac{\lambda^2(\tau_0)}{\lambda^2(\sigma)}\xi = c\eta}
\]
whence to a gain of $(\frac{\lambda(\sigma)}{\lambda(\tau_0)})^{2\delta_0}\eta^{\frac{1}{2}+}$, and so we reduce to the $\eta$-integral 
\begin{align*}
\int_{[\frac{\lambda^{2}(\sigma)}{\lambda^{2}(\tau_0)}]^{-1}}^1\eta^{0+}\frac{\lambda^{2(1+\delta_0)}(\sigma)}{\lambda^{2(1+\delta_0)}(\tau_0)}x_1(\frac{\lambda^2(\sigma)}{\lambda^2(\tau_0)}\eta)\,d\eta = \int_{[\frac{\lambda^2(\sigma)}{\lambda^2(\tau_0)}]^{-1}}^1\eta^{-\frac{1}{2+}}\frac{\lambda^{2(1+\delta_0)}(\sigma)}{\lambda^{2(1+\delta_0)}(\tau_0)}\eta^{\frac{1}{2}}x_1(\frac{\lambda^2(\sigma)}{\lambda^2(\tau_0)}\eta)\,d\eta,
\end{align*}
which via Cauchy-Schwarz can be bounded by 
\[
\big|\int_{[\frac{\lambda^2(\sigma)}{\lambda^2(\tau_0)}]^{-1}}^1\eta^{-\frac{1}{2+}}\frac{\lambda^{2(1+\delta_0)}(\sigma)}{\lambda^{2(1+\delta_0)}(\tau_0)}\eta^{\frac{1}{2}}x_1(\frac{\lambda^2(\sigma)}{\lambda^2(\tau_0)}\eta)\,d\eta\big|\lesssim (\frac{\lambda(\sigma)}{\lambda(\tau_0)})^{2\delta_0}\big\|x_1\big\|_{S_2}. 
\]
It follows that the full expression is bounded by 
\[
\big(\int_{\tau_0}^\infty(\frac{\lambda(\sigma)}{\lambda(\tau_0)})^{2\delta_0}\beta_{\nu}^2(\sigma)\,d\sigma\big)\big\|x_1\big\|_{S_2}\lesssim \tau_0^{-1}\big\|x_1\big\|_{S_2}. 
\]

{\bf{(5(i).b)}}: $\frac{\lambda^2(\tau_0)}{\lambda^2(\sigma)}\xi\ll \eta,\,\eta>1$.  Here, we use 
\begin{align*}
\big|\langle\xi\rangle^{1+}\frac{\rho^{\frac{1}{2}}(\frac{\lambda^2(\tau_0)}{\lambda^2(\sigma)}\xi)}{\rho^{\frac{1}{2}}(\xi)}\frac{\sin[\lambda(\tau_0)\xi^{\frac{1}{2}}\int_{\tau_0}^\sigma\lambda^{-1}(u)\,du]}{\xi^{\frac{1}{2}}}\big|&\lesssim [\frac{\lambda^2(\tau_0)}{\lambda^2(\sigma)}]^{-\frac{1}{4}}\xi^{0+}\chi_{\frac{\lambda^2(\tau_0)}{\lambda^2(\sigma)}\xi<1}\\
& + [\frac{\lambda^2(\tau_0)}{\lambda^2(\sigma)}]^{\frac{1}{4}}\xi^{\frac{1}{2}+}\chi_{\frac{\lambda^2(\tau_0)}{\lambda^2(\sigma)}\xi>1}\\
\end{align*}
and also (under our current assumptions)
\begin{align*}
\big|\frac{F(\frac{\lambda^2(\tau_0)}{\lambda^2(\sigma)}\xi, \eta)\rho(\eta)}{\frac{\lambda^2(\tau_0)}{\lambda^2(\sigma)}\xi - \eta}\frac{\rho^{\frac{1}{2}}(\frac{\lambda^2(\sigma)}{\lambda^2(\tau_0)}\eta)}{\rho^{\frac{1}{2}}(\eta)}\frac{\sin[\lambda(\sigma)\eta^{\frac{1}{2}}\int_{\tau_0}^\sigma\lambda^{-1}(u)\,du]}{\eta^{\frac{1}{2}}}\big|\lesssim\langle \frac{\lambda^2(\tau_0)}{\lambda^2(\sigma)}\xi\rangle^{-2}\eta^{-2}[\frac{\lambda^2(\sigma)}{\lambda^2(\tau_0)}]^{\frac{1}{4}}. 
\end{align*}
Then we observe that (recalling $1+ = 1+\delta_0, \frac12+ = \frac12 + \delta_0$)
\begin{align*}
&\frac{\lambda^{\frac32}(\tau_0)}{\lambda^{\frac32}(\sigma)}\big\|[\frac{\lambda^2(\tau_0)}{\lambda^2(\sigma)}]^{-\frac{1}{4}}\xi^{0+}\chi_{\frac{\lambda^2(\tau_0)}{\lambda^2(\sigma)}\xi<1}\big\|_{L^2_{d\xi}} + \frac{\lambda^{\frac32}(\tau_0)}{\lambda^{\frac32}(\sigma)}\big\|\langle \frac{\lambda^2(\tau_0)}{\lambda^2(\sigma)}\xi\rangle^{-2}[\frac{\lambda^2(\tau_0)}{\lambda^2(\sigma)}]^{\frac{1}{4}}\xi^{\frac{1}{2}+}\chi_{\frac{\lambda^2(\tau_0)}{\lambda^2(\sigma)}\xi>1}\big\|_{L^2_{d\xi}}\\
&\lesssim [\frac{\lambda(\sigma)}{\lambda(\tau_0)}]^{2\delta_0},
\end{align*}
while also 
\[
\big|\int_1^\infty [\frac{\lambda^2(\sigma)}{\lambda^2(\tau_0)}]^{\frac{1}{4}}\eta^{-2} \frac{\lambda^{\frac32}(\sigma)}{\lambda^{\frac32}(\tau_0)}x_1(\frac{\lambda^2(\sigma)}{\lambda^2(\tau_0)}\eta)\,d\eta\big|\lesssim (\frac{\lambda(\tau_0)}{\lambda(\sigma)})^{2\delta_0}\big\|\langle\cdot\rangle^{\frac12+}x_1\big\|_{L^2}
\]
In total we bound the preceding case by the expression 
\[
\lesssim \int_{\tau_0}^\infty \beta_{\nu}^2(\sigma)[\frac{\lambda(\sigma)}{\lambda(\tau_0)}]^{2\delta_0}\cdot (\frac{\lambda(\tau_0)}{\lambda(\sigma)})^{2\delta_0}\big\|\langle\cdot\rangle^{\frac12+}x_1\big\|_{L^2}\,d\sigma\lesssim \tau_0^{-1}\big\|x_1\big\|_{S_2}. 
\]

{\bf{(5(i).c)}}: $\frac{\lambda^2(\tau_0)}{\lambda^2(\sigma)}\xi\sim\eta$. Here we reduce to bounding the $L^2_{d\xi}$-norm of the function 
\begin{align*}
&\chi_{\frac{\lambda^2(\tau_0)}{\lambda^2(\sigma)}\xi<1}\frac{\lambda(\tau_0)}{\lambda(\sigma)}\xi^{0+}\int_{\eta\sim \frac{\lambda^2(\tau_0)}{\lambda^2(\sigma)}\xi}\frac{F(\frac{\lambda^2(\tau_0)}{\lambda^2(\sigma)}\xi, \eta)\eta^{-\frac{1}{2}}}{\frac{\lambda^2(\tau_0)}{\lambda^2(\sigma)}\xi - \eta}\frac{\lambda^2(\sigma)}{\lambda^2(\tau_0)}x_1(\frac{\lambda^2(\sigma)}{\lambda^2(\tau_0)}\eta)\,d\eta\\
& + \chi_{\frac{\lambda^2(\tau_0)}{\lambda^2(\sigma)}\xi>1}\frac{\lambda^2(\tau_0)}{\lambda^2(\sigma)}\xi^{\frac{1}{2}+}\int_{\eta\sim \frac{\lambda^2(\tau_0)}{\lambda^2(\sigma)}\xi}\frac{F(\frac{\lambda^2(\tau_0)}{\lambda^2(\sigma)}\xi, \eta)}{\frac{\lambda^2(\tau_0)}{\lambda^2(\sigma)}\xi - \eta}\frac{\lambda^2(\sigma)}{\lambda^2(\tau_0)}x_1(\frac{\lambda^2(\sigma)}{\lambda^2(\tau_0)}\eta)\,d\eta,\\
\end{align*}
where we have omitted for now the $\sigma$ integral, as well as the weight $\beta_{\nu}^2(\sigma)$. Then for the first expression observe that under the localisations there have have 
\[
\big|F(\frac{\lambda^2(\tau_0)}{\lambda^2(\sigma)}\xi, \eta)\eta^{-\frac{1}{2}}\big|\lesssim \eta^{\frac{1}{2}}, 
\]
and so we find 
\begin{align*}
&\big\|\chi_{\frac{\lambda^2(\tau_0)}{\lambda^2(\sigma)}\xi<1}\frac{\lambda(\tau_0)}{\lambda(\sigma)}\xi^{0+}\int_{\eta\sim \frac{\lambda^2(\tau_0)}{\lambda^2(\sigma)}\xi}\frac{F(\frac{\lambda^2(\tau_0)}{\lambda^2(\sigma)}\xi, \eta)\eta^{-\frac{1}{2}}}{\frac{\lambda^2(\tau_0)}{\lambda^2(\sigma)}\xi - \eta}\frac{\lambda^2(\sigma)}{\lambda^2(\tau_0)}x_1(\frac{\lambda^2(\sigma)}{\lambda^2(\tau_0)}\eta)\,d\eta\big\|_{L^2_{d\xi}}\\
&\lesssim (\frac{\lambda(\sigma)}{\lambda(\tau_0)})^{2\delta_0}\big\|\langle\cdot\rangle^{\frac{1}{2}}x_1\big\|_{L^2_{d\eta}}
\end{align*}
For the second integral expression above, one gets under the localisations indicated there 
\[
\big|\frac{\lambda^2(\tau_0)}{\lambda^2(\sigma)}\xi^{\frac{1}{2}+}F(\frac{\lambda^2(\tau_0)}{\lambda^2(\sigma)}\xi, \eta)\big|\lesssim (\frac{\lambda(\sigma)}{\lambda(\tau_0)})^{2\delta_0}\frac{\lambda(\tau_0)}{\lambda(\sigma)}
\]
and since also $\eta\gtrsim 1$ one gets the same bound as for the preceding term. Then performing the $\sigma$-integral results in the upper bound 
\[
\big\|x_1\big\|_{S_2}\int_{\tau_0}^\infty  (\frac{\lambda(\sigma)}{\lambda(\tau_0)})^{2\delta_0}\beta_{\nu}^2(\sigma)\,d\sigma\lesssim \tau_0^{-1}\big\|x_1\big\|_{S_2}. 
\]

{\bf{(5(i).d)}}: $\frac{\lambda^2(\tau_0)}{\lambda^2(\sigma)}\xi\gg\eta$. Here we again take advantage of the vanishing property of $x_1$. Use the same decomposition as in preceding case splitting into the cases 
\[
\frac{\lambda^2(\tau_0)}{\lambda^2(\sigma)}\xi<1, \frac{\lambda^2(\tau_0)}{\lambda^2(\sigma)}\xi>1. 
\]
Both are handled similarly, so we consider the contribution of the former case. Write (under the current assumption on frequencies)
\begin{equation}\label{eq:1001}
\frac{F(\frac{\lambda^2(\tau_0)}{\lambda^2(\sigma)}\xi, \eta)\rho^{\frac{1}{2}}(\eta)}{\frac{\lambda^2(\tau_0)}{\lambda^2(\sigma)}\xi - \eta} = c\frac{F(\frac{\lambda^2(\tau_0)}{\lambda^2(\sigma)}\xi, 0)\eta^{-\frac{1}{4}}}{\frac{\lambda^2(\tau_0)}{\lambda^2(\sigma)}\xi} + O(\frac{\eta^{\frac{3}{4}}}{\frac{\lambda^2(\tau_0)}{\lambda^2(\sigma)}\xi}) + O(\eta^{\frac14}).
\end{equation}
Inserting the error term at the end into the first long expression in {\it{(5(i).c)}} and evaluating the $L^2_{d\xi}$ integral gives a factor $(\frac{\lambda(\sigma)}{\lambda(\tau_0)})^{2\delta_0}$ and reduces to bounding the integral 
\begin{align*}
&\big|\int_0^1\frac{\lambda^{\frac{3}{2}}(\sigma)}{\lambda^{\frac{3}{2}}(\tau_0)}\eta^{-\frac{1}{4}}\rho^{\frac{1}{2}}(\frac{\lambda^2(\sigma)}{\lambda^2(\tau_0)}\eta)x_1(\frac{\lambda^2(\sigma)}{\lambda^2(\tau_0)}\eta)\,d\eta\big|\\
&\lesssim (\frac{\lambda(\sigma)}{\lambda(\tau_0)})^{2\delta_0}\big\|x_1\big\|_{S_2}
\end{align*}
The integral over $\sigma$ is then again bounded as in the preceding case. 
\\
This reduces things to the contribution of the principal part $c\frac{F(\frac{\lambda^2(\tau_0)}{\lambda^2(\sigma)}\xi, 0)\eta^{-\frac{1}{4}}}{\frac{\lambda^2(\tau_0)}{\lambda^2(\sigma)}\xi}$ (in lieu of the expression on the left in \eqref{eq:1001}
) , which is the expression  
\begin{align*}
&\chi_{\frac{\lambda^2(\tau_0)}{\lambda^2(\sigma)}\xi<1}\frac{\lambda(\tau_0)}{\lambda(\sigma)}\xi^{0+}\int_{\eta\ll\frac{\lambda^2(\tau_0)}{\lambda^2(\sigma)}\xi}\frac{F(\frac{\lambda^2(\tau_0)}{\lambda^2(\sigma)}\xi, 0)\eta^{-\frac{1}{4}}}{\frac{\lambda^{2}(\tau_0)}{\lambda^{2}(\sigma)}\xi}\frac{\lambda^{\frac{3}{2}}(\sigma)}{\lambda^{\frac{3}{2}}(\tau_0)}\rho^{\frac{1}{2}}(\frac{\lambda^2(\sigma)}{\lambda^2(\tau_0)}\eta)\\
&\hspace{5cm}\cdot\frac{\sin[\lambda(\sigma)\eta^{\frac{1}{2}}\int_{\tau_0}^\sigma\lambda^{-1}(u)\,du]}{\eta^{\frac{1}{2}}}x_1(\frac{\lambda^2(\sigma)}{\lambda^2(\tau_0)}\eta)\,d\eta
\end{align*}
But then extending the integral to the full range $[0,\infty)$ up to errors treatable like in {\it{(5(i).a)}} - {\it{(5(i).c)}}, and then performing the change of variable $\tilde{\eta} = \frac{\lambda^2(\sigma)}{\lambda^2(\tau_0)}\eta$, we are led to the expression 
\begin{align*}
\chi_{\frac{\lambda^2(\tau_0)}{\lambda^2(\sigma)}\xi<1}\frac{\lambda(\tau_0)}{\lambda(\sigma)}\xi^{0+}\int_0^\infty \frac{\lambda(\sigma)}{\lambda(\tau_0)}\frac{\sin[\lambda(\tau_0)\tilde{\eta}^{\frac{1}{2}}\int_{\tau_0}^\sigma\lambda^{-1}(u)\,du]}{\tilde{\eta}^{\frac{3}{4}}}(\rho^{\frac{1}{2}}x_1)(\tilde{\eta})d\tilde{\eta}. 
\end{align*}
However, the bound 
\[
\big|\int_0^\infty \frac{\lambda(\sigma)}{\lambda(\tau_0)}\frac{\sin[\lambda(\tau_0)\tilde{\eta}^{\frac{1}{2}}\int_{\tau_0}^\sigma\lambda^{-1}(u)\,du]}{\tilde{\eta}^{\frac{3}{4}}}(\rho^{\frac{1}{2}}x_1)(\tilde{\eta})d\tilde{\eta}\big|\lesssim \sigma\big\|x_1\big\|_{S_2}
\]
is not quite good enough to handle its contribution leading to a divergent integral in $\sigma$. Here note that if we restrict further to 
\[
\frac{\lambda^2(\tau_0)}{\lambda^2(\sigma)}\xi<\sigma^{-\delta}
\]
for some $\delta>0$ (but large in comparison to $\delta_0$), then we still get convergence as we gain $\sigma^{-\frac{\delta}{2}}$ from the $L^2_{d\xi}$-norm of the output. It follows that one may reduce to the contribution of 
\begin{align*}
&\chi_{\sigma^{-\delta}<\frac{\lambda^2(\tau_0)}{\lambda^2(\sigma)}\xi<1}\frac{\lambda(\tau_0)}{\lambda(\sigma)}\xi^{0+}\int_{\eta\ll\frac{\lambda^2(\tau_0)}{\lambda^2(\sigma)}\xi}\frac{F(\frac{\lambda^2(\tau_0)}{\lambda^2(\sigma)}\xi, 0)\eta^{-\frac{1}{4}}}{\frac{\lambda^{2}(\tau_0)}{\lambda^{2}(\sigma)}\xi}\frac{\lambda^{\frac{3}{2}}(\sigma)}{\lambda^{\frac{3}{2}}(\tau_0)}\rho^{\frac{1}{2}}(\frac{\lambda^2(\sigma)}{\lambda^2(\tau_0)}\eta)\\
&\hspace{5cm}\cdot\frac{\sin[\lambda(\sigma)\eta^{\frac{1}{2}}\int_{\tau_0}^\sigma\lambda^{-1}(u)\,du]}{\eta^{\frac{1}{2}}}x_1(\frac{\lambda^2(\sigma)}{\lambda^2(\tau_0)}\eta)\,d\eta
\end{align*}
To avoid a divergence in the eventual $\sigma$-integral here, we have to keep track of all oscillating factors, and in particular the outer one which we suppressed until now: 
\[
\sin[\lambda(\tau_0)\xi^{\frac{1}{2}}\int_{\tau_0}^\sigma\lambda^{-1}(u)\,du]. 
\]
Combining the two oscillatory phases, passing to the variable $\tilde{\eta} = \frac{\lambda^2(\sigma)}{\lambda^2(\tau_0)}\eta$ and performing an integration by parts with respect to $\sigma$ then results in a gain of 
\[
\sigma^{-1}[\frac{\lambda^2(\tau_0)}{\lambda^2(\sigma)}\xi]^{-1}\lesssim \sigma^{\delta-1}, 
\]
and so we again win provided $\delta<1$ and $1-\delta\gg \delta_0$. 
\\

This then concludes controlling the contribution of the source term $\beta_{\nu}^2(\sigma)\mathcal{K}x^{(0)}$ to the estimate 
\[
\big\|\chi_{\xi>1}(\triangle \tilde{x}^{(1)}_0, \triangle \tilde{x}^{(1)}_{1})\big\|_{S}\lesssim \big\|(x_0, x_1)\big\|_{\tilde{S}}.
\]

{\bf{Step 5(ii)}}: {\it{Contribution of $ \beta_{\nu}(\tau)\mathcal{K}_{cc}\mathcal{D}_{\tau}x^{(0)}$}}. This case is potentially delicate since we only gain one power of $\sigma$ here, but expect to lose small powers of the form $(\frac{\lambda(\sigma)}{\lambda(\tau_0)})^{2\delta_0}$. However, we then do not lose a power of $\sigma$ as in case {\it{(5(i).c)}} before due to the less singular nature of $\mathcal{D}_{\sigma}x^{(0)}$ (one gains a factor $\eta^{\frac{1}{2}}$). In fact, using an integration by parts argument with respect to $\sigma$ as in case {\it{(5(i).c)}} before, one easily disposes of the cases 
\[
\frac{\lambda^2(\tau_0)}{\lambda^2(\sigma)}\xi\ll \eta,\,\frac{\lambda^2(\tau_0)}{\lambda^2(\sigma)}\xi\gg \eta.
\]
This then reduces the contribution of the term $\beta_{\nu}(\sigma)\mathcal{K}\mathcal{D}_{\sigma}x^{(0)}$ to the diagonal case $\frac{\lambda^2(\tau_0)}{\lambda^2(\sigma)}\xi\sim\eta$. Then, due to the vanishing and decay properties of the kernel $F(\cdot, \cdot)$, we may essentially (up to terms which are treated similarly) reduce to the case $\frac{\lambda^2(\tau_0)}{\lambda^2(\sigma)}\xi\sim 1\sim \eta$. Thus consider now the expression 
\begin{align*}
&\big\|\langle\xi\rangle^{\frac{1}{2}+}\int_{\tau_0}^\infty\chi_{\frac{\lambda^2(\tau_0)}{\lambda^2(\sigma)}\xi\sim 1}\frac{\lambda^2(\tau_0)}{\lambda^2(\sigma)}\beta_{\nu}(\sigma)\\&\hspace{3cm}\cdot\int_{\eta\sim 1}\frac{\lambda^{\frac32}(\sigma)}{\lambda^{\frac32}(\tau_0)}\frac{F(\frac{\lambda^2(\tau_0)}{\lambda^2(\sigma)}\xi, \eta)}{\frac{\lambda^{2}(\tau_0)}{\lambda^{2}(\sigma)}\xi - \eta}\rho^{\frac{1}{2}}(\frac{\lambda^2(\sigma)}{\lambda^2(\tau_0)}\eta)x_1(\frac{\lambda^2(\sigma)}{\lambda^2(\tau_0)}\eta)\,d\eta d\sigma\big\|_{L^2_{d\xi}}
\end{align*}
In fact, bounding this will then imply the needed bound for $\triangle \tilde{x}^{(1)}_0$. By orthogonality, we can equate this expression's square with 
\begin{align*}
&\sum_{j\geq 0}\big\|\langle\xi\rangle^{\frac{1}{2}+}\int_{\tau_0}^\infty\chi_{\frac{\lambda^2(\tau_0)}{\lambda^2(\sigma)}\xi\sim 1}\frac{\lambda^2(\tau_0)}{\lambda^2(\sigma)}\beta_{\nu}(\sigma)\\&\hspace{3cm}\cdot\int_{\eta\sim 1}\frac{F(\frac{\lambda^2(\tau_0)}{\lambda^2(\sigma)}\xi, \eta)}{\frac{\lambda^{2}(\tau_0)}{\lambda^{2}(\sigma)}\xi - \eta}\frac{\lambda^2(\sigma)}{\lambda^2(\tau_0)}x_1(\frac{\lambda^2(\sigma)}{\lambda^2(\tau_0)}\eta)\,d\eta d\sigma\big\|_{L^2_{d\xi}(\xi\sim 2^j)}^2\\
& = \sum_{j\geq 0}\big\|\langle\xi\rangle^{\frac{1}{2}+}\int_{\sigma\sim 2^{\frac{\nu j}{2(1+\nu)}}\tau_0}\chi_{\frac{\lambda^2(\tau_0)}{\lambda^2(\sigma)}\xi\sim 1}\frac{\lambda^2(\tau_0)}{\lambda^2(\sigma)}\beta_{\nu}(\sigma)\\&\hspace{3cm}\cdot\int_{\eta\sim 1}\frac{F(\frac{\lambda^2(\tau_0)}{\lambda^2(\sigma)}\xi, \eta)}{\frac{\lambda^{2}(\tau_0)}{\lambda^{2}(\sigma)}\xi - \eta}\frac{\lambda^2(\sigma)}{\lambda^2(\tau_0)}x_1(\frac{\lambda^2(\sigma)}{\lambda^2(\tau_0)}\eta)\,d\eta d\sigma\big\|_{L^2_{d\xi}(\xi\sim 2^j)}^2\\
\end{align*}
Then observe that in the above expression we have $\langle\xi\rangle^{\frac{1}{2}+}\frac{\lambda(\tau_0)}{\lambda(\sigma)}\lesssim (\frac{\lambda(\sigma)}{\lambda(\tau_0)})^{2\delta_0}$, and furthermore, for fixed $j$, we have 
\[
\frac{\lambda^2(\sigma)}{\lambda^2(\tau_0)}\eta\sim 2^j,\,\frac{\lambda(\sigma)}{\lambda(\tau_0)}\big\|(\frac{\lambda(\sigma)}{\lambda(\tau_0)})^{2\delta_0}\big(\eta\frac{\lambda^2(\sigma)}{\lambda^2(\tau_0)}\big)^{\frac12}x_1(\frac{\lambda^2(\sigma)}{\lambda^2(\tau_0)}\eta)\big\|_{L^2_{d\eta}(\eta\sim 1)}\lesssim \big\|x_1\big\|_{S_2}. 
\]
Thus, for fixed $j$, we get 
 \begin{align*}
&\big\|\langle\xi\rangle^{\frac{1}{2}+}\int_{\sigma\sim 2^{\frac{\nu j}{1+\nu}}\tau_0}\chi_{\frac{\lambda^2(\tau_0)}{\lambda^2(\sigma)}\xi\sim 1}\frac{\lambda^2(\tau_0)}{\lambda^2(\sigma)}\beta_{\nu}(\sigma)\\&\hspace{3cm}\cdot\int_{\eta\sim 1}\frac{F(\frac{\lambda^2(\tau_0)}{\lambda^2(\sigma)}\xi, \eta)}{\frac{\lambda^{2}(\tau_0)}{\lambda^{2}(\sigma)}\xi - \eta}\frac{\lambda^2(\sigma)}{\lambda^2(\tau_0)}x_1(\frac{\lambda^2(\sigma)}{\lambda^2(\tau_0)}\eta)\,d\eta d\sigma\big\|_{L^2_{d\xi}(\xi\sim 2^j)}\\
&\lesssim \big\|\chi_{\eta\sim 2^j}x_1\big\|_{S_2},
 \end{align*}
 and the desired final bound follows after square summation over $j$. This finally concludes our arguments for the bound 
 \[
 \big\|\chi_{\xi>1}(\triangle x^{(1)}_0, \triangle x^{(1)}_{1})\big\|_{S}\lesssim \big\|(x_0, x_1)\big\|_{\tilde{S}},
\]
and thereby the proof of Proposition~\ref{prop:firstiteratelincont}. 

\end{proof}

In the following proposition, we shall strengthen the bounds of the preceding proposition a bit. In fact, this will be necessary to finally arrive at a framework of estimates which we can perpetuate across the iteration. The norms involving square-sums over dyadic time intervals shall be the key technical device allowing us to formulate the correct inductive steps for the iteration.  
 \begin{prop}\label{prop:firstiteratelincontimprov} Using the same terminology as in the statement of the preceding proposition, and also recalling \eqref{eq:Snorm}, we have the low frequency bound 
 \begin{equation}\label{eq:firstiteratelincontimprov1}\begin{split}
&\big(\sum_{\substack{N\gtrsim \tau_0\\N\,\text{dyadic}}}\sup_{\tau\sim N}(\frac{\lambda(\tau)}{\lambda(\tau_0)})^{4\delta_0}\big\|\xi^{-0+}\mathcal{D}_{\tau}\triangle x_{>\tau}^{(1)}(\tau, \cdot)\big\|_{L^2_{d\xi}(\xi<1)}^2\big)^{\frac12}\\
&\lesssim \big\|(x_0, x_1)\big\|_{\tilde{S}} + \big|x_{0d}\big|. 
 \end{split}\end{equation}
 Moreover, recalling 
 \begin{align*}
 &\triangle x^{(1)}(\tau, \xi) - S(\tau)( \triangle\tilde{\tilde{x}}_0^{(1)}, \triangle\tilde{\tilde{x}}_1^{(1)}) \\
 &=\int_{\tau_0}^\tau\frac{\lambda^{\frac{3}{2}}(\tau)}{\lambda^{\frac{3}{2}}(\sigma)}\frac{\rho^{\frac{1}{2}}(\frac{\lambda^2(\tau)}{\lambda^2(\sigma)}\xi)}{\rho^{\frac{1}{2}}(\xi)}\frac{\sin[\lambda(\tau)\xi^{\frac{1}{2}}\int_{\tau}^\sigma\lambda^{-1}(u)\,du]}{\xi^{\frac{1}{2}}}\mathcal{R}(\sigma, \underline{x}^{(0)})(\frac{\lambda^2(\tau)}{\lambda^2(\sigma)}\xi)\,d\sigma, 
 \end{align*}
we also have the high frequency bound 
  \begin{equation}\label{eq:firstiteratelincontimprov2}\begin{split}
&\big(\sum_{\substack{N\gtrsim \tau_0\\N\,\text{dyadic}}}\sup_{\tau\sim N}(\frac{\lambda(\tau)}{\lambda(\tau_0)})^{4\delta_0}\big\|\xi^{\frac12+}\mathcal{D}_{\tau}\big[\triangle x^{(1)}(\tau, \xi) - S(\tau)(\triangle\tilde{\tilde{x}}_0^{(1)}, \triangle\tilde{\tilde{x}}_1^{(1)})\big](\tau, \cdot)\big\|_{L^2_{d\xi}(\xi>1)}^2\big)^{\frac12}\\
&\lesssim \big\|(x_0, x_1)\big\|_{\tilde{S}} + \big|x_{0d}\big|. 
 \end{split}\end{equation}
 \end{prop}
 \begin{proof} We treat the low frequency estimate and the high frequency estimate separately: 
 \\
 
 {\bf{Step 1}}: {\it{Proof of \eqref{eq:firstiteratelincontimprov1}}}. Recall the representation formula 
 \begin{align*}
 &\triangle x_{>\tau}^{(1)}(\tau, \xi)\\
 &= -\int_{\tau}^\infty\frac{\lambda^{\frac{3}{2}}(\tau)}{\lambda^{\frac{3}{2}}(\sigma)}\frac{\rho^{\frac{1}{2}}(\frac{\lambda^2(\tau)}{\lambda^2(\sigma)}\xi)}{\rho^{\frac{1}{2}}(\xi)}\frac{\sin[\lambda(\tau)\xi^{\frac{1}{2}}\int_{\tau}^\sigma\lambda^{-1}(u)\,du]}{\xi^{\frac{1}{2}}}\mathcal{R}(\sigma, \underline{x}^{(0)})(\frac{\lambda^2(\tau)}{\lambda^2(\sigma)}\xi)\,d\sigma\\
 \end{align*}
 and furthermore the splitting
 \[
 \mathcal{R}(\tau, \underline{x}^{(0)})(\xi) = \mathcal{R}_1(\tau, x^{(0)})(\xi) + \mathcal{R}_2(\tau, x^{(0)}_d)(\xi)
 \]
 as well as the schematic decomposition
 \[
 \mathcal{R}_1(\tau, x^{(0)})(\xi) = \beta_{\nu}^2(\tau)\mathcal{K}_{cc}x^{(0)} + \beta_{\nu}(\tau)\mathcal{K}_{cc}\mathcal{D}_{\tau}x^{(0)} +  \beta_{\nu}^2(\tau)\mathcal{K}_{cc}^2x^{(0)}. 
 \]
 The contribution of the exponentially decaying mode $\mathcal{R}_2(\tau, x^{(0)}_d)(\xi)$ is again straightforward to handle, and so we focus on the contribution of $\mathcal{R}_1(\tau, x^{(0)})(\xi)$. Here we shall consider the contribution of $\beta_{\nu}(\tau)\mathcal{K}_{cc}\mathcal{D}_{\tau}x^{(0)}$, that of $ \beta_{\nu}^2(\tau)\mathcal{K}_{cc}x^{(0)}, \beta_{\nu}^2(\tau)\mathcal{K}_{cc}^2x^{(0)}$ being similar. Also, we shall only consider the contribution to $x^{(0)}$ by the second term on the right in \eqref{eq:linhomparam1}, as the first term is handled just the same.
 \\
 Then write 
 \begin{equation}\label{eq:Step1etaintegral1}\begin{split}
 &\beta_{\nu}(\sigma)\big(\mathcal{K}_{cc}\mathcal{D}_{\sigma}x^{(0)}\big)(\sigma, \frac{\lambda^2(\tau)}{\lambda^2(\sigma)}\xi)\\& =  \beta_{\nu}(\sigma)\int_0^\infty \frac{F(\frac{\lambda^2(\tau)}{\lambda^2(\sigma)}\xi, \eta)\rho(\eta)}{\frac{\lambda^2(\tau)}{\lambda^2(\sigma)}\xi - \eta}\frac{\lambda^{\frac{3}{2}}(\sigma)}{\lambda^{\frac{3}{2}}(\tau_0)}\frac{\rho^{\frac{1}{2}}(\frac{\lambda^2(\sigma)}{\lambda^2(\tau_0)}\eta)}{\rho^{\frac{1}{2}}(\eta)}\\&\hspace{4cm}\cdot\cos[\lambda(\sigma)\eta^{\frac{1}{2}}\int_{\tau_0}^\sigma\lambda^{-1}(u)\,du]x_1(\frac{\lambda^2(\sigma)}{\lambda^2(\tau_0)}\eta)\,d\eta
 \end{split}\end{equation}
 Also, use that 
 \begin{align*}
 &\mathcal{D}_{\tau}\triangle x_{>\tau}^{(1)}(\tau, \xi)\\
 &= -\int_{\tau}^\infty\frac{\lambda^{\frac{3}{2}}(\tau)}{\lambda^{\frac{3}{2}}(\sigma)}\frac{\rho^{\frac{1}{2}}(\frac{\lambda^2(\tau)}{\lambda^2(\sigma)}\xi)}{\rho^{\frac{1}{2}}(\xi)}\cos[\lambda(\tau)\xi^{\frac{1}{2}}\int_{\tau}^\sigma\lambda^{-1}(u)\,du]\mathcal{R}(\sigma, \underline{x}^{(0)})(\frac{\lambda^2(\tau)}{\lambda^2(\sigma)}\xi)\,d\sigma\\
 \end{align*}
and we replace $\mathcal{R}(\sigma, \underline{x}^{(0)})(\frac{\lambda^2(\tau)}{\lambda^2(\sigma)}\xi)$ by $\beta_{\nu}(\sigma)\big(\mathcal{K}_{cc}\mathcal{D}_{\sigma}x^{(0)}\big)(\sigma, \frac{\lambda^2(\tau)}{\lambda^2(\sigma)}\xi)$. 
 
 {\bf{(1i)}}: $\frac{\lambda^2(\tau)}{\lambda^2(\sigma)}\xi\ll \eta$, $\frac{\lambda^2(\sigma)}{\lambda^2(\tau_0)}\eta<1$. Here we get 
\[
\big|\frac{\lambda^{\frac{3}{2}}(\sigma)}{\lambda^{\frac{3}{2}}(\tau_0)}\frac{\rho^{\frac{1}{2}}(\frac{\lambda^2(\sigma)}{\lambda^2(\tau_0)}\eta)}{\rho^{\frac{1}{2}}(\eta)}\big|\sim \frac{\lambda(\sigma)}{\lambda(\tau_0)},\,\frac{\lambda^{\frac{3}{2}}(\tau)}{\lambda^{\frac{3}{2}}(\sigma)}\frac{\rho^{\frac{1}{2}}(\frac{\lambda^2(\tau)}{\lambda^2(\sigma)}\xi)}{\rho^{\frac{1}{2}}(\xi)}\sim \frac{\lambda(\tau)}{\lambda(\sigma)}.
\]
The $L^2_{d\xi}$-integral furnishes a gain of $(\frac{\lambda^2(\tau)}{\lambda^2(\sigma)}\xi)^{\frac{1}{2+}}$ which offsets the loss from $\rho(\eta)$ up to $\eta^{-0+}$, and the $\eta$-integral, handled via Cauchy-Schwarz, furnishes a gain of $\eta^{\frac{1}{2+}}<(\frac{\lambda^2(\sigma)}{\lambda^2(\tau_0)})^{-\frac{1}{2+}}$. In total, we gain  
\[
[\frac{\lambda(\tau_0)}{\lambda(\tau)}]^{1-}\cdot (\frac{\lambda(\tau)}{\lambda(\sigma)})^{1-}\frac{\eta^{\frac{1}{2+}}}{(\frac{\lambda^2(\sigma)}{\lambda^2(\tau_0)})^{-\frac{1}{2+}}}
\]
Here the weight $[\frac{\lambda(\tau_0)}{\lambda(\tau)}]^{1-}$ is much better than the required $[\frac{\lambda(\tau_0)}{\lambda(\tau)}]^{2\delta_0}$, which then also ensures square-summability over dyadic time intervals. 
\\

 {\bf{(1ii)}}: $\frac{\lambda^2(\tau)}{\lambda^2(\sigma)}\xi\ll \eta$, $\frac{\lambda^2(\sigma)}{\lambda^2(\tau_0)}\eta\geq 1, \eta<1$. Here we have
 \[
 \big|\frac{\lambda^{\frac{3}{2}}(\sigma)}{\lambda^{\frac{3}{2}}(\tau_0)}\frac{\rho^{\frac{1}{2}}(\frac{\lambda^2(\sigma)}{\lambda^2(\tau_0)}\eta)}{\rho^{\frac{1}{2}}(\eta)}\big|\sim \frac{\lambda^2(\sigma)}{\lambda^2(\tau_0)}\eta^{\frac12},\,\frac{\lambda^{\frac{3}{2}}(\tau)}{\lambda^{\frac{3}{2}}(\sigma)}\frac{\rho^{\frac{1}{2}}(\frac{\lambda^2(\tau)}{\lambda^2(\sigma)}\xi)}{\rho^{\frac{1}{2}}(\xi)}\sim \frac{\lambda(\tau)}{\lambda(\sigma)}.
\]
The $L^2_{d\xi}$-norm of the output again gains $[\frac{\lambda^2(\tau)}{\lambda^2(\sigma)}\xi]^{\frac{1}{2+}}$, and so the expression \eqref{eq:firstiteratelincontimprov1} in the present situation gets bounded by 
\begin{equation}\label{eq:delicate111}
\big(\sum_{\substack{N\gtrsim \tau_0\\N\,\text{dyadic}}}\sup_{\tau\sim N}(\frac{\lambda(\tau)}{\lambda(\tau_0)})^{4\delta_0}\big[\int_{\tau}^\infty\beta_{\nu}(\sigma)\int_0^1\eta^{-0+}\frac{\lambda^2(\sigma)}{\lambda^2(\tau_0)}\eta^{\frac12}x_1(\frac{\lambda^2(\sigma)}{\lambda^2(\tau_0)}\eta)\,d\eta d\sigma\big]^2\big)^{\frac12}
\end{equation}
Using the Cauchy-Schwarz inequality and the definition of the norm $\|\cdot\|_{S_2}$, we have 
\begin{align*}
\int_0^1\eta^{-0+}\frac{\lambda^2(\sigma)}{\lambda^2(\tau_0)}\eta^{\frac12}x_1(\frac{\lambda^2(\sigma)}{\lambda^2(\tau_0)}\eta)\,d\eta&\lesssim \big(\frac{\lambda(\tau_0)}{\lambda(\sigma)}\big)^{2\delta_0}\sum_{j<0}2^{\frac{j}{2+}}\big\|\chi_{\cdot\sim 2^j\frac{\lambda^2(\sigma)}{\lambda^2(\tau_0)}}x_1\big\|_{S_2}\\
&\lesssim \big(\frac{\lambda(\tau_0)}{\lambda(\sigma)}\big)^{2\delta_0}\big(\sum_{j<0}2^{\frac{j}{2+}}\big\|\chi_{\cdot\sim 2^j\frac{\lambda^2(\sigma)}{\lambda^2(\tau_0)}}x_1\big\|_{S_2}^2\big)^{\frac12}
\end{align*}
It follows that (reiterating Cauchy-Schwarz)
\begin{align*}
&\big[\int_{\tau}^\infty\beta_{\nu}(\sigma)\int_0^1\eta^{-0+}\frac{\lambda^2(\sigma)}{\lambda^2(\tau_0)}\eta^{\frac12}x_1(\frac{\lambda^2(\sigma)}{\lambda^2(\tau_0)}\eta)\,d\eta d\sigma\big]^2\\
&\lesssim \big(\frac{\lambda(\tau_0)}{\lambda(\tau)}\big)^{4\delta_0}\sum_{\substack{M\gtrsim \tau\\ M\,\text{dyadic}}}\sup_{\sigma \sim M}(\frac{\lambda(\tau)}{\lambda(\sigma)})^{2\delta_0}\big(\sum_{j<0}2^{\frac{j}{2+}}\big\|\chi_{\cdot\sim 2^j\frac{\lambda^2(\sigma)}{\lambda^2(\tau_0)}}x_1\big\|_{S_2}^2\big),
\end{align*}
and so we can bound the square of \eqref{eq:delicate111} by 
\[
\sum_{\substack{N\gtrsim \tau_0\\N\,\text{dyadic}}}\sum_{\substack{M\gtrsim N\\ M\,\text{dyadic}}}\sup_{\tau\sim N}\sup_{\sigma \sim M}(\frac{\lambda(\tau)}{\lambda(\sigma)})^{2\delta_0}\big(\sum_{j<0}2^{\frac{j}{2+}}\big\|\chi_{\cdot\sim 2^j\frac{\lambda^2(\sigma)}{\lambda^2(\tau_0)}}x_1\big\|_{S_2}^2\big)
\]
This last expression is bounded by $\lesssim \big\|x_1\big\|_{S_2}^2$ by a simple orthogonality argument. 
\\

 {\bf{(1iii)}}: $\frac{\lambda^2(\tau)}{\lambda^2(\sigma)}\xi\ll \eta$, $\eta\geq 1$. Here \eqref{eq:firstiteratelincontimprov1} is easily seen to be bounded by 
 \[
 \big(\sum_{\substack{N\gtrsim \tau_0\\N\,\text{dyadic}}}\sup_{\tau\sim N}(\frac{\lambda(\tau)}{\lambda(\tau_0)})^{4\delta_0}\big[\int_{\tau}^\infty\beta_{\nu}(\sigma)\int_1^\infty\eta^{-2}\frac{\lambda^2(\sigma)}{\lambda^2(\tau_0)}x_1(\frac{\lambda^2(\sigma)}{\lambda^2(\tau_0)}\eta)\,d\eta d\sigma\big]^2\big)^{\frac12}
 \]
 This is handled by analogy to the preceding case, by localising $\eta$ to dyadic size $2^j$, $j>0$. 
\\

{\bf{(1iv)}}: $\frac{\lambda^2(\tau)}{\lambda^2(\sigma)}\xi\sim\eta$, $\frac{\lambda^2(\sigma)}{\lambda^2(\tau_0)}\eta<1$. Use that 
\[
\big|\xi^{-0+}F(\frac{\lambda^2(\tau)}{\lambda^2(\sigma)}\xi, \eta)\rho(\eta)\big|\lesssim \eta^{\frac12-}\lesssim [\frac{\lambda(\tau_0)}{\lambda(\sigma)}]^{\frac12-} [\frac{\lambda(\tau_0)}{\lambda(\tau)}]^{2\delta_0}
\]
in this case(provided we choose $\delta_0$ small enough). In turn this implies that 
\begin{align*}
&\big\|\xi^{-0+}\beta_{\nu}(\sigma)\int_0^\infty \frac{F(\frac{\lambda^2(\tau)}{\lambda^2(\sigma)}\xi, \eta)\rho(\eta)}{\frac{\lambda^2(\tau)}{\lambda^2(\sigma)}\xi - \eta}\frac{\lambda^{\frac{3}{2}}(\sigma)}{\lambda^{\frac{3}{2}}(\tau_0)}\frac{\rho^{\frac{1}{2}}(\frac{\lambda^2(\sigma)}{\lambda^2(\tau_0)}\eta)}{\rho^{\frac{1}{2}}(\eta)}\\&\hspace{3.5cm}\cdot\cos[\lambda(\sigma)\eta^{\frac{1}{2}}\int_{\tau_0}^\sigma\lambda^{-1}(u)\,du]\chi_{\frac{\lambda^2(\sigma)}{\lambda^2(\tau_0)}\eta<1}x_1(\frac{\lambda^2(\sigma)}{\lambda^2(\tau_0)}\eta)\,d\eta\big\|_{L^2_{d\xi}}\\
&\lesssim \beta_{\nu}(\sigma)[\frac{\lambda(\tau)}{\lambda(\sigma)}]^{-1} [\frac{\lambda(\tau_0)}{\lambda(\sigma)}]^{\frac12-} [\frac{\lambda(\tau_0)}{\lambda(\tau)}]^{2\delta_0}\big\|x_1\big\|_{S_2}, 
\end{align*}
and hence the corresponding contribution to \eqref{eq:firstiteratelincontimprov1} is bounded by 
\[
\lesssim \big\|x_1\big\|_{S_2}\big(\sum_{\substack{N\gtrsim \tau_0\\N\,\text{dyadic}}}\sup_{\tau\sim N}(\frac{\lambda(\tau_0)}{\lambda(\tau)})^{1-}\big)^{\frac12}\lesssim  \big\|x_1\big\|_{S_2}. 
\]

{\bf{(1v)}}: $\frac{\lambda^2(\tau)}{\lambda^2(\sigma)}\xi\sim\eta$, $\frac{\lambda^2(\sigma)}{\lambda^2(\tau_0)}\eta\geq 1$. This is similar to case {\it{(1ii)}}. One can absorb an extra $(\frac{\lambda^2(\sigma)}{\lambda^2(\tau_0)}\eta)^{0+}$ into $x_1$, which results in a gain of $[\frac{\lambda(\tau_0)}{\lambda(\sigma)}]^{2\delta_0}$ when combining with the bound 
\[
\big|F(\frac{\lambda^2(\tau)}{\lambda^2(\sigma)}\xi, \eta)\rho(\eta)\big|\lesssim \eta^{\frac12}. 
\]
{\bf{(1vi)}}: $\frac{\lambda^2(\tau)}{\lambda^2(\sigma)}\xi\gg\eta$. Switching variables in \eqref{eq:Step1etaintegral1} we have 
\begin{align*}
&\int_0^\infty \frac{F(\frac{\lambda^2(\tau)}{\lambda^2(\sigma)}\xi, \eta)\rho(\eta)}{\frac{\lambda^2(\tau)}{\lambda^2(\sigma)}\xi - \eta}\frac{\lambda^{\frac{3}{2}}(\sigma)}{\lambda^{\frac{3}{2}}(\tau_0)}\frac{\rho^{\frac{1}{2}}(\frac{\lambda^2(\sigma)}{\lambda^2(\tau_0)}\eta)}{\rho^{\frac{1}{2}}(\eta)}\cos[\lambda(\sigma)\eta^{\frac{1}{2}}\int_{\tau_0}^\sigma\lambda^{-1}(u)\,du]x_1(\frac{\lambda^2(\sigma)}{\lambda^2(\tau_0)}\eta)\,d\eta\\
& = \int_0^\infty \frac{F(\frac{\lambda^2(\tau)}{\lambda^2(\sigma)}\xi, \frac{\lambda^2(\tau)}{\lambda^2(\sigma)}\tilde{\eta})\rho(\frac{\lambda^2(\tau)}{\lambda^2(\sigma)}\tilde{\eta})}{\xi - \tilde{\eta}}\frac{\lambda^{\frac{3}{2}}(\sigma)}{\lambda^{\frac{3}{2}}(\tau_0)}\frac{\rho^{\frac{1}{2}}(\frac{\lambda^2(\tau)}{\lambda^2(\tau_0)}\tilde{\eta})}{\rho^{\frac{1}{2}}(\frac{\lambda^2(\tau)}{\lambda^2(\sigma)}\tilde{\eta})}\cos[\lambda(\tau)\tilde{\eta}^{\frac{1}{2}}\int_{\tau_0}^\sigma\lambda^{-1}(u)\,du]x_1(\frac{\lambda^2(\tau)}{\lambda^2(\tau_0)}\tilde{\eta})\,d\tilde{\eta}\\
\end{align*}
Then, returning to the full expression also involving the $\sigma$-integral, we intend to perform an integration by parts with respect to $\sigma$, provided $\frac{\lambda^2(\tau)}{\lambda^2(\sigma)}\xi>\sigma^{-2\delta}$. In fact, under the latter assumption, we write 
\begin{align*}
&\cos[\lambda(\tau)\xi^{\frac{1}{2}}\int_{\tau}^\sigma\lambda^{-1}(u)\,du]\cos[\lambda(\tau)\tilde{\eta}^{\frac{1}{2}}\int_{\tau_0}^\sigma\lambda^{-1}(u)\,du]\\
& = \sum_{\pm}\frac{1}{\frac{\lambda(\tau)}{\lambda(\sigma)}\xi^{\frac{1}{2}}\pm \frac{\lambda(\tau)}{\lambda(\sigma)}\tilde{\eta}^{\frac{1}{2}}}\partial_{\sigma}\big[\cos[\lambda(\tau)\xi^{\frac{1}{2}}\int_{\tau}^\sigma\lambda^{-1}(u)\,du \pm \lambda(\tau)\tilde{\eta}^{\frac{1}{2}}\int_{\tau_0}^\sigma\lambda^{-1}(u)\,du]\big]
\end{align*}
and so integration by parts with respect to $\sigma$ will result in a gain of $\sigma^{\delta - 1}$. To conclude this case then, one again distinguishes between the situations $\frac{\lambda^2(\sigma)}{\lambda^2(\tau_0)}\eta><1$. To bound the $\eta$-integral by Cauchy-Schwarz, one needs to pay a small power of $\frac{\lambda^2(\sigma)}{\lambda^2(\tau_0)}$, which is counteracted by the gain in $\sigma$. 
\\

This reduces things to the case $\frac{\lambda^2(\tau)}{\lambda^2(\sigma)}\xi<\sigma^{-2\delta}$. Here an additional gain in $\sigma$ comes from the $\xi$-integral, which indeed produces a gain 
\[
[\xi^{\frac{1}{2}}\frac{\lambda(\tau)}{\lambda(\sigma)}]^{\frac{1}{2+}} <\sigma^{-\delta},
\]
This suffices to absorb the factor $(\frac{\lambda(\tau)}{\lambda(\tau_0)})^{4\delta_0}$ in \eqref{eq:firstiteratelincontimprov1} and also to be able to square-sum over dyadic time intervals. This completes the proof of \eqref{eq:firstiteratelincontimprov1}.
\\

 {\bf{Step 2}}: {\it{Proof of \eqref{eq:firstiteratelincontimprov2}}}. Again we omit the simple contribution by the exponentially decaying term due to $x_{0d}$. Here we have $\frac{\lambda^2(\tau)}{\lambda^2(\sigma)}\xi>1$ on account of $\tau\geq \sigma$ and $\xi>1$ (recall the definition of $\triangle x^{(1)}$ from the previous proposition). We shall again treat the contribution of $x_1$ and the source term $\beta_{\nu}(\sigma)\mathcal{K}_{cc}\mathcal{D}_{\sigma}x^{(0)}$ in detail, the other cases being similar. Throughout in this case, we have 
\[
\frac{\lambda^{\frac{3}{2}}(\tau)}{\lambda^{\frac{3}{2}}(\sigma)}\frac{\rho^{\frac{1}{2}}(\frac{\lambda^2(\tau)}{\lambda^2(\sigma)}\xi)}{\rho^{\frac{1}{2}}(\xi)}\sim \frac{\lambda^2(\tau)}{\lambda^2(\sigma)}.
\]
{\bf{(2i)}}: $\frac{\lambda^2(\tau)}{\lambda^2(\sigma)}\xi\ll \eta$.  Square integration over $\xi$ gives an extra factor $\xi^{\frac{1}{2}}$, and we have (for $\xi>1$, $\frac{\lambda^2(\tau)}{\lambda^2(\sigma)}\xi\ll \eta$, $\tau>\sigma$)
 \[
\big|\xi^{\frac{1}{2}}\langle\xi\rangle^{\frac{1}{2}+}\frac{\lambda^2(\tau)}{\lambda^2(\sigma)} F(\frac{\lambda^2(\tau)}{\lambda^2(\sigma)}\xi, \eta)\rho(\eta)\big|\lesssim \eta^{-1}\big(\frac{\lambda^2(\sigma)}{\lambda^2(\tau)}\big)^2
\]
due to the strong decay properties for $F(\cdot, \cdot)$ for large variables, see \cite{KST}. Furthermore, write (for $\eta>1$)
\begin{align*}
\big|\frac{\lambda^{\frac{3}{2}}(\sigma)}{\lambda^{\frac{3}{2}}(\tau_0)}\frac{\rho^{\frac{1}{2}}(\frac{\lambda^2(\sigma)}{\lambda^2(\tau_0)}\eta)}{\rho^{\frac{1}{2}}(\eta)}x_1(\frac{\lambda^2(\sigma)}{\lambda^2(\tau_0)}\eta)\big|
\lesssim \big(\frac{\lambda(\tau_0)}{\lambda(\sigma)}\big)^{2\delta_0}\frac{\lambda(\sigma)}{\lambda(\tau_0)}\big|\big(\frac{\lambda^2(\sigma)}{\lambda^2(\tau_0)}\eta\big)^{\frac{1}{2}+}x_1(\frac{\lambda^2(\sigma)}{\lambda^2(\tau_0)}\eta)\big|
\end{align*}
Then it follows that 
\begin{align*}
&\big\|\langle\xi\rangle^{\frac{1}{2}+}\int_{\tau_0}^\tau\frac{\lambda^{\frac{3}{2}}(\tau)}{\lambda^{\frac{3}{2}}(\sigma)}\frac{\rho^{\frac{1}{2}}(\frac{\lambda^2(\tau)}{\lambda^2(\sigma)}\xi)}{\rho^{\frac{1}{2}}(\xi)}\cos[\lambda(\tau)\xi^{\frac{1}{2}}\int_{\tau}^\sigma\lambda^{-1}(u)\,du]\beta_{\nu}(\sigma)\\&\hspace{4cm}\cdot\big[\mathcal{K}_{\frac{\lambda^2(\tau)}{\lambda^2(\sigma)}\xi\ll \eta}\mathcal{D}_{\sigma}\big(S(\sigma)(0, x_1)\big)\big](\frac{\lambda^2(\tau)}{\lambda^2(\sigma)}\xi)\,d\sigma\big\|_{L^2_{d\xi}(\xi>1)}\\
&\lesssim \big(\frac{\lambda(\tau_0)}{\lambda(\tau)}\big)^{2\delta_0}\sum_{j>1}2^{-j}\big(\sum_{\substack {\sigma\sim N\lesssim \tau\\ N\,\text{dyadic}}}\frac{\lambda^2(N)}{\lambda^2(\tau)}\sup_{\sigma\sim N}\big\|x_1\big\|_{S_2(\cdot\sim 2^j\frac{\lambda^2(\sigma)}{\lambda^2(\tau_0)})}^2\big)^{\frac12}
\end{align*}
 where $S(\tau)$ is the propagator given explicitly by \eqref{eq:linhomparam1}. In the last step we have used Cauchy-Schwarz for the time integral. Then the desired bound \eqref{eq:firstiteratelincontimprov2} for this contribution follows by square-summing over dyadic $\tau$-intervals (after taking suprema over such) and a simple orthogonality argument. 
\\

{\bf{(2ii)}}: $\frac{\lambda^2(\tau)}{\lambda^2(\sigma)}\xi\sim\eta$, $\xi>1, \tau\geq \sigma$. Here the estimate is a bit more delicate, since we only get 
\[
\big|\langle\xi\rangle^{\frac12+}F(\frac{\lambda^2(\tau)}{\lambda^2(\sigma)}\xi, \eta)\rho(\eta)\big|\lesssim \frac{\lambda(\sigma)}{\lambda(\tau)}\xi^{0+}
\]
on account of the bounds established in \cite{KST}. Then using another orthogonality argument as in the preceding case, we obtain the similar bound 
\begin{align*}
&\big\|\langle\xi\rangle^{\frac{1}{2}+}\int_{\tau_0}^\tau\frac{\lambda^{\frac{3}{2}}(\tau)}{\lambda^{\frac{3}{2}}(\sigma)}\frac{\rho^{\frac{1}{2}}(\frac{\lambda^2(\tau)}{\lambda^2(\sigma)}\xi)}{\rho^{\frac{1}{2}}(\xi)}\cos[\lambda(\tau)\xi^{\frac{1}{2}}\int_{\tau}^\sigma\lambda^{-1}(u)\,du]\beta_{\nu}(\sigma)\\&\hspace{4cm}\cdot\big[\mathcal{K}_{\frac{\lambda^2(\tau)}{\lambda^2(\sigma)}\xi\sim\eta}\mathcal{D}_{\sigma}\big(S(\sigma)(0, x_1)\big)\big](\frac{\lambda^2(\tau)}{\lambda^2(\sigma)}\xi)\,d\sigma\big\|_{L^2_{d\xi}(\xi>1)}\\
&\lesssim \big(\frac{\lambda(\tau_0)}{\lambda(\tau)}\big)^{2\delta_0}\sum_{j>0}2^{-\frac{j}{2}}\big\|x_1\big\|_{S_2(\cdot\sim 2^j\frac{\lambda^2(\tau)}{\lambda^2(\tau_0)})}
\end{align*}
 Then the desired square-sum bound \eqref{eq:firstiteratelincontimprov2} again follows from an orthogonality argument. 
 \\
 
 {\bf{(2iii)}}: $\frac{\lambda^2(\tau)}{\lambda^2(\sigma)}\xi\gg\eta$, $\xi>1, \tau\geq \sigma$. Here as we are in the non-singular regime for the Hilbert transform type operator giving $\mathcal{K}_{cc}$, we can take advantage of the better decay properties of the kernel $F$, see \cite{KST}, namely 
 \[
 \big|\langle\xi\rangle^{\frac12}\frac{\lambda^2(\tau)}{\lambda^2(\sigma)}F(\frac{\lambda^2(\tau)}{\lambda^2(\sigma)}\xi, \eta)\big|\lesssim \big(\frac{\lambda^2(\tau)}{\lambda^2(\sigma)}\big)^{-N}\xi^{-N}. 
 \]
On the other hand, the $\eta$-integral furnishes a gain of $\eta^{\delta_0}[\frac{\lambda(\tau_0)}{\lambda(\sigma)}]^{2\delta_0}$ for $\eta<1$. One then easily infers the bound 
\begin{align*}
&\big\|\langle\xi\rangle^{\frac{1}{2}+}\int_{\tau_0}^\tau\frac{\lambda^{\frac{3}{2}}(\tau)}{\lambda^{\frac{3}{2}}(\sigma)}\frac{\rho^{\frac{1}{2}}(\frac{\lambda^2(\tau)}{\lambda^2(\sigma)}\xi)}{\rho^{\frac{1}{2}}(\xi)}\cos[\lambda(\tau)\xi^{\frac{1}{2}}\int_{\tau}^\sigma\lambda^{-1}(u)\,du]\beta_{\nu}(\sigma)\\&\hspace{4cm}\cdot\big[\mathcal{K}_{\frac{\lambda^2(\tau)}{\lambda^2(\sigma)}\xi\gg\eta}\mathcal{D}_{\sigma}\big(S(\sigma)(0, x_1)\big)\big](\frac{\lambda^2(\tau)}{\lambda^2(\sigma)}\xi)\,d\sigma\big\|_{L^2_{d\xi}(\xi>1)}\\
&\lesssim \big(\frac{\lambda(\tau_0)}{\lambda(\tau)}\big)^{2\delta_0}\sum_{j\in \Z}\min\{2^{\delta_0 j}, 2^{-Mj}\}\big(\sum_{\substack{\sigma\sim N\lesssim \tau\\ N\,\text{dyadic}}}[\frac{\lambda(N)}{\lambda(\tau)}]^M\sup_{\sigma\sim N}\big\|x_1\big\|_{S_2(\cdot\sim 2^j\frac{\lambda^2(\sigma)}{\lambda^2(\tau_0)})}^2\big)^{\frac12}
\end{align*}
where $M$ in effect may be chosen arbitrarily, but $M= 1$ suffices. Then the required bound \eqref{eq:firstiteratelincontimprov2} follows easily from an orthogonality argument. 
 
 \end{proof}

 We have now almost concluded the treatment of the contribution to the first iterate from the linear term $\mathcal{R}(\tau, \underline{x}^{(0)})$, except for the contribution to the discrete spectral part of $x^{(1)}$. Calling this contribution 
 $\triangle x^{(1)}_{d}(\tau)$, we recall from \eqref{eq:x_d} and a straightforward bootstrap argument that it suffices to bound 
 \[
 \big|\int_{\tau_0}^\infty H_d(\tau, \sigma)f_d(\sigma)\,d\sigma\big|,\,H_d(\tau, \sigma) = -\frac{1}{2}|\xi_d|^{-\frac{1}{2}}e^{-|\xi_d|^{\frac{1}{2}}|\tau-\sigma|},
 \]
 where 
\[
f_d(\sigma):= \beta_{\nu}(\sigma)\mathcal{K}_{dc}\mathcal{D}_{\sigma}x^{(0)} + \beta_{\nu}^2(\sigma)\mathcal{K}_{dc}x^{(0)} +  \beta_{\nu}^2(\sigma)\mathcal{K}_{dc}\mathcal{K}_cx^{(0)} + \mathcal{K}_{dd}x^{(0)}_{d}, 
\]
and $x^{(0)}$ here refers to the continuous spectral part of $\underline{x}^{(0)}$, while $x^{(0)}_{d}$ refers to its discrete part. Assuming the operator $\mathcal{K}_{dc}$ is given by kernel function $K_{d}(\eta)$ and again restricting to the contribution by $x_1$ to $x^{(0)}$ and the expression $ \beta_{\nu}^2(\sigma)\mathcal{K}_{dc}x^{(0)}$, say, we need to bound the absolute value of the expression
\begin{align*}
&\int_{\tau_0}^\infty H_d(\tau, \sigma)\beta_{\nu}^2(\sigma)\int_0^\infty K_{d}(\eta)\frac{\lambda^{\frac{3}{2}}(\sigma)}{\lambda^{\frac{3}{2}}(\tau_0)}\frac{\rho^{\frac{1}{2}}(\frac{\lambda^2(\sigma)}{\lambda^2(\tau_0)}\eta)}{\rho^{\frac{1}{2}}(\eta)}\\&\hspace{4cm}\cdot\frac{\sin[\lambda(\sigma)\eta^{\frac{1}{2}}\int_{\tau_0}^\sigma\lambda^{-1}(u)\,du]}{\eta^{\frac{1}{2}}}x_1(\frac{\lambda^2(\sigma)}{\lambda^2(\tau_0)}\eta)\,d\eta
\end{align*}
 Using
 \begin{align*}
 \big|\frac{\lambda^{\frac{3}{2}}(\sigma)}{\lambda^{\frac{3}{2}}(\tau_0)}\frac{\rho^{\frac{1}{2}}(\frac{\lambda^2(\sigma)}{\lambda^2(\tau_0)}\eta)}{\rho^{\frac{1}{2}}(\eta)}\big|&\lesssim \chi_{\frac{\lambda^2(\sigma)}{\lambda^2(\tau_0)}\eta<1}\frac{\lambda(\sigma)}{\lambda(\tau_0)} + \chi_{\frac{\lambda^2(\tau_0)}{\lambda^2(\sigma)}<\eta<1}\frac{\lambda^2(\sigma)}{\lambda^2(\tau_0)}\eta^{\frac{1}{2}}\\
 & + \chi_{\eta>1}\frac{\lambda^2(\sigma)}{\lambda^2(\tau_0)}
 \end{align*}
 and the boundedness and decay properties of $K_{d}(\eta)$ established in \cite{KST}, and arguing as in the proof of Proposition~\ref{prop:firstiteratelincont}, we get 
 \begin{align*}
&\big|\int_0^\infty K_{d}(\eta)\frac{\lambda^{\frac{3}{2}}(\sigma)}{\lambda^{\frac{3}{2}}(\tau_0)}\frac{\rho^{\frac{1}{2}}(\frac{\lambda^2(\sigma)}{\lambda^2(\tau_0)}\eta)}{\rho^{\frac{1}{2}}(\eta)}\cdot\beta_{\nu}^2(\sigma)\frac{\sin[\lambda(\sigma)\eta^{\frac{1}{2}}\int_{\tau_0}^\sigma\lambda^{-1}(u)\,du]}{\eta^{\frac{1}{2}}}x_1(\frac{\lambda^2(\sigma)}{\lambda^2(\tau_0)}\eta)\,d\eta\big|\\
&\lesssim (\frac{\lambda(\sigma)}{\lambda(\tau_0)})^{0+}\sigma\cdot\beta_{\nu}^2(\sigma)\big\|x_1\big\|_{S_2}. 
 \end{align*}
 It then follows that 
 \begin{align*}
 &\big|\int_{\tau_0}^\infty H_d(\tau, \sigma)\beta_{\nu}^2(\sigma)\int_0^\infty K_{d}(\eta)\frac{\lambda^{\frac{3}{2}}(\sigma)}{\lambda^{\frac{3}{2}}(\tau_0)}\frac{\rho^{\frac{1}{2}}(\frac{\lambda^2(\sigma)}{\lambda^2(\tau_0)}\eta)}{\rho^{\frac{1}{2}}(\eta)}\\&\hspace{4cm}\cdot\frac{\sin[\lambda(\sigma)\eta^{\frac{1}{2}}\int_{\tau_0}^\sigma\lambda^{-1}(u)\,du]}{\eta^{\frac{1}{2}}}x_1(\frac{\lambda^2(\sigma)}{\lambda^2(\tau_0)}\eta)\,d\eta\big|\\
 &\lesssim \big\|x_1\big\|_{S_2}\int_{\tau_0}^\infty H_d(\tau, \sigma)\sigma\beta_{\nu}^2(\sigma)\cdot (\frac{\lambda(\sigma)}{\lambda(\tau_0)})^{0+}\,d\sigma\\
 &\lesssim \tau^{-1+}\big\|x_1\big\|_{S_2}. 
 \end{align*}
Similarly, one obtains the bound $\tau^{-1+}\big\|x_1\big\|_{S_2}$ for the corresponding contribution of $\beta_{\nu}(\sigma)\mathcal{K}_{dc}\mathcal{D}_{\sigma}x^{(0)}$. Finally, the contribution of the discrete part $x^{(0)}_d$ is again elementary due to its exponential decay. 
\\

Let us summarise the preceding observations in the following 
\begin{prop}\label{prop:firstiteratelindisc} Define the function $\triangle x^{(1)}_d(\tau)$ implicitly by the equation 
\[
\triangle x^{(1)}_d(\tau) = \int_{\tau_0}^\infty H_d(\tau, \sigma)\cdot [\mathcal{R}_d(\sigma, \underline{x}^{(0)}) - \beta_{\nu}(\sigma)\partial_{\sigma}\triangle x^{(1)}_d(\sigma)]\,d\sigma
\]
where we set $\mathcal{R}_d(\sigma, \underline{x}^{(0)}) = \langle \mathcal{R}(\sigma, \underline{x}^{(0)}), \phi_d\rangle$, and  impose the vanishing condition $\lim_{\tau\rightarrow\infty}\triangle x^{(1)}_d(\tau)  = 0$. Then we have 
\[
\tau^{1-}\big[\big|\triangle x^{(1)}_d(\tau)\big| + \big|\partial_{\tau}\triangle x^{(1)}_d(\tau)\big|\big]\lesssim \big\|(x_0, x_1)\big\|_{\tilde{S}} + \big|x_{0d}\big|. 
\]

\end{prop}
 
\section{Control of the first iterate; contribution of the nonlinear terms} 

Here we show the following result which controls the contribution to the first iterate of the source term $N_{\nu}(\epsilon^{(0)})$, where we put 
\[
\tilde{\epsilon}^{(0)} = R\epsilon^{(0)} = x^{(0)}_d(\tau)\phi_d(R) + \int_0^\infty \phi(R, \xi)x^{(0)}(\tau, \xi)\rho(\xi)\,d\xi, 
\]
and $x^{(0)}$ as before stands for the zeroth iterate given by Lemma~\ref{lem:linhom}:
\begin{prop}\label{prop:firstiteratenonlin} There is a choice of $(\triangle\tilde{\tilde{x}}_0^{(1)}, \triangle\tilde{\tilde{x}}_1^{(1)})\in \tilde{S}$ with 
\begin{equation}\label{eq:keysmallnl}
\big\|(\triangle\tilde{\tilde{x}}_0^{(1)},\triangle\tilde{\tilde{x}}_1^{(1)})\big\|_{\tilde{S}}\lesssim \tau_0^{-(1-)}[\big\|(x_0, x_1)\big\|_{\tilde{S}} + |x_{0d}|]
\end{equation}
such that setting $f(\tau, \xi): = \mathcal{F}\big(\lambda^{-2}(\tau)RN_{\nu}(\tilde{\epsilon}^{(0)}) \big)(\xi)$ as well as 
 \begin{align*}
 &\triangle x^{(1)}(\tau, \xi): \\
 &=\int_{\tau_0}^\tau\frac{\lambda^{\frac{3}{2}}(\tau)}{\lambda^{\frac{3}{2}}(\sigma)}\frac{\rho^{\frac{1}{2}}(\frac{\lambda^2(\tau)}{\lambda^2(\sigma)}\xi)}{\rho^{\frac{1}{2}}(\xi)}\frac{\sin[\lambda(\tau)\xi^{\frac{1}{2}}\int_{\tau}^\sigma\lambda^{-1}(u)\,du]}{\xi^{\frac{1}{2}}}f(\sigma, (\frac{\lambda^2(\tau)}{\lambda^2(\sigma)}\xi)\,d\sigma\\
 & + S(\tau)(\triangle\tilde{\tilde{x}}_0^{(1)}, \triangle\tilde{\tilde{x}}_1^{(1)}) 
 \end{align*}
 we have the high frequency bound 
 \begin{equation}\label{eq:firstiterhighnl}
  \big\|\chi_{\xi>1}\big(\triangle x^{(1)}(\tau, \xi), \mathcal{D}_{\tau}\triangle x^{(1)}(\tau, \xi)\big)\big\|_{S}\lesssim \tau_0^{-(1-)}\big[\big\|x_0, x_1\big\|_{\tilde{S}} + |x_{0d}|\big]
  \end{equation}
  uniformly in $\tau\geq \tau_0$, as well as the improved bound 
  \begin{equation}\label{eq:firstiterhighimprovednl}\begin{split}
 &\big(\sum_{\substack{N\gtrsim \tau_0\\N\,\text{dyadic}}}\sup_{\tau\sim N}(\frac{\lambda(\tau)}{\lambda(\tau_0)})^{4\delta_0}\big\|\xi^{\frac12+}\mathcal{D}_{\tau}\big[\triangle x^{(1)}(\tau, \xi) - S(\tau)(\triangle\tilde{\tilde{x}}_0^{(1)}, \triangle\tilde{\tilde{x}}_1^{(1)})\big](\tau, \cdot)\big\|_{L^2_{d\xi}(\xi>1)}^2\big)^{\frac12}\\
&\lesssim \tau_0^{-(1-)}\big[\big\|(x_0, x_1)\big\|_{\tilde{S}} + \big|x_{0d}\big|\big].  
\end{split}\end{equation}

  Moreover, there is a splitting 
  \begin{equation}\label{eq:splitting1nl}
  \triangle x^{(1)}(\tau, \xi) =  \triangle x^{(1)}_{>\tau}(\tau, \xi) + S(\tau)\big(\triangle \tilde{x}^{(1)}_0(\xi), \triangle \tilde{x}^{(1)}_1(\xi)\big)
  \end{equation}
  such that we have 
  \begin{equation}\label{eq:firstiterlownl}\begin{split}
 &\sup_{\tau\geq \tau_0}(\frac{\tau}{\tau_0})^{-\kappa}\big\|\chi_{\xi<1}\triangle x^{(1)}_{>\tau}(\tau, \xi)\big\|_{S_1} + \sup_{\tau\geq \tau_0}\big(\frac{\tau}{\tau_0}\big)^{\kappa}\big\|\chi_{\xi<1}\mathcal{D}_{\tau}\triangle x^{(1)}_{>\tau}(\tau, \xi)\big\|_{S_2}\\&
 + \big\|\big(\triangle \tilde{x}^{(1)}_0(\xi), \triangle \tilde{x}^{(1)}_1(\xi)\big)\big\|_{S}\lesssim  \big\|(x_0, x_1)\big\|_{\tilde{S}} + |x_{0d}|.
  \end{split}\end{equation}
 The functions $\triangle \tilde{x}^{(1)}_{0,1}$ satisfy the vanishing relations \eqref{eq:vanishing}. 
Here we set as in Proposition~\ref{prop:firstiteratelincont} $\kappa = 2(1+\nu^{-1})\delta_0$. Furthermore, there is the following improvement:
 \begin{equation}\label{eq:firstiteratelincontimprov1nl}\begin{split}
&\big(\sum_{\substack{N\gtrsim \tau_0\\N\,\text{dyadic}}}\sup_{\tau\sim N}(\frac{\lambda(\tau)}{\lambda(\tau_0)})^{4\delta_0}\big\|\xi^{-0+}\mathcal{D}_{\tau}\triangle x_{>\tau}^{(1)}(\tau, \cdot)\big\|_{L^2_{d\xi}(\xi<1)}^2\big)^{\frac12}\\
&\lesssim \tau_0^{-\gamma}\big[\big\|(x_0, x_1)\big\|_{\tilde{S}} + \big|x_{0d}\big|\big]. 
 \end{split}\end{equation}
 for some $\gamma>0$. 
 \end{prop}
 \begin{rem}\label{rem:firstiteratenonlin} Observe that in \eqref{eq:firstiterlownl}  we do not gain a smallness factor here, in spite of the multilinear nature of the source term. This is due to the contribution of the term $E_2$ in the statement of Proposition~\ref{prop:nonlinbounds}, which in turn arises due to the interaction of the corrections $u_{\nu} - u_0$ with the part of $\tilde{\epsilon}$ which grows linearly in time $\tau$, i. e. a suitable multiple of the resonance $\phi(R,0)$. More precisely, the lack of a smallness gain arises for this source term in the low frequency regime, and only for the un-differentiated output (i. e. without the $\mathcal{D}_{\tau}$). This will later ensure that {\it{re-iteration}} of our modified Duhamel type parametrix to this bad term will produce better terms with a gain. 
\end{rem}
\begin{proof}(Proposition~\ref{prop:firstiteratenonlin}) This roughly follows the steps of the proof of Proposition~\ref{prop:firstiteratelincont}, with Proposition~\ref{prop:nonlinbounds} playing an important role. The choice of the correction term $\triangle\tilde{\tilde{x}}_1^{(1)}$ will be made when controlling the low frequency part of $\triangle x^{(1)}$, so we shall assume for now that such a choice has been made satisfying \eqref{eq:keysmallnl}. 
\\

{\bf{Step 1}}: {\it{Proof of the high frequency bounds \eqref{eq:firstiterhighnl}, \eqref{eq:firstiterhighimprovednl}}}. In light of Lemma~\ref{lem:freeSbound}, the proof of \eqref{eq:firstiterhighnl} will follows once we establish it for $\triangle x^{(1)}(\tau, \cdot) - S(\tau)( \triangle\tilde{\tilde{x}}_0^{(1)}, \triangle\tilde{\tilde{x}}_1^{(1)})$. Here we shall in fact establish the somewhat stronger bound \eqref{eq:firstiterhighimprovednl} for the time derivative of $\triangle x^{(1)}$, the remaining bounds being in effect similar but simpler. Then we may invoke Proposition~\ref{prop:nonlinbounds} where we find ourselves in situation (i) of that proposition. Then it remains to distinguish between two situations: 
\\

{\bf{(1a)}}: {\it{Situation (i) in the statement of Proposition~\ref{prop:nonlinbounds} with a source term of type $E_1$}}. Here we need to bound (with respect to the square-sum norm as in \eqref{eq:firstiterhighimprovednl}) the expression 
\begin{align*}
&X_1(\tau, \xi):\\& = \int_{\tau_0}^\tau\frac{\lambda^{\frac{3}{2}}(\tau)}{\lambda^{\frac{3}{2}}(\sigma)}\frac{\rho^{\frac{1}{2}}(\frac{\lambda^2(\tau)}{\lambda^2(\sigma)}\xi)}{\rho^{\frac{1}{2}}(\xi)}\cos[\lambda(\tau)\xi^{\frac{1}{2}}\int_{\tau}^\sigma\lambda^{-1}(u)\,du]\big(\langle E_1(\sigma,\cdot), \phi(\cdot, \cdot)\rangle\big)(\frac{\lambda^2(\tau)}{\lambda^2(\sigma)}\xi)\,d\sigma
\end{align*}
Observe that (using $\tau\geq \sigma$)
\begin{align*}
\big\|\langle\xi\rangle^{\frac12+}X_1(\tau, \xi)\big\|_{L^2_{d\xi}(\xi>1)}&\lesssim \int_{\tau_0}^\tau\big\|\frac{\lambda^{2}(\tau)}{\lambda^{2}(\sigma)}\langle\xi\rangle^{\frac12+}\big(\langle E_1(\sigma,\cdot), \phi(\cdot, \cdot)\rangle\big)(\frac{\lambda^2(\tau)}{\lambda^2(\sigma)}\xi)\big\|_{L^2_{d\xi}(\xi>1)}\,d\sigma\\
&\lesssim \int_{\tau_0}^\tau\big(\frac{\lambda(\sigma)}{\lambda(\tau)}\big)^{2\delta_0}\big(\frac{\lambda(\sigma)}{\lambda(\tau)}\big)^{\frac12}\big\|\xi^{\frac12}\langle E_1(\sigma,\cdot), \phi(\cdot, \xi)\rangle\big\|_{L^2_{d\rho}(\xi>1)}\,d\sigma\\
&\lesssim \int_{\tau_0}^\tau\big(\frac{\lambda(\sigma)}{\lambda(\tau)}\big)^{\frac12+2\delta_0}\big\|E_1(\sigma, \cdot)\big\|_{H^1}\,d\sigma.
\end{align*}
We conclude that 
\begin{align*}
&\big(\sum_{\substack{\tau\sim N\gtrsim \tau_0\\ N\,\text{dyadic}}}\big(\frac{\lambda(\tau)}{\lambda(\tau_0)}\big)^{4\delta_0}\sup_{\tau\sim N}\big\|\langle\xi\rangle^{\frac12+}X_1(\tau, \xi)\big\|_{L^2_{d\xi}(\xi>1)}^2\big)^{\frac12}\\
&\lesssim \big(\sum_{\substack{\tau\sim N\gtrsim \tau_0\\ N\,\text{dyadic}}}\big(\frac{\lambda(\tau)}{\lambda(\tau_0)}\big)^{4\delta_0}\sup_{\tau\sim N}\big[\int_{\tau_0}^\tau\big(\frac{\lambda(\sigma)}{\lambda(\tau)}\big)^{\frac12+2\delta_0}\big\|E_1(\sigma, \cdot)\big\|_{H^1}\,d\sigma\big]^{2}\big)^{\frac12}\\
&\lesssim \big[\big\|(x_0, x_1)\big\|_{\tilde{S}} + \big|x_{0d}\big|\big]\big(\sum_{\substack{\tau\sim N\gtrsim \tau_0\\ N\,\text{dyadic}}}\big(\frac{\lambda(\tau)}{\lambda(\tau_0)}\big)^{4\delta_0}\sup_{\tau\sim N}\big[\int_{\tau_0}^\tau\big(\frac{\lambda(\sigma)}{\lambda(\tau)}\big)^{\frac12+2\delta_0}\sigma^{-2+}\,d\sigma\big]^{2}\big)^{\frac12}\\
&\lesssim \tau_0^{-(1-)}\big[\big\|(x_0, x_1)\big\|_{\tilde{S}} + \big|x_{0d}\big|\big].
\end{align*}
provided $\delta_0\ll 1$ is sufficiently small. 
\\

{\bf{(1b)}}: {\it{Situation (i) in the statement of Proposition~\ref{prop:nonlinbounds} with a source term of type $E_2$}}. Here we get a much worse point wise decay for the $H^1$-norm of the source term, and in particular we cannot integrate the time integral in the absolute sense. Instead, we shall have to exploit some oscillation in it. To carry this out, we observe that (we use the notation used in the statement of Proposition~\ref{prop:nonlinbounds})
\begin{align*}
&\big\|\partial_{\sigma}\big[\big(\langle E_2(\sigma,\cdot), \phi(\cdot, \cdot)\rangle\big)(\frac{\lambda^2(\tau)}{\lambda^2(\sigma)}\xi)\big]\big\|_{L^2_{d\rho}(\xi>1)}\\&\lesssim \frac{\lambda(\sigma)}{\lambda(\tau)}[\sigma^{-3}|g(\sigma)| + \sigma^{-2}\partial_{\sigma}g]\big\|\langle R\rangle^{-3}g(R, a)\big\|_{H^1_{dR}},
\end{align*}
and further, 
\begin{align*}
&X_2(\tau, \xi):\\& = \int_{\tau_0}^\tau\frac{\lambda^{\frac{3}{2}}(\tau)}{\lambda^{\frac{3}{2}}(\sigma)}\frac{\rho^{\frac{1}{2}}(\frac{\lambda^2(\tau)}{\lambda^2(\sigma)}\xi)}{\rho^{\frac{1}{2}}(\xi)}\cos[\lambda(\tau)\xi^{\frac{1}{2}}\int_{\tau}^\sigma\lambda^{-1}(u)\,du]\big(\langle E_2(\sigma,\cdot), \phi(\cdot, \cdot)\rangle\big)(\frac{\lambda^2(\tau)}{\lambda^2(\sigma)}\xi)\,d\sigma\\
& =  \int_{\tau_0}^\tau \sin[\lambda(\tau)\xi^{\frac{1}{2}}\int_{\tau}^\sigma\lambda^{-1}(u)\,du]\\&\hspace{2cm}\cdot\partial_{\sigma}\big[\xi^{-\frac12}\frac{\lambda(\sigma)}{\lambda(\tau)}\frac{\lambda^{\frac{3}{2}}(\tau)}{\lambda^{\frac{3}{2}}(\sigma)}\frac{\rho^{\frac{1}{2}}(\frac{\lambda^2(\tau)}{\lambda^2(\sigma)}\xi)}{\rho^{\frac{1}{2}}(\xi)}\big(\langle E_2(\sigma,\cdot), \phi(\cdot, \cdot)\rangle\big)(\frac{\lambda^2(\tau)}{\lambda^2(\sigma)}\xi)\big]\,d\sigma\\
& + \text{boundary terms}.
\end{align*}
But then, using 
\[
\big|\sigma^{-3}|g(\sigma)| + \sigma^{-2}\partial_{\sigma}g\big|\lesssim \sigma^{-2}\big[\big\|(x_0, x_1)\big\|_{\tilde{S}} + \big|x_{0d}\big|\big],
\]
we easily infer the bound 
\begin{align*}
&\big(\sum_{\substack{\tau\sim N\gtrsim \tau_0\\ N\,\text{dyadic}}}\big(\frac{\lambda(\tau)}{\lambda(\tau_0)}\big)^{4\delta_0}\sup_{\tau\sim N}\big\|\langle\xi\rangle^{\frac12+}X_2(\tau, \xi)\big\|_{L^2_{d\xi}(\xi>1)}^2\big)^{\frac12}\\
&\lesssim \tau_0^{-1}\big[\big\|(x_0, x_1)\big\|_{\tilde{S}} + \big|x_{0d}\big|\big].
\end{align*}
This then concludes {\it{Step 1}}, i. e. the proof of the high frequency bounds. 
\\

{\bf{Step 2}}: {\it{Proof of the low frequency bounds \eqref{eq:firstiterlownl}, \eqref{eq:firstiteratelincontimprov1nl}.}} This involves in particular defining the correction terms $\triangle\tilde{\tilde{x}}_0^{(1)}, \triangle\tilde{\tilde{x}}_1^{(1)}$ and thereby the data $(\triangle \tilde{x}^{(1)}_0(\xi), \triangle \tilde{x}^{(1)}_1)$. We do this in close analogy to the proof of Proposition~\ref{prop:firstiteratelincont}, and hence shall be correspondingly more brief here. To begin with, we set 
\begin{align*}
\triangle \tilde{x}^{(1)}_0(\xi) =& \int_{\tau_0}^\infty\frac{\lambda^{\frac{3}{2}}(\tau_0)}{\lambda^{\frac{3}{2}}(\sigma)}\frac{\rho^{\frac{1}{2}}(\frac{\lambda^2(\tau_0)}{\lambda^2(\sigma)}\xi)}{\rho^{\frac{1}{2}}(\xi)}\frac{\sin[\lambda(\tau_0)\xi^{\frac{1}{2}}\int_{\tau_0}^\sigma\lambda^{-1}(u)\,du]}{\xi^{\frac{1}{2}}}f(\sigma, \frac{\lambda^2(\tau_0)}{\lambda^2(\sigma)}\xi)\,d\sigma\\
& + \triangle\tilde{\tilde{x}}_0^{(1)}\\
&=:\widetilde{\triangle} \tilde{x}^{(1)}_0(\xi) + \triangle\tilde{\tilde{x}}_0^{(1)}(\xi), 
\end{align*}
as well as 
\begin{align*}
\triangle \tilde{x}^{(1)}_1(\xi) =&\int_{\tau_0}^\infty\frac{\lambda^{\frac{3}{2}}(\tau_0)}{\lambda^{\frac{3}{2}}(\sigma)}\frac{\rho^{\frac{1}{2}}(\frac{\lambda^2(\tau_0)}{\lambda^2(\sigma)}\xi)}{\rho^{\frac{1}{2}}(\xi)}\cos[\lambda(\tau_0)\xi^{\frac{1}{2}}\int_{\tau_0}^\sigma\lambda^{-1}(u)\,du]f(\sigma, \underline{x}^{(0)})(\frac{\lambda^2(\tau_0)}{\lambda^2(\sigma)}\xi)\,d\sigma\\& + \triangle\tilde{\tilde{x}}_1^{(1)}(\xi)\\
&=:\widetilde{\triangle} \tilde{x}^{(1)}_1(\xi) + \triangle\tilde{\tilde{x}}_1^{(1)}(\xi).
\end{align*}
with $f(\sigma, \xi): = \mathcal{F}\big(\lambda^{-2}(\sigma)RN_{\nu}(\epsilon^{(0)}) \big)(\xi)$.  Then we have 

\begin{lem}\label{lem:choiceofcorrectionnl}There exists $\triangle\tilde{\tilde{x}}_0^{(1)}(\xi),\triangle\tilde{\tilde{x}}_1^{(1)}(\xi)$ satisfying \eqref{eq:keysmallnl} and such that we have 
\begin{equation}\label{eq:vanishing2nl}
\int_0^\infty \frac{(\rho^{\frac{1}{2}}\triangle\tilde{x}^{(1)}_0)(\xi)}{\xi^{\frac{1}{4}}}\cos[\nu\tau_0\xi^{\frac{1}{2}}]\,d\xi = 0,\,\int_0^\infty \frac{(\rho^{\frac{1}{2}}\triangle\tilde{x}^{(1)}_1)(\xi)}{\xi^{\frac{3}{4}}}\sin[\nu\tau_0\xi^{\frac{1}{2}}]\,d\xi = 0.
\end{equation}
\end{lem}
\begin{proof}(lemma) We show this for $\triangle\tilde{\tilde{x}}_1^{(1)}(\xi)$, the argument for $\triangle\tilde{\tilde{x}}_0^{(1)}(\xi)$ being similar. It suffices to show (see the proof of Lemma~\ref{lem:choiceofcorrection}) that 
\begin{align*}
\big|\int_0^\infty \frac{(\rho^{\frac{1}{2}}\tilde{\triangle}\tilde{x}^{(1)}_1)(\xi)}{\xi^{\frac{3}{4}}}\sin[\nu\tau_0\xi^{\frac{1}{2}}]\,d\xi\big|\lesssim \tau_0^{0+}\big[\big\|(x_0, x_1)\big\|_{\tilde{S}} + \big|x_{0d}\big|\big]
\end{align*}
To see this, use for $\xi<1$ that from Proposition~\ref{prop:nonlinbounds}
\[
\big|f(\sigma, \underline{x}^{(0)})(\frac{\lambda^2(\tau_0)}{\lambda^2(\sigma)}\xi)\big|\lesssim \sigma^{-1}\big[\big\|(x_0, x_1)\big\|_{\tilde{S}} + \big|x_{0d}\big|\big],
\]
and so (for $\xi<1$)
\[
\big|\widetilde{\triangle} \tilde{x}^{(1)}_1(\xi)\big|\lesssim \int_{\tau_0}^\infty\sigma^{-1}\frac{\lambda(\tau_0)}{\lambda(\sigma)}\,d\sigma\cdot\big[\big\|(x_0, x_1)\big\|_{\tilde{S}} + \big|x_{0d}\big|\big]
\]
Then we get 
\begin{align*}
\big|\int_0^1 \frac{(\rho^{\frac{1}{2}}\tilde{\triangle}\tilde{x}^{(1)}_1)(\xi)}{\xi^{\frac{3}{4}}}\sin[\nu\tau_0\xi^{\frac{1}{2}}]\,d\xi\big|&\lesssim \tau_0^{0+}\int_{\tau_0}^\infty\sigma^{-1}\frac{\lambda(\tau_0)}{\lambda(\sigma)}\,d\sigma\big[\big\|(x_0, x_1)\big\|_{\tilde{S}} + \big|x_{0d}\big|\big]\\
&\lesssim \tau_0^{0+}\cdot \big[\big\|(x_0, x_1)\big\|_{\tilde{S}} + \big|x_{0d}\big|\big]. 
\end{align*}
On the other hand, when $\xi>1$, use that 
\[
\big|\widetilde{\triangle} \tilde{x}^{(1)}_1(\xi)\big|\lesssim \big[\chi_{\frac{\lambda^2(\tau_0)}{\lambda^2(\sigma)}\xi<1}\frac{\lambda(\tau_0)}{\lambda(\sigma)}\xi^{-\frac12} + \chi_{\frac{\lambda^2(\tau_0)}{\lambda^2(\sigma)}\xi>1}\frac{\lambda^2(\tau_0)}{\lambda^2(\sigma)}\big]\cdot \big|f(\sigma, x^{(0)})(\frac{\lambda^2(\tau_0)}{\lambda^2(\sigma)}\xi)\big|
\]
Then in the regime $\frac{\lambda^2(\tau_0)}{\lambda^2(\sigma)}\xi<1$ we use the bound 
\[
\big|f(\sigma, x^{(0)})(\frac{\lambda^2(\tau_0)}{\lambda^2(\sigma)}\xi)\big|\lesssim \sigma^{-1}\big[\big\|(x_0, x_1)\big\|_{\tilde{S}} + \big|x_{0d}\big|\big], 
\]
while in the regime $\frac{\lambda^2(\tau_0)}{\lambda^2(\sigma)}\xi>1$, we use the $H^{0+}_{d\xi}$-bound for $f(\sigma, \xi)$ following from Proposition~\ref{prop:nonlinbounds}. This gives the bound 
\begin{align*}
\big|\int_1^\infty \frac{(\rho^{\frac{1}{2}}\tilde{\triangle}\tilde{x}^{(1)}_1)(\xi)}{\xi^{\frac{3}{4}}}\sin[\nu\tau_0\xi^{\frac{1}{2}}]\,d\xi\big|&\lesssim \tau_0^{0+}\int_{\tau_0}^\infty\sigma^{-1}[\frac{\lambda(\tau_0)}{\lambda(\sigma)}]^{1-}\,d\sigma\big[\big\|(x_0, x_1)\big\|_{\tilde{S}} + \big|x_{0d}\big|\big]\\
&\lesssim \tau_0^{0+}\cdot \big[\big\|(x_0, x_1)\big\|_{\tilde{S}} + \big|x_{0d}\big|\big]. 
\end{align*}
This completes the proof of the lemma. 
\end{proof}

Following the same sequence of steps as in the proof of Proposition~\ref{prop:firstiteratelincont}, we next establish the bound \eqref{eq:firstiterlownl} for the first two terms on the left:
\\

{\it{The estimate $\sup_{\tau\geq \tau_0}(\frac{\tau}{\tau_0})^{-\kappa}\big\|\chi_{\xi<1}\triangle x^{(1)}_{>\tau}(\tau, \xi)\big\|_{S_1} \lesssim \big\|(x_0, x_1)\big\|_{\tilde{S}} + \big|x_{0d}\big|$.}}
We need to bound 
\begin{align*}
\big\|\xi^{-0+}\int_{\tau}^\infty\frac{\lambda^{\frac{3}{2}}(\tau)}{\lambda^{\frac{3}{2}}(\sigma)}\frac{\rho^{\frac{1}{2}}(\frac{\lambda^2(\tau)}{\lambda^2(\sigma)}\xi)}{\rho^{\frac{1}{2}}(\xi)}\frac{\sin[\lambda(\tau)\xi^{\frac{1}{2}}\int_{\tau}^\sigma\lambda^{-1}(u)\,du]}{\xi^{\frac12}}f(\sigma, \underline{x}^{(0)})(\frac{\lambda^2(\tau)}{\lambda^2(\sigma)}\xi)\,d\sigma\big\|_{L^2_{d\xi}(\xi<1)}
\end{align*}
This we can do by exploiting the point wise bound on $f(\sigma, x^{(0)})(\cdot)$ used before, with a small twist. Observe the simple bound (for $\sigma\geq \tau$)
\begin{align*}
&\big\|\xi^{-0+}\frac{\lambda^{\frac{3}{2}}(\tau)}{\lambda^{\frac{3}{2}}(\sigma)}\frac{\rho^{\frac{1}{2}}(\frac{\lambda^2(\tau)}{\lambda^2(\sigma)}\xi)}{\rho^{\frac{1}{2}}(\xi)}\frac{\sin[\lambda(\tau)\xi^{\frac{1}{2}}\int_{\tau}^\sigma\lambda^{-1}(u)\,du]}{\xi^{\frac12}}\big\|_{L^2_{d\xi}(\xi<1)}\\
&\lesssim \tau^{2\delta_0}\cdot \frac{\lambda(\tau)}{\lambda(\sigma)},
\end{align*}
Note that we lose a factor $\tau^{2\delta_0}$ here, and we need to compensate for it. In fact, we claim that we get the point wise bound 
\[
\big|f(\sigma, \underline{x}^{(0)})(\frac{\lambda^2(\tau)}{\lambda^2(\sigma)}\xi)\big|\lesssim (\frac{\sigma}{\tau_0})^{\kappa}\sigma^{-1-\delta_0}\cdot [\big\|(x_0, x_1)\big\|_{\tilde{S}} + \big|x_{0d}\big|].
\]
To get this, we revisit the bad cases {\it{(2)}}, {\it{(4)}} in the proof of Proposition~\ref{prop:lingrowthcond}, which are responsible for the linear growth of $\tilde{\epsilon}$. Considering for example case {\it{(4)}}, and dividing into the cases 
$\tilde{\xi}^{\frac12}\tau_0^{1+\nu^{-1}}\tau^{-\nu^{-1}}\gtrsim 1$ or $\tilde{\xi}^{\frac12}\tau_0^{1+\nu^{-1}}\tau^{-\nu^{-1}}\lesssim 1$, we easily infer the bound (recalling $\kappa = 2\delta_0(1+\nu^{-1})$)
\begin{align*}
&\frac{\lambda(\tau)}{\lambda(\tau_0)}\big|\int_0^\infty\frac{\rho^{\frac{1}{2}}(\tilde{\xi})x_1(\tilde{\xi})}{\tilde{\xi}^{\frac{3}{4}}}\sin[\nu\tau_0\tilde{\xi}^{\frac{1}{2}}]\big(\cos[\nu\tau_0^{1+\nu^{-1}}\tau^{-\nu^{-1}}\tilde{\xi}^{\frac{1}{2}}] - 1\big)\,d\tilde{\xi}\big|\\
& +  \frac{\lambda(\tau)}{\lambda(\tau_0)}\big|\int_0^\infty\frac{\rho^{\frac{1}{2}}(\tilde{\xi})x_1(\tilde{\xi})}{\tilde{\xi}^{\frac{3}{4}}}\cos[\nu\tau_0\tilde{\xi}^{\frac{1}{2}}]\sin[\nu\tau_0^{1+\nu^{-1}}\tau^{-\nu^{-1}}\tilde{\xi}^{\frac{1}{2}}]\,d\tilde{\xi}\big|\\
&\lesssim (\frac{\tau}{\tau_0})^{\kappa}\cdot\tau^{1-2\delta_0}\big\|x_1\big\|_{S_2},
\end{align*}
which then easily implies the claimed bound on $f(\sigma, \underline{x}^{(0)})(\frac{\lambda^2(\tau)}{\lambda^2(\sigma)}\xi)$, see the proof of Proposition~\ref{prop:nonlinbounds}. One argues similarly for the contribution of $x_0$, i. e. in case {\it{(2)}}.  Armed with the preceding estimates,
we get 
\begin{align*}
&\big\|\xi^{-0+}\int_{\tau}^\infty\frac{\lambda^{\frac{3}{2}}(\tau)}{\lambda^{\frac{3}{2}}(\sigma)}\frac{\rho^{\frac{1}{2}}(\frac{\lambda^2(\tau)}{\lambda^2(\sigma)}\xi)}{\rho^{\frac{1}{2}}(\xi)}\frac{\sin[\lambda(\tau)\xi^{\frac{1}{2}}\int_{\tau}^\sigma\lambda^{-1}(u)\,du]}{\xi^{\frac12}}f(\sigma, \underline{x}^{(0)})(\frac{\lambda^2(\tau)}{\lambda^2(\sigma)}\xi)\,d\sigma\big\|_{L^2_{d\xi}(\xi<1)}\\
&\lesssim \tau^{2\delta_0}\int_{\tau}^\infty\frac{\lambda(\tau)}{\lambda(\sigma)}(\frac{\sigma}{\tau_0})^{\kappa}\sigma^{-(1+2\delta_0)}\,d\sigma\cdot \big[\big\|(x_0, x_1)\big\|_{\tilde{S}} + \big|x_{0d}\big|\big]\\
&\lesssim\big(\frac{\tau}{\tau_0}\big)^{\kappa}\cdot \big[\big\|(x_0, x_1)\big\|_{\tilde{S}} + \big|x_{0d}\big|\big], 
\end{align*}
as desired 
\\

{\it{The estimate $\sup_{\tau\geq \tau_0}(\frac{\tau}{\tau_0})^{\kappa}\big\|\chi_{\xi<1}\mathcal{D}_{\tau}\triangle x^{(1)}_{>\tau}(\tau, \xi)\big\|_{S_1} \lesssim  \big\|(x_0, x_1)\big\|_{\tilde{S}} + \big|x_{0d}\big|$.}} In fact, we immediately pass to the stronger bound \eqref{eq:firstiteratelincontimprov1nl}, which implies the former bound. 
\\
Here, we need to bound the following norm 
\begin{align*}
\big\|\xi^{-0+}\int_{\tau}^\infty\frac{\lambda^{\frac{3}{2}}(\tau)}{\lambda^{\frac{3}{2}}(\sigma)}\frac{\rho^{\frac{1}{2}}(\frac{\lambda^2(\tau)}{\lambda^2(\sigma)}\xi)}{\rho^{\frac{1}{2}}(\xi)}\cos[\lambda(\tau)\xi^{\frac{1}{2}}\int_{\tau}^\sigma\lambda^{-1}(u)\,du]f(\sigma, \underline{x}^{(0)})(\frac{\lambda^2(\tau)}{\lambda^2(\sigma)}\xi)\,d\sigma\big\|_{L^2_{d\xi}(\xi<1)}
\end{align*}
and more precisely we need to then evaluate the square-sum over dyadic time intervals as in \eqref{eq:firstiteratelincontimprov1nl}. For this estimate we need to distinguish between the two parts $E_1, E_2$ contributing to $f(\sigma, \cdot)$ according to Proposition~\ref{prop:nonlinbounds}. Assume that 
\[
f(\sigma, \cdot) = \langle \phi(R, \cdot), E_1(\sigma, R)\rangle
\]
Then according to Proposition~\ref{prop:nonlinbounds} we get 
\begin{align*}
&\big\|\xi^{-0+}\int_{\tau}^\infty\frac{\lambda^{\frac{3}{2}}(\tau)}{\lambda^{\frac{3}{2}}(\sigma)}\frac{\rho^{\frac{1}{2}}(\frac{\lambda^2(\tau)}{\lambda^2(\sigma)}\xi)}{\rho^{\frac{1}{2}}(\xi)}\cos[\lambda(\tau)\xi^{\frac{1}{2}}\int_{\tau}^\sigma\lambda^{-1}(u)\,du]f(\sigma, \underline{x}^{(0)})(\frac{\lambda^2(\tau)}{\lambda^2(\sigma)}\xi)\,d\sigma\big\|_{L^2_{d\xi}(\xi<1)}\\
&\lesssim \int_{\tau}^\infty\frac{\lambda(\tau)}{\lambda(\sigma)}\big\|f(\sigma, \cdot)\big\|_{L^\infty_{d\xi}}\,d\sigma\\
&\lesssim \big[\big\|(x_0, x_1)\big\|_{\tilde{S}} + \big|x_{0d}\big|\big]\cdot\int_{\tau}^\infty\frac{\lambda(\tau)}{\lambda(\sigma)}\sigma^{-(2-)}\,d\sigma
\end{align*}
But we have 
\begin{align*}
\big(\sum_{\substack{\tau\sim N\gtrsim \tau_0\\ N\,\text{dyadic}}}\big(\frac{\lambda(\tau)}{\lambda(\tau_0)}\big)^{4\delta_0}\sup_{\tau\sim N}\big[\int_{\tau}^\infty\frac{\lambda(\tau)}{\lambda(\sigma)}\sigma^{-(2-)}\,d\sigma\big]^2\big)^{\frac12}\lesssim \tau_0^{-(1-)}, 
\end{align*}
and so the desired bound \eqref{eq:firstiteratelincontimprov1nl} follows for this contribution. 
\\

Next, consider the case when 
\[
f(\sigma, \cdot) = \langle \phi(R, \cdot), E_2(\sigma, R)\rangle. 
\]
Here we obtain an extra gain in $\tau$ by performing an integration by parts with respect to $\sigma$. More precisely, we do this in case $\xi^{\frac12}\frac{\lambda(\tau)}{\lambda(\sigma)}\geq \sigma^{-\delta}$, while in the case $\xi^{\frac12}\frac{\lambda(\tau)}{\lambda(\sigma)}<\sigma^{-\delta}$, the gain follows from the $L^2_{d\xi}$-integral. 
\\

At this stage, the proof of the proposition will be complete provided we also show the following:
\\

{\it{The estimate $ \big\|\big(\triangle \tilde{x}^{(1)}_0(\xi), \triangle \tilde{x}^{(1)}_1(\xi)\big)\big\|_{S}\lesssim \big\|(x_0,x_1)\big\|_{\tilde{S}} + \big|x_{0d}\big|$.}} This estimate follows in the low frequency regime $\xi<1$ from the preceding estimates 
for $\triangle x^{(1)}_{>\tau}$. It therefore suffices to prove the bound 
\[
 \big\|\chi_{\xi\geq 1}\big(\triangle \tilde{x}^{(1)}_0(\xi), \triangle \tilde{x}^{(1)}_1(\xi)\big)\big\|_{S}\lesssim \big\|(x_0,x_1)\big\|_{\tilde{S}} + \big|x_{0d}\big|.
 \]
 Again we distinguish between the two parts $E_1, E_2$ contributing to $f(\sigma, \cdot)$ according to Proposition~\ref{prop:nonlinbounds}. Assuming 
 \[
 f(\sigma, \cdot) = \langle \phi(R, \cdot), \lambda^{-2}(\sigma)E_1(\sigma, R)\rangle, 
\]
we get (considering the estimate for $\triangle \tilde{x}^{(1)}_1(\xi)$, the one for $\triangle \tilde{x}^{(1)}_0(\xi)$ being similar)
\begin{align*}
&\big\|\xi^{\frac12+}\int_{\tau_0}^\infty\frac{\lambda^{\frac{3}{2}}(\tau_0)}{\lambda^{\frac{3}{2}}(\sigma)}\frac{\rho^{\frac{1}{2}}(\frac{\lambda^2(\tau_0)}{\lambda^2(\sigma)}\xi)}{\rho^{\frac{1}{2}}(\xi)}\cos[\lambda(\tau_0)\xi^{\frac{1}{2}}\int_{\tau_0}^\sigma\lambda^{-1}(u)\,du]f(\sigma, \underline{x}^{(0)})(\frac{\lambda^2(\tau_0)}{\lambda^2(\sigma)}\xi)\,d\sigma\big\|_{L^2_{d\xi}(\xi\geq 1)}\\
&\lesssim \int_{\tau_0}^\infty\big(\frac{\lambda(\sigma)}{\lambda(\tau_0)}\big)^{0+}\big\|\langle\xi\rangle^{\frac12+}f(\sigma, \underline{x}^{(0)})\big\|_{L^2_{d\xi}}\,d\sigma\\
&\lesssim  \big[\big\|(x_0, x_1)\big\|_{\tilde{S}} + \big|x_{0d}\big|\big]\cdot\int_{\tau_0}^\infty\big(\frac{\lambda(\tau_0)}{\lambda(\sigma)}\big)^{-0+}\sigma^{-(2-)}\,d\sigma\\
&\lesssim \big\|(x_0, x_1)\big\|_{\tilde{S}} + \big|x_{0d}\big|. 
\end{align*}
The case when 
\[
 f(\sigma, \cdot) = \langle \phi(R, \cdot), \lambda^{-2}(\sigma)E_2(\sigma, R)\rangle, 
\]
is more complicated, since we only get $\sigma^{-1}$-decay. Here we need to exploit the better smoothness properties of $E_2(\sigma, R)$ with respect to $\sigma$, and perform integration by parts with respect to $\sigma$. Specifically, we use that 
\begin{align*}
&\big\|\frac{\lambda(\tau_0)}{\lambda(\sigma)}\cdot \frac{\partial}{\partial_{\sigma}}\big(\langle  \phi(R, \frac{\lambda^2(\tau_0)}{\lambda^2(\sigma)}\xi), \lambda^{-2}(\sigma)E_2(\sigma, R)\rangle\big)\big\|_{H^{0+}_{d\xi}}\\
&\lesssim \big(\frac{\lambda(\tau_0)}{\lambda(\sigma)}\big)^{-0-}\sigma^{-2}\cdot [ \big\|(x_0, x_1)\big\|_{\tilde{S}} + \big|x_{0d}\big|].
\end{align*}
Then we obtain the desired bound via integration by parts with respect to $\sigma$ in the regime $\frac{\lambda^2(\tau_0)}{\lambda^2(\sigma)}\xi>\sigma^{-\delta}$ for some small enough $\delta>0$ (independent of $\delta_2$), while in the regime 
\[
\frac{\lambda^2(\tau_0)}{\lambda^2(\sigma)}\xi<\sigma^{-\delta}
\]
we obtain a gain in $\sigma^{-1}$ directly from the $L^2_{d\xi}$-integral, as well as exploiting $L^\infty_{d\xi}$ control over $f(\sigma, \cdot)$. Observe that the preceding reasoning in fact gives a better bound 
\[
 \big\|\chi_{\xi\geq 1}\big(\triangle \tilde{x}^{(1)}_0(\xi), \triangle \tilde{x}^{(1)}_1(\xi)\big)\big\|_{S}\lesssim \tau_0^{-\delta}[\big\|(x_0,x_1)\big\|_{\tilde{S}} + \big|x_{0d}\big|].
 \]
This concludes the proof of Proposition~\ref{prop:firstiteratenonlin} . 
\end{proof}

In analogy to Proposition~\ref{prop:firstiteratelindisc}, we also get a straightforward bound for the discrete part of the first iterate: write
\[
f_d(\sigma, \underline{x}^{(0)}) = \langle \phi_d(R), \lambda^{-2}(\sigma)N_{\nu}(\epsilon^{(0)})\rangle
\]
\begin{prop}\label{prop:firstiteratenonlindisc} Define the function $\triangle x^{(1)}_d(\tau)$ implicitly by the equation 
\[
\triangle x^{(1)}_d(\tau) = \int_{\tau_0}^\infty H_d(\tau, \sigma)\cdot [f_d(\sigma, \underline{x}^{(0)}) - \beta_{\nu}(\sigma)\partial_{\sigma}\triangle x^{(1)}_d(\sigma)]\,d\sigma,
\]
and  impose the vanishing condition $\lim_{\tau\rightarrow\infty}\triangle x^{(1)}_d(\tau)  = 0$. Then we have 
\[
\tau^{1-}\big[\big|\triangle x^{(1)}_d(\tau)\big| + \big|\partial_{\tau}\triangle x^{(1)}_d(\tau)\big|\big]\lesssim \big\|(x_0, x_1)\big\|_{\tilde{S}} + \big|x_{0d}\big|. 
\]

\end{prop}

\section{Iterative step; a priori control of higher iterates}

Here, we take the conclusions about the structure of the first iterate from the preceding sections as a priori given and recover them for the next higher iterate. More precisely, at first, we show this for our approximate Duhamel parameterix applied to the linear source terms $\mathcal{R}(\sigma, \cdot)$. As this step will be used infinitely often, we replace the notation $\triangle x^{(1)}$ from before by 
$\triangle x^{(input)}$, and the next iterate will be denoted by $\triangle x^{(output)}$. The role of the correction terms to the data $\triangle \tilde{\tilde{x}}^{(1)}_0, \triangle \tilde{\tilde{x}}^{(1)}_1$ will be played by functions $\triangle\tilde{\tilde{x}}^{(input)}_{0,1}(\xi)$. The following proposition only takes the continuous spectral part of the full operator $\mathcal{R}(\sigma,\underline{\triangle x}^{(input)})$ into account:

\begin{prop}\label{prop:lineariterstep}Assume that there are functions $\triangle\tilde{\tilde{x}}^{(input)}_0(\xi), \triangle\tilde{\tilde{x}}^{(input)}_1(\xi)$ with 
\begin{equation}
\big\|(\triangle\tilde{\tilde{x}}^{(input)}_0(\xi), \triangle\tilde{\tilde{x}}^{(input)}_1(\xi))\big\|_{\tilde{S}}\leq A,
\end{equation}
and such that the function $\triangle x^{(input)}(\tau, \xi)$ satisfies the bound (as before $\kappa = 2(1+\nu^{-1})\delta_0$)
\begin{align*}
&\sup_{\tau\geq \tau_0}\big[\big(\frac{\tau_0}{\tau}\big)^{\kappa}\big\|\chi_{\xi>1}\triangle x^{(input)}(\tau, \xi)\big\|_{S_1}\\& + \big(\sum_{\substack{N\gtrsim \tau_0\\ N\,\text{dyadic}}}[\sup_{\tau\sim N}\big(\frac{\tau}{\tau_0}\big)^{2\kappa}\big\|\chi_{\xi>1}\mathcal{D}_{\tau}[\triangle x^{(input)} (\tau, \xi) - S(\tau)(\triangle\tilde{\tilde{x}}^{(input)}_0, \triangle\tilde{\tilde{x}}^{(input)}_1)]\big\|_{S_2}]^2\big)^{\frac{1}{2}}\\&\leq A,
\end{align*}
and furthermore that there exist $(\triangle\tilde{x}^{(input)}_0(\xi), \triangle\tilde{x}^{(input)}_1(\xi))$ with 
\begin{align*}
\big\|(\triangle\tilde{x}^{(input)}_0, \triangle\tilde{x}^{(input)}_1)\big\|_{S}\leq A,\,&\int_0^\infty\frac{(\rho^{\frac{1}{2}}\triangle\tilde{x}^{(input)}_1)(\xi)}{\xi^{\frac{3}{4}}}\sin[\nu\tau_0\xi^{\frac{1}{2}}]\,d\xi = 0,\\
&\int_0^\infty\frac{(\rho^{\frac{1}{2}}\triangle\tilde{x}^{(input)}_0)(\xi)}{\xi^{\frac{1}{4}}}\cos[\nu\tau_0\xi^{\frac{1}{2}}]\,d\xi = 0,
\end{align*}
such that upon writing 
\begin{equation}\label{eq:trianglexinput}
\triangle x^{(input)} (\tau, \xi) = \triangle_{>\tau} x^{(input)} (\tau, \xi) + S(\tau)\big(\triangle\tilde{x}^{(input)}_0, \triangle\tilde{x}^{(input)}_1\big),
\end{equation}
we have 
\begin{align*}
&\sup_{\tau\geq \tau_0}\big[\big(\frac{\tau_0}{\tau}\big)^{\kappa}\big\|\chi_{\xi<1}\triangle_{>\tau} x^{(input)}(\tau, \xi)\big\|_{S_1}\\& + \big(\sum_{\substack{N\gtrsim \tau_0\\ N\,\text{dyadic}}}\sup_{\tau\sim N}\big((\frac{\tau}{\tau_0})^{\kappa}\big\|\chi_{\xi<1}\mathcal{D}_{\tau}\triangle_{>\tau} x^{(input)} (\tau, \xi)\big\|_{S_2}\big)^2\big)^{\frac12}\leq A
\end{align*}
Then there exist $\triangle\tilde{\tilde{x}}^{(output)}_0, \triangle\tilde{\tilde{x}}^{(output)}_1$ satisfying 
\begin{equation}\label{eq:newcorrection}
\big\|(\triangle\tilde{\tilde{x}}^{(output)}_0, \triangle\tilde{\tilde{x}}^{(output)}_1)\big\|_{\tilde{S}}\lesssim \tau_0^{-(1-)}A, 
\end{equation}
and such that if we introduce 
 \begin{align*}
 &\triangle x^{(output)}(\tau, \xi): \\
 &=\int_{\tau_0}^\tau\frac{\lambda^{\frac{3}{2}}(\tau)}{\lambda^{\frac{3}{2}}(\sigma)}\frac{\rho^{\frac{1}{2}}(\frac{\lambda^2(\tau)}{\lambda^2(\sigma)}\xi)}{\rho^{\frac{1}{2}}(\xi)}\frac{\sin[\lambda(\tau)\xi^{\frac{1}{2}}\int_{\tau}^\sigma\lambda^{-1}(u)\,du]}{\xi^{\frac{1}{2}}}\mathcal{R}_1(\sigma, \triangle x^{(input)} (\sigma, \xi) )(\frac{\lambda^2(\tau)}{\lambda^2(\sigma)}\xi)\,d\sigma\\
 & + S(\tau)\big(\triangle\tilde{\tilde{x}}^{(output)}_0, \triangle\tilde{\tilde{x}}^{(output)}_1\big),
 \end{align*}
then we have the high-frequency bound 
 \begin{equation}\label{eq:apriorihighfreq}\begin{split}
  &\sup_{\tau\geq \tau_0}(\frac{\tau_0}{\tau})^{\kappa}\big\|\chi_{\xi>1}\triangle x^{(output)}(\tau, \xi)\big\|_{S_1}\\& +  \big(\sum_{\substack{N\gtrsim \tau_0\\ N\,\text{dyadic}}}\sup_{\tau\sim N}(\frac{\tau}{\tau_0})^{2\kappa}\big\|\chi_{\xi>1}\mathcal{D}_{\tau}[\triangle x^{(output)}(\tau, \xi) - S(\tau)\big(\triangle\tilde{\tilde{x}}^{(output)}_0, \triangle\tilde{\tilde{x}}^{(output)}_1\big)]\big\|_{S_2}^2\big)^{\frac12}
 \lesssim A.
  \end{split}\end{equation}
  Moreover, there is a splitting 
  \begin{equation}\label{eq:splitting1}
  \triangle x^{(output)}(\tau, \xi) =  \triangle x^{(output)}_{>\tau}(\tau, \xi) + S(\tau)\big(\triangle \tilde{x}^{(output)}_0(\xi), \triangle \tilde{x}^{(output)}_1(\xi)\big)
  \end{equation}
  such that we have 
  \begin{equation}\label{eq:apriorilowfreq}\begin{split}
 &\sup_{\tau\geq \tau_0}(\frac{\tau}{\tau_0})^{-\kappa}\big\|\chi_{\xi<1}\triangle x^{(output)}_{>\tau}(\tau, \xi)\big\|_{S_1} + \big(\sum_{\substack{N\gtrsim \tau_0\\ N\,\text{dyadic}}}[\sup_{\tau\sim N}(\frac{\tau}{\tau_0})^{\kappa}\big\|\chi_{\xi<1}\mathcal{D}_{\tau}\triangle_{>\tau}x^{(output)}(\tau, \xi)\big\|_{S_2}]^2\big)^{\frac{1}{2}}\\&
 + \big\|\big(\triangle \tilde{x}^{(output)}_0(\xi), \triangle \tilde{x}^{(output)}_1(\xi)\big)\big\|_{S}\lesssim A,
  \end{split}\end{equation}
  and furthermore we have the orthogonality conditions 
\[
\int_0^\infty\frac{(\rho^{\frac{1}{2}}\triangle \tilde{x}^{(output)}_0)(\xi)}{\xi^{\frac{1}{4}}}\cos[\nu\tau_0\xi^{\frac{1}{2}}]\,d\xi = 0,\,\int_0^\infty\frac{(\rho^{\frac{1}{2}}\triangle \tilde{x}^{(output)}_1)(\xi)}{\xi^{\frac{3}{4}}}\sin[\nu\tau_0\xi^{\frac{1}{2}}]\,d\xi = 0.
\]  
\end{prop}
\begin{rem}\label{rem:proplineariterstep} One can obtain the same conclusion by imposing the weaker condition
\begin{align*}
\big(\frac{\xi^{\frac12-}}{\langle\xi\rangle^{\frac12-}}\triangle \tilde{x}^{(input)}_0(\xi), \triangle \tilde{x}^{(input)}_1(\xi)\big)\big\|_{S} \lesssim A.
\end{align*}
\end{rem}
\begin{proof} 

We follow the pattern of proof of Proposition~\ref{prop:firstiteratenonlin}, Proposition~\ref{prop:firstiteratelincont}. 
\\

{\bf{Step 1}}: {\it{High frequency estimates}}. Here we prove the bound \eqref{eq:apriorihighfreq}. We use the analogue of \eqref{eq:R_1structure}, which for us becomes 
\[
\mathcal{R}_1(\tau, \triangle x^{(input)})(\xi) = \beta_{\nu}^2(\tau)\mathcal{K}_{cc}\triangle x^{(input)} + \beta_{\nu}(\tau)\mathcal{K}_{cc}\mathcal{D}_{\tau}\triangle x^{(input)} +  \beta_{\nu}^2(\tau)\mathcal{K}_{cc}^2\triangle x^{(input)}
\]
We shall in detail treat the contribution of the first two terms on the right, the last one being more of the same and similar to the first term. We further split this into the following tasks: 
\\

{\bf{(1a)}}: {\it{The contribution of $\beta_{\nu}^2(\tau)\mathcal{K}_{cc}\triangle x^{(input)}$ to the term $(\frac{\tau_0}{\tau})^{\kappa}\big\|\chi_{\xi>1}\triangle x^{(output)}(\tau, \xi)\big\|_{S_1}$.}}
We need to bound 
\[
\big\|\chi_{\xi>1}\int_{\tau_0}^\tau\frac{\lambda^{\frac{3}{2}}(\tau)}{\lambda^{\frac{3}{2}}(\sigma)}\frac{\rho^{\frac{1}{2}}(\frac{\lambda^2(\tau)}{\lambda^2(\sigma)}\xi)}{\rho^{\frac{1}{2}}(\xi)}\frac{\sin[\lambda(\tau)\xi^{\frac{1}{2}}\int_{\tau}^\sigma\lambda^{-1}(u)\,du]}{\xi^{\frac{1}{2}}}\beta_{\nu}^2(\sigma)\big(\mathcal{K}_{cc}\triangle x^{(input)}\big)(\sigma, \frac{\lambda^2(\tau)}{\lambda^2(\sigma)}\xi)\,d\sigma\big\|_{\langle\xi\rangle^{-1-}L^2_{d\xi}}
\]
We distinguish between the following cases for the integral defining $\mathcal{K}_{cc}$: 
\\

{\it{(1a.i): $\frac{\lambda^2(\tau)}{\lambda^2(\sigma)}\xi\ll \eta$. In particular, $\eta>1$.}}  Here we can use the high-frequency bound for $\triangle x^{(input)}$, as well as the fast decay of the kernel $F(\cdot, \cdot)$ in $\mathcal{K}_{cc}$, which indeed amounts to (see \cite{KST})
\[
\big([\frac{\lambda^2(\tau)}{\lambda^2(\sigma)}\xi]^2\eta\big)^{-N}.
\]
This factor more than compensates for the growing term 
\[
\frac{\lambda^{\frac{3}{2}}(\tau)}{\lambda^{\frac{3}{2}}(\sigma)}\frac{\rho^{\frac{1}{2}}(\frac{\lambda^2(\tau)}{\lambda^2(\sigma)}\xi)}{\rho^{\frac{1}{2}}(\xi)}\sim \frac{\lambda^2(\tau)}{\lambda^2(\sigma)}
\]
Then using 
\[
\big|\big(\mathcal{K}_{\frac{\lambda^2(\tau)}{\lambda^2(\sigma)}\xi\ll \eta}\triangle x^{(input)}\big)(\sigma, \frac{\lambda^2(\tau)}{\lambda^2(\sigma)}\xi)\big|\lesssim [\frac{\lambda^2(\tau)}{\lambda^2(\sigma)}\xi]^{-2}\big\|\chi_{\eta>1}\langle\eta\rangle^{1+}\triangle x^{(input)}(\sigma, \cdot)\big\|_{L^2_{d\eta}},
\]
we get 
\begin{align*}
&\big\|\chi_{\xi>1}\int_{\tau_0}^\tau\frac{\lambda^{\frac{3}{2}}(\tau)}{\lambda^{\frac{3}{2}}(\sigma)}\frac{\rho^{\frac{1}{2}}(\frac{\lambda^2(\tau)}{\lambda^2(\sigma)}\xi)}{\rho^{\frac{1}{2}}(\xi)}\frac{\sin[\lambda(\tau)\xi^{\frac{1}{2}}\int_{\tau}^\sigma\lambda^{-1}(u)\,du]}{\xi^{\frac{1}{2}}}\\&\hspace{3cm}\cdot\beta_{\nu}^2(\sigma)\big(\mathcal{K}_{\frac{\lambda^2(\tau)}{\lambda^2(\sigma)}\xi\ll \eta}\triangle x^{(input)}\big)(\sigma, \frac{\lambda^2(\tau)}{\lambda^2(\sigma)}\xi)\,d\sigma\big\|_{\langle\xi\rangle^{-1-}L^2_{d\xi}}\\
&\lesssim A\int_{\tau_0}^\tau\sigma^{-2}\big(\frac{\sigma}{\tau_0}\big)^{\kappa}\,d\sigma\lesssim A,
\end{align*}
which indeed improves the bound needed, provided $\delta_0$ is chosen sufficiently small. 
\\

{\it{(1a.ii): $\frac{\lambda^2(\tau)}{\lambda^2(\sigma)}\xi\sim\eta$. In particular, $\eta\gtrsim 1$.}} Here one needs to be more careful since the decay of the kernel $F(\cdot, \cdot)$ is weaker in the singular region, see \cite{KST}. Still, we can easily close, by using 
\[
\frac{\lambda(\tau)}{\lambda(\sigma)}\cdot\big\|(\frac{\lambda^2(\tau)}{\lambda^2(\sigma)}\xi)^{\frac12+}\mathcal{K}_{\frac{\lambda^2(\tau)}{\lambda^2(\sigma)}\xi\sim\eta}\triangle x^{(input)}(\sigma, \cdot)\big\|_{L^2_{d\xi}}\lesssim \big\|\langle\eta\rangle^{0+}\triangle x^{(input)}(\sigma, \cdot)\big\|_{L^2_{d\eta}},
\]
which is a consequence of the bound 
\[
\big|\xi^{\frac12+}F(\xi, \eta)\rho(\eta)\big|\lesssim \xi^{0+}.
\]
It follows that 
\begin{align*}
&\big\|\chi_{\xi>1}\int_{\tau_0}^\tau\frac{\lambda^{\frac{3}{2}}(\tau)}{\lambda^{\frac{3}{2}}(\sigma)}\frac{\rho^{\frac{1}{2}}(\frac{\lambda^2(\tau)}{\lambda^2(\sigma)}\xi)}{\rho^{\frac{1}{2}}(\xi)}\frac{\sin[\lambda(\tau)\xi^{\frac{1}{2}}\int_{\tau}^\sigma\lambda^{-1}(u)\,du]}{\xi^{\frac{1}{2}}}\\&\hspace{3cm}\cdot\beta_{\nu}^2(\sigma)\big(\mathcal{K}_{\frac{\lambda^2(\tau)}{\lambda^2(\sigma)}\xi\sim\eta}\triangle x^{(input)}\big)(\sigma, \frac{\lambda^2(\tau)}{\lambda^2(\sigma)}\xi)\,d\sigma\big\|_{\langle\xi\rangle^{-1-}L^2_{d\xi}}\\
&\lesssim A\int_{\tau_0}^\tau\sigma^{-2}\big(\frac{\sigma}{\tau_0}\big)^{\kappa}\,d\sigma\lesssim A,
\end{align*}

{\it{(1a.iii): $\frac{\lambda^2(\tau)}{\lambda^2(\sigma)}\xi\gg\eta$.}} Here one distinguishes between the regimes $\eta>1, \eta<1$ for defining $\mathcal{K}_{\frac{\lambda^2(\tau)}{\lambda^2(\sigma)}\xi\gg\eta}\triangle x^{(input)}(\sigma, \cdot)$. In the former case, one argues exactly as in situation {\it{(1a.i)}}. In the latter case we need to exploit the structure of $\triangle x^{(input)}$ given by \eqref{eq:trianglexinput}. The contribution of $S(\tau)\big(\triangle\tilde{x}^{(input)}_0, \triangle\tilde{x}^{(input)}_1\big)$ is handled as in case {\it{(1(i))}} in the proof of Proposition~\ref{prop:firstiteratelincont}, with $S(\tau)\big(\triangle\tilde{x}^{(input)}_0, \triangle\tilde{x}^{(input)}_1\big)$  replacing $x^{(0)}$. Thus we then reduce to bounding the expression 
\begin{align*}
&\big\|\chi_{\xi>1}\int_{\tau_0}^\tau\frac{\lambda^{\frac{3}{2}}(\tau)}{\lambda^{\frac{3}{2}}(\sigma)}\frac{\rho^{\frac{1}{2}}(\frac{\lambda^2(\tau)}{\lambda^2(\sigma)}\xi)}{\rho^{\frac{1}{2}}(\xi)}\frac{\sin[\lambda(\tau)\xi^{\frac{1}{2}}\int_{\tau}^\sigma\lambda^{-1}(u)\,du]}{\xi^{\frac{1}{2}}}\\&\hspace{3cm}\cdot\beta_{\nu}^2(\sigma)\big(\mathcal{K}_{\eta<1}\triangle_{>\sigma}x^{(input)}\big)(\sigma, \frac{\lambda^2(\tau)}{\lambda^2(\sigma)}\xi)\,d\sigma\big\|_{\langle\xi\rangle^{-1-}L^2_{d\xi}}.
\end{align*}
This is straightforward since 
\[
\big|\big(\mathcal{K}_{\eta<1}\triangle_{>\sigma}x^{(input)}\big)(\sigma, \frac{\lambda^2(\tau)}{\lambda^2(\sigma)}\xi)\big|\lesssim \big(\frac{\lambda^2(\tau)}{\lambda^2(\sigma)}\xi\big)^{-N}\big\|\eta^{-0-}\triangle_{>\sigma}x^{(input)}\big\|_{L^2_{d\eta}},
\]
which then gives 
\begin{align*}
&\big\|\chi_{\xi>1}\int_{\tau_0}^\tau\frac{\lambda^{\frac{3}{2}}(\tau)}{\lambda^{\frac{3}{2}}(\sigma)}\frac{\rho^{\frac{1}{2}}(\frac{\lambda^2(\tau)}{\lambda^2(\sigma)}\xi)}{\rho^{\frac{1}{2}}(\xi)}\frac{\sin[\lambda(\tau)\xi^{\frac{1}{2}}\int_{\tau}^\sigma\lambda^{-1}(u)\,du]}{\xi^{\frac{1}{2}}}\\&\hspace{3cm}\cdot\beta_{\nu}^2(\sigma)\big(\mathcal{K}_{\eta<1}\triangle_{>\sigma}x^{(input)}\big)(\sigma, \frac{\lambda^2(\tau)}{\lambda^2(\sigma)}\xi)\,d\sigma\big\|_{\langle\xi\rangle^{-1-}L^2_{d\xi}}\\
&\lesssim  A\int_{\tau_0}^\tau\sigma^{-2}\big(\frac{\sigma}{\tau_0}\big)^{\kappa}\,d\sigma\lesssim A.
\end{align*}

{\bf{(1b)}}: {\it{The contribution of $\beta_{\nu}^2(\tau)\mathcal{K}_{cc}\triangle x^{(input)}$ to the term 
\[
\big(\sum_{\substack{N\gtrsim \tau_0\\ N\,\text{dyadic}}}\sup_{\tau\sim N}(\frac{\tau}{\tau_0})^{2\kappa}\big\|\chi_{\xi>1}\mathcal{D}_{\tau}[\triangle x^{(output)}(\tau, \xi) - S(\tau)\big(\triangle\tilde{\tilde{x}}^{(output)}_0, \triangle\tilde{\tilde{x}}^{(output)}_1\big)]\big\|_{S_2}^2\big)^{\frac12}.
\]
}} 
Note that 
\[
\triangle x^{(output)}(\tau, \xi) - S(\tau)\big(\triangle\tilde{\tilde{x}}^{(output)}_0, \triangle\tilde{\tilde{x}}^{(output)}_1\big)
\]
is simply the pure Duhamel term 
\[
\int_{\tau_0}^\tau\frac{\lambda^{\frac{3}{2}}(\tau)}{\lambda^{\frac{3}{2}}(\sigma)}\frac{\rho^{\frac{1}{2}}(\frac{\lambda^2(\tau)}{\lambda^2(\sigma)}\xi)}{\rho^{\frac{1}{2}}(\xi)}\frac{\sin[\lambda(\tau)\xi^{\frac{1}{2}}\int_{\tau}^\sigma\lambda^{-1}(u)\,du]}{\xi^{\frac{1}{2}}}\mathcal{R}_1(\sigma, \triangle x^{(input)} (\sigma, \xi) )(\frac{\lambda^2(\tau)}{\lambda^2(\sigma)}\xi)\,d\sigma.
\]
We then briefly go through the cases {\it{(1b.i)}} - {\it{(1b.iii)}} which correspond to  {\it{(1a.i)}} - {\it{(1a.iii)}}: 
\\

{\it{(1b.i) and (1b.iii)}}. One gains $\big(\frac{\lambda(\tau)}{\lambda(\sigma)}\big)^{-N}$, and the preceding analysis gives a gain of $\sigma^{-2}$. This allows to include a weight $(\frac{\tau}{\tau_0})^{\kappa}$ (in fact much more) and square-sum over dyadic $\tau\geq \tau_0$. 
\\

{\it{(1b.ii)}}. Here the gain is smaller but still more than enough to close. One uses that 
\begin{align*}
&\frac{\lambda(\tau)}{\lambda(\sigma)}\cdot\big\|(\frac{\lambda^2(\tau)}{\lambda^2(\sigma)}\xi)^{\frac12+}\mathcal{K}_{\frac{\lambda^2(\tau)}{\lambda^2(\sigma)}\xi\sim\eta}\triangle x^{(input)}(\sigma, \cdot)\big\|_{L^2_{d\xi}}
\\&
\lesssim \frac{\lambda(\sigma)}{\lambda(\tau)}\big\|\langle\eta\rangle^{\frac12+}\triangle x^{(input)}(\sigma, \cdot)\big\|_{L^2_{d\eta}}
\end{align*}
and continues as in {\it{(1a.ii)}}. 
\\

{\bf{(1c)}}: {\it{The contribution of $\beta_{\nu}(\tau)\mathcal{K}_{cc}\mathcal{D}_{\tau}\triangle x^{(input)}$ to the term $(\frac{\tau_0}{\tau})^{\kappa}\big\|\chi_{\xi>1}\triangle x^{(output)}(\tau, \xi)\big\|_{S_1}$.}}
Here we need to bound the expression 
\[
\big\|\chi_{\xi>1}\int_{\tau_0}^\tau\frac{\lambda^{\frac{3}{2}}(\tau)}{\lambda^{\frac{3}{2}}(\sigma)}\frac{\rho^{\frac{1}{2}}(\frac{\lambda^2(\tau)}{\lambda^2(\sigma)}\xi)}{\rho^{\frac{1}{2}}(\xi)}\frac{\sin[\lambda(\tau)\xi^{\frac{1}{2}}\int_{\tau}^\sigma\lambda^{-1}(u)\,du]}{\xi^{\frac{1}{2}}}\beta_{\nu}(\sigma)\big(\mathcal{K}_{cc}\mathcal{D}_{\sigma}\triangle x^{(input)}\big)(\sigma, \frac{\lambda^2(\tau)}{\lambda^2(\sigma)}\xi)\,d\sigma\big\|_{\langle\xi\rangle^{-1-}L^2_{d\xi}},
\]
which we do in close analogy to case {\bf{(1a)}}. In fact, while we only have weight $\beta_{\nu}(\sigma)$ instead of $\beta_{\nu}^2(\sigma)$, we get better decay by assumption for $\mathcal{D}_{\sigma}\triangle x^{(input)}$. 
\\

{\bf{(1d)}}: {\it{The contribution of $\beta_{\nu}(\tau)\mathcal{K}_{cc}\mathcal{D}_{\tau}\triangle x^{(input)}$ to the term 
\[
\big(\sum_{\substack{N\gtrsim \tau_0\\ N\,\text{dyadic}}}\sup_{\tau\sim N}(\frac{\tau}{\tau_0})^{2\kappa}\big\|\chi_{\xi>1}\mathcal{D}_{\tau}[\triangle x^{(output)}(\tau, \xi) - S(\tau)\big(\triangle\tilde{\tilde{x}}^{(output)}_0, \triangle\tilde{\tilde{x}}^{(output)}_1\big)]\big\|_{S_2}^2\big)^{\frac12}.
\]
}} 
This case is a bit more delicate than the preceding {\it{(1a)}} - {\it{(1c)}}, in that the estimate is sharp here. First, using the already established Proposition~\ref{prop:firstiteratelincontimprov}, we may replace $\triangle x^{(input)}$ in the high-frequency regime $\xi>1$ by 
\[
\widetilde{\triangle x^{(input)}}: = \triangle x^{(input)} - S(\tau)\big(\triangle\tilde{\tilde{x}}^{(input)}_0, \triangle\tilde{\tilde{x}}^{(input)}_1\big)
\]
and in the low frequency regime $\xi<1$ by $\triangle_{>\tau}x^{(input)}$.

Then, following the steps in {\it{(1a)}} and using the improvement in {\it{(1b.ii)}} above as well as the Cauchy-Schwarz inequality to reduce over a square-sum over dyadic $\sigma$, we get the bound 
\begin{align*}
&\big(\sum_{\substack{N\gtrsim \tau_0\\ N\,\text{dyadic}}}\sup_{\tau\sim N}(\frac{\tau}{\tau_0})^{2\kappa}\big\|\chi_{\xi>1}\mathcal{D}_{\tau}[\triangle x^{(output)}(\tau, \xi) - S(\tau)\big(\triangle\tilde{\tilde{x}}^{(output)}_0, \triangle\tilde{\tilde{x}}^{(output)}_1\big)]\big\|_{S_2}^2\big)^{\frac12}\\
&\lesssim \big(\sum_{\substack{N\gtrsim \tau_0\\ N\,\text{dyadic}}}\sup_{\tau\sim N}(\frac{\tau}{\tau_0})^{2\kappa}\sum_{\substack{\tau_0\lesssim M\lesssim N\\ M\,\text{dyadic}}}(\frac{M}{\tau_0})^{-2\kappa}\big(\frac{\lambda(M)}{\lambda(N)}\big)^2\sup_{\sigma\sim M}(\frac{\sigma}{\tau_0})^{2\kappa}\big\|\mathcal{D}_{\sigma}\widetilde{\triangle x^{(input)}}\big\|_{S_2(\xi>1)}^2\big)^{\frac12}\\
& + \big(\sum_{\substack{N\gtrsim \tau_0\\ N\,\text{dyadic}}}\sup_{\tau\sim N}(\frac{\tau}{\tau_0})^{2\kappa}\sum_{\substack{\tau_0\lesssim M\lesssim N\\ M\,\text{dyadic}}}(\frac{M}{\tau_0})^{-2\kappa}\big(\frac{\lambda(M)}{\lambda(N)}\big)^2\sup_{\sigma\sim M}(\frac{\sigma}{\tau_0})^{2\kappa}\big\|\mathcal{D}_{\sigma}\triangle_{>\sigma}x^{(input)}\big\|_{S_2(\xi<1)}^2\big)^{\frac12}\\
&\lesssim \big(\sum_{\substack{\tau_0\lesssim M\\ M\,\text{dyadic}}}\sup_{\sigma\sim M}(\frac{\sigma}{\tau_0})^{2\kappa}\big\|\mathcal{D}_{\sigma}\widetilde{\triangle x^{(input)}}\big\|_{S_2(\xi>1)}^2\big)^{\frac12} + \big(\sum_{\substack{\tau_0\lesssim M\\ M\,\text{dyadic}}}\sup_{\sigma\sim M}(\frac{\sigma}{\tau_0})^{2\kappa}\big\|\mathcal{D}_{\sigma}\triangle_{>\sigma} x^{(input)}\big\|_{S_2(\xi<1)}^2\big)^{\frac12}
\end{align*}
The last expression is $\lesssim A$ according to the assumptions of the proposition. This concludes {\it{Step 1}}. 
\\

{\bf{Step 2}}: {\it{The choice of $\triangle\tilde{\tilde{x}}^{(output)}_0, \triangle\tilde{\tilde{x}}^{(output)}_1$.}} Introduce 
\begin{align*}
&\triangle_{>\tau}x^{(output)}(\tau, \xi): =\\& -\int_{\tau}^\infty\frac{\lambda^{\frac{3}{2}}(\tau)}{\lambda^{\frac{3}{2}}(\sigma)}\frac{\rho^{\frac{1}{2}}(\frac{\lambda^2(\tau)}{\lambda^2(\sigma)}\xi)}{\rho^{\frac{1}{2}}(\xi)}\frac{\sin[\lambda(\tau)\xi^{\frac{1}{2}}\int_{\tau}^\sigma\lambda^{-1}(u)\,du]}{\xi^{\frac{1}{2}}}\mathcal{R}_1(\sigma, \triangle x^{(input)} (\sigma, \xi) )(\frac{\lambda^2(\tau)}{\lambda^2(\sigma)}\xi)\,d\sigma
\end{align*}
and further write 
\begin{align*}
&\triangle \tilde{x}^{(output)}_0(\xi): =\\& \int_{\tau_0}^\infty\frac{\lambda^{\frac{3}{2}}(\tau_0)}{\lambda^{\frac{3}{2}}(\sigma)}\frac{\rho^{\frac{1}{2}}(\frac{\lambda^2(\tau_0)}{\lambda^2(\sigma)}\xi)}{\rho^{\frac{1}{2}}(\xi)}\frac{\sin[\lambda(\tau_0)\xi^{\frac{1}{2}}\int_{\tau_0}^\sigma\lambda^{-1}(u)\,du]}{\xi^{\frac{1}{2}}}\mathcal{R}_1(\sigma, \triangle x^{(input)} (\sigma, \xi) )(\frac{\lambda^2(\tau_0)}{\lambda^2(\sigma)}\xi)\,d\sigma\\
& + \triangle\tilde{\tilde{x}}^{(output)}_0,
\end{align*}
\begin{align*}
&\triangle \tilde{x}^{(output)}_1(\xi): = \\
&\int_{\tau_0}^\infty\frac{\lambda^{\frac{3}{2}}(\tau_0)}{\lambda^{\frac{3}{2}}(\sigma)}\frac{\rho^{\frac{1}{2}}(\frac{\lambda^2(\tau_0)}{\lambda^2(\sigma)}\xi)}{\rho^{\frac{1}{2}}(\xi)}\cos[\lambda(\tau_0)\xi^{\frac{1}{2}}\int_{\tau_0}^\sigma\lambda^{-1}(u)\,du]\mathcal{R}_1(\sigma, \triangle x^{(input)} (\sigma, \xi) )(\frac{\lambda^2(\tau_0)}{\lambda^2(\sigma)}\xi)\,d\sigma\\& + \triangle\tilde{\tilde{x}}^{(output)}_1. 
\end{align*}
Here the choice of the correction terms $\triangle\tilde{\tilde{x}}^{(output)}_0, \triangle\tilde{\tilde{x}}^{(output)}_1$ will be made according to the following 
\begin{lem}\label{lem:choiceofcorrectioniter} There exist $\triangle\tilde{\tilde{x}}^{(output)}_0, \triangle\tilde{\tilde{x}}^{(output)}_1$ satisfying \eqref{eq:newcorrection} and such that we have 
\[
\int_0^\infty \frac{(\rho^{\frac{1}{2}}\triangle \tilde{x}^{(output)}_0)(\xi)}{\xi^{\frac{1}{4}}}\cos[\nu\tau_0\xi^{\frac{1}{2}}]\,d\xi = 0, \int_0^\infty \frac{(\rho^{\frac{1}{2}}\triangle \tilde{x}^{(output)}_1)(\xi)}{\xi^{\frac{3}{4}}}\sin[\nu\tau_0\xi^{\frac{1}{2}}]\,d\xi = 0,
\]
where $\triangle \tilde{x}^{(output)}_{0,1}(\xi)$ are defined through the formula preceding the lemma. 
\end{lem}
\begin{proof}(lemma) We provide it for $\triangle\tilde{\tilde{x}}^{(output)}_1$, the argument for $\triangle\tilde{\tilde{x}}^{(output)}_0$ being identical. Following closely the proof of Lemma~\ref{lem:choiceofcorrection}, it suffices to prove the estimate 
\begin{equation}\label{lem:choiceofcorrectioniterkey}
\big|\int_0^\infty \frac{(\rho^{\frac{1}{2}}\tilde{\triangle} \tilde{x}^{(output)}_1)(\xi)}{\xi^{\frac{3}{4}}}\sin[\nu\tau_0\xi^{\frac{1}{2}}]\,d\xi\big|\lesssim \tau_0^{0+}\cdot A, 
\end{equation}
where we set 
\begin{align*}
&\tilde{\triangle} \tilde{x}^{(output)}_1: =\\& \int_{\tau_0}^\infty\frac{\lambda^{\frac{3}{2}}(\tau_0)}{\lambda^{\frac{3}{2}}(\sigma)}\frac{\rho^{\frac{1}{2}}(\frac{\lambda^2(\tau_0)}{\lambda^2(\sigma)}\xi)}{\rho^{\frac{1}{2}}(\xi)}\cos[\lambda(\tau_0)\xi^{\frac{1}{2}}\int_{\tau_0}^\sigma\lambda^{-1}(u)\,du]\mathcal{R}_1(\sigma, \triangle x^{(input)} (\sigma, \xi) )(\frac{\lambda^2(\tau_0)}{\lambda^2(\sigma)}\xi)\,d\sigma
\end{align*}
Also, observe that in light of Lemma~\ref{lem:choiceofcorrection} and its proof, we may replace the input $\chi_{\xi>1}\triangle x^{(input)}$ by $\chi_{\xi>1}\widetilde{\triangle x^{(input)}}$, while we can replace $\chi_{\xi<1}\triangle x^{(input)}$ by 
$\chi_{\xi<1}\triangle_{>\tau} x^{(input)}$. 
\\
We shall do so below without further mention. 
\\

As usual, write schematically 
\[
 \mathcal{R}_1(\sigma, \triangle x^{(input)} (\tau, \xi) ) = \beta_{\nu}^2(\sigma)\mathcal{K}_{cc}\triangle x^{(input)} (\sigma, \xi)  + \beta_{\nu}(\sigma)\mathcal{K}_{cc}\mathcal{D}_\sigma\triangle x^{(input)} (\sigma, \xi),  
 \]
where we omit the contribution of the error term $ \beta_{\nu}^2(\sigma)\mathcal{K}_{cc}^2\triangle x^{(input)} (\sigma, \xi)$ whose contribution is handled like the one of the first term on the right. 
\\

{\bf{(2.i)}} {\it{Contribution of $\chi_{\xi>1}\widetilde{\triangle x^{(input)}}(\tau, \xi)$.}} Here we divide further into different regimes for the output frequency variable $\xi$: 
\\

{\it{(2.i.1): Output frequency $\xi<1$.}} Here we have a non-resonant interaction inside $\mathcal{K}_{cc}\ldots$ since we then also have $\frac{\lambda^2(\tau_0)}{\lambda^2(\sigma)}\xi<1$, and so we use the point wise bound (recall the decay of the kernel of $\mathcal{K}_{cc}$)
 \begin{align*}
 \big|\beta_{\nu}^2(\sigma)\mathcal{K}_{cc}(\chi_{\eta>1}\widetilde{\triangle x^{(input)}}) (\sigma, \frac{\lambda^2(\tau_0)}{\lambda^2(\sigma)}\xi)\big|&\lesssim \beta_{\nu}^2(\sigma)\big\|\chi_{\eta>1}\widetilde{\triangle x^{(input)}}\big\|_{L^2_{d\eta}}\\
 &\lesssim \beta_{\nu}^2(\sigma)(\frac{\sigma}{\tau_0})^{\kappa}\cdot A. 
 \end{align*}
 Further, bound 
 \[
 \big|\frac{\lambda^{\frac{3}{2}}(\tau_0)}{\lambda^{\frac{3}{2}}(\sigma)}\frac{\rho^{\frac{1}{2}}(\frac{\lambda^2(\tau_0)}{\lambda^2(\sigma)}\xi)}{\rho^{\frac{1}{2}}(\xi)}\big|\lesssim \frac{\lambda(\tau_0)}{\lambda(\sigma)}.
 \]
 Finally, we easily get for the contribution of 
 \[
 \beta_{\nu}^2(\sigma)\mathcal{K}_{cc}\big(\chi_{\xi>1}\widetilde{\triangle x^{(input)}}\big)
 \]
 to the left hand side of \eqref{lem:choiceofcorrectioniterkey} the bound
 the bound 
 \begin{align*}
 \big|\int_0^1 \frac{(\rho^{\frac{1}{2}}\tilde{\triangle}\tilde{x}^{(output)}_1)(\xi)}{\xi^{\frac{3}{4}}}\sin[\nu\tau_0\xi^{\frac{1}{2}}]\,d\xi \big|&\lesssim \tau_0^{0+}\int_{\tau_0}^\infty \beta_{\nu}^2(\sigma)\frac{\lambda(\tau_0)}{\lambda(\sigma)}(\frac{\sigma}{\tau_0})^{\kappa}\,d\sigma\cdot A\\
 &\lesssim \tau_0^{0+}\cdot A. 
 \end{align*}
 The contribution of $\beta_{\nu}(\sigma)\mathcal{K}_{cc}(\chi_{\eta>1}\mathcal{D}_\sigma\widetilde{\triangle x^{(input)}}) $ is handled similarly.\\
\\

{\it{(2.i.2): Output frequency $\xi>1$.}} Observe that if $\frac{\lambda^2(\tau_0)}{\lambda^2(\sigma)}\xi<1$, then 
\[
 \big|\frac{\lambda^{\frac{3}{2}}(\tau_0)}{\lambda^{\frac{3}{2}}(\sigma)}\frac{\rho^{\frac{1}{2}}(\frac{\lambda^2(\tau_0)}{\lambda^2(\sigma)}\xi)}{\rho^{\frac{1}{2}}(\xi)}\big|\lesssim \xi^{-\frac12}\cdot\frac{\lambda(\tau_0)}{\lambda(\sigma)}
 \]
 Combined with a factor $\rho^{\frac12}\xi^{-\frac34}\sim \xi^{-\frac12}$, the $\xi$-integral converges absolutely up to a factor $\big|\log(\frac{\lambda(\tau_0)}{\lambda(\sigma)})\big|$. Thus the corresponding contribution to the left hand side of \eqref{lem:choiceofcorrectioniterkey} is bounded by 
 \[
 \lesssim A\cdot\int_{\tau_0}^\infty \big|\log(\frac{\lambda(\tau_0)}{\lambda(\sigma)})\big|\big[\beta_{\nu}^2(\sigma)\frac{\lambda(\tau_0)}{\lambda(\sigma)}(\frac{\sigma}{\tau_0})^{\kappa} + \beta_{\nu}(\sigma)\frac{\lambda(\tau_0)}{\lambda(\sigma)}\big]\,d\sigma,
 \]
 which in turn is $\lesssim A$. On the other hand, when $\frac{\lambda^2(\tau_0)}{\lambda^2(\sigma)}\xi\geq 1$, we get 
\[
 \big|\frac{\lambda^{\frac{3}{2}}(\tau_0)}{\lambda^{\frac{3}{2}}(\sigma)}\frac{\rho^{\frac{1}{2}}(\frac{\lambda^2(\tau_0)}{\lambda^2(\sigma)}\xi)}{\rho^{\frac{1}{2}}(\xi)}\big|\lesssim\frac{\lambda^2(\tau_0)}{\lambda^2(\sigma)}, 
 \]
and unless we are in the resonant case $\eta\sim \frac{\lambda^2(\tau_0)}{\lambda^2(\sigma)}\xi$ for the convolution type operator $\mathcal{K}_{cc}$, we can easily bound the $\xi$-integral absolutely, using the kernel decay from \cite{KST}, while keeping a factor $\frac{\lambda(\tau_0)}{\lambda(\sigma)}$ which suffices to ensure convergence of the $\sigma$-integral using crude bounds. 
\\
We are thus reduced here to the contribution of the resonant case $\eta\sim \frac{\lambda^2(\tau_0)}{\lambda^2(\sigma)}\xi$, for which we use 
\begin{align*}
\mathcal{K}_{\frac{\lambda^2(\tau_0)}{\lambda^2(\sigma)}\xi\sim \eta}(\chi_{\eta>1}\widetilde{\triangle x^{(input)}}) (\sigma, \frac{\lambda^2(\tau_0)}{\lambda^2(\sigma)}\xi)\big\|_{L^2_{d\xi}}
 \lesssim \frac{\lambda(\sigma)}{\lambda(\tau_0)}\big\|\chi_{\eta>1}\widetilde{\triangle x^{(input)}}(\sigma, \cdot)\big\|_{L^2_{d\eta}}
 \end{align*}
In fact, one may include a further factor $\big(\frac{\lambda^2(\tau_0)}{\lambda^2(\sigma)}\xi\big)^{\frac12}$ into $\mathcal{K}_{\frac{\lambda^2(\tau_0)}{\lambda^2(\sigma)}\xi\sim \eta}(\ldots)$ due to the decay of its kernel, and so using the Cauchy-Schwarz inequality we may bound the corresponding contribution to the left hand side of \eqref{lem:choiceofcorrectioniterkey} by 
\[
\lesssim A\cdot\int_{\tau_0}^\infty [\beta_{\nu}^2(\sigma)(\frac{\sigma}{\tau_0})^{\kappa} + \beta_\nu(\sigma)(\frac{\sigma}{\tau_0})^{-\kappa}]\,d\sigma\lesssim A. 
\]
This concludes step {\it{(2.i)}}, i. e. the contribution of the high-frequency term 
\[
\chi_{\xi>1}\widetilde{\triangle x^{(input)}}(\tau, \xi). 
\]

{\bf{(2.ii)}} {\it{Contribution of $\chi_{\xi<1}\triangle_{>\tau} x^{(input)}(\tau, \xi)$.}} We split into the usual interaction types between $\eta$ and $\xi$ inside
\[
\beta_{\nu}^2(\sigma)\mathcal{K}_{cc}(\chi_{\eta<1}\triangle_{>\sigma} x^{(input)}),\,\beta_{\nu}(\sigma)\mathcal{K}_{cc}(\chi_{\eta<1}\mathcal{D}_\sigma\triangle_{>\sigma} x^{(input)})
\]
Throughout $\eta<1$, by assumption.  
\\

{\it{(2.ii.1): $\eta\gg \frac{\lambda^2(\tau_0)}{\lambda^2(\sigma)}\xi$.}} First, assume that $\xi<1$, i. e. consider 
\begin{align*}
\int_0^1 \frac{(\rho^{\frac{1}{2}}\tilde{\triangle} \tilde{x}^{(output)}_1)(\xi)}{\xi^{\frac{3}{4}}}\sin[\nu\tau_0\xi^{\frac{1}{2}}]\,d\xi,  
\end{align*}
where we now set 
\begin{align*}
&\tilde{\triangle} \tilde{x}^{(output)}_1(\xi): =\\
&  \int_{\tau_0}^\infty\frac{\lambda^{\frac{3}{2}}(\tau_0)}{\lambda^{\frac{3}{2}}(\sigma)}\frac{\rho^{\frac{1}{2}}(\frac{\lambda^2(\tau_0)}{\lambda^2(\sigma)}\xi)}{\rho^{\frac{1}{2}}(\xi)}\cos[\lambda(\tau_0)\xi^{\frac{1}{2}}\int_{\tau_0}^\sigma\lambda^{-1}(u)\,du]\mathcal{R}_1(\sigma, \chi_{\eta<1}\triangle_{>\sigma} x^{(input)})(\frac{\lambda^2(\tau_0)}{\lambda^2(\sigma)}\xi)\,d\sigma
\end{align*}
Then we use the point wise bound 
\begin{align*}
&\big|\mathcal{R}_1(\sigma, \chi_{\eta<1}\triangle_{>\sigma} x^{(input)})(\frac{\lambda^2(\tau_0)}{\lambda^2(\sigma)}\xi)\big|\\&\lesssim \beta_{\nu}^2(\sigma)\big\|\eta^{-0+}\triangle_{>\sigma} x^{(input)}\big\|_{L^2_{d\eta}(\eta<1)} + \beta_{\nu}(\sigma)\big\|\eta^{-0+}\mathcal{D}_{\sigma}\triangle_{>\sigma} x^{(input)}\big\|_{L^2_{d\eta}(\eta<1)},
\end{align*}
and so for $\xi<1$ we obtain 
\begin{align*}
\big|\tilde{\triangle}\tilde{x}^{(output)}_1)(\xi)\big|\lesssim \beta_{\nu}^2(\sigma)(\frac{\sigma}{\tau_0})^{\kappa}\cdot \frac{\lambda(\tau_0)}{\lambda(\sigma)}A + \beta_{\nu}(\sigma)(\frac{\sigma}{\tau_0})^{-\kappa}\cdot \frac{\lambda(\tau_0)}{\lambda(\sigma)}A
\end{align*}
It follows after integration over $\sigma\geq \tau_0$ that 
\begin{align*}
\big|\int_0^1 \frac{(\rho^{\frac{1}{2}}\tilde{\triangle}\tilde{x}^{(output)}_1)(\xi)}{\xi^{\frac{3}{4}}}\sin[\nu\tau_0\xi^{\frac{1}{2}}]\,d\xi\big|\lesssim A\log \tau_0.
\end{align*}
Next, consider the case $\xi>1$. Taking advantage of the decay properties of the kernel of $\mathcal{K}$, we have for $\xi>1$ the crude bound
\begin{align*}
\big|\triangle \tilde{x}^{(output)}_1)(\xi)\big|\lesssim \langle\frac{\lambda(\tau_0)}{\lambda(\sigma)}\xi\rangle^{-1}\big[\beta_{\nu}^2(\sigma)(\frac{\sigma}{\tau_0})^{\kappa}\cdot \frac{\lambda(\tau_0)}{\lambda(\sigma)}A + \beta_{\nu}(\sigma)(\frac{\sigma}{\tau_0})^{-\kappa}\cdot \frac{\lambda(\tau_0)}{\lambda(\sigma)}A\big]. 
\end{align*}
From here it easily follows that for this contribution we get 
\begin{align*}
\big|\int_1^\infty \frac{(\rho^{\frac{1}{2}}\tilde{\triangle} \tilde{x}^{(output)}_1)(\xi)}{\xi^{\frac{3}{4}}}\sin[\nu\tau_0\xi^{\frac{1}{2}}]\,d\xi\big|\lesssim A.
\end{align*}

{\it{(2.ii.2): $\eta\sim\frac{\lambda^2(\tau_0)}{\lambda^2(\sigma)}\xi$.}} Here we use that 
\begin{align*}
&\big\|\langle\frac{\lambda(\tau_0)}{\lambda(\sigma)}\xi^{\frac12}\rangle\frac{\langle\xi\rangle^{\frac{1}{2}}}{\xi^{\frac{1}{2}}}\mathcal{R}_1(\sigma, \chi_{\eta<1}\triangle_{>\sigma} x^{(input)})(\frac{\lambda^2(\tau_0)}{\lambda^2(\sigma)}\xi)\big\|_{L^2_{d\xi}}\\&\lesssim (\frac{\lambda(\tau_0)}{\lambda(\sigma)})^{-1}\big[\beta_{\nu}^2(\sigma)\big\|\eta^{-0+}\triangle_{>\sigma} x^{(input)}\big\|_{L^2_{d\eta}} + \beta_{\nu}(\sigma)\big\|\eta^{-0+}\mathcal{D}_{\sigma}\triangle_{>\sigma} x^{(input)}\big\|_{L^2_{d\eta}}\big],
\end{align*}
where we use the fact that we can absorb the singular weight $\xi^{-\frac{1}{2}}$ into the kernel defining $\mathcal{K}_{\eta\sim\frac{\lambda^2(\tau_0)}{\lambda^2(\sigma)}\xi}$. 
This in turn yields
\begin{align*}
\big\|\frac{\langle\xi\rangle^{\frac{1}{2}+}}{\xi^{\frac{1}{2}}}\tilde{\triangle}\tilde{x}^{(output)}_1)(\sigma, \xi)\big\|_{L^2_{d\xi}}\lesssim \beta_{\nu}^2(\sigma)\big\|\eta^{-0+}\triangle_{>\sigma} \tilde{x}^{(1)}\big\|_{L^2_{d\eta}} + \beta_{\nu}(\sigma)\big\|\eta^{-0+}\mathcal{D}_{\sigma}\triangle_{>\sigma} \tilde{x}^{(1)}\big\|_{L^2_{d\eta}}.
\end{align*}
From here one finally infers that for this contribution we also get 
\begin{align*}
\big|\int_0^\infty \frac{(\rho^{\frac{1}{2}}\tilde{\triangle}\tilde{x}^{(output)}_1)(\xi)}{\xi^{\frac{3}{4}}}\sin[\nu\tau_0\xi^{\frac{1}{2}}]\,d\xi\big|\lesssim A.
\end{align*}

{\it{(2.ii.3): $\eta\gg\frac{\lambda^2(\tau_0)}{\lambda^2(\sigma)}\xi$.}} This is similar to case {\it{(2.ii.1)}} and omitted. 
\\

The proof of Lemma~\ref{lem:choiceofcorrectioniter} is now completed as in {\it{(2)}} of the proof of Lemma~\ref{lem:choiceofcorrection}. 

\end{proof}

{\bf{Step 3}}: {\it{Proof of required $S$-bounds for the initial data $(\triangle \tilde{x}^{(output)}_0, \triangle \tilde{x}^{(output)}_1)$.}} Recall that these are defined at the beginning of {\it{Step 2}}. In light of Proposition~\ref{prop:firstiteratelincont} and its proof, it suffices to replace the high-frequency input $\chi_{\xi>1}\triangle x^{(input)}(\tau, \xi)$ by $\chi_{\xi>1}\widetilde{\triangle x^{(input)}}(\tau, \xi)$ with the latter defined in {\it{(1d)}} of {\it{Step 1}} further above, while the low-frequency input $\chi_{\xi<1}\triangle x^{(input)}(\tau, \xi)$ may be replaced by $\chi_{\xi<1}\triangle_{>\tau} x^{(input)}(\tau, \xi)$.
\\
We now encapsulate the required $S$-bound in the following 
\begin{lem}\label{lem:hereneednorm} Under the hypotheses of Proposition~\ref{prop:lineariterstep}, we have 
\begin{align*}
\big\|(\triangle \tilde{x}^{(output)}_0, \triangle \tilde{x}^{(output)}_1)\big\|_{S}\lesssim A.
\end{align*}
\end{lem}
\begin{proof}(lemma)  We shall consider in detail the slightly more difficult term $\triangle \tilde{x}^{(output)}_0$, on account of its more singular behaviour near $\xi = 0$. The second term $ \triangle \tilde{x}^{(output)}_1$ is handled similarly. 

{\it{(3i): High output frequencies $\xi>1$.}} Here we need to bound 
\[
\big\|\xi^{1+}\chi_{\xi>1}\triangle \tilde{x}^{(output)}_0\big\|_{L^2_{d\xi}}. 
\]
Distinguish between different interactions in $\mathcal{K}_{cc}$: 
\\

{\it{(3i.1): $\frac{\lambda^2(\tau_0)}{\lambda^2(\sigma)}\xi\ll \eta$, $\eta<1$.}} Observe that in this case 
\[
\frac{\lambda^{\frac{3}{2}}(\tau_0)}{\lambda^{\frac{3}{2}}(\sigma)}\frac{\rho^{\frac{1}{2}}(\frac{\lambda^2(\tau_0)}{\lambda^2(\sigma)}\xi)}{\rho^{\frac{1}{2}}(\xi)}\sim \frac{\lambda(\tau_0)}{\lambda(\sigma)}\xi^{-\frac{1}{2}},
\]
and so we have the bound 
\begin{align*}
\big|\xi^{1+}\chi_{\xi>1}\frac{\lambda^{\frac{3}{2}}(\tau_0)}{\lambda^{\frac{3}{2}}(\sigma)}\frac{\rho^{\frac{1}{2}}(\frac{\lambda^2(\tau_0)}{\lambda^2(\sigma)}\xi)}{\rho^{\frac{1}{2}}(\xi)}\frac{\sin[\lambda(\tau_0)\xi^{\frac{1}{2}}\int_{\tau_0}^\sigma\lambda^{-1}(u)\,du]}{\xi^{\frac{1}{2}}}\big|\lesssim (\frac{\lambda^2(\tau_0)}{\lambda^2(\sigma)}\xi)^{0+}[\frac{\lambda(\tau_0)}{\lambda(\sigma)}]^{1-}
\end{align*}
Then recalling that the basic constituents of $\mathcal{R}_1(\sigma, \triangle x^{(input)})$ are $\beta_{\nu}^2(\sigma)\mathcal{K}_{cc}\triangle x^{(input)}$ as well as $\beta_{\nu}(\sigma)\mathcal{K}_{cc}\mathcal{D}_{\sigma}\triangle x^{(input)}$, as well as the fact that for low frequencies, $\eta<1$, we may replace $\triangle x^{(input)}$ by $\triangle_{>\sigma} x^{(input)}$, we can bound the contribution of 
\[
\beta_{\nu}^2(\sigma)\mathcal{K}_{\frac{\lambda^2(\tau_0)}{\lambda^2(\sigma)}\xi\ll \eta}(\chi_{\cdot<1}\triangle_{>\tau} x^{(input)})
\]
as follows\footnote{Throughout recall the notation explained after \eqref{eq:Snorm}}
: 
\begin{align*}
&\big\|\xi^{1+}\chi_{\xi>1}\frac{\lambda^{\frac{3}{2}}(\tau_0)}{\lambda^{\frac{3}{2}}(\sigma)}\frac{\rho^{\frac{1}{2}}(\frac{\lambda^2(\tau_0)}{\lambda^2(\sigma)}\xi)}{\rho^{\frac{1}{2}}(\xi)}\frac{\sin[\lambda(\tau_0)\xi^{\frac{1}{2}}\int_{\tau_0}^\sigma\lambda^{-1}(u)\,du]}{\xi^{\frac{1}{2}}}\\&\hspace{3.5cm}\cdot\beta_{\nu}^2(\sigma)\mathcal{K}_{\frac{\lambda^2(\tau_0)}{\lambda^2(\sigma)}\xi\ll \eta}(\chi_{\eta<1}\triangle_{>\sigma} x^{(input)})(\sigma, \frac{\lambda^2(\tau_0)}{\lambda^2(\sigma)}\xi)\big\|_{L^2_{d\xi}}\\
& = \big\|\xi^{1+}\chi_{\xi>1}\frac{\lambda^{\frac{3}{2}}(\tau_0)}{\lambda^{\frac{3}{2}}(\sigma)}\frac{\rho^{\frac{1}{2}}(\frac{\lambda^2(\tau_0)}{\lambda^2(\sigma)}\xi)}{\rho^{\frac{1}{2}}(\xi)}\frac{\sin[\lambda(\tau_0)\xi^{\frac{1}{2}}\int_{\tau_0}^\sigma\lambda^{-1}(u)\,du]}{\xi^{\frac{1}{2}}}\\&\hspace{3.5cm}\cdot\beta_{\nu}^2(\sigma)\mathcal{K}_{\frac{\lambda^2(\tau_0)}{\lambda^2(\sigma)}\xi\ll \eta}(\chi_{\eta<1}\triangle_{>\sigma} x^{(input)})(\sigma, \frac{\lambda^2(\tau_0)}{\lambda^2(\sigma)}\xi)\big\|_{L^2_{d\xi}(\frac{\lambda^2(\tau_0)}{\lambda^2(\sigma)}\xi<1)}\\
&\lesssim (\frac{\lambda(\sigma)}{\lambda(\tau_0)})^{2\delta_0}\beta_{\nu}^2(\sigma)\big\|\mathcal{K}_{\frac{\lambda^2(\tau_0)}{\lambda^2(\sigma)}\xi\ll \eta}(\chi_{\eta<1}\triangle_{>\sigma} x^{(input)})(\sigma, \cdot)\big\|_{L^\infty_{d\xi}}.
\end{align*}
Recalling the definition of $\mathcal{K}_{cc}$, this in turn can be bounded by 
\begin{align*}
&(\frac{\lambda(\sigma)}{\lambda(\tau_0)})^{2\delta_0}\beta_{\nu}^2(\sigma)\big\|\mathcal{K}_{\frac{\lambda^2(\tau_0)}{\lambda^2(\sigma)}\xi\ll \eta}(\chi_{\eta<1}\triangle_{>\sigma} x^{(input)})(\sigma, \cdot)\big\|_{L^\infty_{d\xi}}\\
&\lesssim (\frac{\lambda(\sigma)}{\lambda(\tau_0)})^{2\delta_0}\beta_{\nu}^2(\sigma)\big\|\eta^{-0+}\chi_{\eta<1}\triangle_{>\sigma}x^{(input)})(\sigma, \cdot)\big\|_{L^2_{d\eta}}\\
&\lesssim (\frac{\lambda(\sigma)}{\lambda(\tau_0)})^{4\delta_0}\beta_{\nu}^2(\sigma)\cdot A.
\end{align*}
Integrating over $\sigma\geq \tau_0$ yields the desired bound for this contribution, provided $\delta_0\ll 1$. Note that we take crucial advantage here of the decay $\beta_{\nu}^2(\sigma)$, and so the argument for the second type of contribution, $\beta_{\nu}(\sigma)\mathcal{K}_{cc}(\chi_{\cdot<1}\mathcal{D}_{\sigma}\triangle_{>\sigma} x^{(input)})$, which is just borderline integrable (in light of our assumptions on $\mathcal{D}_{\sigma}\triangle_{>\sigma} x^{(input)}$), will be more delicate. 
\\

Specifically, we observe that for fixed dyadic frequency $2^j<1$, we have 
\begin{align*}
&\big\|\xi^{1+}\chi_{\xi>1}\frac{\lambda^{\frac{3}{2}}(\tau_0)}{\lambda^{\frac{3}{2}}(\sigma)}\frac{\rho^{\frac{1}{2}}(\frac{\lambda^2(\tau_0)}{\lambda^2(\sigma)}\xi)}{\rho^{\frac{1}{2}}(\xi)}\frac{\sin[\lambda(\tau_0)\xi^{\frac{1}{2}}\int_{\tau_0}^\sigma\lambda^{-1}(u)\,du]}{\xi^{\frac{1}{2}}}\\&\hspace{2.5cm}\cdot\beta_{\nu}(\sigma)\mathcal{K}_{\frac{\lambda^2(\tau_0)}{\lambda^2(\sigma)}\xi\ll \eta}(\chi_{\eta<1}\mathcal{D}_{\sigma}\triangle_{>\sigma} x^{(input)})(\sigma, \frac{\lambda^2(\tau_0)}{\lambda^2(\sigma)}\xi)\big\|_{L^2_{d\xi}(\frac{\lambda^2(\tau_0)}{\lambda^2(\sigma)}\xi\sim 2^j)}\\
&\lesssim (\frac{\lambda(\sigma)}{\lambda(\tau_0)})^{2\delta_0}2^{\frac{(1-)}{2}j}\beta_{\nu}(\sigma)\big\|\eta^{-0+}\chi_{\eta<1}\mathcal{D}_{\sigma}\triangle_{>\sigma} x^{(input)})(\sigma, \cdot)\big\|_{L^2_{d\eta}}.\\
\end{align*}
Then using orthogonality it follows that 
\begin{align*}
&\big\|\xi^{1+}\chi_{\xi>1}\int_{\tau_0}^\infty\chi_{\frac{\lambda^2(\tau_0)}{\lambda^2(\sigma)}\xi<1}\frac{\lambda^{\frac{3}{2}}(\tau_0)}{\lambda^{\frac{3}{2}}(\sigma)}\frac{\rho^{\frac{1}{2}}(\frac{\lambda^2(\tau_0)}{\lambda^2(\sigma)}\xi)}{\rho^{\frac{1}{2}}(\xi)}\frac{\sin[\lambda(\tau_0)\xi^{\frac{1}{2}}\int_{\tau_0}^\sigma\lambda^{-1}(u)\,du]}{\xi^{\frac{1}{2}}}\\&\hspace{2.5cm}\cdot\beta_{\nu}(\sigma)\mathcal{K}_{\frac{\lambda^2(\tau_0)}{\lambda^2(\sigma)}\xi\ll \eta}(\chi_{\eta<1}\mathcal{D}_{\sigma}\triangle_{>\sigma} x^{(input)})(\sigma, \frac{\lambda^2(\tau_0)}{\lambda^2(\sigma)}\xi)\,d\sigma\big\|_{L^2_{d\xi}}\\
&\lesssim \sum_{j<0}\big(\sum_{k\in Z}\big\|\xi^{1+}\chi_{\xi>1}\int_{\tau_0}^\infty\chi_{\frac{\lambda^2(\tau_0)}{\lambda^2(\sigma)}\xi\sim 2^j}\frac{\lambda^{\frac{3}{2}}(\tau_0)}{\lambda^{\frac{3}{2}}(\sigma)}\frac{\rho^{\frac{1}{2}}(\frac{\lambda^2(\tau_0)}{\lambda^2(\sigma)}\xi)}{\rho^{\frac{1}{2}}(\xi)}\frac{\sin[\lambda(\tau_0)\xi^{\frac{1}{2}}\int_{\tau_0}^\sigma\lambda^{-1}(u)\,du]}{\xi^{\frac{1}{2}}}\\&\hspace{2.5cm}\cdot\beta_{\nu}(\sigma)\mathcal{K}_{\frac{\lambda^2(\tau_0)}{\lambda^2(\sigma)}\xi\ll \eta}(\chi_{\eta<1}\mathcal{D}_{\sigma}\triangle_{>\sigma}x^{(input)})(\sigma, \frac{\lambda^2(\tau_0)}{\lambda^2(\sigma)}\xi)\,d\sigma\big\|_{L^2_{d\xi}(\xi\sim 2^k)}^2\big)^{\frac{1}{2}}\\
\end{align*}
This last expression can be bounded in turn by (using the previous estimates) 
\begin{align*}
& \sum_{j<0}\big(\sum_{k\in Z}\big\|\xi^{1+}\chi_{\xi>1}\int_{\tau_0}^\infty\chi_{\frac{\lambda^2(\tau_0)}{\lambda^2(\sigma)}\xi\sim 2^j}\frac{\lambda^{\frac{3}{2}}(\tau_0)}{\lambda^{\frac{3}{2}}(\sigma)}\frac{\rho^{\frac{1}{2}}(\frac{\lambda^2(\tau_0)}{\lambda^2(\sigma)}\xi)}{\rho^{\frac{1}{2}}(\xi)}\frac{\sin[\lambda(\tau_0)\xi^{\frac{1}{2}}\int_{\tau_0}^\sigma\lambda^{-1}(u)\,du]}{\xi^{\frac{1}{2}}}\\&\hspace{2.5cm}\cdot\beta_{\nu}(\sigma)\mathcal{K}_{\frac{\lambda^2(\tau_0)}{\lambda^2(\sigma)}\xi\ll \eta}(\chi_{\eta<1}\mathcal{D}_{\sigma}\triangle_{>\sigma}x^{(input)})(\sigma, \frac{\lambda^2(\tau_0)}{\lambda^2(\sigma)}\xi)\,d\sigma\big\|_{L^2_{d\xi}(\xi\sim 2^k)}^2\big)^{\frac{1}{2}}\\
&\lesssim  \sum_{j<0} 2^{\frac{(1-)}{2}j}\big(\sum_{k\in Z}\sup_{\lambda^2(\sigma)\sim \lambda^2(\tau_0)2^{k-j}}(\frac{\lambda(\sigma)}{\lambda(\tau_0)})^{4\delta_0}\big\|\eta^{-0+}\chi_{\eta<1}\chi_{\sigma\geq \tau_0}\mathcal{D}_{\sigma}\triangle_{>\sigma}x^{(input)}(\sigma, \cdot)\big\|_{L^2_{d\eta}}^2\big)^{\frac{1}{2}}\\
&\lesssim A. 
\end{align*}

{\it{(3i.2): $\frac{\lambda^2(\tau_0)}{\lambda^2(\sigma)}\xi\ll \eta$, $\eta>1$.}} Here one replaces the norms 
\[
\big\|\eta^{-0+}\chi_{\eta<1}\chi_{\sigma\geq \tau_0}\triangle_{>\sigma} x^{(input)}(\sigma, \cdot)\big\|_{L^2_{d\eta}}
\]
by 
\[
\big\|\chi_{\eta>1}\chi_{\sigma\geq \tau_0}\widetilde{\triangle x^{(input)}}(\sigma, \cdot)\big\|_{L^2_{d\eta}}.
\]
Also, since now we can also have $\frac{\lambda^2(\tau_0)}{\lambda^2(\sigma)}\xi>1$, but we can always absorb a large weight $\frac{\lambda^2(\tau_0)}{\lambda^2(\sigma)}\xi$ into the kernel of $\mathcal{K}_{\frac{\lambda^2(\tau_0)}{\lambda^2(\sigma)}\xi\ll \eta}$ and even gain $(\frac{\lambda^2(\tau_0)}{\lambda^2(\sigma)}\xi)^{-1}$, upon localising $\frac{\lambda^2(\tau_0)}{\lambda^2(\sigma)}\xi\sim 2^j$, we can replace the gain $2^{\frac{(1-)j}{2}}$ in the last part of {\it{(3i.1)}} by the gain 
\[
\min\{2^{\frac{(1-)}{2}j}, 2^{-j}\}.
\]
Otherwise, one proceeds just as in {\it{(3i.1)}}. 
\\

{\it{(3i.3): $\frac{\lambda^2(\tau_0)}{\lambda^2(\sigma)}\xi\sim\eta$, $\eta<1$.}} Here we follow a similar pattern, exploiting an orthogonality argument to control the contribution of $\beta_{\nu}(\sigma)\mathcal{K}_{cc}\mathcal{D}_{\sigma}\triangle_{>\sigma} x^{(input)}$. We spell this case out in detail, omitting the simpler contribution of $\beta_{\nu}^2(\sigma)\mathcal{K}_{cc}\triangle_{>\sigma}x^{(input)}$. Thus we use that for $2^j<1$ we have 
\begin{align*}
&\big\|\beta_{\nu}(\sigma)(\mathcal{K}_{\frac{\lambda^2(\tau_0)}{\lambda^2(\sigma)}\xi\sim\eta}\chi_{\eta\sim 2^j}\mathcal{D}_{\sigma}\triangle_{>\sigma}x^{(input)})(\sigma, \frac{\lambda^2(\tau_0)}{\lambda^2(\sigma)}\xi)\big\|_{L^2_{d\xi}}
\\&\lesssim  ( \frac{\lambda(\tau_0)}{\lambda(\sigma)})^{-1}\beta_{\nu}(\sigma)2^{\frac{j}{2}}\big\|\eta^{-0+}\chi_{\eta<1}\mathcal{D}_{\sigma}\triangle_{>\sigma} x^{(input)}(\sigma, \cdot)\big\|_{L^2_{d\eta}}
\end{align*}
Here the gain of $2^{\frac{j}{2}}$ is a consequence of the decay properties of the kernel defining $\mathcal{K}_{cc}$. We combine this with 
\begin{align*}
\big|\xi^{1+}\chi_{\xi>1}\frac{\lambda^{\frac{3}{2}}(\tau_0)}{\lambda^{\frac{3}{2}}(\sigma)}\frac{\rho^{\frac{1}{2}}(\frac{\lambda^2(\tau_0)}{\lambda^2(\sigma)}\xi)}{\rho^{\frac{1}{2}}(\xi)}\frac{\sin[\lambda(\tau_0)\xi^{\frac{1}{2}}\int_{\tau_0}^\sigma\lambda^{-1}(u)\,du]}{\xi^{\frac{1}{2}}}\big|\lesssim \xi^{0+}\frac{\lambda(\tau_0)}{\lambda(\sigma)}
\end{align*}
provided $\frac{\lambda^2(\tau_0)}{\lambda^2(\sigma)}\xi<1$ and use an orthogonality argument to obtain 
\begin{align*}
&\big\|\xi^{1+}\chi_{\xi>1}\int_{\tau_0}^\infty\frac{\lambda^{\frac{3}{2}}(\tau_0)}{\lambda^{\frac{3}{2}}(\sigma)}\frac{\rho^{\frac{1}{2}}(\frac{\lambda^2(\tau_0)}{\lambda^2(\sigma)}\xi)}{\rho^{\frac{1}{2}}(\xi)}\frac{\sin[\lambda(\tau_0)\xi^{\frac{1}{2}}\int_{\tau_0}^\sigma\lambda^{-1}(u)\,du]}{\xi^{\frac{1}{2}}}\\&\hspace{2.5cm}\cdot\beta_{\nu}(\sigma)\mathcal{K}_{\frac{\lambda^2(\tau_0)}{\lambda^2(\sigma)}\xi\sim\eta}(\chi_{\eta<1}\mathcal{D}_{\sigma}\triangle_{>\sigma} x^{(input)})(\sigma, \frac{\lambda^2(\tau_0)}{\lambda^2(\sigma)}\xi)\,d\sigma\big\|_{L^2_{d\xi}}\\
&\lesssim \sum_{j<0}\big(\sum_{k\in Z}\big\|\xi^{1+}\chi_{\xi>1}\int_{\tau_0}^\infty\frac{\lambda^{\frac{3}{2}}(\tau_0)}{\lambda^{\frac{3}{2}}(\sigma)}\frac{\rho^{\frac{1}{2}}(\frac{\lambda^2(\tau_0)}{\lambda^2(\sigma)}\xi)}{\rho^{\frac{1}{2}}(\xi)}\frac{\sin[\lambda(\tau_0)\xi^{\frac{1}{2}}\int_{\tau_0}^\sigma\lambda^{-1}(u)\,du]}{\xi^{\frac{1}{2}}}\\&\hspace{2.5cm}\cdot\beta_{\nu}(\sigma)\mathcal{K}_{\frac{\lambda^2(\tau_0)}{\lambda^2(\sigma)}\xi\sim\eta}(\chi_{\eta\sim 2^j}\mathcal{D}_{\sigma}\triangle_{>\sigma} x^{(input)})(\sigma, \frac{\lambda^2(\tau_0)}{\lambda^2(\sigma)}\xi)\,d\sigma\big\|_{L^2_{d\xi}(\xi\sim 2^k)}^2\big)^{\frac{1}{2}},\\
\end{align*}
and further 
\begin{align*}
&\sum_{j<0}\big(\sum_{k\in Z}\big\|\xi^{1+}\chi_{\xi>1}\int_{\tau_0}^\infty\frac{\lambda^{\frac{3}{2}}(\tau_0)}{\lambda^{\frac{3}{2}}(\sigma)}\frac{\rho^{\frac{1}{2}}(\frac{\lambda^2(\tau_0)}{\lambda^2(\sigma)}\xi)}{\rho^{\frac{1}{2}}(\xi)}\frac{\sin[\lambda(\tau_0)\xi^{\frac{1}{2}}\int_{\tau_0}^\sigma\lambda^{-1}(u)\,du]}{\xi^{\frac{1}{2}}}\\&\hspace{2.5cm}\cdot\beta_{\nu}(\sigma)\mathcal{K}_{\frac{\lambda^2(\tau_0)}{\lambda^2(\sigma)}\xi\sim\eta}(\chi_{\eta\sim 2^j}\mathcal{D}_{\sigma}\triangle_{>\sigma}x^{(input)})(\sigma, \frac{\lambda^2(\tau_0)}{\lambda^2(\sigma)}\xi)\,d\sigma\big\|_{L^2_{d\xi}(\xi\sim 2^k)}^2\big)^{\frac{1}{2}}\\
&\lesssim \sum_{j<0}2^{\frac{j}{2}}\big(\sum_{k}\sup_{\lambda^2(\sigma)\sim 2^{k-j}\lambda^2(\tau_0)}(\frac{\lambda^2(\sigma)}{\lambda^2(\tau_0)})^{2\delta_0}\big\|\eta^{-0+}\chi_{\eta<1}\mathcal{D}_{\sigma}\triangle_{>\sigma} x^{(input)})(\sigma,\cdot)\big\|_{L^2_{d\eta}}^2\big)^{\frac{1}{2}}\\
&\lesssim A. 
\end{align*}

{\it{(3i.4): $\frac{\lambda^2(\tau_0)}{\lambda^2(\sigma)}\xi\sim\eta$, $\eta>1$.}} This case is handled just like the preceding one except that now one uses decay in $\eta^{-1}$ for the kernel of $\mathcal{K}_{cc}$. 
\\

{\it{(3i.5): $\frac{\lambda^2(\tau_0)}{\lambda^2(\sigma)}\xi\gg\eta$.}} This is analogous to cases {\it{(3i.1), (3i.2)}} and omitted. 
\\

{\it{(3ii): Low output frequencies $\xi<1$.}} Here we need to bound 
\[
\big\|\xi^{-0+}\chi_{\xi<1}\triangle x^{(output)}_0\big\|_{L^2_{d\xi}}. 
\]
We follow the scheme of the high-frequency case {\it{(3i)}}, and only consider the more delicate contribution of $\beta_{\nu}(\sigma)\mathcal{K}_{cc}\mathcal{D}_{\sigma}\triangle x^{(input)}$: 
\\

{\it{(3ii.1):  $\frac{\lambda^2(\tau_0)}{\lambda^2(\sigma)}\xi\ll \eta$, $\eta<1$.   }} We perform integration by parts with respect to $\sigma$ here in order not to lose a potentially large factor $\tau_0^{0+}$. Thus write 
\begin{align*}
&\mathcal{K}_{\frac{\lambda^2(\tau_0)}{\lambda^2(\sigma)}\xi\ll \eta}(\mathcal{D}_{\sigma}\triangle_{>\sigma} x^{(input)})\\& = \int_0^\infty \chi_{\frac{\lambda^2(\tau_0)}{\lambda^2(\sigma)}\xi\ll \eta}\frac{F(\frac{\lambda^2(\tau_0)}{\lambda^2(\sigma)}\xi, \eta)\rho(\eta)}{\frac{\lambda^2(\tau_0)}{\lambda^2(\sigma)}\xi - \eta}(\mathcal{D}_{\sigma}\triangle_{>\sigma} x^{(input)})(\sigma, \eta)\,d\eta\\
& = \int_0^\infty \chi_{\xi\ll \tilde{\eta}}\frac{F(\frac{\lambda^2(\tau_0)}{\lambda^2(\sigma)}\xi, \frac{\lambda^2(\tau_0)}{\lambda^2(\sigma)}\tilde{\eta})\rho(\frac{\lambda^2(\tau_0)}{\lambda^2(\sigma)}\tilde{\eta})}{\xi - \tilde{\eta}}\partial_{\sigma}\big((\triangle_{>\sigma} x^{(input)})(\sigma,\frac{\lambda^2(\tau_0)}{\lambda^2(\sigma)}\tilde{\eta})\big)\,d\eta\\
& + \text{error},
\end{align*}
where the term $\text{error}$ gains an extra weight $\sigma^{-1}$ and is hence handled like the contribution of $\beta_{\nu}^2(\sigma)\mathcal{K}_{cc}\ldots$.\\
Inserting this expression into the formula for $\triangle x^{(output)}_0$ and performing an integration by parts with respect to $\sigma$ either gains $\sigma^{-1}$ or else a factor $\frac{\lambda(\tau_0)}{\lambda(\sigma)}\xi^{\frac12}$, but at the expense of replacing 
$\mathcal{D}_{\sigma}\triangle x^{(input)}$ by $\triangle x^{(input)}$. In either case, estimating the $\xi$-integral for the output brute force, we bound this contribution to 
\[
\big\|\xi^{-0+}\chi_{\xi<1}\triangle x^{(output)}_0\big\|_{L^2_{d\xi}}. 
\]
by the expression 
\begin{align*}
\lesssim \int_{\tau_0}^\infty (\frac{\sigma}{\tau_0})^{\kappa}[\tau_0^{0+}\sigma^{-2} + \sigma^{-1}\frac{\lambda(\tau_0)}{\lambda(\sigma)}](\frac{\sigma}{\tau_0})^{-\kappa}\big\|\eta^{-0+}\chi_{\eta<1}\triangle_{>\sigma} x^{(input)})(\sigma,\cdot)\big\|_{L^2_{d\eta}}\,d\sigma\\
&\lesssim A,
\end{align*}
provided $\delta_0\ll 1$. Observe that the factor $\tau_0^{0+}$ here comes from absorbing the singular weight $\xi^{-0+}$ into 
\[
\sin[\lambda(\tau_0)\xi^{\frac{1}{2}}\int_{\tau_0}^\sigma\lambda^{-1}(u)\,du]. 
\]

{\it{(3ii.2):  $\frac{\lambda^2(\tau_0)}{\lambda^2(\sigma)}\xi\ll \eta$, $\eta>1$.   }} This is analogous to the preceding case, one replaces $\triangle_{>\sigma}x^{(input)}$ by $\widetilde{\triangle x^{(input)}}$ (defined as in {\it{(1d)}} above). 
\\

{\it{(3ii.3):  $\frac{\lambda^2(\tau_0)}{\lambda^2(\sigma)}\xi\sim\eta$, in particular $\eta<1$}}. Here we observe that the singular weight $\xi^{-0+}$ is better than $\eta^{-0+}$, and the weight $\xi^{-\frac12}$ may be absorbed into the kernel of $\mathcal{K}_{\frac{\lambda^2(\tau_0)}{\lambda^2(\sigma)}\xi\sim\eta}$. Then the corresponding contribution to $\big\|\xi^{-0+}\chi_{\xi<1}\triangle x^{(output)}_0\big\|_{L^2_{d\xi}}$ is bounded by 
\begin{align*}
&\lesssim \int_{\tau_0}^\infty\beta_{\nu}(\sigma)\big\|\eta^{-0+}\mathcal{D}_{\sigma}\triangle_{>\sigma} x^{(input)}(\sigma, \cdot)\big\|_{L^2_{d\eta}(\eta<1)}\,d\sigma\\
&\lesssim  \int_{\tau_0}^\infty\beta_{\nu}(\sigma)(\frac{\sigma}{\tau_0})^{-\kappa}\cdot(\frac{\sigma}{\tau_0})^{\kappa}\big\|\eta^{-0+}\mathcal{D}_{\sigma}\triangle_{>\sigma} x^{(input)}(\sigma, \cdot)\big\|_{L^2_{d\eta}(\eta<1)}\,d\sigma\\
&\lesssim A\cdot \int_{\tau_0}^\infty\beta_{\nu}(\sigma)(\frac{\sigma}{\tau_0})^{-\kappa}\,d\sigma\lesssim A. 
\end{align*}

{\it{(3ii.4):  $\frac{\lambda^2(\tau_0)}{\lambda^2(\sigma)}\xi\gg\eta$, in particular $\eta<1$}}. Here we can trade the singular weight $\xi^{-0+}$ for a better singular weight $\eta^{-0+}$ and then estimate the term like in case {\it{(3ii.3)}}. 
\\

This completes case {\it{(3ii)}}, and the proof of the lemma is complete. 
\end{proof}

{\bf{Step 4}}: {\it{Proof of the remaining low frequency bound 
\begin{align*}
&\sup_{\tau\geq \tau_0}(\frac{\tau}{\tau_0})^{-\kappa}\big\|\chi_{\xi<1}\triangle x^{(output)}_{>\tau}(\tau, \xi)\big\|_{S_1}\\&\hspace{2cm} + \big(\sum_{\substack{N\gtrsim \tau_0\\ N\,\text{dyadic}}}[\sup_{\tau\sim N}(\frac{\tau}{\tau_0})^{\kappa}\big\|\chi_{\xi<1}\mathcal{D}_{\tau}\triangle_{>\tau}x^{(output)}(\tau, \xi)\big\|_{S_2}]^2\big)^{\frac{1}{2}}\lesssim A. 
\end{align*}
}}
The bound on the first expression 
\[
\sup_{\tau\geq \tau_0}(\frac{\tau}{\tau_0})^{-\kappa}\big\|\chi_{\xi<1}\triangle x^{(output)}_{>\tau}(\tau, \xi)\big\|_{S}
\]
follows precisely as in the proof of the preceding lemma, provided one replaces $\tau_0$ by $\tau$ in the Duhamel integral.  Thus consider now the more complicated second expression involving the square-sum. We shall again only consider the more difficult contribution of $\beta_{\nu}(\sigma)\mathcal{K}_{cc}(\mathcal{D}_{\sigma}\triangle x^{(input)})$. Restricting this further to 
\[
\beta_{\nu}(\sigma)\mathcal{K}_{cc}(\chi_{\cdot <1}\mathcal{D}_{\sigma}\triangle_{>\sigma} x^{(input)}),
\]
we arrive at the expression 
\begin{align*}
&\chi_{\xi<1}\int_{\tau}^\infty\frac{\lambda^{\frac{3}{2}}(\tau)}{\lambda^{\frac{3}{2}}(\sigma)}\frac{\rho^{\frac{1}{2}}(\frac{\lambda^2(\tau)}{\lambda^2(\sigma)}\xi)}{\rho^{\frac{1}{2}}(\xi)}\cos[\lambda(\tau)\xi^{\frac{1}{2}}\int_{\tau}^\sigma\lambda^{-1}(u)\,du]\\&\hspace{3cm}\cdot\beta_{\nu}(\sigma)\mathcal{K}_{cc}(\chi_{\eta<1}\mathcal{D}_{\sigma}\triangle_{>\sigma} x^{(input)})(\sigma, \frac{\lambda^2(\tau)}{\lambda^2(\sigma)}\xi)\,d\sigma
\end{align*}
Observe that for dyadic $N\gtrsim \tau_0$ and $\tau\sim N$ we have (here $M$ also denotes dyadic numbers)
\begin{align*}
&(\frac{N}{\tau_0})^{\kappa}\big\|\xi^{-0+}\chi_{\xi<1}\int_{\tau}^\infty\frac{\lambda^{\frac{3}{2}}(\tau)}{\lambda^{\frac{3}{2}}(\sigma)}\frac{\rho^{\frac{1}{2}}(\frac{\lambda^2(\tau)}{\lambda^2(\sigma)}\xi)}{\rho^{\frac{1}{2}}(\xi)}\cos[\lambda(\tau)\xi^{\frac{1}{2}}\int_{\tau}^\sigma\lambda^{-1}(u)\,du]\\&\hspace{3cm}\cdot\beta_{\nu}(\sigma)\mathcal{K}(\chi_{\eta<1}\mathcal{D}_{\sigma}\triangle_{>\sigma} x^{(input)})(\sigma, \frac{\lambda^2(\tau)}{\lambda^2(\sigma)}\xi)\,d\sigma\big\|_{L^2_{d\xi}}\\
&\lesssim (\frac{N}{\tau_0})^{\kappa}\sum_{\substack{M\gtrsim N\\M\,\text{dyadic}}}\big(\frac{\lambda(N}{\lambda(M)})\sup_{\sigma\sim M}\big\|\eta^{-0+}\chi_{\eta<1}\mathcal{D}_{\sigma}\triangle_{>\sigma} x^{(input)})(\sigma,\cdot)\big\|_{L^2_{d\eta}}
\end{align*}
and from here we easily infer 
\begin{align*}
&\big(\sum_{\substack{N\gtrsim \tau_0\\ N\,\text{dyadic}}}(\frac{N}{\tau_0})^{\kappa}\big\|\xi^{-0+}\ldots\big\|_{L^2_{d\xi}}^2\big)^{\frac{1}{2}}\\&\lesssim \big(\sum_{\substack{M\gtrsim \tau_0\\ M\,\text{dyadic}}}(\frac{M}{\tau_0})^{2\kappa}\sup_{\sigma\sim M}\big\|\eta^{-0+}\chi_{\eta<1}\mathcal{D}_{\sigma}\triangle_{>\sigma} x^{(input)})(\sigma,\cdot)\big\|_{L^2_{d\eta}}^2\big)^{\frac{1}{2}}\\
&\lesssim A, 
\end{align*}
as required. The contribution of the high frequency input 
\[
\beta_{\nu}(\sigma)\mathcal{K}_{cc}(\chi_{\cdot >1}\mathcal{D}_{\sigma}\widetilde{\triangle x^{(input)}})
\]
is handled similarly. 
\\

This concludes {\it{Step 4}} and thereby the proof of Proposition~\ref{prop:lineariterstep}. 

\end{proof}

The proof of the remark following this proposition is immediate in light of Proposition~\ref{prop:firstiteratelincont}. 
\\

The preceding proposition only dealt with the transition from the continuous spectral part $\triangle x^{(input)}(\tau, \xi)$ to the continuous spectral part $\triangle x^{(output)}$, but we also need to study the transition for the discrete spectral part. Here the natural hypothesis for the inductive step is furnished by the conclusion of Proposition~\ref{prop:firstiteratelindisc}, Proposition~\ref{prop:firstiteratenonlindisc}: 

\begin{prop}\label{prop:lineariterstepdisc}Let $\mathcal{K}_{cd}$ be defined as in \cite{KST}, \cite{KS1}. Then if $\sup_{\tau\geq \tau_0}\tau^{(1-)}|\triangle x^{(input)}_d(\tau)| + \sup_{\tau\geq \tau_0}\tau^{(1-)}|\partial_{\tau}\triangle x^{(input)}_d(\tau)|\leq A$, and we denote schematically 
\[
\big(\mathcal{R}_{cd}(\triangle x^{(input)}_d)\big)(\tau, \xi): = \beta_{\nu}^2(\tau)\mathcal{K}_{cd}(\xi)\triangle x^{(input)}_d(\tau) +  \beta_{\nu}(\tau)\mathcal{K}_{cd}(\xi)\partial_{\tau}\triangle x^{(input)}_d(\tau) 
\]
Then there are $\triangle\tilde{\tilde{x}}^{(output)}_0(\xi), \triangle\tilde{\tilde{x}}^{(output)}_1(\xi)$ satisfying 
\[
\big\|(\triangle\tilde{\tilde{x}}^{(output)}_0, \triangle\tilde{\tilde{x}}^{(ouptut)}_1)\big\|_{\langle\xi\rangle^{-\frac{1}{2}-}\xi^{0+}L^2_{d\xi}}\lesssim \tau_0^{-(2-)}A
\]
and such that if we put 
\begin{align*}
 &\triangle x^{(output)}(\tau, \xi): \\
 &=\int_{\tau_0}^\tau\frac{\lambda^{\frac{3}{2}}(\tau)}{\lambda^{\frac{3}{2}}(\sigma)}\frac{\rho^{\frac{1}{2}}(\frac{\lambda^2(\tau)}{\lambda^2(\sigma)}\xi)}{\rho^{\frac{1}{2}}(\xi)}\frac{\sin[\lambda(\tau)\xi^{\frac{1}{2}}\int_{\tau}^\sigma\lambda^{-1}(u)\,du]}{\xi^{\frac{1}{2}}}\big(\mathcal{R}_{cd}(\triangle x^{(input)}_d)\big)(\sigma, \cdot)\big)(\frac{\lambda^2(\tau)}{\lambda^2(\sigma)}\xi)\,d\sigma\\
 & + S(\tau)(\triangle\tilde{\tilde{x}}^{(output)}_0, \triangle\tilde{\tilde{x}}^{(output)}_1),\\
 \end{align*}
then the exact same conclusions obtain as in Proposition~\ref{prop:lineariterstep} but with $A$ replaced by the better $\tau_0^{-(1-)}A$. 
\end{prop}

The proof of this proposition is similar in outline but in effect much simpler than the one of Proposition~\ref{prop:lineariterstep}, and hence omitted.  
\\

To finish the inductive step for the linear source terms of the schematic form $\mathcal{R}(\tau, \underline{\triangle x})$, we also have to study their projection onto the unstable spectral part, i. e. the operators $\mathcal{R}_d(\tau, \underline{\triangle x})$. 
Borrowing notation from \cite{KST}, \cite{KS1}, introduce the schematic operator 
\begin{align*}
\mathcal{R}_d(\sigma, \underline{\triangle x}) &= \beta_{\nu}^2(\sigma)\big[\mathcal{K}_{dc}(\triangle x(\sigma, \cdot)) + \mathcal{K}_{dd}\triangle x_d(\sigma)\big]\\
& + \beta_{\nu}(\sigma)\big[\mathcal{K}_{dc}(\mathcal{D}_{\sigma}\triangle x(\sigma, \cdot)) + \mathcal{K}_{dd}\partial_{\sigma}\triangle x_d(\sigma)\big]\\
\end{align*}
Then we have the following
\begin{prop}\label{prop:discreteintodiscrete} Assume that $\triangle x^{(input)}(\tau, \xi)$ is as in the statement of Proposition~\ref{prop:lineariterstep}, and that moreover we have 
\[
|\triangle x^{(input)}_d(\tau)| + |\partial_{\tau}\triangle x^{(input)}_d(\tau)|\leq \tau^{-(1-)}A,\,\forall\tau\geq \tau_0.
\]
Then writing $\underline{\triangle x}^{(input)} = \left(\begin{array}{c}\triangle x^{(input)}_d\\ \triangle x^{(input)}\end{array}\right)$ and setting 
\[
\triangle x_d^{(output)}(\tau): = \int_\tau^\infty H_d(\tau, \sigma)[\mathcal{R}_d(\sigma, \underline{\triangle x}^{(input)}) - \beta_{\nu}(\sigma)\triangle_d x^{(input)}]\,d\sigma, 
\]
we have 
\[
\sup_{\tau\geq \tau_0}\big(\tau|\triangle \tilde{x}_d^{(2)}(\tau)| + \tau|\partial_{\tau}\triangle \tilde{x}_d^{(2)}(\tau)|\big)\lesssim A. 
\]

\end{prop}

The three propositions Proposition~\ref{prop:discreteintodiscrete}, Proposition~\ref{prop:lineariterstepdisc}, Proposition~\ref{prop:lineariterstep}, combined with Proposition~\ref{prop:firstiteratelincont} completely describe the contribution of the linear source term 
\[
\mathcal{R}(\tau, \underline{\triangle x}^{(j-1)})
\]
 to the next iterate $\underline{\triangle x}^{(j)}$ in the iterative step \eqref{eq:incrementj}. To complete control over the iterative step, we then also need to control the contribution of the source term
 \[
 \underline{\triangle f}^{(j-1)},
 \]
 which we recall is given by the expression 
 \begin{align*}
\underline{\triangle f}^{(j-1)}(\tau,\xi): = \left(\begin{array}{c}\langle \phi_d, \lambda^{-2}(\tau)RN_{\nu}(\epsilon^{(j-1)})\rangle - \lambda^{-2}(\tau)RN_{\nu}(\epsilon^{(j-2)})\rangle\\ \mathcal{F}\big(\lambda^{-2}(\tau)RN_{\nu}(\epsilon^{(j-1)}) \big)(\xi) - \lambda^{-2}(\tau)RN_{\nu}(\epsilon^{(j-2)}) \big)(\xi)\end{array}\right).
\end{align*}
Thus we now formulate the corresponding iterative step for this contribution. Here we carefully single out the part where we do not gain a smallness factor, and the part for which we do gain smallness. Observe that the contribution of the discrete spectral part always leads to smallness gains and is hence in some sense negligible in the iteration scheme later on: 

\begin{prop}\label{prop:nonlineariterstep} Assume that 
\[
\tilde{\epsilon}^{(k)}(\tau, R) = x^{(k)}_d(\tau)\phi_d(R) + \int_0^\infty x^{(k)}(\tau, \xi)\phi(R, \xi)\rho(\xi)\,d\xi,\,1\leq k\leq j-1,
\]
with 
\[
\underline{x}^{(k)}(\tau, \xi) = \left(\begin{array}{c}x_d^{(k)}(\tau)\\ x^{(k)}(\tau, \xi)\end{array}\right) = \underline{x}^{(0)}+ \sum_{1\leq l\leq k}\underline{\triangle x}^{(l)}(\tau, \xi),\,1\leq l\leq k, 
\]
and $\underline{x}^{(0)} = \left(\begin{array}{c}x_d(\tau)\\ x(\tau, \xi)\end{array}\right)$ as in the statement of Lemma~\ref{lem:linhom}. Assume that we have 
\[
\sup_{1\leq k\leq j-1}\big[\big|x_d^{(k)}(\tau)\big| + \big|\partial_{\tau}x_d^{(k)}(\tau)\big|\big]\leq B\tau^{-(1-)},
\]
and furthermore each of the $x^{(k)}(\tau, \xi)$, $1\leq k\leq j-1$, admits a structure and bounds like the function $\triangle x^{(input)}$ in Proposition~\ref{prop:lineariterstep} with a constant $B\ll 1$ (instead of $A$ there). Further, assume that 
\[
\big|\triangle x_d^{(j-1)}(\tau)\big| + \big|\partial_{\tau}\triangle x_d^{(j-1)}(\tau)\big|\big]\leq A_d\tau^{-(1-)}, 
\]
and that the difference $\triangle x^{(j-1)}(\tau, \xi)$ admits the same structure and estimates as $\triangle x^{(input)}$ in Proposition~\ref{prop:lineariterstep} with constant $A$. Then we can conclude the following:
\begin{itemize}
\item For the discrete contribution to the next iterate $\triangle x^{(j)}_d(\tau)$, we get 
\begin{align*}
\big|\triangle x^{(j)}_d(\tau)\big| + \big|\partial_{\tau}\triangle x^{(j)}_d(\tau)\big|&\lesssim \big|\int_{\tau}^\infty H_d(\tau, \sigma)\triangle f^{(j-1)}(\sigma)\,d\sigma\big|\\
&\lesssim \tau^{-(1-)}\big[A_d\tau_0^{-1} + A\big]
\end{align*}
\item For the continuous spectral part, there is a splitting 
\[
\triangle f^{(j-1)}(\tau,\xi) = \triangle f^{(j-1)}_{bad}(\tau,\xi) +  \triangle f^{(j-1)}_{good}(\tau,\xi), 
\]
such that we have the following conclusion: there exist $\triangle \tilde{\tilde{x}}^{(j)}_{good,\kappa}, \triangle \tilde{\tilde{x}}^{(j)}_{bad,\kappa}$, $\kappa = 0,1$, satisfying the bounds 
\[
\big\|(\triangle \tilde{\tilde{x}}^{(j)}_{good,0}, \triangle \tilde{\tilde{x}}^{(j)}_{good,1})\big\|_{S_2}\lesssim \tau_0^{-(1-)}[AB + A_d\tau_0^{-1}],\,\big\|(\triangle \tilde{\tilde{x}}^{(j)}_{bad,0},\triangle \tilde{\tilde{x}}^{(j)}_{bad,1})\big\|_{S_2}\lesssim \tau_0^{-(1-)}[A + A_d\tau_0^{-1}],
\]
and such that if we put 
\begin{align*}
\triangle x^{(j)}_{good}(\tau, \xi)&:= \int_{\tau_0}^\tau\frac{\lambda^{\frac{3}{2}}(\tau)}{\lambda^{\frac{3}{2}}(\sigma)}\frac{\rho^{\frac{1}{2}}(\frac{\lambda^2(\tau)}{\lambda^2(\sigma)}\xi)}{\rho^{\frac{1}{2}}(\xi)}\frac{\sin[\lambda(\tau)\xi^{\frac{1}{2}}\int_{\tau}^\sigma\lambda^{-1}(u)\,du]}{\xi^{\frac{1}{2}}}\triangle f^{(j-1)}_{good}(\sigma, \frac{\lambda^2(\tau)}{\lambda^2(\sigma)}\xi)\,d\sigma\\
 & + S(\tau)\big(\triangle\tilde{\tilde{x}}^{(j)}_{good,0}, \triangle\tilde{\tilde{x}}^{(j)}_{good,1}\big),
 \end{align*}
and analogously 
\begin{align*}
\triangle x^{(j)}_{bad}(\tau, \xi)&:= \int_{\tau_0}^\tau\frac{\lambda^{\frac{3}{2}}(\tau)}{\lambda^{\frac{3}{2}}(\sigma)}\frac{\rho^{\frac{1}{2}}(\frac{\lambda^2(\tau)}{\lambda^2(\sigma)}\xi)}{\rho^{\frac{1}{2}}(\xi)}\frac{\sin[\lambda(\tau)\xi^{\frac{1}{2}}\int_{\tau}^\sigma\lambda^{-1}(u)\,du]}{\xi^{\frac{1}{2}}}\triangle f^{(j-1)}_{bad}(\sigma, \frac{\lambda^2(\tau)}{\lambda^2(\sigma)}\xi)\,d\sigma\\
 & + S(\tau)\big( \triangle\tilde{\tilde{x}}^{(j)}_{bad,0}, \triangle\tilde{\tilde{x}}^{(j)}_{bad,1}\big),
 \end{align*}
then $\triangle x^{(j)}_{good}, \triangle x^{(j)}_{bad}$ admit the same structure and estimates as the function $\triangle x^{(output)}$ in Proposition~\ref{prop:lineariterstep}, with $ \triangle\tilde{\tilde{x}}^{(j)}_{good,\kappa}, \triangle\tilde{\tilde{x}}^{(j)}_{bad,\kappa}$ replacing $\triangle \tilde{\tilde{x}}^{(output)}_{\kappa}$, respectively, and with the constant $A$ in Proposition~\ref{prop:lineariterstep} replaced by 
\[
AB + A_d\tau_0^{-1},\,A + A_d\tau_0^{-(1-)}, 
\]
respectively. 
\end{itemize}
\end{prop}
\begin{rem}\label{rem:nonlineariterstep} As we saw for the proof of Proposition~\ref{prop:firstiteratenonlin}, the reason for the lack of a smallness gain in the constant for the contribution of the bad term $\triangle f^{(j-1)}_{bad}$ is the fact that this term is in effect the Fourier transform of the linear interaction term $5\lambda^{-2}(\tau)(u_{\nu}^4 - u_0^4)\triangle\tilde{\epsilon}^{(j-1)}$, with the latter given by 
\[
\triangle\tilde{\epsilon}^{(j-1)}(\tau, R) = \triangle x^{(j-1)}_d(\tau)\phi_d(R) + \int_0^\infty \triangle x^{(j-1)}(\tau, \xi)\phi(R, \xi)\rho(\xi)\,d\xi, 
\]
and the lack of smallness gain comes precisely from the low-frequency regime $\xi<1$ and the linear in $\tau$ growth of $\triangle\tilde{\epsilon}^{(j-1)}$. Observe that this linear growth occurs only for the continuous spectral part of $\triangle\tilde{\epsilon}^{(j-1)}$, while the discrete spectral part is actually such that it decays to zero asymptotically, by assumption on $\triangle x^{(j-1)}_d(\tau)$. 

\end{rem}

The proof of Proposition~\ref{prop:nonlineariterstep} is very close to the one of Proposition~\ref{prop:firstiteratenonlin}, and hence we omit it here. 
This completes establishing the necessary a priori bounds in a first approximation, although we will have to refine these considerations a bit to obtain the desired convergence of the iterates.

\section{Preparations for the proof of convergence; refined estimates}

Observe that in the key inductive Propositions~\ref{prop:lineariterstep} - Proposition~\ref{prop:nonlineariterstep} we have not improved the constant $A$ controlling the inputs and outputs. This is entirely due to the source terms which are {\it{linear}} in the perturbation $\triangle x^{(j-1)}$ (or alternatively $\triangle \tilde{\epsilon}^{(j-1)}$), and we shall have to rely on delicate {\it{re-iteration arguments}} to obtain convergence of the series $\sum\triangle x^{(j)}$. However, before being able to set up such an argument, which eventually reduces the convergence issue to integrations over high-dimensional simplices, we have to provide certain tools which shall allow us to reduce to a very specific situation at the end, namely essentially unit frequencies throughout. For this, we shall have to re-visit the proofs of some of the earlier propositions and refine them a bit. 
\\

To begin with, we observe that for the time differentiated Duhamel parametrix, we can gain a smallness factor when restricting to very low frequencies: 

\begin{lem}\label{lem:smallnessgainoutputfreqsmall} Let $\triangle x^{(output)}, \triangle_{>\tau} x^{(output)}$ be defined as in Proposition~\ref{prop:lineariterstep}. Then for $0<\kappa_*\ll 1$, we have 
\[
\big(\sum_{\substack{N\geq \tau_0\\ N\,\text{dyadic}}}\sup_{\tau\sim N}(\frac{\tau}{\tau_0})^{2\kappa}\big\|\chi_{\xi<\kappa_*}\mathcal{D}_{\tau}\triangle_{>\tau}x^{(output)}\big\|_{S_2}^2\big)^{\frac12}\lesssim \kappa_*^{\gamma}A. 
\]
for a suitable absolute constant $\gamma>0$. Moreover, if $\triangle x^{(k)},\,1\leq k\leq j$, are as in Proposition~\ref{prop:nonlineariterstep}, then we have 
\[
\big(\sum_{\substack{N\geq \tau_0\\ N\,\text{dyadic}}}\sup_{\tau\sim N}(\frac{\tau}{\tau_0})^{2\kappa}\big\|\chi_{\xi<\kappa_*}\mathcal{D}_{\tau}\triangle_{>\tau}x^{(j)}\big\|_{S_2}^2\big)^{\frac12}\lesssim \kappa_*^{\gamma}[A + A_d\tau_0^{-(1-)}]. 
\]
In both inequalities one may replace the operator $\mathcal{D}_\tau$ by a weight $\frac{\xi^{\frac12-}}{\langle\xi\rangle^{\frac12-}}$. 
\end{lem}
\begin{proof}(lemma)
{\it{(1) Proof of the first inequality.}} Recall that the input $\triangle x^{(input)}$ admits for both low as well as high frequencies a decomposition into a 'free part' of the form $S(\tau)(\ldots)$ plus a term with less structure but better bounds. Thus to prove the first inequality, we distinguish between the case where $\triangle x^{(input)}$ is in effect of the form $S(\tau)(\ldots)$, as well as the remaining situation. In the former situation, the proof of Proposition~\ref{prop:firstiteratelincontimprov} is relevant, and we shall
freely borrow from it. Thus replace for now $\triangle x^{(input)}$ by $x^{(0)}$ with data of norm $\lesssim A$. As in the proof of Proposition~\ref{prop:firstiteratelincontimprov} we only consider the contribution of the delicate source term $\beta_{\nu}(\tau)\mathcal{K}_{cc}\mathcal{D}_{\tau}x^{(0)}$. Then we need to eke out an extra gain $\kappa_*^{\gamma}$ from this proof provided we restrict the output frequency to $\xi<\kappa_*$. Referring to the steps of that proof, consider case {\it{(1i)}}. 
There we use the gain of $(\frac{\lambda^2(\tau)}{\lambda^2(\sigma)}\xi)^{\frac{1}{2+}}$ to cancel the singular factor $\rho(\eta)$, but we can keep track of this gain and obtain the more precise gain 
\[
(\frac{\lambda^2(\tau)}{\lambda^2(\sigma)}\xi)^{\frac{1}{2+}}\cdot\eta^{-\frac{1}{2+}}\cdot\eta^{\frac{1}{2+}}
\]
where the factor $\eta^{\frac{1}{2+}}$ comes from application of Cauchy-Schwarz to the $\eta$-integral. We can also write this as 
\[
(\frac{\lambda^2(\tau)}{\lambda^2(\sigma)}\xi)^{\frac{1}{4+}}\cdot (\frac{\lambda^2(\tau)}{\lambda^2(\sigma)}\xi)^{\frac{1}{4}}\cdot\eta^{-\frac14}\cdot\eta^{\frac14}
\]
Then using the assumption $\frac{\lambda^2(\sigma)}{\lambda^2(\tau_0)}\eta<1$ in that case, we use the factor $\eta^{\frac14}$ to gain $\frac{\lambda^{\frac12}(\tau_0)}{\lambda^{\frac12}(\tau)}$ (weaker than what we had previously in case {\it{(1i)}}, but still enough), while we also gain 
\[
(\frac{\lambda^2(\tau)}{\lambda^2(\sigma)}\xi)^{\frac{1}{4+}}\lesssim \frac{\lambda^{\frac14}(\tau)}{\lambda^{\frac14}(\sigma)}\cdot \kappa_*^{\frac14}
\]
which gives the required gain with $\gamma = \frac14$. 
\\
Case {\it{(1ii)}} is similar in that one trades a power of $2^j$, such as $2^{\frac{j}{4}}$, for a power of $\kappa_*$, such as $\kappa_*^{\frac14}$, while in case {\it{(1iii)}} one gets an extra power of $\kappa_*$ without any losses. 
In case {\it{(1iv)}} one uses the modified bound (assuming $\xi<\kappa_*$)
\[
\big|\xi^{-0+}F(\frac{\lambda^2(\tau)}{\lambda^2(\sigma)}\xi, \eta)\rho(\eta)\big|\lesssim \eta^{\frac12-}\lesssim \kappa_*^{\frac14}[\frac{\lambda(\tau_0)}{\lambda(\sigma)}]^{\frac14-} [\frac{\lambda(\tau_0)}{\lambda(\tau)}]^{2\delta_0}
\]
and again the desired conclusion is reached with $\gamma = \frac14$. The remaining cases {\it{(1v)}}, {\it{(1vi)}} are more of the same.
\\
This deals with the case when $\triangle x^{(input)}$ is replaced by the 'free part' $S(\tau)(\ldots)$. In case that $\triangle x^{(input)}$ is of the form $\triangle_{>\tau} x^{(input)}$ or else $\widetilde{\triangle x^{(input)}}$, one needs to revisit {\it{Step 4}} of the proof of Proposition~\ref{prop:lineariterstep}, where the required gain is easily seen to come from directly from the $\xi$-integral via Holder upon using an $L^\infty_{d\xi}$-bound for the expression, or else exploring the vanishing property of the kernel of $\mathcal{K}_{cc}$, depending on the situation. 
\\

{\it{(2) Proof of the second inequality.}} This follows by revisiting the proofs of Proposition~\ref{prop:firstiteratenonlin}, Proposition~\ref{prop:nonlineariterstep} in for low frequency outputs $\xi<1$. The idea is that for an expression such as 
\begin{align*}
&\big\|\xi^{-0+}\int_{\tau}^\infty\frac{\lambda^{\frac{3}{2}}(\tau)}{\lambda^{\frac{3}{2}}(\sigma)}\frac{\rho^{\frac{1}{2}}(\frac{\lambda^2(\tau)}{\lambda^2(\sigma)}\xi)}{\rho^{\frac{1}{2}}(\xi)}\cos[\lambda(\tau)\xi^{\frac{1}{2}}\int_{\tau}^\sigma\lambda^{-1}(u)\,du]f(\sigma, \underline{x}^{(0)})(\frac{\lambda^2(\tau)}{\lambda^2(\sigma)}\xi)\,d\sigma\big\|_{L^2_{d\xi}(\xi<1)}\\
\end{align*}
which occurs in the proof of Proposition~\ref{prop:firstiteratenonlin}, one gains smallness in $\kappa_*$ (upon restricting to $\xi<\kappa_*$) by estimating the expression in $L^\infty_{d\xi}$ and invoking Holder. This is possible for $\mathcal{D}_{\tau}\triangle x^{(j)}$ but of course not for $\triangle x^{(j)}$. 
We omit the simple details. 
\end{proof}

While the preceding lemma dealt with small output frequencies, we also need a lemma gaining smallness for small input frequencies, but in a more restrictive context: 

\begin{lem}\label{lem:smallnessgaininputfreqsmall} Assume that $\triangle x^{(input)}(\tau, \xi)$ is as in Proposition~\ref{prop:lineariterstep}. Then for $0<\kappa_*\ll 1$, we have 
\begin{align*}
&\big\|\int_{\tau_0}^\infty \frac{\xi^{\frac12}}{\langle\xi\rangle^{\frac12}}\frac{\lambda^{\frac{3}{2}}(\tau_0)}{\lambda^{\frac{3}{2}}(\sigma)}\frac{\rho^{\frac{1}{2}}(\frac{\lambda^2(\tau_0)}{\lambda^2(\sigma)}\xi)}{\rho^{\frac{1}{2}}(\xi)}\frac{\sin[\lambda(\tau_0)\xi^{\frac{1}{2}}\int_{\tau_0}^\sigma\lambda^{-1}(u)\,du]}{\xi^{\frac{1}{2}}}\mathcal{R}_1(\sigma, \chi_{\eta<\kappa_*}\triangle x^{(input)})\,d\sigma\big\|_{S_1}\\&\lesssim [\kappa_*^{\gamma} + \tau_0^{-\gamma}]A
\end{align*}
\begin{align*}
&\big\|\int_{\tau_0}^\infty \frac{\lambda^{\frac{3}{2}}(\tau_0)}{\lambda^{\frac{3}{2}}(\sigma)}\frac{\rho^{\frac{1}{2}}(\frac{\lambda^2(\tau_0)}{\lambda^2(\sigma)}\xi)}{\rho^{\frac{1}{2}}(\xi)}\cos[\lambda(\tau_0)\xi^{\frac{1}{2}}\int_{\tau_0}^\sigma\lambda^{-1}(u)\,du]\mathcal{R}_1(\sigma, \chi_{\eta<\kappa_*}\triangle x^{(input)})\,d\sigma\big\|_{S_2}\\&\lesssim [\kappa_*^{\gamma}+ \tau_0^{-\gamma}]A
\end{align*}
for a suitable $\gamma>0$. 
\end{lem}
\begin{proof}(lemma) Observe that in the first expression we have included a crucial additional factor $\frac{\xi^{\frac12}}{\langle\xi\rangle^{\frac12}}$ which gives a gain for low frequencies $\xi<1$. Then, depending on whether we replace $\triangle x^{(input)}(\tau, \xi)$ by its 'free part' $S(\tau)(\ldots)$ or by a less structured but smaller error part (as in the assumptions in Proposition~\ref{prop:lineariterstep}), the relevant gain is obtained by revisiting {\it{Step 4, 5}} in the proof of Proposition~\ref{prop:firstiteratelincont}, respectively the proof of Lemma~\ref{lem:hereneednorm}. Assuming first that $\triangle x^{(input)}$ is 'free' and consulting {\it{Step 4}} in the proof of Proposition~\ref{prop:firstiteratelincont}, one sees that in the situations $\xi\cdot\frac{\lambda^2(\tau)}{\lambda^2(\sigma)}\ll \eta$, $\xi\cdot\frac{\lambda^2(\tau)}{\lambda^2(\sigma)}\sim\eta$, the extra factor $\frac{\xi^{\frac12}}{\langle\xi\rangle^{\frac12}}$ translates into a gain of 
\[
\eta^{\frac{1}{2}}\frac{\lambda(\sigma)}{\lambda(\tau)}\lesssim \kappa_*^{\frac12} \frac{\lambda(\sigma)}{\lambda(\tau)}, 
\]
keeping in mind our additional restriction $\eta<\kappa_*$. We may replace the preceding also by a factor involving a weaker gain in $\kappa_*$, 
\[
[\kappa_*^{\frac12} \frac{\lambda(\sigma)}{\lambda(\tau)}]^{2\gamma}, 
\]
which leads to a good bound following {\it{(4(i).a)}}, {\it{(4(ii).b)}} in {\it{Step 4}} in the proof of Proposition~\ref{prop:firstiteratelincont}, provided 
\[
\int_{\tau}^\infty [\frac{\lambda(\sigma)}{\lambda(\tau)}]^{2\gamma}\beta_{\nu}^2(\sigma)(\frac{\lambda(\sigma)}{\lambda(\tau_0)})^{2\delta_0}\,d\sigma
\]
converges, which is the case provided $0<\gamma<\gamma_0(\nu)$. On the other hand, in case {\it{(4(ii).c)}} one performs integration by parts with respect to $\sigma$ to translate the gain in $\xi$ (for $\xi<1$) into a gain in $\eta$, whence in $\kappa_*$.  
This deals with the case when $\triangle x^{(input)}(\tau, \xi)$ is 'free' and we restrict to small output frequencies $\xi<1$. On the other hand, in the case of large output frequencies $\xi>1$, following {\it{Step 5}} in the proof of Proposition~\ref{prop:firstiteratelincont}, one directly infers either a gain $\tau_0^{-\gamma}$ or else a gain $\kappa_*^{\gamma}$, which the concludes the case when $\triangle x^{(input)}(\tau, \xi)$ is 'free'. 
\\
On the other hand, if $\triangle x^{(input)}$ is of the form $\triangle_{>\tau} x^{(input)}$ (for low frequencies) or $\widetilde{\triangle x^{(input)}}$ (for high frequencies), then the required gain follows directly from the proof of Lemma~\ref{lem:hereneednorm}. 
\end{proof}

The preceding two lemmas have dealt with improvements to estimates tied to the linear source term $\mathcal{R}_1(\sigma, \triangle x^{(input)})$. The next lemma deals with an improvement for the terms arising from the nonlinear source terms $\triangle f$. Observe that we jut break even in the conclusion of Proposition~\ref{prop:firstiteratenonlin}, as well as Proposition~\ref{prop:nonlineariterstep}, without smallness gain. This  is entirely due to the case of low frequencies of the output, and hence inclusion of an additional weight there improves the estimate correspondingly: 
\begin{lem}\label{lem:smallnessgainnl} Using the same notation and making the same assumptions as in the statement of Proposition~\ref{prop:nonlineariterstep}, put 
\[
\triangle x^{(j)}: = \triangle x^{(j)}_{good} + \triangle x^{(j)}_{bad},\,\triangle \tilde{\tilde{x}}^{(j)}_{\rho}: = \triangle \tilde{\tilde{x}}^{(j)}_{good,\rho} + \triangle \tilde{\tilde{x}}^{(j)}_{bad,\rho},\,\rho = 0, 1. 
\]
Then setting 
\[
\widetilde{\triangle x^{(j)}}(\tau, \cdot): = \triangle x^{(j)}(\tau, \cdot) - S(\tau)\big(\triangle \tilde{\tilde{x}}^{(j)}_0, \triangle \tilde{\tilde{x}}^{(j)}_1\big),\,\triangle_{>\tau} x^{(j)}: =  \triangle_{>\tau} x^{(j)}_{good} + \triangle_{>\tau} x^{(j)}_{bad}, 
\]
as well as 
\[
\triangle \tilde{x}^{(j)}_\rho(\xi): = \triangle \tilde{x}^{(j)}_{good,\rho}(\xi) + \triangle \tilde{x}^{(j)}_{bad,\rho}(\xi),\,\rho = 0,\,1, 
\]
we have the improved estimates 
 \begin{equation}\begin{split}
  \sup_{\tau\geq \tau_0}(\frac{\tau_0}{\tau})^{\kappa}\big\|\chi_{\xi>1}\triangle x^{(j)}(\tau, \xi)\big\|_{S_1}+  \big(\sum_{\substack{N\gtrsim \tau_0\\ N\,\text{dyadic}}}\sup_{\tau\sim N}(\frac{\tau}{\tau_0})^{2\kappa}\big\|\chi_{\xi>1}\mathcal{D}_{\tau}\widetilde{\triangle x^{(j)}}\big\|_{S_2}^2\big)^{\frac12}
 \lesssim [A+A_d]\cdot\tau_0^{-\gamma}.
  \end{split}\end{equation}
as well as
 \begin{equation}\begin{split}
 &\sup_{\tau\geq \tau_0}(\frac{\tau}{\tau_0})^{-\kappa}\big\|\chi_{\xi<1}\frac{\xi^{\frac12}}{\langle\xi\rangle^{\frac12}}\triangle x^{(j)}_{>\tau}(\tau, \xi)\big\|_{S_1}\\& + \big(\sum_{\substack{N\gtrsim \tau_0\\ N\,\text{dyadic}}}[\sup_{\tau\sim N}(\frac{\tau}{\tau_0})^{\kappa}\big\|\chi_{\xi<1}\mathcal{D}_{\tau}\triangle_{>\tau}x^{(j)}(\tau, \xi)\big\|_{S_2}]^2\big)^{\frac{1}{2}}\\&
 + \big\|\big(\frac{\xi^{\frac12-}}{\langle\xi\rangle^{\frac12-}}\triangle \tilde{x}^{(j)}_0(\xi), \triangle \tilde{x}^{(j)}_1(\xi)\big)\big\|_{S}\lesssim [A+A_d]\cdot\tau_0^{-\gamma},
  \end{split}\end{equation}
for suitable $\gamma>0$. 
\end{lem}
The proof of this lemma follows again in close analogy to the one of Proposition~\ref{prop:firstiteratenonlin}, the point being that inclusion of the extra weight $\frac{\xi^{\frac12-}}{\langle\xi\rangle^{\frac12-}}$ in the un-differentiated terms $\triangle x^{(j)}_{>\tau},\,\triangle \tilde{x}^{(j)}_0(\xi)$ allows one to perform integration by parts with respect to $\sigma$ in the Duhamel parametrix applied to the 'bad' source terms $E_2$, recalling Proposition~\ref{prop:nonlinbounds}. 
\\

In the next section, we shall complement the preceding three lemmas with further improvements of 're-iteration type', i. e. exploiting cancellations when applying the iterative step twice, as opposed to only once(like for the preceding three lemmas). 

\section{Improvements upon re-iteration}

Recall the key iterative step \eqref{eq:incrementj}. Each of the two terms on the right contributes to the next iterate $\underline{\triangle x}^{(j)}$, in a way specified by Proposition~\ref{prop:lineariterstep}, Proposition~\ref{prop:nonlineariterstep}, and we can correspondingly write the continuous spectral part of $\triangle x^{(j)}$ in the following manner(changing the notation for $\triangle\tilde{\tilde{x}}^{(j)}_{0,1}$, which now incorporate the contributions from both source terms)
\begin{equation}\label{eq:trianglexjformula}\begin{split}
&\triangle x^{(j)}(\tau, \xi)\\& =  \int_{\tau_0}^\tau\frac{\lambda^{\frac{3}{2}}(\tau)}{\lambda^{\frac{3}{2}}(\sigma)}\frac{\rho^{\frac{1}{2}}(\frac{\lambda^2(\tau)}{\lambda^2(\sigma)}\xi)}{\rho^{\frac{1}{2}}(\xi)}\frac{\sin[\lambda(\tau)\xi^{\frac{1}{2}}\int_{\tau}^\sigma\lambda^{-1}(u)\,du]}{\xi^{\frac{1}{2}}}\\&\hspace{4cm}\cdot\big[\triangle f^{(j-1)}(\sigma, \cdot) + \mathcal{R}(\sigma, \underline{\triangle x}^{(j-1)})\big](\frac{\lambda^2(\tau)}{\lambda^2(\sigma)}\xi)\,d\sigma\\
 & + S(\tau)\big(\triangle\tilde{\tilde{x}}^{(j)}_{0}, \triangle\tilde{\tilde{x}}^{(j)}_{1}\big),
\end{split}\end{equation}
The goal of this section is to show that upon {\it{re-iteration}} of \eqref{eq:trianglexjformula}, the contribution of the nonlinear source term $\triangle f^{(j-1)}$ becomes essentially negligible, and we can reduce to the contribution of the principal linear source term $\beta_{\nu}(\sigma)\mathcal{K}_{cc}\mathcal{D}_{\sigma}\triangle x^{(j-1)}$. 
\\
To make things precise, we commence by introducing a quantity $\triangle A_j\in \R_+$ which measures the size of the correction $\triangle x^{(j)}$ as well as all the related functions, in accordance with Proposition~\ref{prop:lineariterstep}, Proposition~\ref{prop:nonlineariterstep}. Thus we make 
\begin{defn}\label{defn:sizeofcorrection} We define the quantity $\triangle A_j$, which controls the size of the correction at stage $j$, by 
\begin{align*}
&\triangle A_j:=\\& \sup_{\tau\geq \tau_0}(\frac{\tau_0}{\tau})^{\kappa}\big\|\chi_{\xi>1}\triangle x^{(j)}(\tau, \xi)\big\|_{S_1} + \big(\sum_{\substack{N\gtrsim \tau_0\\ N\,\text{dyadic}}}[\sup_{\tau\sim N}(\frac{\tau}{\tau_0})^{\kappa}\big\|\chi_{\xi>1}\mathcal{D}_{\tau}\widetilde{\triangle x^{(j)}}(\tau, \xi)\big\|_{S_2}]^2\big)^{\frac{1}{2}}\\
& + \sup_{\tau\geq \tau_0}(\frac{\tau_0}{\tau})^{\kappa}\big\|\chi_{\xi<1}\triangle_{>\tau}\triangle x^{(j)}(\tau, \xi)\big\|_{S_1} + \big(\sum_{\substack{N\gtrsim \tau_0\\ N\,\text{dyadic}}}[\sup_{\tau\sim N}(\frac{\tau}{\tau_0})^{\kappa}\big\|\chi_{\xi<1}\mathcal{D}_{\tau}\triangle_{>\tau}\triangle x^{(j)}(\tau, \xi)\big\|_{S_2}]^2\big)^{\frac{1}{2}}\\
& + \big\|(\triangle\tilde{x}_0^{(j)}, \triangle\tilde{x}_1^{(j)})\big\|_{\tilde{S}} + \big\|(\triangle\tilde{\tilde{x}}^{(j)}_0, \triangle\tilde{\tilde{x}}^{(j)}_1)\big\|_{\tilde{S}} + \sup_{\tau\geq\tau_0}\tau^{(1-)}|\triangle x_d^{(j)}(\tau)| + \sup_{\tau\geq\tau_0}\tau^{(1-)}|\partial_{\tau}\triangle x_d^{(j)}(\tau)|
\end{align*}
\end{defn}

Now pick a very small frequency cutoff $0<\kappa_*\ll 1$, to be specified later. We already point out that $\kappa_*$ will be chosen independently of $\tau_0$, and the latter will then be picked at the very end large enough to kill any constants depending on $\kappa_*$. 
Recalling that 
\begin{align*}
&\widetilde{\triangle x^{(j)}}(\tau, \xi) = \triangle x^{(j)}(\tau, \xi) - S(\tau)\big(\triangle\tilde{\tilde{x}}^{(j)}_{0}, \triangle\tilde{\tilde{x}}^{(j)}_{1}\big)\\
& = \int_{\tau_0}^\tau\frac{\lambda^{\frac{3}{2}}(\tau)}{\lambda^{\frac{3}{2}}(\sigma)}\frac{\rho^{\frac{1}{2}}(\frac{\lambda^2(\tau)}{\lambda^2(\sigma)}\xi)}{\rho^{\frac{1}{2}}(\xi)}\frac{\sin[\lambda(\tau)\xi^{\frac{1}{2}}\int_{\tau}^\sigma\lambda^{-1}(u)\,du]}{\xi^{\frac{1}{2}}}\\&\hspace{4cm}\cdot\big[\triangle f^{(j-1)}(\sigma, \cdot) + \mathcal{R}(\sigma, \underline{\triangle x}^{(j-1)})\big](\frac{\lambda^2(\tau)}{\lambda^2(\sigma)}\xi)\,d\sigma,\\
\end{align*}
we can now state the following important {\it{re-iteration lemma}}: 
\begin{lem}\label{lem:Phireduction} Introduce the operator 
\begin{align*}
&\Phi(f)(\tau, \xi): = \\&\int_{\tau_0}^\tau \frac{\lambda^{\frac{3}{2}}(\tau)}{\lambda^{\frac{3}{2}}(\sigma)}\frac{\rho^{\frac{1}{2}}(\frac{\lambda^2(\tau)}{\lambda^2(\sigma)}\xi)}{\rho^{\frac{1}{2}})\xi)}\cos[\lambda(\tau)\xi^{\frac{1}{2}}\int_{\tau}^\sigma\lambda^{-1}(u)\,du]\beta_{\nu}(\sigma)(\mathcal{K}_{cc}f)(\frac{\lambda^2(\tau)}{\lambda^2(\sigma)}\xi)\,d\sigma.
\end{align*}
Borrowing the notation from Proposition~\ref{prop:nonlineariterstep}, assume that 
\[
\sum_{k=1}^{j-1}\triangle A_k\lesssim1.
\]
Then 
\begin{equation}\label{eq:keyPhiinductive}\begin{split}
&\big(\sum_{\substack{N\gtrsim \tau_0\\ N\,\text{dyadic}}}[\sup_{\tau\sim N}(\frac{\tau}{\tau_0})^{\kappa}\big\|\chi_{\xi>\kappa_*}\mathcal{D}_{\tau}\widetilde{\triangle x^{(j)}}\big\|_{S_2}]^2\big)^{\frac{1}{2}}\\&\lesssim \big(\sum_{\substack{N\gtrsim \tau_0\\ N\,\text{dyadic}}}[\sup_{\tau\sim N}(\frac{\tau}{\tau_0})^{\kappa}\big\|\chi_{\xi>\kappa_*}\Phi\chi_{\xi>\kappa_*}\Phi(\mathcal{D}_{\tau}\triangle x^{(j-2)})\big\|_{S_2}]^2\big)^{\frac{1}{2}} + C_{\kappa_*}\tau_0^{-\gamma}[\triangle A_{j-1} + \triangle A_{j-2}]
\end{split}\end{equation}
for suitable $\gamma>0$ (absolute constant). 
\end{lem}
\begin{proof} We use the relation \eqref{eq:trianglexjformula} at both $j$ and $j-1$, which then gives $\triangle x^{(j)}$ in terms of $\underline{\triangle x}^{(j-2)}$ as well as $\underline{\triangle x}^{(k)},\,k\leq j-2$. Then recall
\[
\underline{\triangle x}^{(j-2)} = \left(\begin{array}{c}\triangle x^{(j-2)}_d\\ \triangle x^{(j-2)}\end{array}\right),\,\underline{\triangle x}^{(j-1)} = \left(\begin{array}{c}\triangle x^{(j-1)}_d\\ \triangle x^{(j-1)}\end{array}\right),
\]
and invoke \eqref{eq:trianglexjformula} at $j$. Then Proposition~\ref{prop:lineariterstepdisc} allows us to write\footnote{The notation is somewhat loose here, the error being in the sense of the norm used on the left in \eqref{eq:keyPhiinductive}} 
\begin{align*}
&\mathcal{D}_{\tau}\widetilde{\triangle x^{(j)}}(\tau, \xi)\\& = \int_{\tau_0}^\tau\frac{\lambda^{\frac{3}{2}}(\tau)}{\lambda^{\frac{3}{2}}(\sigma)}\frac{\rho^{\frac{1}{2}}(\frac{\lambda^2(\tau)}{\lambda^2(\sigma)}\xi)}{\rho^{\frac{1}{2}}(\xi)}\cos[\lambda(\tau)\xi^{\frac{1}{2}}\int_{\tau}^\sigma\lambda^{-1}(u)\,du]\\&\hspace{4cm}\cdot\big[\triangle f^{(j-1)}(\sigma, \cdot) + \mathcal{R}_1(\sigma, \triangle x^{(j-1)})\big](\frac{\lambda^2(\tau)}{\lambda^2(\sigma)}\xi)\,d\sigma\\& + \tau_0^{-1}O_{\kappa_*}(\triangle A_{j-1}). 
\end{align*}
Furthermore, recalling \eqref{eq:R_1structure}, as well as the proof of Proposition~\ref{prop:firstiteratelincont}, Proposition~\ref{prop:lineariterstep}, one easily infers 
\begin{equation}\label{eq:intermediatereduction1}\begin{split}
&\chi_{\xi>\kappa_*}\mathcal{D}_{\tau}\widetilde{\triangle x^{(j)}}(\tau, \xi)\\& = \chi_{\xi>\kappa_*}\int_{\tau_0}^\tau\frac{\lambda^{\frac{3}{2}}(\tau)}{\lambda^{\frac{3}{2}}(\sigma)}\frac{\rho^{\frac{1}{2}}(\frac{\lambda^2(\tau)}{\lambda^2(\sigma)}\xi)}{\rho^{\frac{1}{2}}(\xi)}\cos[\lambda(\tau)\xi^{\frac{1}{2}}\int_{\tau}^\sigma\lambda^{-1}(u)\,du]\\&\hspace{4cm}\cdot\big[\triangle f^{(j-1)}(\sigma, \cdot) + \beta_{\nu}(\sigma)\mathcal{K}_{cc}(\mathcal{D}_{\sigma}\triangle x^{(j-1)})\big](\frac{\lambda^2(\tau)}{\lambda^2(\sigma)}\xi)\,d\sigma\\& + \tau_0^{-1}O_{\kappa_*}(\triangle A_{j-1}). 
\end{split}\end{equation}
In fact, Remark~\ref{rem:prop71highfreqimprov} shows that the contribution of $\beta_{\nu}^2(\sigma)\mathcal{K}_{cc}(\triangle x^{(j-1)})$ gains a power of $\tau_0^{-1}$ upon restricting the output to frequencies $\gtrsim 1$. 
\\ 

Then use \eqref{eq:trianglexjformula} at $j-1$, resulting via the Duhamel type parametrix in a contribution from $\triangle f^{(j-2)}(\sigma, \cdot)$ as well as one from $\mathcal{R}_1(\sigma, \triangle x^{(j-2)})$. To begin with, consider the contribution when we replace $\triangle x^{(j-1)}$ on the right hand side in \eqref{eq:intermediatereduction1} by 
\begin{align*}
&\int_{\tau_0}^\tau\frac{\lambda^{\frac{3}{2}}(\tau)}{\lambda^{\frac{3}{2}}(\sigma)}\frac{\rho^{\frac{1}{2}}(\frac{\lambda^2(\tau)}{\lambda^2(\sigma)}\xi)}{\rho^{\frac{1}{2}}(\xi)}\frac{\sin[\lambda(\tau)\xi^{\frac{1}{2}}\int_{\tau}^\sigma\lambda^{-1}(u)\,du]}{\xi^{\frac{1}{2}}}\triangle f^{(j-2)}(\sigma, \frac{\lambda^2(\tau)}{\lambda^2(\sigma)}\xi)\,d\sigma\\& + S(\tau)\big(\triangle\tilde{\tilde{x}}^{(j-1)}_0, \triangle\tilde{\tilde{x}}^{(j-1)}_1\big),
 \end{align*}
 Here we are somewhat loose with the notation and assume that the correction term $\triangle\tilde{\tilde{x}}^{(j-1)}_1$ is the one coming from Proposition~\ref{prop:nonlineariterstep} at $j-1$ instead of $j$. Then using Lemma~\ref{lem:smallnessgainnl} as well as Remark~\ref{rem:proplineariterstep}, and inspecting the proof of Proposition~\ref{prop:nonlineariterstep}, Proposition~\ref{prop:firstiteratenonlin}, we see that the corresponding contribution to the left hand side of \eqref{eq:keyPhiinductive} is of size $\tau_0^{-\gamma}O_{\kappa_*}(\triangle A_{j-2})$. 
 \\
 Similarly, replacing  
$\triangle x^{(j-1)}$ in $\triangle f^{(j-1)}(\sigma, \cdot)$ on the right hand side in \eqref{eq:intermediatereduction1} by 
\begin{align*}
&\int_{\tau_0}^\tau\frac{\lambda^{\frac{3}{2}}(\tau)}{\lambda^{\frac{3}{2}}(\sigma)}\frac{\rho^{\frac{1}{2}}(\frac{\lambda^2(\tau)}{\lambda^2(\sigma)}\xi)}{\rho^{\frac{1}{2}}(\xi)}\frac{\sin[\lambda(\tau)\xi^{\frac{1}{2}}\int_{\tau}^\sigma\lambda^{-1}(u)\,du]}{\xi^{\frac{1}{2}}}\mathcal{R}_1(\sigma, \triangle x^{(j-2)}))(\frac{\lambda^2(\tau)}{\lambda^2(\sigma)}\xi)\,d\sigma\\& + S(\tau)\big(\triangle\tilde{\tilde{x}}^{(j-1)}_0, \triangle\tilde{\tilde{x}}^{(j-1)}_1\big),
 \end{align*}
where now the corrections $\triangle\tilde{\tilde{x}}^{(j-1)}_{0,1}$ are in accordance with Proposition~\ref{prop:lineariterstep}, and invoking Lemma~\ref{lem:smallnessgainnl}, we arrive at the contribution of size $\tau_0^{-\gamma}O_{\kappa_*}(\triangle A_{j-2})$. 
\\
At this stage we have shown that 
\begin{equation}\label{eq:intermediatereduction2}\begin{split}
&\chi_{\xi>\kappa_*}\mathcal{D}_{\tau}\widetilde{\triangle x^{(j)}}(\tau, \xi)\\& = \chi_{\xi>\kappa_*}\int_{\tau_0}^\tau\frac{\lambda^{\frac{3}{2}}(\tau)}{\lambda^{\frac{3}{2}}(\sigma)}\frac{\rho^{\frac{1}{2}}(\frac{\lambda^2(\tau)}{\lambda^2(\sigma)}\xi)}{\rho^{\frac{1}{2}}(\xi)}\cos[\lambda(\tau)\xi^{\frac{1}{2}}\int_{\tau}^\sigma\lambda^{-1}(u)\,du]\\&\hspace{4cm}\cdot\beta_{\nu}(\sigma)\mathcal{K}_{cc}(\mathcal{D}_{\sigma}\triangle x^{(j-1)}_{{\mathcal{R}}})(\frac{\lambda^2(\tau)}{\lambda^2(\sigma)}\xi)\,d\sigma\\& + \tau_0^{-\gamma}O_{\kappa_*}(\triangle A_{j-1} + \triangle A_{j-2}), 
\end{split}\end{equation}
where 
\begin{align*}
\triangle x^{(j-1)}_{{\mathcal{R}}}(\tau, \xi): &= \int_{\tau_0}^\tau\frac{\lambda^{\frac{3}{2}}(\tau)}{\lambda^{\frac{3}{2}}(\sigma)}\frac{\rho^{\frac{1}{2}}(\frac{\lambda^2(\tau)}{\lambda^2(\sigma)}\xi)}{\rho^{\frac{1}{2}}(\xi)}\frac{\sin[\lambda(\tau)\xi^{\frac{1}{2}}\int_{\tau}^\sigma\lambda^{-1}(u)\,du]}{\xi^{\frac12}}\\&\hspace{5cm}\cdot \mathcal{R}_1(\sigma, \triangle x^{(j-2)})(\frac{\lambda^2(\tau)}{\lambda^2(\sigma)}\xi)\,d\sigma\\& + S(\tau)\big(\triangle\tilde{\tilde{x}}^{(j-1)}_0, \triangle\tilde{\tilde{x}}^{(j-1)}_1\big)
\end{align*}
and here, by abuse of notation, the correction terms $\triangle\tilde{\tilde{x}}^{(j-1)}_{0,1}$ are in accordance with Proposition~\ref{prop:lineariterstep}. 
\\

At this stage we have almost achieved the conclusion of the lemma, except that we still need to reduce $\mathcal{D}_{\sigma}\triangle x^{(j-1)}_{{\mathcal{R}}}$ to 
\[
\chi_{\xi>\kappa_*}[\mathcal{D}_{\sigma}\triangle x^{(j-1)}_{{\mathcal{R}}} - \mathcal{D}_{\sigma}S(\sigma)\big(\triangle\tilde{\tilde{x}}^{(j-1)}_0, \triangle\tilde{\tilde{x}}^{(j-1)}_1\big)],
\]
and further reduce $\mathcal{R}_1$ in the definition of the latter to the principal part 
\[
\beta_{\nu}(\sigma)\mathcal{K}_{cc}(\mathcal{D}_{\sigma}\triangle x^{(j-2)}).
\]
The crucial part here is the inclusion of the cutoff $\chi_{\xi>\kappa_*}$ in front of $\mathcal{D}_{\sigma}\triangle x^{(j-1)}_{{\mathcal{R}}}$, and this will be achieved by combining the various oscillatory phases in the integral. 
\\

{\it{(1): Reduction of $\mathcal{D}_{\sigma}\triangle x^{(j-1)}_{{\mathcal{R}}}$ to $\chi_{\cdot>\kappa_*}\mathcal{D}_{\sigma}\triangle x^{(j-1)}_{{\mathcal{R}}}$.}} Consider the function 
\begin{align*}
&\tilde{\Phi}(\tau, \xi): = \chi_{\xi>\kappa_*}\int_{\tau_0}^\tau\frac{\lambda^{\frac{3}{2}}(\tau)}{\lambda^{\frac{3}{2}}(\sigma)}\frac{\rho^{\frac{1}{2}}(\frac{\lambda^2(\tau)}{\lambda^2(\sigma)}\xi)}{\rho^{\frac{1}{2}}(\xi)}\cos[\lambda(\tau)\xi^{\frac{1}{2}}\int_{\tau}^\sigma\lambda^{-1}(u)\,du]\\&\hspace{4cm}\cdot\beta_{\nu}(\sigma)\mathcal{K}_{cc}(\chi_{\cdot<\kappa_*}\mathcal{D}_{\sigma}\triangle x^{(j-1)}_{{\mathcal{R}}})(\frac{\lambda^2(\tau)}{\lambda^2(\sigma)}\xi)\,d\sigma
\end{align*}
and recall that 
\begin{equation}\label{eq:Kccintermediate}\begin{split}
&\chi_{\xi>\kappa_*}\mathcal{K}_{cc}(\chi_{\cdot<\kappa_*}\mathcal{D}_{\sigma}\triangle x^{(j-1)}_{{\mathcal{R}}})(\frac{\lambda^2(\tau)}{\lambda^2(\sigma)}\xi)\\& =  \chi_{\xi>\kappa_*}\int_0^\infty\frac{F(\frac{\lambda^2(\tau)}{\lambda^2(\sigma)}\xi, \eta)\rho(\eta)}{\frac{\lambda^2(\tau)}{\lambda^2(\sigma)}\xi - \eta}\chi_{\eta<\kappa_*}\mathcal{D}_{\sigma}\triangle x^{(j-1)}_{{\mathcal{R}}}\,d\eta.
\end{split}\end{equation}
Then the idea is that since $\tau\geq\sigma$, we also have $\frac{\lambda^2(\tau)}{\lambda^2(\sigma)}\xi>\eta$, and hence {\it{combining the oscillatory phases}} (one of which is inherent in the definition of $\triangle x^{(j-1)}_{{\mathcal{R}}}$) we can gain by performing an integration by parts with respect to $\sigma$. Specifically, write 
\[
\chi_{\eta<\kappa_*}\mathcal{D}_{\sigma}\triangle x^{(j-1)}_{{\mathcal{R}}} = \chi_{\eta<\kappa_*}\mathcal{D}_{\sigma}\big[\triangle_{>\sigma} x^{(j-1)}_{{\mathcal{R}}} + S(\sigma)\big(\tilde{\triangle} x^{(j-1)}_{{\mathcal{R}},0},\,\tilde{\triangle} x^{(j-1)}_{{\mathcal{R}},1}\big)\big],
\]
where we recall that 
\begin{align*}
&\triangle_{>\sigma} x^{(j-1)}_{{\mathcal{R}}}(\sigma, \eta)\\& = \int_{\sigma}^\infty\frac{\lambda^{\frac{3}{2}}(\sigma)}{\lambda^{\frac{3}{2}}(\sigma_1)}\frac{\rho^{\frac{1}{2}}(\frac{\lambda^2(\sigma)}{\lambda^2(\sigma_1)}\eta)}{\rho^{\frac{1}{2}}(\eta)}\frac{\sin[\lambda(\sigma)\eta^{\frac{1}{2}}\int_{\sigma}^{\sigma_1}\lambda^{-1}(u)\,du]}{\eta^{\frac12}}\\&\hspace{5cm}\cdot \mathcal{R}_1(\sigma_1, \triangle x^{(j-2)})(\frac{\lambda^2(\sigma)}{\lambda^2(\sigma_1)}\eta)\,d\sigma_1
\end{align*}
Inserting this in \eqref{eq:Kccintermediate} for $\triangle x^{(j-1)}_{{\mathcal{R}}}$ and performing a change of variable $\eta = \frac{\lambda^2(\tau)}{\lambda^2(\sigma)}\tilde{\eta}$, and inserting this in turn into the expression for $\tilde{\Phi}(\tau, \xi)$ above, leads to the $\sigma$-oscillatory factor 
\[
\cos[\lambda(\tau)\xi^{\frac{1}{2}}\int_{\tau}^\sigma\lambda^{-1}(u)\,du]\cdot \sin[\lambda(\tau)\tilde{\eta}^{\frac12}\int_{\sigma}^{\sigma_1}\lambda^{-1}(u)\,du]
\]
and since $\xi>\eta$ we have a fortiori $\xi>\tilde{\eta}$ on account of $\sigma\leq \tau$. Then by performing integration by parts with respect to $\sigma$ (this works if $\sigma<\tau - \tau^{\frac12}$, say, to force a lower bound on $\xi-\tilde{\eta}$, while if $\sigma\in [\tau - \tau^{\frac12}, \tau]$, one gains directly in $\tau$) and repeating the arguments in Proposition~\ref{prop:lineariterstep}, one infers that this contribution to $\tilde{\Phi}$ is of size $\tau_0^{-\gamma}O_{\kappa_*}(\triangle A_{j-2})$. Replacing $\triangle x^{(j-1)}_{{\mathcal{R}}}$ by $S(\sigma)\big(\tilde{\triangle} x^{(j-1)}_{{\mathcal{R}},0},\,\tilde{\triangle} x^{(j-1)}_{{\mathcal{R}},1}\big)$ is handled similarly, using the algebra in {\it{Step 1}} of the proof of Proposition~\ref{prop:firstiteratelincont}. 
\\

{\it{(2): Reduction of $\chi_{\cdot>\kappa_*}\mathcal{D}_{\sigma}\triangle x^{(j-1)}_{{\mathcal{R}}}$ to $\chi_{\xi>\kappa_*}\Phi(\mathcal{D}_{\tau}\triangle x^{(j-2)})$.}} Here one first disposes of the contribution of 
\[
\mathcal{D}_{\sigma}S(\sigma)\big(\triangle\tilde{\tilde{x}}^{(j-1)}_0, \triangle\tilde{\tilde{x}}^{(j-1)}_1\big),
\]
using that $\big\|\big(\triangle\tilde{\tilde{x}}^{(j-1)}_0, \triangle\tilde{\tilde{x}}^{(j-1)}_1\big)\big\|_{S}\lesssim \tau_0^{-(1-)}\triangle A_{j-2}$ as well as the fact that all frequencies are now essentially large(which means we no longer need the free term $S(\sigma)\big(\triangle\tilde{\tilde{x}}^{(j-1)}_0, \triangle\tilde{\tilde{x}}^{(j-1)}_1\big)$ to cancel out part of the Duhamel term), and then one easily disposes of the term $\beta_{\nu}^2(\sigma)\mathcal{K}_{cc}(\triangle x^{(j-2)})$ inside $\mathcal{R}_1(\sigma, \triangle x^{(j-2)})(\frac{\lambda^2(\tau)}{\lambda^2(\sigma)}\xi)$, exploiting again the fact that all frequencies are now $\gtrsim 1$, see {\it{(1a)}} in the proof of Proposition~\ref{prop:lineariterstep}, where a gain $\tau_0^{-1}$ is implicit. 
\\

This completes the proof of Lemma~\ref{lem:Phireduction}. 
\end{proof}

\section{Convergence of the iterative scheme}

In this section, we finally show that the sequence of iterates defined by \eqref{eq:increment1}, \eqref{eq:incrementj} {\it{and with the modifications of the initial data due to the $\triangle\tilde{\tilde{x}}^{(j)}_{0,1}$}} converges, which will then imply the proof of Theorem~\ref{thm:MainTechnical}. The crux of the matter will be played on the one hand by Lemma~\ref{lem:Phireduction}, which in some sense achieves a {\it{diagonalization}} of the high-frequency part of the once repeated iteration step, together with a crucial proposition of somewhat combinatorial type below which shows that {\it{re-iterating the operator $\chi_{\xi>\kappa}\Phi$ many times gains smallness}}. This last fact has close analogues in \cite{KST1} as well as \cite{DoMHKS}. 
Recalling the quantities $\triangle A_j$ from Definition~\ref{defn:sizeofcorrection}, the main result of this section is the following 
\begin{prop}\label{prop:convergence}  Given $\kappa_*>0$ and $\epsilon>0$ sufficiently small, we have for any $k\geq 1$ sufficiently large and $j\geq 2k+2$ the bound
\begin{align*}
\triangle A_j\leq \kappa_*^{\gamma} [\triangle A_{j-1}  + \triangle A_{j-2}] + \kappa_*^{-C\delta_0}\epsilon^{\gamma_1 k}e^{\epsilon^{-2}}\cdot\triangle A_{j-2k} + C_{\kappa_*,k}\tau_0^{-\gamma}[\sum_{l=j-2k-1}^j\triangle A_{l}]
\end{align*}
provided $\sum_{l=1}^{j}\triangle A_l <\delta_1$, where $\delta_1 = \delta_1(\tau_0)>0$ is sufficiently small. Here $\gamma>0$ is an absolute constant. 
\end{prop}

As a consequence of the proposition, we get convergence in the following sense: 
\begin{cor}\label{cor:convergence} We have 
\begin{align*}
\sum_{j=1}^\infty\triangle A_j\lesssim \big\|(x_0, x_1)\big\|_{\tilde{S}} + \big|x_{0d}\big|,
\end{align*}
provided that $\big\|(x_0, x_1)\big\|_{\tilde{S}} + \big|x_{0d}\big|$ is sufficiently small. 
\end{cor}
\begin{proof} To begin with, we note that the preceding proposition implies that given $\delta>0$, there are $j_*(\delta)$ and $\tau_0$ large enough such that we have for any $j\geq j_*(\delta)+2$
\[
\triangle A_j\leq\delta[\triangle A_{j-1}  + \triangle A_{j-2} + \triangle A_{j-j_*(\delta)}] + \frac{\delta}{ j_*(\delta)}\sum_{l=j-j_*(\delta) -1}^j\triangle A_{l}
\]
In fact, first one picks $\kappa_*$ small enough such that $\kappa_*^{\gamma}<\delta$, and then letting $\epsilon<1$ small enough such that the conclusion of the preceding proposition holds we pick $k$ large enough such that 
\[
 \kappa_*^{-C\delta_0}\epsilon^{\gamma_1 k}e^{\epsilon^{-2}}<\delta
 \]
Then set $j_*(\delta) = 2k$ and finally pick $\tau_0$ large enough such that $C_{\kappa_*,k}\tau_0^{-\gamma}<\frac{\delta}{j_*(\delta)}$. 
But then summing over $j\geq j_*(\delta)$ we get 
\[
\sum_{l=1}^\infty \triangle A_l\leq \sum_{l=1}^{j_*(\delta)} \triangle A_l + 4\delta\sum_{l=1}^\infty \triangle A_l,
\]
which then gives 
\[
\sum_{l=1}^\infty \triangle A_l\leq \frac{1}{1-4\delta} \sum_{l=1}^{j_*(\delta)} \triangle A_l\lesssim_{\delta}\big\|(x_0, x_1)\big\|_{\tilde{S}} + \big|x_{0d}\big|, 
\]
and the latter can be made $\ll \delta_1(\tau_0)$ by picking the data small enough, hence establishing the convergence and necessary smallness condition behind the proposition. 
\end{proof}

We can in fact easily sharpen the preceding corollary, in that we get exponential decay of the $\triangle A_j$ with respect to $j$, and we can also strengthen the bounds concerning the corrections $\triangle\tilde{\tilde{x}}^{(j)}_{\kappa}$, $\kappa = 0,1$, in light of the iterative step. We record this in the following 
\begin{cor}\label{for:corrections decay} For any $\delta>0$, and $\tau_0 = \tau_0(\delta)$ large enough, we have 
\[
\triangle A_j\lesssim_{\delta}\delta^j(\big\|(x_0, x_1\big\|_{\tilde{S}} + \big|x_{0d}\big|).
\]
Similarly, we have 
\[
\big\|(\triangle\tilde{\tilde{x}}^{(j)}_0, \triangle\tilde{\tilde{x}}^{(j)}_1)\big\|_{\tilde{S}}\lesssim_{\delta}\tau_0^{-(1-)}\delta^j(\big\|(x_0, x_1\big\|_{\tilde{S}} + \big|x_{0d}\big|).
\]

\end{cor}

It remains to prove Proposition~\ref{prop:convergence}. As already mentioned, a key role will be played by the fact that {\it{re-iterating}} the operator $\chi_{\cdot>\kappa_*}\Phi$ many times leads to an improved estimate. The following is a precise version of what we need. Its proof will be relegated to the end of this section. 
\begin{prop}\label{prop:mainestimate} For any $\epsilon>0$ sufficiently small, and a suitable $\gamma_1>0$, there is for any $k\geq 1$ a splitting 
\[
(\chi_{\xi>\kappa_*}\Phi)^k(f) = \Phi_1^{(k)}(f) + \Phi_2^{(k)}(f),
\]
such that we have the bounds 
\[
\big(\sum_{\substack{N\gtrsim \tau_0\\ N\,\text{dyadic}}}[\sup_{\tau\sim N}(\frac{\tau}{\tau_0})^{\kappa}\big\|\Phi_1^{(k)}(f)\big\|_{S_3}]^2\big)^{\frac{1}{2}}\leq \epsilon^{\gamma_1 k}e^{\epsilon^{-2}}\big(\sum_{\substack{N\gtrsim \tau_0\\ N\,\text{dyadic}}}[\sup_{\tau\sim N}(\frac{\tau}{\tau_0})^{\kappa}\big\|f(\tau, \cdot)\|_{S_2}]^2\big)^{\frac{1}{2}}
\]
\[
\big(\sum_{\substack{N\gtrsim \tau_0\\ N\,\text{dyadic}}}[\sup_{\tau\sim N}(\frac{\tau}{\tau_0})^{\kappa}\big\|\Phi_2^{(k)}(f)\big\|_{S_2}]^2\big)^{\frac{1}{2}}\lesssim_{k,\kappa_*} \tau_0^{-\gamma_1}\big(\sum_{\substack{N\gtrsim \tau_0\\ N\,\text{dyadic}}}[\sup_{\tau\sim N}(\frac{\tau}{\tau_0})^{\kappa}\big\|f(\tau, \cdot)\|_{S_2}]^2\big)^{\frac{1}{2}}
\]
Here we have introduced the norm 
\[
\big\|f\big\|_{S_3}: = \|\min\{\xi^{0+}, 1\}\langle\xi\rangle^{\frac{1}{2}+}f(\xi)\|_{L^2_{d\xi}}.
\]
\end{prop}

Back to proving  Proposition~\ref{prop:convergence}, we shall do so by obtaining a smallness gain for the various constituents of $\triangle A_j$, according to Definition~\ref{defn:sizeofcorrection}. We shall first consider the time differentiated expression $\mathcal{D}_{\tau}\triangle x^{(j)}$, and show a smallness gain for it after many iterations: 

\begin{lem}\label{lem:Dtauxconvergence} We have 
\begin{align*}
&\big(\sum_{\substack{N\gtrsim \tau_0\\ N\,\text{dyadic}}}[\sup_{\tau\sim N}(\frac{\tau}{\tau_0})^{\kappa}\big[\big\|\chi_{\xi>1}\mathcal{D}_{\tau}\widetilde{\triangle x^{(j)}}\|_{S_2} + \|\chi_{\xi<1}\mathcal{D}_{\tau}\triangle_{>\tau} x^{(j)}\|_{S_2}]^2\big)^{\frac{1}{2}}\\
&\leq \kappa_*^{\gamma} [\triangle A_{j-1}  + \triangle A_{j-2}] + \kappa_*^{-C\delta_0}\epsilon^{\gamma_1 k}e^{\epsilon^{-2}}\cdot\triangle A_{j-2k} + C_{\kappa_*,k}\tau_0^{-\gamma}[\sum_{l=j-2k-1}^j\triangle A_{l}]
\end{align*}
provided $\sum_{l=1}^{j}\triangle A_l <\delta_1$, where $\delta_1 = \delta_1(\tau_0)>0$ is sufficiently small.
\end{lem}
\begin{proof}(Lemma~\ref{lem:Dtauxconvergence})
\\

{\bf{Step 1}}: Using Lemma~\ref{lem:smallnessgainoutputfreqsmall} and choosing $\kappa_* = \kappa_*(\delta)$ small enough, we get 
\[
\big(\sum_{\substack{N\gtrsim \tau_0\\ N\,\text{dyadic}}}[\sup_{\tau\sim N}(\frac{\tau}{\tau_0})^{\kappa} \|\chi_{\xi<\kappa_*}\mathcal{D}_{\tau}\triangle_{>\tau} x^{(j)}\|_{S_2}]^2\big)^{\frac{1}{2}}\ll \kappa_*^\gamma \triangle A_{j-1}. 
\]
{\bf{Step 2}}: From Lemma~\ref{lem:Phireduction} and its proof, we infer for any $k\geq 1$
\begin{align*}
&\big(\sum_{\substack{N\gtrsim \tau_0\\ N\,\text{dyadic}}}[\sup_{\tau\sim N}(\frac{\tau}{\tau_0})^{\kappa}\big\|\chi_{\xi>\kappa_*}\mathcal{D}_{\tau}\widetilde{\triangle x^{(j)}}\big\|_{S_2}]^2\big)^{\frac{1}{2}}\\&\lesssim \big(\sum_{\substack{N\gtrsim \tau_0\\ N\,\text{dyadic}}}[\sup_{\tau\sim N}(\frac{\tau}{\tau_0})^{\kappa}\big\|(\chi_{\xi>\kappa_*}\Phi)^{2k}(\mathcal{D}_{\tau}\triangle x^{(j-2k)})\big\|_{S_2}]^2\big)^{\frac{1}{2}} + C_{\kappa_*,k}\tau_0^{-\gamma}[\sum_{l=j-2k}^j\triangle A_{l}]
\end{align*}
Invoking Proposition~\ref{prop:mainestimate}, we can estimate this in turn by 
\begin{align*}
&\big(\sum_{\substack{N\gtrsim \tau_0\\ N\,\text{dyadic}}}[\sup_{\tau\sim N}(\frac{\tau}{\tau_0})^{\kappa}\big\|(\chi_{\xi>\kappa_*}\Phi)^{2k}(\mathcal{D}_{\tau}\triangle x^{(j-2k)})\big\|_{S_2}]^2\big)^{\frac{1}{2}} + C_{\kappa_*,k}\tau_0^{-\gamma}[\sum_{l=j-2k}^j\triangle A_{l}]\\
&\lesssim \kappa_*^{-C\delta_0}\epsilon^{\gamma_1 k}e^{\epsilon^{-2}}\cdot\triangle A_{j-2k} + C_{\kappa_*,k}\tau_0^{-\gamma}[\sum_{l=j-2k-1}^j\triangle A_{l}]
\end{align*}
provided $0<\epsilon\leq \epsilon_*$. Then we have almost completed the proof of the lemma, exempt that we don't control $\chi_{\kappa_*<\xi<1}\mathcal{D}_{\tau}\triangle_{>\tau} x^{(j)}$ yet. For this we have to pass from $\chi_{\kappa_*<\xi<1}\mathcal{D}_{\tau}\widetilde{\triangle x^{(j)}}$ to 
$\chi_{\kappa_*<\xi<1}\mathcal{D}_{\tau}\triangle_{>\tau} x^{(j)}$. 
\\

{\bf{Step 3}}: Write
\[
\mathcal{D}_{\tau}\triangle_{>\tau}x^{(j)} = \mathcal{D}_{\tau}\widetilde{\triangle x^{(j)}} + \mathcal{D}_{\tau}S(\tau)(\triangle\tilde{\tilde{x}}^{(j)}_0, \triangle\tilde{\tilde{x}}^{(j)}_1) - \mathcal{D}_{\tau}S(\tau)(\triangle\tilde{x}^{(j)}_0, \triangle\tilde{x}^{(j)}_1)
\]
Then it follows that
\begin{align*}
&\big(\sum_{\substack{N\gtrsim \tau_0\\ N\,\text{dyadic}}}[\sup_{\tau\sim N}(\frac{\tau}{\tau_0})^{\kappa}\big[\big\|\chi_{\kappa_*<\xi<1}\mathcal{D}_{\tau}\triangle_{>\tau} x^{(j)}\|_{S_2}]^2\big)^{\frac{1}{2}}\\&\leq \big(\sum_{\substack{N\gtrsim \tau_0\\ N\,\text{dyadic}}}[\sup_{\tau\sim N}(\frac{\tau}{\tau_0})^{\kappa}\big[\big\|\chi_{\kappa_*<\xi<1} \mathcal{D}_{\tau}\widetilde{\triangle x^{(j)}}\|_{S_2}]^2\big)^{\frac{1}{2}}\\
& +  \big(\sum_{\substack{N\gtrsim \tau_0\\ N\,\text{dyadic}}}[\sup_{\tau\sim N}(\frac{\tau}{\tau_0})^{\kappa}\big[\big\|\chi_{\kappa_*<\xi<1}\mathcal{D}_{\tau}S(\tau)(\triangle\tilde{\tilde{x}}^{(j)}_0, \triangle\tilde{\tilde{x}}^{(j)}_1) \|_{S_2}]^2\big)^{\frac{1}{2}}\\
& + \big(\sum_{\substack{N\gtrsim \tau_0\\ N\,\text{dyadic}}}[\sup_{\tau\sim N}(\frac{\tau}{\tau_0})^{\kappa}\big[\big\|\chi_{\kappa_*<\xi<1}\mathcal{D}_{\tau}S(\tau)(\triangle\tilde{x}^{(j)}_0, \triangle\tilde{x}^{(j)}_1)\|_{S_2}]^2\big)^{\frac{1}{2}}\\
\end{align*}
Here the first term on the right is bounded using {\it{Step 2}}, and the second term on the right is bounded by $\big|\log\kappa_*\big|\tau_0^{-(1-)}\triangle A_{j-1}$ due to Proposition~\ref{prop:lineariterstep}, Proposition~\ref{prop:nonlineariterstep}. Note that free evolutions $\mathcal{D}_{\tau}S(\tau)(\ldots)$ with data in $S$ cannot be placed into the square-sum type space in general, but restricting output frequencies to a finite dyadic range enables such an estimate, which explains the factor $\big|\log\kappa_*\big|$.  
Thus to complete the proof of Lemma~\ref{lem:Dtauxconvergence}, it remains to bound the last term on the right, which will follow from 
\begin{equation}\label{eq:triangletildeximprov}\begin{split}
&\big\|\big(\triangle\tilde{x}^{(j)}_0, \triangle\tilde{x}^{(j)}_1\big)\big\|_{S(\xi>\kappa_*)}\\&\lesssim \kappa_*^{\gamma}[\triangle A_{j-1} + \triangle A_{j-2}] + \kappa_*^{-C\delta_0}\epsilon^{\gamma_1 k}e^{\epsilon^{-2}}\cdot\triangle A_{j-2k} + C_{\kappa_*,k}\tau_0^{-\gamma}[\sum_{l=j-2k-1}^j\triangle A_{l}]
\end{split}\end{equation}
This will be accomplished in the next lemma
\end{proof}

\begin{lem}\label{lem:triangletildeximprov} The bound \eqref{eq:triangletildeximprov} holds. 
\end{lem}
\begin{proof}(Lemma~\ref{lem:triangletildeximprov}) We commence by deducing the bound for $\big\|\triangle\tilde{x}^{(j)}_1\big\|_{S_2}$. Recall from the proofs of Proposition~\ref{prop:lineariterstep}, Proposition~\ref{prop:nonlineariterstep} that we have 
\begin{align*}
\triangle\tilde{x}^{(j)}_1 &= \int_{\tau_0}^\infty\frac{\lambda^{\frac32}(\tau_0)}{\lambda^{\frac32}(\sigma)}\frac{\rho^{\frac12}(\frac{\lambda^2(\tau_0)}{\lambda^2(\sigma)}\xi)}{\rho^{\frac12}(\xi)}\cos[\lambda(\tau_0)\xi^{\frac12}\int_{\tau_0}^{\sigma}\lambda^{-1}(u)\,du]\\&\hspace{3cm}\cdot\big[\mathcal{R}(\sigma, \underline{\triangle x}^{(j-1)}) + \triangle f^{(j-1)}(\sigma, \cdot)\big](\frac{\lambda^2(\tau_0)}{\lambda^2(\sigma)}\xi)\,d\sigma\\
& + \triangle\tilde{\tilde{x}}^{(j)}_1. 
\end{align*}
Moreover, we immediately have $\big\| \triangle\tilde{\tilde{x}}^{(j)}_1\big\|_{S_2}\lesssim \tau_0^{-(1-)}\triangle A_{j-1}$. As for the integral expression on the right, relying on Proposition~\ref{prop:lineariterstepdisc}, Proposition~\ref{prop:nonlineariterstep}, we see that the contribution of the discrete spectral part of $\underline{\triangle x}^{(j-1)}$ to the left hand side is of size $\lesssim \tau_0^{-(1-)}\triangle A_{j-1}$, and in particular we may and shall replace $\mathcal{R}(\sigma, \underline{\triangle x}^{(j-1)})$ by $\mathcal{R}_1(\sigma, \triangle x^{(j-1)})$. Then further decompose 
\[
\mathcal{R}_1(\sigma, \triangle x^{(j-1)}) = \mathcal{R}_1(\sigma, \chi_{\eta>\kappa_*}\triangle x^{(j-1)}) + \mathcal{R}_1(\sigma, \chi_{\eta<\kappa_*}\triangle x^{(j-1)})
\]
For the first term on the right, we use the bound established in the preceding lemma(without any information on $(\triangle\tilde{x}^{(j)}_0, \triangle\tilde{x}^{(j)}_1)$) and using the argument in Proposition~\ref{prop:lineariterstep} we deduce the desired bound for this contribution to $\triangle\tilde{x}^{(j)}_1$. On the other hand, for the contribution of $\mathcal{R}_1(\sigma, \chi_{\eta<\kappa_*}\triangle x^{(j-1)})$ we see that Lemma~\ref{lem:smallnessgaininputfreqsmall} furnishes the required bound. Furthermore, for the contribution of $\triangle f^{(j-1)}(\sigma, \cdot)$ to $\triangle\tilde{x}^{(j)}_1$, we obtain the improved estimate using Lemma~\ref{lem:smallnessgainnl}. This completes the bound for $\big\|\triangle\tilde{x}^{(j)}_1\big\|_{S_2}$.
\\

Next, we proceed by analogy for the term 
\begin{align*}
\triangle\tilde{x}^{(j)}_0 &= \int_{\tau_0}^\infty\frac{\lambda^{\frac32}(\tau_0)}{\lambda^{\frac32}(\sigma)}\frac{\rho^{\frac12}(\frac{\lambda^2(\tau_0)}{\lambda^2(\sigma)}\xi)}{\rho^{\frac12}(\xi)}\frac{\sin[\lambda(\tau_0)\xi^{\frac12}\int_{\tau_0}^{\sigma}\lambda^{-1}(u)\,du]}{\xi^{\frac12}}\\&\hspace{3cm}\cdot\big[\mathcal{R}(\sigma, \underline{\triangle x}^{(j-1)}) + \triangle f^{(j-1)}(\sigma, \cdot)\big](\frac{\lambda^2(\tau_0)}{\lambda^2(\sigma)}\xi)\,d\sigma\\
& + \triangle\tilde{\tilde{x}}^{(j)}_0
\end{align*}
Again we can dispose immediately of the contribution of the discrete spectral part of $\underline{\triangle x}^{(j-1)}$, as well as of the term $\triangle\tilde{\tilde{x}}^{(j)}_0$. Moreover, using the preceding lemma we can replace $\mathcal{R}(\sigma, \underline{\triangle x}^{(j-1)})$ by 
\[
\mathcal{R}_1(\sigma, \chi_{\eta<\kappa_*}\triangle x^{(j-1)}).
\]
Consider then the integral expression 
\begin{align*}
&\int_{\tau_0}^\infty\frac{\lambda^{\frac32}(\tau_0)}{\lambda^{\frac32}(\sigma)}\frac{\rho^{\frac12}(\frac{\lambda^2(\tau_0)}{\lambda^2(\sigma)}\xi)}{\rho^{\frac12}(\xi)}\frac{\sin[\lambda(\tau_0)\xi^{\frac12}\int_{\tau_0}^{\sigma}\lambda^{-1}(u)\,du]}{\xi^{\frac12}}\mathcal{R}_1(\sigma,  \chi_{\eta<\kappa_*}\triangle x^{(j-1)})(\frac{\lambda^2(\tau_0)}{\lambda^2(\sigma)}\xi)\,d\sigma\\
\end{align*}
Here since $\chi_{\eta<\kappa_*}\triangle x^{(j-1)}$ is in the low frequency regime (since $\kappa_*\ll 1$), we can write 
\begin{align*}
\chi_{\eta<\kappa_*}\triangle x^{(j-1)}& = \chi_{\eta<\kappa_*}\triangle_{>\sigma} x^{(j-1)}  +  \chi_{\eta<\kappa_*}S(\sigma)\big(\triangle\tilde{x}^{(j-1)}_0, \triangle\tilde{x}^{(j-1)}_1\big)\\
&= \chi_{\eta<\kappa_*}\triangle_{>\sigma} x^{(j-1)}  +  \chi_{\eta<\kappa_*}S(\sigma)\big(\triangle\tilde{x}^{(j-1)}_0, 0\big)\\
& +  \chi_{\eta<\kappa_*}S(\sigma)\big(0, \triangle\tilde{x}^{(j-1)}_1\big). 
\end{align*}
The contribution of the last term on the right to the preceding integral expression is of the desired type due to our improved bound for $\triangle\tilde{x}^{(j-1)}_1$ from the first part of the proof. 
\\
Next, expanding out 
\begin{align*}
&\mathcal{R}_1(\sigma,  \chi_{\eta<\kappa_*}\triangle_{>\sigma} x^{(j-1)})\\& = \beta_{\nu}^2(\sigma)\mathcal{K}_{cc}( \chi_{\eta<\kappa_*}\triangle_{>\sigma} x^{(j-1)}) + \beta_{\nu}(\sigma)\mathcal{K}_{cc}( \chi_{\eta<\kappa_*}\mathcal{D}_{\sigma}\triangle_{>\sigma} x^{(j-1)}),
\end{align*}
we get the desired bound for the contribution of the last term on the right due to {\it{Step 1}} in the proof of Lemma~\ref{lem:Dtauxconvergence}. On the other hand, arguing as in the proof of Lemma~\ref{lem:hereneednorm}, the contribution of 
\[
\beta_{\nu}^2(\sigma)\mathcal{K}_{cc}( \chi_{\eta<\kappa_*}\triangle_{>\sigma} x^{(j-1)})
\]
to the preceding integral expression is bounded by $\lesssim \tau_0^{-1}\triangle A_{j-1}$. Finally, consider
\begin{align*}
&\int_{\tau_0}^\infty\frac{\lambda^{\frac32}(\tau_0)}{\lambda^{\frac32}(\sigma)}\frac{\rho^{\frac12}(\frac{\lambda^2(\tau_0)}{\lambda^2(\sigma)}\xi)}{\rho^{\frac12}(\xi)}\frac{\sin[\lambda(\tau_0)\xi^{\frac12}\int_{\tau_0}^{\sigma}\lambda^{-1}(u)\,du]}{\xi^{\frac12}}\\&\hspace{4cm}\cdot\mathcal{R}_1(\sigma,  \chi_{\eta<\kappa_*}S(\sigma)\big(\triangle\tilde{x}^{(j-1)}_0, 0\big))(\frac{\lambda^2(\tau_0)}{\lambda^2(\sigma)}\xi)\,d\sigma\\
\end{align*}
By distinguishing between $\xi^{\frac12}\frac{\lambda(\tau_0)}{\lambda(\sigma)}\gtrless \sigma^{-\delta}$ and either exploiting a smallness gain coming from the $L^2_{d\xi}$-norm of the output, or else performing an integration by parts, while always keeping in mind that $\xi>\kappa_*$ for the output, the preceding term leads to a contribution bounded by $\lesssim \tau_0^{-\gamma}\triangle A_{j-1}$.

This completes discussion of the contribution of $\mathcal{R}_1(\sigma, \chi_{\eta<\kappa_*}\triangle x^{(j-1)})$ to $\triangle\tilde{x}^{(j)}_0 $. 
\\

We have now reduced the proof of the improved bound for $\triangle\tilde{x}^{(j)}_0 $ to controlling the expression 
\begin{align*}
\int_{\tau_0}^\infty\frac{\lambda^{\frac32}(\tau_0)}{\lambda^{\frac32}(\sigma)}\frac{\rho^{\frac12}(\frac{\lambda^2(\tau_0)}{\lambda^2(\sigma)}\xi)}{\rho^{\frac12}(\xi)}\frac{\sin[\lambda(\tau_0)\xi^{\frac12}\int_{\tau_0}^{\sigma}\lambda^{-1}(u)\,du]}{\xi^{\frac12}}\triangle f^{(j-1)}(\sigma, \frac{\lambda^2(\tau_0)}{\lambda^2(\sigma)}\xi)\,d\sigma.\\
\end{align*}
Here the desired gain comes from Lemma~\ref{lem:smallnessgainnl} in light of the fact that we restrict to $\xi>\kappa_*$, and we are done.
\end{proof}

In light of the preceding two lemmas, conclusion of the proof of Proposition~\ref{prop:convergence} will require improving bounds for 
\[
\big\|(\chi_{\xi<\kappa_*}\triangle \tilde{x}_0^{(j)}, \chi_{\xi<\kappa_*}\triangle \tilde{x}_1^{(j)})\big\|_{\tilde{S}}, 
\]
the undifferentiated quantities 
\[
\chi_{\xi>1}\widetilde{\triangle x^{(j)}}, \chi_{\xi<1}\triangle_{>\tau}x^{(j)}, 
\]
as well as the discrete part $x_d(\tau), \partial_\tau x_d(\tau)$. 
The improved bound for 
\[
\big\|(\chi_{\xi<\kappa_*}\triangle \tilde{x}_0^{(j)}, \chi_{\xi<\kappa_*}\triangle \tilde{x}_1^{(j)})\big\|_{\tilde{S}}, 
\]
follows directly from Lemma~\ref{lem:smallnessgainoutputfreqsmall} , Lemma~\ref{lem:smallnessgainnl}. For the remaining quantities, 
we accomplish this in the following two lemmas: 
\begin{lem}\label{lem:triangleximprov}
We have the bound 
\begin{align*}
&\sup_{\tau\geq \tau_0}(\frac{\tau}{\tau_0})^{-\kappa}\big\|\chi_{\xi>1}\widetilde{\triangle x^{(j)}}(\tau, \cdot)\big\|_{S_1} + \sup_{\tau\geq \tau_0}(\frac{\tau}{\tau_0})^{-\kappa}\big\|\chi_{\xi<1}\triangle_{>\tau}x^{(j)}(\tau, \cdot)\big\|_{S_1}\\&\lesssim
\kappa_*^{\gamma}[\triangle A_{j-1} + \triangle A_{j-2}] + \kappa_*^{-C\delta_0}\epsilon^{\gamma_1 k}e^{\epsilon^{-2}}\cdot\triangle A_{j-2k} + C_{\kappa_*,k}\tau_0^{-\gamma}[\sum_{l=j-2k-1}^j\triangle A_{l}]
\end{align*}
\end{lem}
\begin{proof} This is now straightforward given our preceding work: write 
\begin{align*}
&\triangle_{>\tau}\triangle x^{(j)}(\tau, \xi)\\& = \int_{\tau}^\infty\frac{\lambda^{\frac32}(\tau)}{\lambda^{\frac32}(\sigma)}\frac{\rho^{\frac12}(\frac{\lambda^2(\tau)}{\lambda^2(\sigma)}\xi)}{\rho^{\frac12}(\xi)}\frac{\sin[\lambda(\tau)\xi^{\frac12}\int_{\tau}^{\sigma}\lambda^{-1}(u)\,du]}{\xi^{\frac12}}\\&\hspace{3cm}\cdot\big[\mathcal{R}(\sigma, \underline{\triangle x}^{(j-1)}) + \triangle f^{(j-1)}(\sigma, \cdot)\big](\frac{\lambda^2(\tau)}{\lambda^2(\sigma)}\xi)\,d\sigma.
\end{align*}
Then as in the proof of the preceding lemma, one may immediately reduce 
\[
\mathcal{R}(\sigma, \underline{\triangle x}^{(j-1)})
\]
to the contribution of the continuous spectral part 
\[
\mathcal{R}_1(\sigma, \triangle x^{(j-1)}) = \beta_{\nu}^2(\sigma)\mathcal{K}_{cc}(\triangle x^{(j-1)}) + \beta_{\nu}(\sigma)\mathcal{K}_{cc}(\mathcal{D}_{\sigma}\triangle x^{(j-1)})
\]
Then inspection of the proof of Proposition~\ref{prop:lineariterstep} reveals that for the contribution of the first term on the right we gain a power of $\tau_0^{-1}$ except for the contribution of the low-frequency free part of $\triangle x^{(j-1)}$ of the specific form 
\[
\chi_{\eta<1}S(\sigma)\big(\triangle\tilde{x}^{(j-1)}_0, \triangle\tilde{x}^{(j-1)}_1\big).
\]
However, here the preceding lemma (with $j$ replaced by $j-1$) together with the above the remark on 
\[
\big\|(\chi_{\xi<\kappa_*}\triangle \tilde{x}_0^{(j)}, \chi_{\xi<\kappa_*}\triangle \tilde{x}_1^{(j)})\big\|_{\tilde{S}}
\]
in conjunction with the proof of Proposition~\ref{prop:firstiteratelincont} implies the desired improved bound. As for the contribution of 
\[
 \beta_{\nu}(\sigma)\mathcal{K}_{cc}(\mathcal{D}_{\sigma}\triangle x^{(j-1)}),
 \]
 the improved bound follows from the last lemma but one in conjunction with the proof of Proposition~\ref{prop:lineariterstep}. 
 \\
 The contribution of the source term $\triangle f^{(j-1)}(\sigma, \cdot)$ to $\triangle_{>\tau}\triangle x^{(j)}(\tau, \xi)$ is handled analogously: one splits the input 
 \[
 \triangle x^{(j-1)}
 \]
 in the low frequency regime $\xi<1$ into $\triangle_{>\tau} x^{(j-1)}$  as well as the 'free part' $S(\tau)(\triangle\tilde{x}_0^{(j-1)}, \triangle\tilde{x}_1^{(j-1)})$. The latter leads to an improvement, as we have seen (use Proposition~\ref{prop:nonlineariterstep}). Also, we infer that the contribution of $\triangle_{>\tau} x^{(j-1)}$ leads to a gain in $\tau_0^{-1}$, in light of Proposition~\ref{prop:nonlinbounds}. 
\\
Obtaining the improved bound for $\chi_{\xi>1}\widetilde{\triangle x^{(j)}}$ is similar. 
\end{proof}

To complete the proof of Proposition~\ref{prop:convergence}, we finally have 
\begin{lem}\label{lem:trianglexdimprov}
We have the bound 
\begin{align*}
&\sup_{\tau\geq\tau_0}\tau^{(1-)}[\big|\triangle x^{(j)}_d\big| + \big|\partial_{\tau}\triangle x^{(j)}_d\big|]\\
&\lesssim \kappa_*^{\gamma}[\triangle A_{j-1} + \triangle A_{j-2}] + \kappa_*^{-C\delta_0}\epsilon^{\gamma_1 k}e^{\epsilon^{-2}}\cdot\triangle A_{j-2k} + C_{\kappa_*,k}\tau_0^{-\gamma}[\sum_{l=j-2k-1}^j\triangle A_{l}].
\end{align*}
\end{lem}
In fact, one easily checks that the contribution of $\triangle x^{(j-1)}_d$ to the norm on the left is bounded by $\lesssim \tau_0^{-1}\triangle A_{j-1}$, and the contribution of the continuous part of $\triangle x^{(j-1)}$ is controlled in term of the preceding lemmas. 
\\

Since we also have $\big\|\triangle \tilde{\tilde{x}}^{(j)}_1\big\|_{S_2}\lesssim \tau_0^{-(1-)}\triangle A_{j-1}$ on account of Proposition~\ref{prop:lineariterstep}, Proposition~\ref{prop:nonlineariterstep}, we have now completed the proof of Proposition~\ref{prop:convergence} contingent upon Proposition~\ref{prop:mainestimate}. The proof of this latter will fill the remainder of this section.

\begin{proof}(Proof of Proposition~\ref{prop:mainestimate}) To simplify notation, we shall set here $\mathcal{K}_{cc} =:\mathcal{K}$. Also, we shall use the notation
\begin{align*}
\big(\mathcal{D}_{\tau}Uh\big)(\tau, \xi) = \int_{\tau_0}^{\tau}\frac{\lambda^{\frac32}(\tau)}{\lambda^{\frac32}(\sigma)}\frac{\rho^{\frac12}(\frac{\lambda^2(\tau)}{\lambda^2(\sigma)}\xi)}{\rho^{\frac12}(\xi)}\cos[\lambda(\tau)\xi^{\frac12}\int_{\tau}^{\sigma}\lambda^{-1}(u)\,du]h(\sigma, \frac{\lambda^2(\tau)}{\lambda^2(\sigma)}\xi)\,d\sigma. 
\end{align*}

We use a decomposition of the kernel $\mathcal{K}$ into a diagonal and an off-diagonal term. Specifically, for some $n\geq 1$ later to be chosen large, we split
\begin{align*}
\mathcal{K}(\xi, \eta) &= \chi_{|\frac{\xi}{\eta} - 1|<\frac{1}{n}}\mathcal{K}(\xi, \eta)  + \chi_{|\frac{\xi}{\eta} - 1|\geq\frac{1}{n}}\mathcal{K}(\xi, \eta)\\
& = \mathcal{K}^d(\xi, \eta) + \mathcal{K}^{nd}(\xi, \eta).
\end{align*}
Then we first show that for the {\it{two fold composition}} of $\Phi$, the contribution of $\mathcal{K}^{nd}$ is small: 
\begin{lem}\label{lem:off diagonal} There is $\gamma>0$ such that 
\begin{align*}
&\big(\sum_{\substack{N\gtrsim \tau_0\\ N\,\text{dyadic}}}[\sup_{\tau\sim N}(\frac{\tau}{\tau_0})^{\kappa}\big\|\chi_{\xi>\kappa_*}\mathcal{D}_{\tau}U\big(\beta_{\nu}(\tau)\mathcal{K}^{nd}\chi_{\xi>\kappa_*}\mathcal{D}_{\tau}U(f)\big)\big\|_{S_2}]^2\big)^{\frac{1}{2}}\\&\lesssim \tau_0^{-\gamma}\big(\sum_{\substack{N\gtrsim \tau_0\\ N\,\text{dyadic}}}[\sup_{\tau\sim N}(\frac{\tau}{\tau_0})^{\kappa}\|f(\tau, \cdot)\big\|_{S_2}]^2\big)^{\frac{1}{2}}.
\end{align*}
\end{lem}
\begin{proof}(lemma) Here we combine the vanishing properties of the kernel $\mathcal{K}$ with an integration by parts. We need to show 
\begin{align*}
&\big(\sum_{\substack{N\gtrsim \tau_0\\ N\,\text{dyadic}}}[\sup_{\tau\sim N}(\frac{\tau}{\tau_0})^{\kappa}\big\|\int_{\tau_0}^\tau\frac{\lambda^{\frac{3}{2}}(\tau)}{\lambda^{\frac{3}{2}}(\sigma)}\frac{\rho^{\frac{1}{2}}(\frac{\lambda^2(\tau)}{\lambda^2(\sigma)}\xi)}{\rho^{\frac{1}{2}}(\xi)}\\&\hspace{3cm}\cdot\chi_{\xi>\kappa_*}\cos\big[\lambda(\tau)\xi^{\frac{1}{2}}\int_{\tau}^{\sigma}\lambda^{-1}(u)\,du\big]h(\sigma, \frac{\lambda^2(\tau)}{\lambda^2(\sigma)}\xi)\,d\sigma\big\|_{S_2}]^2\big)^{\frac{1}{2}}\\
&\lesssim \tau_0^{-\gamma}\big(\sum_{\substack{N\gtrsim \tau_0\\ N\,\text{dyadic}}}[\sup_{\tau\sim N}(\frac{\tau}{\tau_0})^{\kappa}\|f(\tau, \cdot)\big\|_{S_2}]^2\big)^{\frac{1}{2}},
\end{align*}
where we have 
\[
h(\sigma, \cdot) = \beta_{\nu}(\sigma)\mathcal{K}^{nd}(\chi_{\cdot>\kappa_*}\mathcal{D}_{\sigma}Uf). 
\]
Thus we can write more explicitly 
\begin{align*}
&h(\sigma, \frac{\lambda^2(\tau)}{\lambda^2(\sigma)}\xi) =  \beta_{\nu}(\sigma)\int_0^\infty \mathcal{K}^{nd}(\frac{\lambda^2(\tau)}{\lambda^2(\sigma)}\xi, \eta)\chi_{\eta>\kappa_*}\int_{\tau_0}^\sigma\frac{\lambda^{\frac{3}{2}}(\sigma)}{\lambda^{\frac{3}{2}}(s)}\frac{\rho^{\frac{1}{2}}(\frac{\lambda^2(\sigma)}{\lambda^2(s)}\eta)}{\rho^{\frac{1}{2}}(\eta)}\\
&\hspace{4cm}\cos[\lambda(\sigma)\eta^{\frac{1}{2}}\int_{\sigma}^s\lambda^{-1}(u)\,du] f(s, \frac{\lambda^2(\sigma)}{\lambda^2(s)}\eta)\,dsd\eta
\end{align*}
Using the change of variables $\tilde{\eta} = \frac{\lambda^2(\sigma)}{\lambda^2(s)}\eta$, this becomes
\begin{align*}
&\beta_{\nu}(\sigma)\int_0^\infty\int_{\tau_0}^\sigma \mathcal{K}^{nd}(\frac{\lambda^2(\tau)}{\lambda^2(\sigma)}\xi, \frac{\lambda^2(s)}{\lambda^2(\sigma)}\tilde{\eta})\chi_{\frac{\lambda^2(s)}{\lambda^2(\sigma)}\tilde{\eta}>\kappa_*}\frac{\lambda^{\frac{1}{2}}(s)}{\lambda^{\frac{1}{2}}(\sigma)}\frac{\rho^{\frac{1}{2}}(\tilde{\eta})}{\rho^{\frac{1}{2}}(\frac{\lambda^2(s)}{\lambda^2(\sigma)}\tilde{\eta})}\\
&\hspace{4cm}\cos[\lambda(s)\tilde{\eta}^{\frac{1}{2}}\int_{\sigma}^s\lambda^{-1}(u)\,du]f(s, \tilde{\eta})\,ds d\tilde{\eta}
\end{align*}
By definition, we have 
\[
\frac{1}{n}<\big|\frac{\frac{\lambda^2(\tau)}{\lambda^2(\sigma)}\xi}{\frac{\lambda^2(s)}{\lambda^2(\sigma)}\tilde{\eta}} - 1\big|
\]
on the support of the integrand, which also implies that 
\[
\frac{1}{n}\lesssim \big|\frac{\frac{\lambda^2(s)}{\lambda^2(\sigma)}\tilde{\eta}}{\frac{\lambda^2(\tau)}{\lambda^2(\sigma)}\xi} - 1\big|
\]
We can then bound 
\[
\frac{1}{\big|\frac{\lambda(\tau)}{\lambda(\sigma)}\xi^{\frac{1}{2}} \mp \frac{\lambda(s)}{\lambda(\sigma)}\tilde{\eta}^{\frac{1}{2}}\big|}\lesssim n\min\{\frac{\frac{\lambda(\tau)}{\lambda(\sigma)}\xi^{\frac{1}{2}} +\frac{\lambda(s)}{\lambda(\sigma)}\tilde{\eta}^{\frac{1}{2}}}{\frac{\lambda^2(s)}{\lambda^2(\sigma)}\tilde{\eta}},\,\frac{\frac{\lambda(\tau)}{\lambda(\sigma)}\xi^{\frac{1}{2}} +\frac{\lambda(s)}{\lambda(\sigma)}\tilde{\eta}^{\frac{1}{2}}}{\frac{\lambda^2(\tau)}{\lambda^2(\sigma)}\xi}\}
\]
Combining the two phases inherent in the expression 
\begin{align*}
\int_{\tau_0}^\tau\frac{\lambda^{\frac{3}{2}}(\tau)}{\lambda^{\frac{3}{2}}(\sigma)}\frac{\rho^{\frac{1}{2}}(\frac{\lambda^2(\tau)}{\lambda^2(\sigma)}\xi)}{\rho^{\frac{1}{2}}(\xi)}\cos\big[\lambda(\tau)\xi^{\frac{1}{2}}\int_{\tau}^{\sigma}\lambda^{-1}(u)\,du\big]h(\sigma, \frac{\lambda^2(\tau)}{\lambda^2(\sigma)}\xi)\,d\sigma
\end{align*}
as revealed in the preceding results in a combined phase of the form 
\begin{align*}
&\cos\big[\lambda(\tau)\xi^{\frac{1}{2}}\int_{\tau}^\sigma\lambda^{-1}(u)\,du\mp \lambda(s)\tilde{\eta}^{\frac{1}{2}}\int_{\sigma}^s\lambda^{-1}(u)\,du\big]\\
& = \frac{1}{\frac{\lambda(\tau)}{\lambda(\sigma)}\xi^{\frac{1}{2}} \mp \frac{\lambda(s)}{\lambda(\sigma)}\tilde{\eta}^{\frac{1}{2}}}\frac{\partial}{\partial\sigma}\big(\sin\big[\lambda(\tau)\xi^{\frac{1}{2}}\int_{\tau}^\sigma\lambda^{-1}(u)\,du\mp \lambda(s)\tilde{\eta}^{\frac{1}{2}}\int_{\sigma}^s\lambda^{-1}(u)\,du\big]\big)\\
\end{align*}
Performing an integration by parts with respect to $\sigma$ in the first integral then produces an extra gain of $\sigma^{-1}$, while we replace the kernel $K^{nd}(\xi, \eta)$ with one of the schematic form 
\[
\min\{\frac{1}{\xi^{\frac{1}{2}}}, \frac{1}{\eta^{\frac{1}{2}}}\}K^{nd}(\xi, \eta). 
\]
Note that the vanishing properties of the kernel $K^{nd}$ imply that we can absorb the weight $\min\{\frac{1}{\xi^{\frac{1}{2}}}, \frac{1}{\eta^{\frac{1}{2}}}\}$ into the kernel. The desired bound of the lemma then follows as for the a priori bounds established earlier. 
\end{proof}

If we now absorb into $\Phi^k_2$ all those expressions where at least one kernel $\mathcal{K}^{nd}$ appears, then we get the desired bound for $\Phi_2^k$ as a consequence of the preceding lemma. It then remains to establish the bound for $\Phi_1^k$ where we now set 
\[
\Phi_1(f): = \chi_{\xi>\kappa_*}\int_{\tau_0}^\tau \frac{\lambda^{\frac{3}{2}}(\tau)}{\lambda^{\frac{3}{2}}(\sigma)}\frac{\rho^{\frac{1}{2}}(\frac{\lambda^2(\tau)}{\lambda^2(\sigma)}\xi)}{\rho^{\frac{1}{2}}(\xi)}\cos[\lambda(\tau)\xi^{\frac{1}{2}}\int_{\tau}^\sigma\lambda^{-1}(u)\,du]\beta_{\nu}(\sigma)(\mathcal{K}^df)(\frac{\lambda^2(\tau)}{\lambda^2(\sigma)}\xi)\,d\sigma.
\]
Below we shall omit the cutoff $\chi_{\xi>\kappa_*}$ as it becomes irrelevant. As we still haven't fixed $n$ in the definition of $\Phi_{1,2}$, we shall henceforth put $n = k$. 
\\

To get the desired bound for $\Phi_1^k$, we shall further decompose the kernel (which now lives near the diagonal $\xi = \eta$) into a small frequency and a large frequency case, as well as a remaining intermediate frequency case. This is where the parameter $\epsilon$ comes in: 
\[
\mathcal{K}^d = \mathcal{K}_1^{d,\epsilon} +  \mathcal{K}_2^{d,\epsilon} +  \mathcal{K}_3^{d,\epsilon}. 
\]
Here the operators on the right are given in terms of their respective kernels as follows: 
\[
K_1^{d,\epsilon} = \chi_{\xi<\epsilon}K^d,\,K_3^{d,\epsilon} = \chi_{\xi>\epsilon^{-1}}K^d,\,K_2^{d,\epsilon} = K^d - K_1^{d,\epsilon}  - K_3^{d,\epsilon}. 
\]
Then the 'outer' kernels $K_{1,3}^{d,\epsilon}$ have an important smallness property evidenced by 
\[
 \big(\sum_{\substack{N\gtrsim \tau_0\\ N\,\text{dyadic}}}[\sup_{\tau\sim N}(\frac{\tau}{\tau_0})^{\kappa}\big\|\mathcal{D}_{\tau}U\beta_{\nu}(\sigma)\mathcal{K}_{1,3}^{d, \epsilon}f\big\|_{S_3}]^2\big)^{\frac{1}{2}}\lesssim \epsilon^{\gamma_1} \big(\sum_{\substack{N\gtrsim \tau_0\\ N\,\text{dyadic}}}[\sup_{\tau\sim N}(\frac{\tau}{\tau_0})^{\kappa}\|f(\tau, \cdot)\|_{S_3}]^2\big)^{\frac{1}{2}}
\]
for suitable $\gamma_1>0$. Indeed, this follows from Lemma~\ref{lem:smallnessgaininputfreqsmall} for the small frequency cutoff and is proved similarly for the large frequency cutoff, exploiting the decay properties of the kernel $\mathcal{K}$ at large frequencies as in \cite{KST}.
\\
Moreover, in light of the definition of $\mathcal{D}_{\tau}U$, we have the following key vanishing relations:
\begin{equation}\label{eq:van1}
\mathcal{K}_3^{d,\epsilon}\mathcal{D}_{\tau}U\mathcal{K}_2^{d, (1+\frac{1}{n})\epsilon} = 0
\end{equation}
\begin{equation}\label{eq:van2}
\mathcal{K}_2^{d,(1+\frac{1}{n})\epsilon}\mathcal{D}_{\tau}U\mathcal{K}_1^{d, \epsilon} = 0
\end{equation}
This implies that operators of type $\mathcal{K}_{2,3}^{d,\epsilon}$ can only be followed by a more restrictive class of operators, and in particular, we can 'lock in' a certain amount of gain in the presence of 'off diagonal' operators. However, iterating a large number of diagonal operators will result in smallness thanks to the fact that one essentially integrates over a simplex in high dimension. 
Specifically, we expand 
\begin{align*}
&\big(\beta_{\nu}(\tau)\mathcal{K}^{d}\mathcal{D}_{\tau}U\big)^n\\
& = \big(\beta_{\nu}(\tau)\mathcal{K}_2^{d, \epsilon}\mathcal{D}_{\tau}U\big)^n\\
& + \sum_{k=1}^{n-1} \big(\beta_{\nu}(\tau)\mathcal{K}_2^{d, \epsilon}\mathcal{D}_{\tau}U\big)^k \big(\beta_{\nu}(\tau)\mathcal{K}_1^{d, \epsilon}\mathcal{D}_{\tau}U\big)\big(\beta_{\nu}(\tau)\mathcal{K}^{d}\mathcal{D}_{\tau}U\big)^{n-k-1}\\
& + \sum_{k=1}^{n-1} \big(\beta_{\nu}(\tau)\mathcal{K}_2^{d, \epsilon}\mathcal{D}_{\tau}U\big)^k \big(\beta_{\nu}(\tau)\mathcal{K}_3^{d, \epsilon}\mathcal{D}_{\tau}U\big)\big(\beta_{\nu}(\tau)\mathcal{K}^{d}\mathcal{D}_{\tau}U\big)^{n-k-1}\\
& + \big(\beta_{\nu}(\tau)\mathcal{K}_1^{d, \epsilon}\mathcal{D}_{\tau}U\big)\big(\beta_{\nu}(\tau)\mathcal{K}^{d}\mathcal{D}_{\tau}U\big)^{n-1}\\
&+\big(\beta_{\nu}(\tau)\mathcal{K}_3^{d, \epsilon}\mathcal{D}_{\tau}U\big)\big(\beta_{\nu}(\tau)\mathcal{K}^{d}\mathcal{D}_{\tau}U\big)^{n-1}\\
&=: A+B+C+D+E.
\end{align*}
For the term $C$, observe that we have 
\begin{align*}
&\big(\beta_{\nu}(\tau)\mathcal{K}_3^{d, \epsilon}\mathcal{D}_{\tau}U\big)\big(\beta_{\nu}(\tau)\mathcal{K}^{d}\mathcal{D}_{\tau}U\big)^{n-k-1}\\& = \big(\beta_{\nu}(\tau)\mathcal{K}_3^{d, \epsilon}\mathcal{D}_{\tau}U\big)\big(\beta_{\nu}(\tau)\mathcal{K}_3^{d, 4\epsilon}\mathcal{D}_{\tau}U\big)^{n-k-1},
\end{align*}
and so all terms here are trapped in a high-frequency regime. Thus we get 
\begin{equation}\label{eq:Cequation}\begin{split}
& C = \sum_{k=1}^{n-1} \big(\beta_{\nu}(\tau)\mathcal{K}_2^{d, \epsilon}\mathcal{D}_{\tau}U\big)^k \big(\beta_{\nu}(\tau)\mathcal{K}_3^{d, \epsilon}\mathcal{D}_{\tau}U\big)\big(\beta_{\nu}(\tau)\mathcal{K}^{d}\mathcal{D}_{\tau}U\big)^{n-k-1}\\
& = \sum_{k=1}^{n-1} \big(\beta_{\nu}(\tau)\mathcal{K}_2^{d, \epsilon}\mathcal{D}_{\tau}U\big)^k\big(\beta_{\nu}(\tau)\mathcal{K}_3^{d, \epsilon}\mathcal{D}_{\tau}U\big)\big(\beta_{\nu}(\tau)\mathcal{K}_3^{d, 4\epsilon}\mathcal{D}_{\tau}U\big)^{n-k-1}
\end{split}\end{equation}
Moreover, for the term $B$, we have 
\begin{equation}\label{eq:Bequation}\begin{split}
&B = \sum_{k=1}^{n-1} \big(\beta_{\nu}(\tau)\mathcal{K}_2^{d, \epsilon}\mathcal{D}_{\tau}U\big)^k \big(\beta_{\nu}(\tau)\mathcal{K}_1^{d, \epsilon}\mathcal{D}_{\tau}U\big)\big(\beta_{\nu}(\tau)\mathcal{K}^{d}\mathcal{D}_{\tau}U\big)^{n-k-1}\\
& = \sum_{k=1}^{n-1} \big(\beta_{\nu}(\tau)\mathcal{K}_2^{d, \epsilon}\mathcal{D}_{\tau}U\big)^k \big(\beta_{\nu}(\tau)\mathcal{K}_1^{d, \epsilon}\mathcal{D}_{\tau}U\big)\big(\beta_{\nu}(\tau)\mathcal{K}_2^{d, \frac{\epsilon}{4}}\mathcal{D}_{\tau}U\big)^{n-k-1}\\
&+\sum_{\substack{1\leq k\leq n-1\\ j\leq n-k-2}} \big(\beta_{\nu}(\tau)\mathcal{K}_2^{d, \epsilon}\mathcal{D}_{\tau}U\big)^k \big(\beta_{\nu}(\tau)\mathcal{K}_1^{d, \epsilon}\mathcal{D}_{\tau}U\big)\big(\beta_{\nu}(\tau)\mathcal{K}_2^{d, \frac{\epsilon}{4}}\mathcal{D}_{\tau}U\big)^{j}\\
&\hspace{4cm}\cdot\big(\beta_{\nu}(\tau)\mathcal{K}_3^{d, \frac{\epsilon}{4}}\mathcal{D}_{\tau}U\big)\big(\beta_{\nu}(\tau)\mathcal{K}_3^{d, \epsilon}\mathcal{D}_{\tau}U\big)^{n-j-k-2}
\end{split}\end{equation}
Here we use that the operators $ \big(\beta_{\nu}(\tau)\mathcal{K}_2^{d, \epsilon}\mathcal{D}_{\tau}U\big)$ on the left force large frequencies at the end of the expression, and if only one operator $\big(\beta_{\nu}(\tau)\mathcal{K}_3^{d, \frac{\epsilon}{4}}$ occurs it will force very large frequencies after it. 
\\
For term $E$ we proceed just as for term $C$. Thus write 
\begin{align*}
\big(\beta_{\nu}(\tau)\mathcal{K}_3^{d, \epsilon}\mathcal{D}_{\tau}U\big)\big(\beta_{\nu}(\tau)\mathcal{K}^{d}\mathcal{D}_{\tau}U\big)^{n-1} = \big(\beta_{\nu}(\tau)\mathcal{K}_3^{d, \epsilon}\mathcal{D}_{\tau}U\big)\big(\beta_{\nu}(\tau)\mathcal{K}_3^{d, 4\epsilon}\mathcal{D}_{\tau}U\big)^{n-1}
\end{align*}
The conclusion is that for terms $A, B, C$ and $E$ we can write them in terms of a few consecutive strings of operators of type $\big(\beta_{\nu}(\tau)\mathcal{K}_2^{d, \epsilon}\mathcal{D}_{\tau}U\big)$, $\big(\beta_{\nu}(\tau)\mathcal{K}_3^{d, \epsilon}\mathcal{D}_{\tau}U\big)$, and for the latter we already have observed a smallness gain. 
Finally, for the remaining term $D$, we also write it in terms of a small number of consecutive strings, by writing 
\begin{align*}
&D =  \big(\beta_{\nu}(\tau)\mathcal{K}_1^{d, \epsilon}\mathcal{D}_{\tau}U\big)\big(\beta_{\nu}(\tau)\mathcal{K}^{d}\mathcal{D}_{\tau}U\big)^{n-1}\\
&=\sum_{j=1}^n\big(\beta_{\nu}(\tau)\mathcal{K}_1^{d, \epsilon}\mathcal{D}_{\tau}U\big)^j[A^{(n-j)}+B^{(n-j)} + C^{(n-j)} + E^{(n-j)}]
\end{align*}
where the superscript indicates that these terms are defined just as in $A$, $B$, $C$ and $E$ but with $n$ replaced by $n-j$. 
\\

At this point, we have essentially reduced the problem of bounding $\big(\beta_{\nu}(\tau)\mathcal{K}^{d}\mathcal{D}_{\tau}U\big)^n$ to the problem of bounding $\big(\beta_{\nu}(\tau)\mathcal{K}_2^{d, \epsilon}\mathcal{D}_{\tau}U\big)^n$, and so this is what we now turn to: 
\begin{lem}\label{lem:simplex} Using the preceding notation and assuming $j\leq n$ (the latter as in the definition of $\mathcal{K}^{d, nd}$), we have the bound 
\begin{align*}
&\big(\sum_{\substack{N\gtrsim \tau_0\\ N\,\text{dyadic}}}[\sup_{\tau\sim N}(\frac{\tau}{\tau_0})^{\kappa}\big\|\mathcal{D}_{\tau}U\big(\beta_{\nu}(\tau)\mathcal{K}_2^{d, \epsilon}\mathcal{D}_{\tau}U\big)^j f\big\|_{S_3}]^2\big)^{\frac{1}{2}}\\&\lesssim \frac{\epsilon^{-j}}{j!} \big(\sum_{\substack{N\gtrsim \tau_0\\ N\,\text{dyadic}}}[\sup_{\tau\sim N}(\frac{\tau}{\tau_0})^{\kappa}\big\|f(\tau, \cdot)\big\|_{S_3}]^2\big)^{\frac{1}{2}}.
\end{align*}

\end{lem}
\begin{proof} Write 
\begin{align*}
&\big(\beta_{\nu}(\tau)\mathcal{K}_2^{d, \epsilon}\mathcal{D}_{\tau}U\big)^j(\beta_{\nu}(\tau)\mathcal{K}_2^{d, \epsilon}f)\\& = \beta_{\nu}(\tau)\mathcal{K}_2^{d, \epsilon}\int_{\tau_0}^{\tau}\int_0^\infty d\sigma_1d\eta_1\frac{\lambda^{\frac{3}{2}}(\tau)}{\lambda^{\frac{3}{2}}(\sigma_1)}\frac{\rho^{\frac{1}{2}}(\frac{\lambda^2(\tau)}{\lambda^2(\sigma_1)}\xi)}{\rho^{\frac{1}{2}}(\xi)}\beta_{\nu}(\sigma_1)\cos(\lambda(\tau)\xi^{\frac{1}{2}}\int_{\tau}^{\sigma_1}\lambda^{-1}(u)\,du)\\
&\cdot K_2^{d, \epsilon}(\frac{\lambda^{2}(\tau)}{\lambda^{2}(\sigma_1)}\xi, \eta_1)\int_{\tau_0}^{\sigma_1}\int_0^\infty d\sigma_2d\eta_2\frac{\lambda^{\frac{3}{2}}(\sigma_1)}{\lambda^{\frac{3}{2}}(\sigma_2)}\frac{\rho^{\frac{1}{2}}(\frac{\lambda^2(\sigma_1)}{\lambda^2(\sigma_2)}\eta_1)}{\rho^{\frac{1}{2}}(\eta_1)}\beta_{\nu}(\sigma_2)\cos(\lambda(\sigma_1)\eta_1^{\frac{1}{2}}\int_{\sigma_1}^{\sigma_2}\lambda^{-1}(u)\,du)\\
&\ldots\\
&\cdot K_2^{d, \epsilon}(\frac{\lambda^{2}(\sigma_{j-2})}{\lambda^{2}(\sigma_{j-1})}\eta_{j-2}, \eta_{j-1})\int_{\tau_0}^{\sigma_{j-1}}d\sigma_j\frac{\lambda^{\frac{3}{2}}(\sigma_{j-1})}{\lambda^{\frac{3}{2}}(\sigma_j)}\frac{\rho^{\frac{1}{2}}(\frac{\lambda^2(\sigma_{j-1})}{\lambda^2(\sigma_j)}\eta_{j-1})}{\rho^{\frac{1}{2}}(\eta_{j-1})}\beta_{\nu}(\sigma_j)\\&\hspace{3cm}\cdot\cos(\lambda(\sigma_{j-1})\eta_{j-1}^{\frac{1}{2}}\int_{\sigma_{j-1}}^{\sigma_j}\lambda^{-1}(u)\,du)((\beta_{\nu}(\cdot)\mathcal{K}_2^{d, \epsilon}f)(\sigma_j, \frac{\lambda^2(\sigma_{j-1})}{\lambda^2(\sigma_j)}\eta_{j-1})\\
\end{align*}
Then we carefully recall that by choice of $\mathcal{K}^d$ we have 
\[
\big|\frac{\lambda^2(\sigma_k)\eta_k}{\lambda^2(\sigma_{k+1})\eta_{k+1}} - 1\big|<\frac{1}{n}, \,1\leq k\leq j-1<n. 
\]
and so 
\[
\lambda^2(\sigma_k)\eta_k\geq (1-\frac{1}{n})^k\lambda^2(\tau)\xi,\,1\leq k\leq j-1. 
\]
Since we further have the restrictions $\xi\gtrsim\epsilon, \eta_k<\epsilon^{-1}$ on the support of the full expression, we get 
\[
\sigma_k>\tau\cdot\epsilon,\,1\leq k\leq j. 
\]
for $\nu$ and the  $\epsilon$ small enough, uniformly in $n$.  In particular, we get 
\[
\beta_{\nu}(\sigma_k)\lesssim \epsilon^{-1}\tau^{-1}, 1\leq k\leq j.
\]
Finally, we infer that for fixed $\tau\geq \tau_0$ 
\begin{align*}
&\big\|\big(\beta_{\nu}(\tau)\mathcal{K}_2^{d, \epsilon}\mathcal{D}_{\tau}U\big)^j(\beta_{\nu}(\tau)\mathcal{K}_2^{d, \epsilon}f)\big\|_{S_3}\\
&\lesssim \beta_{\nu}(\tau)\int_{\epsilon\tau}^{\tau}\beta_{\nu}(\sigma_1)\int_{\epsilon\tau}^{\sigma_1}\beta_{\nu}(\sigma_2)\ldots\int_{\epsilon\tau}^{\sigma_j-1}\beta_{\nu}(\sigma_j)\big\|f(\sigma_j, \cdot)\big\|_{S_3}\,d\sigma_j\ldots d\sigma_1\\
&\lesssim \beta_{\nu}(\tau)\frac{\epsilon^{-j}}{j!}\sup_{\sigma\sim \tau}\big\|f(\sigma, \cdot)\big\|_{S_3}
\end{align*}
We finally get the desired conclusion of the lemma 
\begin{align*}
&\big(\sum_{\substack{N\gtrsim \tau_0\\ N\,\text{dyadic}}}[\sup_{\tau\sim N}(\frac{\tau}{\tau_0})^{\kappa}\big\|\mathcal{D}_{\tau}U\big(\beta_{\nu}(\tau)\mathcal{K}_2^{d, \epsilon}\mathcal{D}_{\tau}U\big)^j(\beta_{\nu}(\tau)\mathcal{K}_2^{d, \epsilon}f)\big\|_{S_3}\\&\lesssim  \frac{\epsilon^{-j}}{j!}\big(\sum_{\substack{N\gtrsim \tau_0\\ N\,\text{dyadic}}}[\sup_{\tau\sim N}(\frac{\tau}{\tau_0})^{\kappa}\|f(\tau, \cdot)\big\|_{S_3}]^2\big)^{\frac{1}{2}}
\end{align*}
\end{proof}

We can now conclude the bound for $\Phi_1^k$ which will then complete the proof of Proposition~\ref{prop:mainestimate}, by bounding the terms $A$ - $E$ from before: 
\\

{\it{Bound for $A$.}} From preceding lemma, we have 
\begin{align*}
 \big(\sum_{\substack{N\gtrsim \tau_0\\ N\,\text{dyadic}}}[\sup_{\tau\sim N}(\frac{\tau}{\tau_0})^{\kappa}\big\|Af\big\|_{S_3}]^2\big)^{\frac{1}{2}}&\lesssim \frac{\epsilon^{-n}}{n!} \big(\sum_{\substack{N\gtrsim \tau_0\\ N\,\text{dyadic}}}[\sup_{\tau\sim N}(\frac{\tau}{\tau_0})^{\kappa}\big\|f(\tau, \cdot)\big\|_{S_3}]^2\big)^{\frac{1}{2}}\\
&\ll \epsilon^{\gamma_1 n}e^{\epsilon^{-2}}\big(\sum_{\substack{N\gtrsim \tau_0\\ N\,\text{dyadic}}}[\sup_{\tau\sim N}(\frac{\tau}{\tau_0})^{\kappa}\big\|f(\tau, \cdot)\big\|_{S_3}]^2\big)^{\frac{1}{2}}\\ 
\end{align*}
provided $n$ is sufficiently large in relation to a fixed chosen $\epsilon$. 
\\

{\it{Bound for $B$.}} In light of identity \eqref{eq:Bequation}, we find 
\begin{align*}
&\big(\sum_{\substack{N\gtrsim \tau_0\\ N\,\text{dyadic}}}[\sup_{\tau\sim N}(\frac{\tau}{\tau_0})^{\kappa}\big\|Bf\big\|_{S_3}]^2\big)^{\frac{1}{2}}\\&\lesssim \big[\sum_{k=1}^{n-1}\frac{\epsilon^{-k}}{k!}\epsilon^{\gamma_1}\frac{\epsilon^{-(n-k-1)}}{(n-k-1)!}\\
&\hspace{1cm} + \sum_{k=1}^{n-1}\sum_{0\leq j\leq n-k-2}\frac{\epsilon^{-k}}{k!}\frac{\epsilon^{-j}}{j!}\epsilon^{\gamma_1(n-k-j)}\big]\big(\sum_{\substack{N\gtrsim \tau_0\\ N\,\text{dyadic}}}[\sup_{\tau\sim N}(\frac{\tau}{\tau_0})^{\kappa}\big\|f(\tau, \cdot)\big\|_{S_3}]^2\big)^{\frac{1}{2}}\\
&\ll \epsilon^{\gamma_1 n}e^{\epsilon^{-2}}\big(\sum_{\substack{N\gtrsim \tau_0\\ N\,\text{dyadic}}}[\sup_{\tau\sim N}(\frac{\tau}{\tau_0})^{\kappa}\big\|f(\tau, \cdot)\big\|_{S_3}]^2\big)^{\frac{1}{2}}\\
\end{align*}
provided $\epsilon$ is sufficiently small and $n$ large enough. 
\\

{\it{Bound for $C, D, E$.}}The term $C$ is similar in light of relation \eqref{eq:Cequation}, as is term $E$. Finally, for the term $D$, we can bound it by 
\begin{align*}
&\big(\sum_{\substack{N\gtrsim \tau_0\\ N\,\text{dyadic}}}[\sup_{\tau\sim N}(\frac{\tau}{\tau_0})^{\kappa}\big\|Df\big\|_{S_3}]^2\big)^{\frac{1}{2}}\\&\lesssim[\sum_{j=1}^n \epsilon^{\gamma_1j}\epsilon^{\gamma_1(n-j)}\epsilon^{\epsilon^{-2}}]\big(\sum_{\substack{N\gtrsim \tau_0\\ N\,\text{dyadic}}}[\sup_{\tau\sim N}(\frac{\tau}{\tau_0})^{\kappa}\big\|f(\tau, \cdot)\big\|_{S_3}]^2\big)^{\frac{1}{2}}\\&\leq \epsilon^{\gamma_2n}e^{\epsilon^{-2}}\big(\sum_{\substack{N\gtrsim \tau_0\\ N\,\text{dyadic}}}[\sup_{\tau\sim N}(\frac{\tau}{\tau_0})^{\kappa}\big\|f(\tau, \cdot)\big\|_{S_3}]^2\big)^{\frac{1}{2}}
\end{align*}
for any $\gamma_2<\gamma_1$, provided $\epsilon<1$ and $n$ is sufficiently large. 
\end{proof}

\section{Proof of Theorem~\ref{thm:MainTechnical}}

By Corollary~\ref{cor:convergence}, the function $\underline{x}(\tau, \xi): = \underline{x}^{(0)}(\tau, \xi) + \sum_{j=1}^\infty \underline{\triangle x}^{(j)}(\tau, \xi)$ converges in the sense of Corollary~\ref{cor:convergence}. In light of the definition of $\triangle A_j$, it is then easily seen that 
\[
\tilde{\epsilon}(\tau, R): = \big[\sum_{j=1}^\infty\triangle_d x^{(j)}(\tau)\big]\phi_d(R) + \int_0^\infty\phi(R, \xi)\big[\underline{x}^{(0)}(\tau, \xi) + \sum_{j=1}^\infty\triangle x^{(j)}(\tau, \xi)\big]\rho(\xi)\,d\xi
\]
is in fact a function in $H^{\frac32+}_{dR}$,  which by construction of $\underline{x}$ constitutes a solution of \eqref{eq:Rtauwave1}. In view of the regularity properties of $u_{approx}^{k,\nu}$ (\cite{KST}, \cite{KS1}), we then obtain that 
\[
u_{approx}^{k,\nu} + \epsilon
\]
is a solution of \eqref{eq:Main} of regularity $H^{1+\frac{\nu}{2}-}_{\R^3}$. Furthermore, writing (see Theorem~\ref{thm:Main})
\[
u_{approx}^{k,\nu} + \epsilon = W_{\lambda(t)}(x) + v_1(t,x)
\]
and recalling \eqref{eq:u2k-1bound}, the proof of the last assertion of Theorem~\ref{thm:Main} concerning $E_{loc}$ reduces to verifying that 
\[
\lim_{t\rightarrow 0}\int_{|x|\leq t}|\nabla_{t,x}\epsilon(t,x)|^2\,dx = 0
\]
However, using that $\epsilon = \frac{\tilde{\epsilon}}{R}$ and the structure estimates of $\tilde{\epsilon}$ coming from Proposition~\ref{prop:lingrowthcond} as well as Proposition~\ref{prop:energybound1} as well as straightforward bounds for the contributions to $\tilde{\epsilon}$ coming from the 'non-free' Fourier components $\triangle_{>\tau}x$, $\widetilde{\triangle x}$, one infers that 
\[
\int_{|x|\leq t}|\nabla_{t,x}\epsilon(t,x)|^2\,dx\lesssim \frac{\tau^2}{\tau^{1+\nu^{-1}}}
\]
which converges to zero for $\nu<1$, and, we in fact had (see the proof of Proposition~\ref{prop:nonlinbounds}) the restriction $\nu\leq \frac13$.  
\\

Back to the level of the Fourier transform, the data of $\underline{x}$ at time $\tau = \tau_0$ no longer coincide with 
\[
\big(\left(\begin{array}{c}x_{0d}\\ x_0\end{array}\right),\,\left(\begin{array}{c}x_{1d}\\ x_1\end{array}\right)\big), 
\]
but can be written as 
\[
\big(\left(\begin{array}{c}x_{0d}\\ x_0\end{array}\right),\,\left(\begin{array}{c}x_{1d}\\ x_1\end{array}\right)\big) + \sum_{j=1}^\infty\big(\left(\begin{array}{c}\triangle x^{(j)}_d(\tau_0)\\ \triangle \tilde{\tilde{x}}^{(j)}_0\end{array}\right),\,\left(\begin{array}{c}\partial_{\tau}\triangle x^{(j)}_d(\tau_0)\\ \triangle \tilde{\tilde{x}}^{(j)}_1\end{array}\right)\big)
\]
Then the arguments from before imply that 
\begin{align*}
\big\|\big( \sum_{j=1}^\infty\triangle \tilde{\tilde{x}}^{(j)}_0, \sum_{j=1}^\infty\triangle \tilde{\tilde{x}}^{(j)}_1\big)\big\|_{\tilde{S}}\lesssim \tau_0^{-(1-)}[\big\|\big(x_0, x_1\big)\big\|_{\tilde{S}} + \big|x_{0d}\big|].
\end{align*}
and further 
\[
\sum_j\big|\triangle x^{(j)}_d(\tau_0)\big|\lesssim \tau_0^{-(1-)}\big|x_{0d}\big| + \big\|\big(x_0, x_1\big)\big\|_{\tilde{S}}
\]
Moreover, passing to differences of solutions gives the Lipschitz-continuity of the data map with small Lipschitz norm. This completes the proof. 

\section{Appendix}

\subsection{The contribution of the first term in \eqref{eq:linhomparam1} in the proof of Proposition~\ref{prop:firstiteratelincont}}

Here we follow the same steps as in the proof alluded to, but only retain the contribution of the source term $x_0$: 
\\

{\bf{Step 1}}: {\it{Proof of the high frequency estimate \eqref{eq:firstiterhigh} for the contribution from $\mathcal{R}_1(\tau, x^{(0)})$.}}
\\

{\bf{Step 1(i)}}: {\it{Contribution of $\beta_{\nu}^2(\tau)\mathcal{K}_{cc}x^{(0)}$}}. Substituting the first term in \eqref{eq:Kx0formula} for $\mathcal{K}_{cc}x^{(0)}$, we arrive at the expression 
\begin{align*}
 \int_{\tau_0}^\tau\frac{\lambda^{\frac{3}{2}}(\tau)}{\lambda^{\frac{3}{2}}(\sigma)}\frac{\rho^{\frac{1}{2}}(\frac{\lambda^2(\tau)}{\lambda^2(\sigma)}\xi)}{\rho^{\frac{1}{2}}(\xi)}\frac{\sin[\lambda(\tau)\xi^{\frac{1}{2}}\int_{\tau}^\sigma\lambda^{-1}(u)\,du]}{\xi^{\frac{1}{2}}}\beta_{\nu}^2(\sigma)A(\xi, \tau, \tau_0, \sigma)\,d\sigma,
 \end{align*}
where this time 
 \begin{align*}
 &A(\xi, \tau, \tau_0, \sigma):\\
 & =  \int_0^\infty\frac{\lambda^{\frac52}(\sigma)}{\lambda^{\frac52}(\tau_0)}\frac{F(\frac{\lambda^2(\tau)}{\lambda^2(\sigma)}\xi, \frac{\lambda^2(\tau)}{\lambda^2(\sigma)}\tilde{\eta})\rho^{\frac{1}{2}}(\frac{\lambda^2(\tau)}{\lambda^2(\sigma)}\tilde{\eta})}{\xi- \tilde{\eta}}\rho^{\frac{1}{2}}(\frac{\lambda^2(\tau)}{\lambda^2(\tau_0)}\tilde{\eta})\cos[\lambda(\tau)\tilde{\eta}^{\frac{1}{2}}\int_{\tau_0}^{\sigma}\lambda^{-1}(u)\,du]\\
 &\hspace{9.5cm}\cdot x_0(\frac{\lambda^2(\tau)}{\lambda^2(\tau_0)}\tilde{\eta})\,d\tilde{\eta}. 
 \end{align*}
To deal with the case $\tilde{\eta}\geq 1$, use that if we put 
\begin{align*}
B(\tau, \sigma, \xi, \tilde{\eta}): = \frac{\lambda^{\frac{3}{2}}(\tau)}{\lambda^{\frac{3}{2}}(\sigma)}\xi^{\frac14+}\rho^{\frac{1}{2}}(\frac{\lambda^2(\tau)}{\lambda^2(\sigma)}\xi)F(\frac{\lambda^2(\tau)}{\lambda^2(\sigma)}\xi, \frac{\lambda^2(\tau)}{\lambda^2(\sigma)}\tilde{\eta})\rho^{\frac{1}{2}}(\frac{\lambda^2(\tau)}{\lambda^2(\sigma)}\tilde{\eta})
\end{align*}
and further restrict to $\xi>1$ (as we do throughout {\it{Step 1}}), we get 
\begin{align*}
\big|B(\tau, \sigma, \xi, \tilde{\eta})\big|\lesssim \max\{\xi, \tilde{\eta}\}^{-\frac14+}(\frac{\lambda(\tau)}{\lambda(\sigma)})^{\frac12}
\end{align*}
Furthermore, we have 
\begin{align*}
&\big\|\frac{\lambda^{\frac52}(\sigma)}{\lambda^{\frac52}(\tau_0)}\rho^{\frac{1}{2}}(\frac{\lambda^2(\tau)}{\lambda^2(\tau_0)}\tilde{\eta})\tilde{\eta}^{-\frac14+}x_0(\frac{\lambda^2(\tau)}{\lambda^2(\tau_0)}\tilde{\eta})\big\|_{L^2_{d\tilde{\eta}}(\tilde{\eta}>1)}
\\&\lesssim \frac{\lambda^{\frac52}(\sigma)\lambda^{\frac12}(\tau)}{\lambda^3(\tau_0)}\cdot\frac{\lambda^2(\tau_0)}{\lambda^2(\tau)}\cdot\frac{\lambda(\tau_0)}{\lambda(\tau)}\big\|x_0\big\|_{\langle\tilde{\eta}\rangle^{-(1+)}L^2_{d\tilde{\eta}}},
\end{align*}
and the operator with kernel 
\[
\frac{\tilde{\eta}^{\frac14-}B(\tau, \sigma, \xi, \tilde{\eta})}{\xi - \tilde{\eta}}
\]
is $L^2$-bounded up to a factor $(\frac{\lambda(\tau)}{\lambda(\sigma)})^{\frac12}$. In all, we then conclude that 
\begin{align*}
&\big\| \int_{\tau_0}^\tau\frac{\lambda^{\frac{3}{2}}(\tau)}{\lambda^{\frac{3}{2}}(\sigma)}\frac{\rho^{\frac{1}{2}}(\frac{\lambda^2(\tau)}{\lambda^2(\sigma)}\xi)}{\rho^{\frac{1}{2}}(\xi)}\frac{\sin[\lambda(\tau)\xi^{\frac{1}{2}}\int_{\tau}^\sigma\lambda^{-1}(u)\,du]}{\xi^{\frac{1}{2}}}\beta_{\nu}^2(\sigma)A_{\tilde{\eta}\geq 1}(\xi, \tau, \tau_0, \sigma)\,d\sigma\big\|_{\langle\xi\rangle^{-1-}L^2_{d\xi}(\xi>1)}\\
&\lesssim \big\|x_0\big\|_{\tilde{S}_1} \int_{\tau_0}^\tau \frac{\lambda^2(\sigma)}{\lambda^2(\tau)}\beta_{\nu}^2(\sigma)\,d\sigma\lesssim \tau_0^{-1} \big\|x_0\big\|_{\tilde{S}_1}. 
\end{align*}
Next, turn to the case $\tilde{\eta}<1$. Then replicating the reductions in {\it{Step 1(i)}} of the proof of Proposition~\ref{prop:firstiteratelincont}, one first replaces $\rho^{\frac12}(\frac{\lambda^2(\tau)}{\lambda^2(\sigma)}\tilde{\eta})$ by 
\[
c\chi_{\frac{\lambda^2(\tau)}{\lambda^2(\sigma)}\tilde{\eta}<1}[\frac{\lambda^2(\tau)}{\lambda^2(\sigma)}\tilde{\eta}]^{-\frac14}. 
\]
To achieve this reduction, recalling the algebra in {\it{Step 1(i)}} earlier,  one has to bound the expression
\begin{align*}
 \int_{\tau_0}^\tau\frac{\lambda^{\frac{3}{2}}(\tau)}{\lambda^{\frac{3}{2}}(\sigma)}\frac{\rho^{\frac{1}{2}}(\frac{\lambda^2(\tau)}{\lambda^2(\sigma)}\xi)}{\rho^{\frac{1}{2}}(\xi)}\frac{\sin[\lambda(\tau)\xi^{\frac{1}{2}}\int_{\tau}^\sigma\lambda^{-1}(u)\,du]}{\xi^{\frac{1}{2}}}\beta_{\nu}^2(\sigma)\tilde{A}_{\tilde{\eta}<1}(\xi, \tau, \tau_0, \sigma)\,d\sigma,
\end{align*}
where we set 
\begin{align*}
&\tilde{A}_{\tilde{\eta}<1}(\xi, \tau, \tau_0, \sigma):\\
 & =  \int_0^1\frac{\lambda^{\frac52}(\sigma)}{\lambda^{\frac52}(\tau_0)}\frac{F(\frac{\lambda^2(\tau)}{\lambda^2(\sigma)}\xi, \frac{\lambda^2(\tau)}{\lambda^2(\sigma)}\tilde{\eta})\chi_{\frac{\lambda^2(\tau)}{\lambda^2(\sigma)}\tilde{\eta}>1}\rho^{\frac{1}{2}}(\frac{\lambda^2(\tau)}{\lambda^2(\sigma)}\tilde{\eta})}{\xi- \tilde{\eta}}\rho^{\frac{1}{2}}(\frac{\lambda^2(\tau)}{\lambda^2(\tau_0)}\tilde{\eta})\\
 &\hspace{4cm}\cdot\cos[\lambda(\tau)\tilde{\eta}^{\frac{1}{2}}\int_{\tau_0}^{\sigma}\lambda^{-1}(u)\,du]\cdot x_0(\frac{\lambda^2(\tau)}{\lambda^2(\tau_0)}\tilde{\eta})\,d\tilde{\eta}\\
 & +  \int_0^1\frac{\lambda^{\frac52}(\sigma)}{\lambda^{\frac52}(\tau_0)}\frac{F(\frac{\lambda^2(\tau)}{\lambda^2(\sigma)}\xi, \frac{\lambda^2(\tau)}{\lambda^2(\sigma)}\tilde{\eta})\chi_{\frac{\lambda^2(\tau)}{\lambda^2(\sigma)}\tilde{\eta}<1}[\frac{\lambda^2(\tau)}{\lambda^2(\sigma)}\tilde{\eta})]^{\frac14}}{\xi- \tilde{\eta}}\rho^{\frac{1}{2}}(\frac{\lambda^2(\tau)}{\lambda^2(\tau_0)}\tilde{\eta})\\
 &\hspace{4cm}\cdot\cos[\lambda(\tau)\tilde{\eta}^{\frac{1}{2}}\int_{\tau_0}^{\sigma}\lambda^{-1}(u)\,du]\cdot x_0(\frac{\lambda^2(\tau)}{\lambda^2(\tau_0)}\tilde{\eta})\,d\tilde{\eta}\\
\end{align*}
For the first expression on the right use that ($\xi\gg1, \tilde{\eta}<1$, $\tau_0\leq \sigma\leq \tau$)
\begin{align*}
\big|\frac{F(\frac{\lambda^2(\tau)}{\lambda^2(\sigma)}\xi, \frac{\lambda^2(\tau)}{\lambda^2(\sigma)}\tilde{\eta})\chi_{\frac{\lambda^2(\tau)}{\lambda^2(\sigma)}\tilde{\eta}>1}\rho^{\frac{1}{2}}(\frac{\lambda^2(\tau)}{\lambda^2(\sigma)}\tilde{\eta})\rho^{\frac{1}{2}}(\frac{\lambda^2(\tau)}{\lambda^2(\tau_0)}\tilde{\eta})}{\xi- \tilde{\eta}}\big|
\lesssim (\frac{\lambda(\tau)}{\lambda(\sigma)})^{-N}\langle\xi\rangle^{-N}\frac{\lambda^{\frac12}(\tau)}{\lambda^{\frac12}(\tau_0)},
\end{align*}
and so 
\begin{align*}
&\big| \int_0^1\frac{\lambda^{\frac52}(\sigma)}{\lambda^{\frac52}(\tau_0)}\frac{F(\frac{\lambda^2(\tau)}{\lambda^2(\sigma)}\xi, \frac{\lambda^2(\tau)}{\lambda^2(\sigma)}\tilde{\eta})\chi_{\frac{\lambda^2(\tau)}{\lambda^2(\sigma)}\tilde{\eta}>1}\rho^{\frac{1}{2}}(\frac{\lambda^2(\tau)}{\lambda^2(\sigma)}\tilde{\eta})}{\xi- \tilde{\eta}}\rho^{\frac{1}{2}}(\frac{\lambda^2(\tau)}{\lambda^2(\tau_0)}\tilde{\eta})\\
 &\hspace{4cm}\cdot\cos[\lambda(\tau)\tilde{\eta}^{\frac{1}{2}}\int_{\tau_0}^{\sigma}\lambda^{-1}(u)\,du]\cdot x_0(\frac{\lambda^2(\tau)}{\lambda^2(\tau_0)}\tilde{\eta})\,d\tilde{\eta}\\
 &\lesssim (\frac{\lambda(\tau)}{\lambda(\sigma)})^{-N}\langle\xi\rangle^{-N}\cdot \frac{\lambda^{\frac52}(\sigma)\lambda^{\frac12}(\tau)}{\lambda^{3}(\tau_0)}\cdot\frac{\lambda(\tau_0)}{\lambda(\tau)}\int_0^1[\frac{\lambda^2(\tau)}{\lambda^2(\tau_0)}\tilde{\eta}]^{\frac12}
 \big| x_0(\frac{\lambda^2(\tau)}{\lambda^2(\tau_0)}\tilde{\eta})\big|\,d\tilde{\eta},\\
\end{align*}
and we further get 
\begin{align*}
&\int_0^1[\frac{\lambda^2(\tau)}{\lambda^2(\tau_0)}\tilde{\eta}]^{\frac12}
 \big| x_0(\frac{\lambda^2(\tau)}{\lambda^2(\tau_0)}\tilde{\eta})\big|\,d\tilde{\eta}
 \lesssim \frac{\lambda^2(\tau_0)}{\lambda^2(\tau)}\big\|x_0\big\|_{\xi^{-\frac12}L^1_{d\xi}}\lesssim  \frac{\lambda^2(\tau_0)}{\lambda^2(\tau)}\big\|x_0\big\|_{\tilde{S}_1},
\end{align*}
via a simple application of the Cauchy-Schwarz inequality. We conclude that 
\begin{equation}\label{eq:appendix1}\begin{split}
&\big| \int_0^1\frac{\lambda^{\frac52}(\sigma)}{\lambda^{\frac52}(\tau_0)}\frac{F(\frac{\lambda^2(\tau)}{\lambda^2(\sigma)}\xi, \frac{\lambda^2(\tau)}{\lambda^2(\sigma)}\tilde{\eta})\chi_{\frac{\lambda^2(\tau)}{\lambda^2(\sigma)}\tilde{\eta}>1}\rho^{\frac{1}{2}}(\frac{\lambda^2(\tau)}{\lambda^2(\sigma)}\tilde{\eta})}{\xi- \tilde{\eta}}\rho^{\frac{1}{2}}(\frac{\lambda^2(\tau)}{\lambda^2(\tau_0)}\tilde{\eta})\\
 &\hspace{4cm}\cdot\cos[\lambda(\tau)\tilde{\eta}^{\frac{1}{2}}\int_{\tau_0}^{\sigma}\lambda^{-1}(u)\,du]\cdot x_0(\frac{\lambda^2(\tau)}{\lambda^2(\tau_0)}\tilde{\eta})\,d\tilde{\eta}\\
 &\lesssim (\frac{\lambda(\tau)}{\lambda(\sigma)})^{-N}\langle\xi\rangle^{-N}\cdot\big\|x_0\big\|_{\tilde{S}_1}.
\end{split}\end{equation}
Similarly, for the term restricted to the low frequency regime $\frac{\lambda^2(\tau)}{\lambda^2(\sigma)}\tilde{\eta}<1$, use that 
\begin{align*}
&\big|\frac{F(\frac{\lambda^2(\tau)}{\lambda^2(\sigma)}\xi, \frac{\lambda^2(\tau)}{\lambda^2(\sigma)}\tilde{\eta})\chi_{\frac{\lambda^2(\tau)}{\lambda^2(\sigma)}\tilde{\eta}<1}[\frac{\lambda^2(\tau)}{\lambda^2(\sigma)}\tilde{\eta})]^{\frac14}}{\xi- \tilde{\eta}}\rho^{\frac{1}{2}}(\frac{\lambda^2(\tau)}{\lambda^2(\tau_0)}\tilde{\eta})\big|\\
&\lesssim (\frac{\lambda(\tau)}{\lambda(\sigma)})^{-N}\langle\xi\rangle^{-N}\big[\chi_{\frac{\lambda^2(\tau)}{\lambda^2(\tau_0)}\tilde{\eta}>1}\frac{\lambda^{\frac12}(\tau)}{\lambda^{\frac12}(\tau_0)}\tilde{\eta}^{\frac12} + \chi_{\frac{\lambda^2(\tau)}{\lambda^2(\tau_0)}\tilde{\eta}<1}\frac{\lambda^{\frac12}(\tau_0)}{\lambda^{\frac12}(\tau)}\big]
\end{align*}
Then using that 
\begin{align*}
\int_0^1\chi_{\frac{\lambda^2(\tau)}{\lambda^2(\tau_0)}\tilde{\eta}<1}\big|x_0(\frac{\lambda^2(\tau)}{\lambda^2(\tau_0)}\tilde{\eta})\big|\,d\tilde{\eta}\lesssim [\frac{\lambda(\tau_0)}{\lambda(\tau)}]^{-2}\big\|x_0\big\|_{\tilde{S}_1}, 
\end{align*}
we also infer 
\begin{equation}\label{eq:appendix2}\begin{split}
&\big|\int_0^1\frac{\lambda^{\frac52}(\sigma)}{\lambda^{\frac52}(\tau_0)}\frac{F(\frac{\lambda^2(\tau)}{\lambda^2(\sigma)}\xi, \frac{\lambda^2(\tau)}{\lambda^2(\sigma)}\tilde{\eta})\chi_{\frac{\lambda^2(\tau)}{\lambda^2(\sigma)}\tilde{\eta}<1}[\frac{\lambda^2(\tau)}{\lambda^2(\sigma)}\tilde{\eta})]^{\frac14}}{\xi- \tilde{\eta}}\rho^{\frac{1}{2}}(\frac{\lambda^2(\tau)}{\lambda^2(\tau_0)}\tilde{\eta})\\
 &\hspace{4cm}\cdot\cos[\lambda(\tau)\tilde{\eta}^{\frac{1}{2}}\int_{\tau_0}^{\sigma}\lambda^{-1}(u)\,du]\cdot x_0(\frac{\lambda^2(\tau)}{\lambda^2(\tau_0)}\tilde{\eta})\,d\tilde{\eta}\big|\\
&\lesssim (\frac{\lambda(\tau)}{\lambda(\sigma)})^{-N}\langle\xi\rangle^{-N}\cdot\big\|x_0\big\|_{\tilde{S}_1}.
\end{split}\end{equation}

However, assuming the bounds \eqref{eq:appendix1}, \eqref{eq:appendix2} for $\tilde{A}_{\tilde{\eta}<1}(\xi, \tau, \tau_0, \sigma)$, one easily infers the bound 
\begin{align*}
&\big\|\int_{\tau_0}^\tau\frac{\lambda^{\frac{3}{2}}(\tau)}{\lambda^{\frac{3}{2}}(\sigma)}\frac{\rho^{\frac{1}{2}}(\frac{\lambda^2(\tau)}{\lambda^2(\sigma)}\xi)}{\rho^{\frac{1}{2}}(\xi)}\frac{\sin[\lambda(\tau)\xi^{\frac{1}{2}}\int_{\tau}^\sigma\lambda^{-1}(u)\,du]}{\xi^{\frac{1}{2}}}\\&\hspace{4cm}\cdot\beta_{\nu}^2(\sigma)\tilde{A}_{\tilde{\eta}<1}(\xi, \tau, \tau_0, \sigma)\,d\sigma\big\|_{\langle\xi\rangle^{-1-}L^2_{d\xi}(\xi>1)}\\
&\lesssim \tau_0^{-1}\big\|x_0\big\|_{\tilde{S}_1}.
\end{align*}
The preceding considerations then allow replacing $\rho^{\frac12}(\frac{\lambda^2(\tau)}{\lambda^2(\sigma)}\tilde{\eta})$ by 
\[
c\chi_{\frac{\lambda^2(\tau)}{\lambda^2(\sigma)}\tilde{\eta}<1}[\frac{\lambda^2(\tau)}{\lambda^2(\sigma)}\tilde{\eta}]^{-\frac14}, 
\]
Similarly, one reduces $\frac{F(\frac{\lambda^2(\tau)}{\lambda^2(\sigma)}\xi, \frac{\lambda^2(\tau)}{\lambda^2(\sigma)}\tilde{\eta})}{\xi - \tilde{\eta}}$ by $\frac{F(\frac{\lambda^2(\tau)}{\lambda^2(\sigma)}\xi, 0)}{\xi}$ up to a negligible error(bounded as before), and so we finally conclude that we may replace $A_{\tilde{\eta}<1}(\xi, \tau, \tau_0, \sigma)$ by a term of the form 
\[
\frac{F(\frac{\lambda^2(\tau)}{\lambda^2(\sigma)}\xi, 0)}{\xi}\int_0^\infty\frac{\lambda^{\frac52}(\sigma)}{\lambda^{\frac52}(\tau_0)}\chi_{\frac{\lambda^2(\tau)}{\lambda^2(\sigma)}\tilde{\eta}<1} \frac{\cos[\lambda(\tau)\tilde{\eta}^{\frac{1}{2}}\int_{\tau_0}^{\sigma}\lambda^{-1}(u)\,du]}{(\frac{\lambda^2(\tau)}{\lambda^2(\sigma)}\tilde{\eta})^{\frac{1}{4}}}\cdot \rho^{\frac{1}{2}}(\frac{\lambda^2(\tau)}{\lambda^2(\tau_0)}\tilde{\eta})x_0(\frac{\lambda^2(\tau)}{\lambda^2(\tau_0)}\tilde{\eta})\,d\tilde{\eta}.
\]
Throughout keep in mind our assumptions $\xi\gg 1$, $\tau\geq \sigma$. Here we can easily get rid of the localiser $\chi_{\frac{\lambda^2(\tau)}{\lambda^2(\sigma)}\tilde{\eta}<1}$ up to a contribution which is bounded as in the preceding, and so we have finally reduced $A_{\tilde{\eta}<1}(\xi, \tau, \tau_0, \sigma)$ to the expression 
\begin{align*}
&\frac{F(\frac{\lambda^2(\tau)}{\lambda^2(\sigma)}\xi, 0)}{\xi}\int_0^\infty\frac{\lambda^{\frac52}(\sigma)}{\lambda^{\frac52}(\tau_0)} \frac{\cos[\lambda(\tau)\tilde{\eta}^{\frac{1}{2}}\int_{\tau_0}^{\sigma}\lambda^{-1}(u)\,du]}{(\frac{\lambda^2(\tau)}{\lambda^2(\sigma)}\tilde{\eta})^{\frac{1}{4}}}\cdot \rho^{\frac{1}{2}}(\frac{\lambda^2(\tau)}{\lambda^2(\tau_0)}\tilde{\eta})x_0(\frac{\lambda^2(\tau)}{\lambda^2(\tau_0)}\tilde{\eta})\,d\tilde{\eta}\\
& = \frac{F(\frac{\lambda^2(\tau)}{\lambda^2(\sigma)}\xi, 0)}{\frac{\lambda^{2}(\tau)}{\lambda^{2}(\sigma)}\xi}\int_0^\infty\frac{\lambda(\sigma)}{\lambda(\tau_0)} \frac{\cos[\lambda(\tau_0)\tilde{\eta}^{\frac{1}{2}}\int_{\tau_0}^{\sigma}\lambda^{-1}(u)\,du]}{\tilde{\eta}^{\frac{1}{4}}}\cdot \rho^{\frac{1}{2}}(\tilde{\eta})x_0(\tilde{\eta})\,d\tilde{\eta}\\
\end{align*}
However, this term being just like the one displayed after \eqref{eq:simplebutimportant}, substituting it for $A_{\tilde{\eta}<1}(\xi, \tau, \tau_0, \sigma)$ leads to the bound 
\begin{align*}
&\big\|\int_{\tau_0}^\tau\frac{\lambda^{\frac{3}{2}}(\tau)}{\lambda^{\frac{3}{2}}(\sigma)}\frac{\rho^{\frac{1}{2}}(\frac{\lambda^2(\tau)}{\lambda^2(\sigma)}\xi)}{\rho^{\frac{1}{2}}(\xi)}\frac{\sin[\lambda(\tau)\xi^{\frac{1}{2}}\int_{\tau}^\sigma\lambda^{-1}(u)\,du]}{\xi^{\frac{1}{2}}}\beta_{\nu}^2(\sigma)A_{\tilde{\eta}<1}(\xi, \tau, \tau_0, \sigma)\,d\sigma\big\|_{\langle\xi\rangle^{-1-}L^2_{d\xi}(\xi>1)}\\
&\lesssim \big\|x_0\big\|_{\tilde{S}_1}\int_{\tau_0}^\tau (\frac{\lambda(\sigma)}{\lambda(\tau)})^{N}\sigma \beta_{\nu}^2(\sigma)\,d\sigma\lesssim \big\|x_0\big\|_{\tilde{S}_1}.
\end{align*}
This concludes {\it{Step 1(i)}} of the proof of Proposition~\ref{prop:firstiteratelincont} for the contribution of the first term in \eqref{eq:Kx0formula}.

\begin{rem}\label{rem:appendixprop71highfreqimprov} We note that in the last step we could have performed an integration by parts with respect to $\sigma$, and gained an extra power $\tau_0^{-(1-)}$. 

\end{rem}

{\bf{Step 1(ii)}}: {\it{Contribution of $\beta_{\nu}(\tau)\mathcal{K}_{cc}\mathcal{D}_{\tau}x^{(0)}$}}. Here we arrive at the expression 
\begin{align*}
 \int_{\tau_0}^\tau\frac{\lambda^{\frac{3}{2}}(\tau)}{\lambda^{\frac{3}{2}}(\sigma)}\frac{\rho^{\frac{1}{2}}(\frac{\lambda^2(\tau)}{\lambda^2(\sigma)}\xi)}{\rho^{\frac{1}{2}}(\xi)}\frac{\sin[\lambda(\tau)\xi^{\frac{1}{2}}\int_{\tau}^\sigma\lambda^{-1}(u)\,du]}{\xi^{\frac{1}{2}}}\beta_{\nu}(\sigma)A(\xi, \tau, \tau_0, \sigma)\,d\sigma,
 \end{align*}
where this time 
 \begin{align*}
 &A(\xi, \tau, \tau_0, \sigma):\\
 & =  \int_0^\infty\frac{\lambda^{\frac52}(\sigma)}{\lambda^{\frac52}(\tau_0)}\frac{F(\frac{\lambda^2(\tau)}{\lambda^2(\sigma)}\xi, \frac{\lambda^2(\tau)}{\lambda^2(\sigma)}\tilde{\eta})\rho^{\frac{1}{2}}(\frac{\lambda^2(\tau)}{\lambda^2(\sigma)}\tilde{\eta})}{\xi- \tilde{\eta}}\rho^{\frac{1}{2}}(\frac{\lambda^2(\tau)}{\lambda^2(\tau_0)}\tilde{\eta})\tilde{\eta}^{\frac12}\sin[\lambda(\tau)\tilde{\eta}^{\frac{1}{2}}\int_{\tau_0}^{\sigma}\lambda^{-1}(u)\,du]\\
 &\hspace{9.5cm}\cdot x_0(\frac{\lambda^2(\tau)}{\lambda^2(\tau_0)}\tilde{\eta})\,d\tilde{\eta}. 
 \end{align*}
Restricting to $\tilde{\eta}\geq 1$, $\xi>1$, and setting 
\begin{align*}
B(\tau, \sigma, \xi, \tilde{\eta}): = \tilde{\eta}^{\frac12}\frac{\lambda^{\frac{3}{2}}(\tau)}{\lambda^{\frac{3}{2}}(\sigma)}\xi^{\frac14+}\rho^{\frac{1}{2}}(\frac{\lambda^2(\tau)}{\lambda^2(\sigma)}\xi)F(\frac{\lambda^2(\tau)}{\lambda^2(\sigma)}\xi, \frac{\lambda^2(\tau)}{\lambda^2(\sigma)}\tilde{\eta})\rho^{\frac{1}{2}}(\frac{\lambda^2(\tau)}{\lambda^2(\sigma)}\tilde{\eta}),
\end{align*}
we have the bound 
\[
\big|B(\tau, \sigma, \xi, \tilde{\eta})\big|\lesssim \tilde{\eta}^{\frac14+}(\frac{\lambda(\tau)}{\lambda(\sigma)})^{\frac12}.
\]
Furthermore, the operator with kernel 
\[
\frac{\tilde{\eta}^{-\frac14-}B(\tau, \sigma, \xi, \tilde{\eta})}{\xi - \tilde{\eta}}
\]
has $L^2$-bound $\lesssim (\frac{\lambda(\tau)}{\lambda(\sigma)})^{\frac12}$, and we have 
\begin{align*}
&\big\|\frac{\lambda^{\frac52}(\sigma)}{\lambda^{\frac52}(\tau_0)}\rho^{\frac{1}{2}}(\frac{\lambda^2(\tau)}{\lambda^2(\tau_0)}\tilde{\eta})\tilde{\eta}^{\frac14+}x_0(\frac{\lambda^2(\tau)}{\lambda^2(\tau_0)}\tilde{\eta})\big\|_{L^2_{d\tilde{\eta}}(\tilde{\eta}>1)}
\\&\lesssim \frac{\lambda^{\frac52}(\sigma)\lambda^{\frac12}(\tau)}{\lambda^3(\tau_0)}\cdot\frac{\lambda^2(\tau_0)}{\lambda^2(\tau)}\cdot\frac{\lambda(\tau_0)}{\lambda(\tau)}\big\|x_0\big\|_{\langle\tilde{\eta}\rangle^{-(1+)}L^2_{d\tilde{\eta}}},
\end{align*}
Then we infer that 
\begin{align*}
&\big\| \int_{\tau_0}^\tau\frac{\lambda^{\frac{3}{2}}(\tau)}{\lambda^{\frac{3}{2}}(\sigma)}\frac{\rho^{\frac{1}{2}}(\frac{\lambda^2(\tau)}{\lambda^2(\sigma)}\xi)}{\rho^{\frac{1}{2}}(\xi)}\frac{\sin[\lambda(\tau)\xi^{\frac{1}{2}}\int_{\tau}^\sigma\lambda^{-1}(u)\,du]}{\xi^{\frac{1}{2}}}\beta_{\nu}(\sigma)A_{\tilde{\eta}\geq 1}(\xi, \tau, \tau_0, \sigma)\,d\sigma\big\|_{L^2_{d\xi}(\xi>1)}\\
&\lesssim \big\|x_0\big\|_{\tilde{S}_1}\cdot\int_{\tau_0}^{\tau}(\frac{\lambda(\tau)}{\lambda(\sigma)})^{\frac12}\cdot  \frac{\lambda^{\frac52}(\sigma)\lambda^{\frac12}(\tau)}{\lambda^3(\tau_0)}\cdot\frac{\lambda^2(\tau_0)}{\lambda^2(\tau)}\cdot\frac{\lambda(\tau_0)}{\lambda(\tau)}\cdot\beta_{\nu}(\sigma)\,d\sigma\\
&\lesssim  \big\|x_0\big\|_{\tilde{S}_1}.
\end{align*}
The remaining situation $\tilde{\eta}<1$ in this case is handled in analogy to the preceding case {\it{Step 1(i)}}; note that on account of the extra factor $\tilde{\eta}^{\frac12}$ we are never in danger of losing a factor $\sigma$. 
\\

{\bf{Step 3}}: Here we deal with the low frequency case $\xi<1$, and in particular, we explain the choice of $\triangle \tilde{\tilde{x}}^{(1)}_{0,1}$ for the contribution of $x^{(0)}$ coming from $x_0$. Here the precise definition of the norm $\big\|\cdot\big\|_{\tilde{S}_1}$ will play a role. The control of the part of $\triangle x_{>\tau}^{(1)}$ as well as that part of  the components $\triangle \tilde{x}_0^{(1)}, \triangle \tilde{x}_1^{(1)}$  contributed by $x_0$ being similar and routine, we shall omit them. To begin with, we need to complete the proof of Lemma~\ref{lem:choiceofcorrection}, which we recall happened in two stages {\it{(1)}} and {\it{(2)}}. Here we show how to choose $\triangle \tilde{\tilde{x}}_0^{(1)}$, by first establishing 
\\

{\it{Completion of proof of Lemma~\ref{lem:choiceofcorrection}, {\bf{(1)}}}}: {\it{proof of the bound}}
 \begin{align*}
\big|\int_0^\infty \frac{(\rho^{\frac{1}{2}}\tilde{\triangle}\tilde{x}^{(1)}_0)(\xi)}{\xi^{\frac{1}{4}}}\cos[\nu\tau_0\xi^{\frac{1}{2}}]\,d\xi\big|\lesssim \tau_0^{0+}\big\|(x_0,x_1)\big\|_{\tilde{S}} + \tau_0^{-(1-)}|x_{0d}|.  
\end{align*}
This is again straightforward for the contribution of the exponentially decaying term $x^{(0)}_d$, and so we reduce to the contribution of $x^{(0)}$. We follow the same sub-cases as the part of the proof of Lemma~\ref{lem:choiceofcorrection} given earlier, but this time only considering the contribution from $x_0$. Recall the formula for $\tilde{\triangle}\tilde{x}^{(1)}_0$ which is given immediately before the statement of Lemma~\ref{lem:choiceofcorrection}.
\\

{\bf{(1i)}}: {\it{Contribution of  $\beta_{\nu}^2(\sigma)\mathcal{K}_{cc}x^{(0)}(\sigma, \xi)$, low frequency $\xi<1$}}.
Throughout we only consider the contribution of the first term in \eqref{eq:Kx0formula} to $x^{(0)}$. Substituting this formula into the expression for $\widetilde{\triangle}\tilde{x}^{(1)}_0$ and then into the first expression in \eqref{eq:vanishing2}, we obtain the following kind of integral expression 
\begin{equation}\label{eq:app:ugly1}\begin{split}
&\int_{\tau_0}^\infty \frac{\lambda^{\frac{3}{2}}(\tau_0)}{\lambda^{\frac{3}{2}}(\sigma)}\beta_{\nu}^2(\sigma)\int_0^\infty\chi_{\xi<1}\frac{\rho^{\frac{1}{2}}(\frac{\lambda^2(\tau_0)}{\lambda^2(\sigma)}\xi)}{\rho^{\frac{1}{2}}(\xi)}\frac{\rho^{\frac{1}{2}}(\xi)\cos[\nu\tau_0\xi^{\frac{1}{2}}]\sin[\lambda(\tau_0)\xi^{\frac{1}{2}}\int_{\tau_0}^\sigma\lambda^{-1}(u)\,du]}{\xi^{\frac{3}{4}}}\\
&\hspace{9cm}\cdot B(\tau_0, \sigma, \xi)\,d\xi d\sigma,
\end{split}\end{equation}
where we have set 
\begin{align*}
&B(\tau_0, \sigma, \xi):\\& = \int_0^\infty \frac{F(\frac{\lambda^2(\tau_0)}{\lambda^2(\sigma)}\xi, \eta)\rho(\eta)}{\frac{\lambda^2(\tau_0)}{\lambda^2(\sigma)}\xi - \eta}\frac{\lambda^{\frac{5}{2}}(\sigma)}{\lambda^{\frac{5}{2}}(\tau_0)}\frac{\rho^{\frac{1}{2}}(\frac{\lambda^2(\sigma)}{\lambda^2(\tau_0)}\eta)}{\rho^{\frac{1}{2}}(\eta)}\cos[\lambda(\sigma)\eta^{\frac{1}{2}}\int_{\tau_0}^\sigma\lambda^{-1}(u)\,du]x_0(\frac{\lambda^2(\sigma)}{\lambda^2(\tau_0)}\eta)\,d\eta
\end{align*}
Here we quickly mimic the cases {\it{(1i.a)}} - {\it{(1i.c)}} from the earlier proof of Lemma~\ref{lem:choiceofcorrection}. 
\\

 {\it{(1i.a): Imbalanced frequencies $\frac{\lambda^2(\tau_0)}{\lambda^2(\sigma)}\xi\ll \eta$.}} Following the same steps as in the earlier proof, we reduce at fixed time $\sigma\geq \tau_0$ to bounding 
\begin{align*}
&\tau_0\beta_{\nu}^2(\sigma)\int_0^\infty \sup_{\frac{\lambda^2(\tau_0)}{\lambda^2(\sigma)}\xi\ll \eta}\big|\eta^{\frac12}\frac{F(\frac{\lambda^2(\tau_0)}{\lambda^2(\sigma)}\xi, \eta)\rho(\eta)}{\frac{\lambda^2(\tau_0)}{\lambda^2(\sigma)}\xi - \eta}\big|\frac{\lambda^{\frac{5}{2}}(\sigma)}{\lambda^{\frac{5}{2}}(\tau_0)}\big|\frac{\rho^{\frac{1}{2}}(\frac{\lambda^2(\sigma)}{\lambda^2(\tau_0)}\eta)}{\rho^{\frac{1}{2}}(\eta)}\\&\hspace{5cm}\cdot\cos[\lambda(\sigma)\eta^{\frac{1}{2}}\int_{\tau_0}^\sigma\lambda^{-1}(u)\,du]x_0(\frac{\lambda^2(\sigma)}{\lambda^2(\tau_0)}\eta)\big|\,d\eta\\
&\lesssim \tau_0\beta_{\nu}^2(\sigma)\int_0^\infty\frac{\lambda^2(\sigma)}{\lambda^2(\tau_0)} \big(\chi_{\frac{\lambda^2(\sigma)}{\lambda^2(\tau_0)}\eta<1}\big|x_0(\frac{\lambda^2(\sigma)}{\lambda^2(\tau_0)}\eta)\big|\\&\hspace{4cm} + \chi_{\frac{\lambda^2(\sigma)}{\lambda^2(\tau_0)}\eta\geq 1}[\frac{\lambda^2(\sigma)}{\lambda^2(\tau_0)}\eta]^{\frac{1}{2}}\big|x_0(\frac{\lambda^2(\sigma)}{\lambda^2(\tau_0)}\eta)\big|\big)\,d\eta\\
&\lesssim \tau_0\beta_{\nu}^2(\sigma)\big\|x_0\big\|_{\tilde{S}_1}. 
\end{align*}
This in turn can be integrated over $\sigma\geq \tau_0$ to give the desired bound. 
\\

 {\it{(1i.b): Balanced frequencies $\frac{\lambda^2(\tau_0)}{\lambda^2(\sigma)}\xi\sim\eta$.}} Arguing as in the earlier proof, we simply need to exploit here that 
 \begin{align*}
\big\|\frac{\lambda^2(\sigma)}{\lambda^2(\tau_0)}\eta^{\frac12}x_0(\frac{\lambda^2(\sigma)}{\lambda^2(\tau_0)}\eta)\big\|_{L^2_{d\eta}}\lesssim [\frac{\lambda(\sigma)}{\lambda(\tau_0)}]^{0+}\big\|x_0\big\|_{\tilde{S}_1}. 
 \end{align*}
 The desired bound easily follows from the earlier argument. 
\\

 {\it{(1i.c): Imbalanced frequencies $\frac{\lambda^2(\tau_0)}{\lambda^2(\sigma)}\xi\gg\eta$.}} Here instead of \eqref{eq:ugly1}, \eqref{eq:ugly2}, one arrives at the quantities
 \begin{align*}
 \int_0^{\infty}\frac{\lambda^{\frac{5}{2}}(\sigma)}{\lambda^{\frac{5}{2}}(\tau_0)}\rho^{\frac{1}{2}}(\frac{\lambda^2(\sigma)}{\lambda^2(\tau_0)}\eta)\frac{\cos[\lambda(\sigma)\eta^{\frac{1}{2}}\int_{\tau_0}^\sigma\lambda^{-1}(u)\,du]}{\eta^{\frac{1}{4}}}x_0(\frac{\lambda^2(\sigma)}{\lambda^2(\tau_0)}\eta)\,d\eta
 \end{align*}
 \begin{align*}
\int_{c\frac{\lambda^2(\tau_0)}{\lambda^2(\sigma)}\xi}^\infty\frac{\lambda^{\frac{5}{2}}(\sigma)}{\lambda^{\frac{5}{2}}(\tau_0)}\rho^{\frac{1}{2}}(\frac{\lambda^2(\sigma)}{\lambda^2(\tau_0)}\eta)\frac{\cos[\lambda(\sigma)\eta^{\frac{1}{2}}\int_{\tau_0}^\sigma\lambda^{-1}(u)\,du]}{\eta^{\frac{1}{4}}}x_0(\frac{\lambda^2(\sigma)}{\lambda^2(\tau_0)}\eta)\,d\eta, 
\end{align*}
Then one controls the contribution of the first expression by means of the first vanishing condition of \eqref{eq:vanishing}, just as in {\it{(1i.c)}} in the earlier proof of Lemma~\ref{lem:choiceofcorrection}, while the second contribution is bounded by 
\begin{align*}
&\big|\int_{c\frac{\lambda^2(\tau_0)}{\lambda^2(\sigma)}\xi}^\infty\frac{\lambda^{\frac{5}{2}}(\sigma)}{\lambda^{\frac{5}{2}}(\tau_0)}\rho^{\frac{1}{2}}(\frac{\lambda^2(\sigma)}{\lambda^2(\tau_0)}\eta)\frac{\cos[\lambda(\sigma)\eta^{\frac{1}{2}}\int_{\tau_0}^\sigma\lambda^{-1}(u)\,du]}{\eta^{\frac{1}{4}}}x_0(\frac{\lambda^2(\sigma)}{\lambda^2(\tau_0)}\eta)\,d\eta\big|\\
&\lesssim \int_{c\frac{\lambda^2(\tau_0)}{\lambda^2(\sigma)}\xi}^{\frac{\lambda^2(\tau_0)}{\lambda^2(\sigma)}}\frac{\lambda^{2}(\sigma)}{\lambda^{2}(\tau_0)}\eta^{-\frac12}\big|x_0(\frac{\lambda^2(\sigma)}{\lambda^2(\tau_0)}\eta)\big|\,d\eta
+\int_{\frac{\lambda^2(\tau_0)}{\lambda^2(\sigma)}}^\infty \frac{\lambda^3(\sigma)}{\lambda^3(\tau_0)}\big|x_0(\frac{\lambda^2(\sigma)}{\lambda^2(\tau_0)}\eta)\big|\,d\eta.
\end{align*}
Then we get 
\begin{align*}
\int_{c\frac{\lambda^2(\tau_0)}{\lambda^2(\sigma)}\xi}^{\frac{\lambda^2(\tau_0)}{\lambda^2(\sigma)}}\frac{\lambda^{2}(\sigma)}{\lambda^{2}(\tau_0)}\eta^{-\frac12}\big|x_0(\frac{\lambda^2(\sigma)}{\lambda^2(\tau_0)}\eta)\big|\,d\eta&\lesssim \xi^{-\frac12+}[\frac{\lambda(\sigma)}{\lambda(\tau_0)}]^{1+}\int_\xi^1\tilde{\eta}^{0-}\big|x_0(\tilde{\eta})\big|\,d\tilde{\eta}\\
&\lesssim \xi^{-\frac12+}\log \xi[\frac{\lambda(\sigma)}{\lambda(\tau_0)}]^{1+}\big\|x_0\big\|_{\tilde{S}_1},
\end{align*}
and further 
\begin{align*}
\big|\int_{\frac{\lambda^2(\tau_0)}{\lambda^2(\sigma)}}^\infty \frac{\lambda^3(\sigma)}{\lambda^3(\tau_0)}\big|x_0(\frac{\lambda^2(\sigma)}{\lambda^2(\tau_0)}\eta)\big|\,d\eta\big|&\lesssim \frac{\lambda(\sigma)}{\lambda(\tau_0)}\int_1^\infty\big|x_0(\tilde{\eta})\big|\,d\tilde{\eta}\\&
\lesssim \frac{\lambda(\sigma)}{\lambda(\tau_0)}\big\|x_0\big\|_{\tilde{S}_1}. 
\end{align*}
If we combine the preceding two bounds with the estimate (for $\xi<1$, $\sigma\geq \tau_0$)
\begin{align*}
&\big|\frac{\lambda^{\frac{3}{2}}(\tau_0)}{\lambda^{\frac{3}{2}}(\sigma)}\beta_{\nu}^2(\sigma)\frac{\rho^{\frac{1}{2}}(\frac{\lambda^2(\tau_0)}{\lambda^2(\sigma)}\xi)}{\rho^{\frac{1}{2}}(\xi)}\frac{\rho^{\frac{1}{2}}(\xi)\cos[\nu\tau_0\xi^{\frac{1}{2}}]\sin[\lambda(\tau_0)\xi^{\frac{1}{2}}\int_{\tau_0}^\sigma\lambda^{-1}(u)\,du]}{\xi^{\frac{3}{4}}}\big|\\
&\lesssim \nu\tau_0\beta_{\nu}^2(\sigma)\cdot \frac{\lambda(\tau_0)}{\lambda(\sigma)}\xi^{-\frac12}, 
\end{align*}
we see that inserting the second expression at the beginning of {\it{(1i.c)}} for $B(\tau_0,\sigma,\xi)$ in \eqref{eq:app:ugly1}, the resulting expression can be bounded by 
\[
\lesssim \big\|x_0\big\|_{\tilde{S}_1}\int_{\tau_0}^\infty  \nu\tau_0\beta_{\nu}^2(\sigma)[ \frac{\lambda(\sigma)}{\lambda(\tau_0)}]^{0+}\,d\sigma\lesssim  \big\|x_0\big\|_{\tilde{S}_1}, 
\]
as desired. 
\\
As the remaining cases {\it{(1ii), (1iii)}}, in analogy to the earlier proof of Lemma~\ref{lem:choiceofcorrection}, are routine variations on the preceding kinds of estimates, we omit them here. This concludes {\it{(1)}} of the completion of proof of Lemma~\ref{lem:choiceofcorrection}. 
\\

{\it{Completion of proof of Lemma~\ref{lem:choiceofcorrection}, {\bf{(2)}}}}: {\it{Choice of $\triangle\tilde{\tilde{x}}_0^{(1)}$.}} Again we shall set 
\[
\triangle\tilde{\tilde{x}}_0^{(1)} = \beta \mathcal{F}(\chi_{R\leq C\tau_0}\phi(R, 0))
\]
for suitable $\beta\in \R$. We easily get 
\[
\int_0^\infty\frac{(\rho^{\frac{1}{2}}\triangle\tilde{\tilde{x}}_0^{(1)})(\xi)}{\xi^{\frac{1}{4}}}\cos[\nu\tau_0\xi^{\frac{1}{2}}]\,d\xi\sim \beta,
\]
while we also have 
\begin{align*}
\big\|(\triangle\tilde{\tilde{x}}_0^{(1)},0)\big\|_{\tilde{S}}\lesssim \beta\tau_0^{-1}. 
\end{align*}

\subsection{The contribution of the source term $\beta_{\nu}^2(\tau)\mathcal{K}_{cc}^2x^{(0)}$ in the proof of Proposition~\ref{prop:firstiteratelincont}}

Here we briefly consider the contribution of the last term in \eqref{eq:R_1structure} in the proof of Proposition~\ref{prop:firstiteratelincont}. More specifically, we shall reduce $x^{(0)}$ to the contribution of the first term in \eqref{eq:linhomparam1}, the second term there being treated similarly, and we shall consider the bound for the low frequency term $\chi_{\xi<1}\triangle_{>\tau} x^{(1)}$ with respect to the norm $\big\|\cdot\big\|_{S_1}$. Following {\it{Step 4}} in the proof of Proposition~\ref{prop:firstiteratelincont}, consider 
the expression 
\begin{equation}\label{eq:app:finalmess}
-\int_{\tau}^\infty\frac{\lambda^{\frac{3}{2}}(\tau)}{\lambda^{\frac{3}{2}}(\sigma)}\frac{\rho^{\frac{1}{2}}(\frac{\lambda^2(\tau)}{\lambda^2(\sigma)}\xi)}{\rho^{\frac{1}{2}}(\xi)}\frac{\sin[\lambda(\tau)\xi^{\frac{1}{2}}\int_{\tau}^\sigma\lambda^{-1}(u)\,du]}{\xi^{\frac{1}{2}}}\beta_{\nu}^2(\sigma)\big(\mathcal{K}_{cc}^2x^{(0)}\big)(\sigma, \frac{\lambda^2(\tau)}{\lambda^2(\sigma)}\xi)\,d\sigma,
\end{equation}
where throughout we shall restrict to $\xi<1$. Write explicitly 
\begin{align*}
&\mathcal{K}_{cc}^2(\frac{\lambda^2(\tau)}{\lambda^2(\sigma)}\xi)x^{(0)}\\& = \int_0^\infty\frac{F(\frac{\lambda^2(\tau)}{\lambda^2(\sigma)}\xi, \frac{\lambda^2(\tau)}{\lambda^2(\sigma)}\zeta)\rho( \frac{\lambda^2(\tau)}{\lambda^2(\sigma)}\zeta)}{\xi - \zeta}
\int_0^\infty \frac{F(\frac{\lambda^2(\tau)}{\lambda^2(\sigma)}\zeta, \frac{\lambda^2(\tau)}{\lambda^2(\sigma)}\tilde{\eta})\rho( \frac{\lambda^2(\tau)}{\lambda^2(\sigma)}\tilde{\eta})}{\zeta - \tilde{\eta}}x^{(0)}(\frac{\lambda^2(\tau)}{\lambda^2(\sigma)}\tilde{\eta})\,d\tilde{\eta}d\zeta, 
\end{align*}
where we now have (under our current simplified version of $x^{(0)}$) 
\[
x^{(0)}(\frac{\lambda^2(\tau)}{\lambda^2(\sigma)}\tilde{\eta}) = \frac{\lambda^{\frac52}(\sigma)}{\lambda^{\frac52}(\tau_0)}\frac{\rho^{\frac12}(\frac{\lambda^{2}(\tau)}{\lambda^2(\tau_0)}\tilde{\eta})}{\rho^{\frac12}(\frac{\lambda^2(\tau)}{\lambda^2(\sigma)}\tilde{\eta})}\cos[\lambda(\tau)\tilde{\eta}^{\frac12}\int_{\tau_0}^{\sigma}\lambda^{-1}(u)\,du]x_0(\frac{\lambda^{2}(\tau)}{\lambda^2(\tau_0)}\tilde{\eta}). 
\]
It remains to estimate $\big\|\eqref{eq:app:finalmess}\big\|_{S_1} = \big\|\xi^{-(0+)}\eqref{eq:app:finalmess}\big\|_{L^2_{d\xi}}$. We prove this again by distinguishing between a number of cases: 
\\

{\it{{\bf{(1)}}: Imbalanced frequencies $\xi\ll \tilde{\eta}$.}} Introduce the auxiliary function (keeping the temporal variables fixed for now)
\begin{align*}
f(\zeta): = &\int_0^\infty \frac{F(\frac{\lambda^2(\tau)}{\lambda^2(\sigma)}\zeta, \frac{\lambda^2(\tau)}{\lambda^2(\sigma)}\tilde{\eta})\rho^{\frac12}( \frac{\lambda^2(\tau)}{\lambda^2(\sigma)}\tilde{\eta})}{\zeta - \tilde{\eta}}\frac{\lambda^{\frac52}(\sigma)}{\lambda^{\frac52}(\tau_0)}\rho^{\frac12}(\frac{\lambda^2(\tau)}{\lambda^2(\tau_0)}\tilde{\eta})\\&\hspace{4cm}\cdot\cos[\lambda(\tau)\tilde{\eta}\int_{\tau_0}^{\sigma}\lambda^{-1}u\,du]x_0(\frac{\lambda^2(\tau)}{\lambda^2(\tau_0)}\tilde{\eta})\,d\tilde{\eta}
\end{align*}
Then using subscripts to indicate further restrictions of the integration range, we get 
\begin{align*}
\big|f_{\zeta\ll\tilde{\eta}}(\zeta)\big|&\lesssim \frac{\lambda^{\frac52}(\sigma)}{\lambda^{\frac52}(\tau_0)}\int_{C\max\{\zeta,\xi\}}^\infty\chi_{\frac{\lambda^2(\tau)}{\lambda^2(\sigma)}\tilde{\eta}<1}\frac{(\frac{\lambda^2(\tau)}{\lambda^2(\sigma)}\tilde{\eta})^{\frac34}}{\tilde{\eta}}
\rho^{\frac12}(\frac{\lambda^2(\tau)}{\lambda^2(\tau_0)}\tilde{\eta})\big|x_0(\frac{\lambda^2(\tau)}{\lambda^2(\tau_0)}\tilde{\eta})\big|\,d\tilde{\eta}\\
& + \frac{\lambda^{\frac52}(\sigma)}{\lambda^{\frac52}(\tau_0)}\int_{C\max\{\zeta,\xi\}}^\infty\chi_{\frac{\lambda^2(\tau)}{\lambda^2(\sigma)}\tilde{\eta}\geq 1}\tilde{\eta}^{-1}(\frac{\lambda^2(\tau)}{\lambda^2(\tau_0)}\tilde{\eta})^{\frac14}\big|x_0(\frac{\lambda^2(\tau)}{\lambda^2(\tau_0)}\tilde{\eta})\big|\,d\tilde{\eta}.\\
\end{align*}
The first integral expression on the right is then bounded by 
\begin{align*}
&\frac{\lambda^{\frac52}(\sigma)}{\lambda^{\frac52}(\tau_0)}\int_{C\max\{\zeta,\xi\}}^\infty\chi_{\frac{\lambda^2(\tau)}{\lambda^2(\sigma)}\tilde{\eta}<1}\frac{(\frac{\lambda^2(\tau)}{\lambda^2(\sigma)}\tilde{\eta})^{\frac34}}{\tilde{\eta}}
\rho^{\frac12}(\frac{\lambda^2(\tau)}{\lambda^2(\tau_0)}\tilde{\eta})\big|x_0(\frac{\lambda^2(\tau)}{\lambda^2(\tau_0)}\tilde{\eta})\big|\,d\tilde{\eta}\\
&\lesssim \frac{\lambda(\sigma)\lambda(\tau)}{\lambda^2(\tau_0)}\int_{C\max\{\zeta,\xi\}}^\infty\chi_{\frac{\lambda^2(\tau)}{\lambda^2(\tau_0)}\tilde{\eta}<1}\tilde{\eta}^{-\frac12}\big|x_0(\frac{\lambda^2(\tau)}{\lambda^2(\tau_0)}\tilde{\eta})\big|\,d\tilde{\eta}\\
& + \frac{\lambda(\sigma)\lambda^{\frac32}(\tau)}{\lambda^{\frac52}(\tau_0)}\int_{C\max\{\zeta,\xi\}}^\infty\chi_{\frac{\lambda^2(\tau)}{\lambda^2(\tau_0)}\tilde{\eta}\geq 1}\tilde{\eta}^{-\frac14}[\frac{\lambda^2(\tau)}{\lambda^2(\tau_0)}\tilde{\eta}]^{\frac14}\big|x_0(\frac{\lambda^2(\tau)}{\lambda^2(\tau_0)}\tilde{\eta})\big|\,d\tilde{\eta}\\
&\lesssim \xi^{-(\frac12-)}\frac{\lambda(\sigma)\lambda(\tau)}{\lambda^2(\tau_0)}[\frac{\lambda(\tau_0)}{\lambda(\tau)}]^{2-}\big\|x_0\big\|_{\tilde{S}_1}
\end{align*}
On the other hand, the second integral in the bound for $\big|f_{\zeta\ll\tilde{\eta}}(\zeta)\big|$ is dominated by 
\begin{align*}
&\big|\frac{\lambda^{\frac52}(\sigma)}{\lambda^{\frac52}(\tau_0)}\int_{C\max\{\zeta,\xi\}}^\infty\chi_{\frac{\lambda^2(\tau)}{\lambda^2(\sigma)}\tilde{\eta}\geq 1}\tilde{\eta}^{-1}(\frac{\lambda^2(\tau)}{\lambda^2(\tau_0)}\tilde{\eta})^{\frac14}\big|x_0(\frac{\lambda^2(\tau)}{\lambda^2(\tau_0)}\tilde{\eta})\big|\,d\tilde{\eta}\big|\\
&\lesssim \frac{\lambda^{\frac52}(\sigma)}{\lambda^{\frac52}(\tau_0)}\cdot\frac{\lambda^{\frac52}(\tau)}{\lambda^{\frac52}(\sigma)}\cdot\frac{\lambda^{\frac52}(\tau_0)}{\lambda^{\frac52}(\tau)}\big\|x_0\big\|_{\tilde{S}_1} = \big\|x_0\big\|_{\tilde{S}_1},
\end{align*}
where we have exploited that we have $\tilde{\eta}\geq \frac{\lambda^2(\sigma)}{\lambda^2(\tau)}$ in the integration range, as well as the fact that the definition of $\big\|\cdot\big\|_{\tilde{S}_1}$ allows us to include an extra weight $(\frac{\lambda^2(\tau)}{\lambda^2(\tau_0)}\tilde{\eta})^{\frac34}$.
The preceding bounds then immediately imply 
\begin{align*}
&\big\|\xi^{-(0+)}\int_0^\infty\chi_{\xi\sim \zeta}\frac{F(\frac{\lambda^2(\tau)}{\lambda^2(\sigma)}\xi, \frac{\lambda^2(\tau)}{\lambda^2(\sigma)}\zeta)\rho( \frac{\lambda^2(\tau)}{\lambda^2(\sigma)}\zeta)}{\xi - \zeta}f_{\zeta\ll\tilde{\eta}}\,d\zeta\big\|_{L^2_{d\xi}(\xi<1)}\\&\lesssim \big\|\zeta^{\frac12-}f_{\zeta\ll\tilde{\eta}}\big\|_{L^2_{d\zeta}(\zeta<1)}\lesssim \frac{\lambda(\sigma)}{\lambda(\tau_0)}[\frac{\lambda(\tau_0)}{\lambda(\tau)}]^{1-} \big\|x_0\big\|_{\tilde{S}_1},
\end{align*}
where we have exploited the vanishing property of $F(\cdot, \cdot)$ to absorb $\xi^{-\frac12}$. It is then easily seen that the contribution of the term $f_{\xi\sim\zeta\ll\tilde{\eta}}$ to \eqref{eq:app:finalmess} leads to the desired bound, using that for $\xi<1$ we have 
\begin{align*}
\big|\frac{\lambda^{\frac{3}{2}}(\tau)}{\lambda^{\frac{3}{2}}(\sigma)}\frac{\rho^{\frac{1}{2}}(\frac{\lambda^2(\tau)}{\lambda^2(\sigma)}\xi)}{\rho^{\frac{1}{2}}(\xi)}\frac{\sin[\lambda(\tau)\xi^{\frac{1}{2}}\int_{\tau}^\sigma\lambda^{-1}(u)\,du]}{\xi^{\frac{1}{2}}}\beta_{\nu}^2(\sigma)\big|\lesssim \frac{\lambda(\tau)}{\lambda(\sigma)}\cdot\tau\cdot \beta_{\nu}^2(\sigma). 
\end{align*}

Similarly, we obtain 
\begin{align*}
&\big|\int_0^\infty\chi_{\xi\ll\zeta}\frac{F(\frac{\lambda^2(\tau)}{\lambda^2(\sigma)}\xi, \frac{\lambda^2(\tau)}{\lambda^2(\sigma)}\zeta)\rho( \frac{\lambda^2(\tau)}{\lambda^2(\sigma)}\zeta)}{\xi - \zeta}f_{\zeta\ll\tilde{\eta}}\,d\zeta\big|\\&
\lesssim \log (\frac{\tilde{\eta}}{\xi})\xi^{-(\frac12-)}\frac{\lambda(\sigma)}{\lambda(\tau_0)}[\frac{\lambda(\tau_0)}{\lambda(\tau)}]^{1-} \big\|x_0\big\|_{\tilde{S}_1}
\end{align*}
Calculating the norm $\big\|\cdot\big\|_{\xi^{0+}L^2_{d\xi}}$ here leads to a logarithmic divergence, which however is easily avoided by means of a simple orthogonality type argument, exploiting gains in $\frac{\xi}{\tilde{\eta}}$ which we neglected before. 
Finally, the expression arising upon substituting the cutoff $\chi_{\zeta\ll \xi}$ is bounded similarly, exploiting the crude bound
\[
\chi_{\zeta\ll \xi}\big|\frac{F(\frac{\lambda^2(\tau)}{\lambda^2(\sigma)}\xi, \frac{\lambda^2(\tau)}{\lambda^2(\sigma)}\zeta)\rho( \frac{\lambda^2(\tau)}{\lambda^2(\sigma)}\zeta)}{\xi-\zeta}\big|\lesssim \zeta^{-\frac12},\,\xi<1.
\]
This concludes the contribution of $|f_{\zeta\ll\tilde{\eta}}|$. 
\\
Next, we consider what happens when we replace $f(\zeta)$ by $f_{\zeta\sim \tilde{\eta}}$. Proceeding as for $f_{\zeta\ll\tilde{\eta}}$, one infers that 
\begin{align*}
\big\|f_{\zeta\sim \tilde{\eta}}\big\|_{L^2_{d\zeta}}\lesssim \frac{\lambda(\sigma)\lambda(\tau)}{\lambda^2(\tau_0)}[\frac{\lambda(\tau_0)}{\lambda(\tau)}]^{2-}\big\|x_0\big\|_{\tilde{S}_1}, 
\end{align*}
Since $\zeta\sim \tilde{\eta}$ implies $\zeta\gg\xi$ under the hypothesis of case {\it{(1)}}, and further (provided $\xi\ll\zeta$)
\[
\big|F(\frac{\lambda^2(\tau)}{\lambda^2(\sigma)}\xi, \frac{\lambda^2(\tau)}{\lambda^2(\sigma)}\zeta)\rho( \frac{\lambda^2(\tau)}{\lambda^2(\sigma)}\zeta)\big|\lesssim 1, 
\]
we have 
\begin{align*}
\big|\int_0^\infty\chi_{\xi\ll\zeta}\frac{F(\frac{\lambda^2(\tau)}{\lambda^2(\sigma)}\xi, \frac{\lambda^2(\tau)}{\lambda^2(\sigma)}\zeta)\rho( \frac{\lambda^2(\tau)}{\lambda^2(\sigma)}\zeta)}{\xi - \zeta}f_{\zeta\sim\tilde{\eta}}\,d\zeta\big|\lesssim 
\frac{\lambda(\sigma)\lambda(\tau)}{\lambda^2(\tau_0)}[\frac{\lambda(\tau_0)}{\lambda(\tau)}]^{2-}\big\|x_0\big\|_{\tilde{S}_1}. 
\end{align*}
Then setting 
\[
\frac{\lambda^{\frac{3}{2}}(\tau)}{\lambda^{\frac{3}{2}}(\sigma)}\frac{\rho^{\frac{1}{2}}(\frac{\lambda^2(\tau)}{\lambda^2(\sigma)}\xi)}{\rho^{\frac{1}{2}}(\xi)}\frac{\sin[\lambda(\tau)\xi^{\frac{1}{2}}\int_{\tau}^\sigma\lambda^{-1}(u)\,du]}{\xi^{\frac{1}{2}}}\beta_{\nu}^2(\sigma)
= \Phi(\tau,\sigma,\xi),
\]
we get 
\begin{align*}
&\big\|\xi^{-(0+)}\int_{\tau}^\infty \Phi(\tau,\sigma,\xi)\int_0^\infty\chi_{\xi\ll\zeta}\frac{F(\frac{\lambda^2(\tau)}{\lambda^2(\sigma)}\xi, \frac{\lambda^2(\tau)}{\lambda^2(\sigma)}\zeta)\rho( \frac{\lambda^2(\tau)}{\lambda^2(\sigma)}\zeta)}{\xi - \zeta}f_{\zeta\sim\tilde{\eta}}\,d\zeta\big\|_{L^2_{d\xi}(\xi<1)}\\
&\lesssim \big\|x_0\big\|_{\tilde{S}_1}\int_{\tau}^\infty \frac{\lambda(\tau)}{\lambda(\sigma)}\tau\beta_{\nu}^2(\sigma)\cdot\frac{\lambda(\sigma)\lambda(\tau)}{\lambda^2(\tau_0)}[\frac{\lambda(\tau_0)}{\lambda(\tau)}]^{2-}\,d\sigma\lesssim (\frac{\tau}{\tau_0})^{\kappa}\big\|x_0\big\|_{\tilde{S}_1}, 
\end{align*}
which takes care of the contribution of $f_{\zeta\sim\tilde{\eta}}$. It remains to deal with the contribution of $f_{\zeta\gg\tilde{\eta}}$, which is treated similarly to that of $f_{\zeta\ll\tilde{\eta}}$, trading a factor $\xi^{\frac12-}$ for a $\tilde{\eta}^{\frac12-}$. 
This concludes case {\it{(1)}}, i. e. the imbalanced case $\xi\ll\tilde{\eta}$. 
\\

{\it{{\bf{(2)}}: Imbalanced frequencies $\xi\gg\tilde{\eta}$.}} We borrow the notation from the preceding case {\it{(1)}}. First consider the contribution of $f_{\zeta\ll\tilde{\eta}}$, which is bounded by 
\begin{align*}
&\zeta^{\frac12}\big|f_{\zeta\ll\tilde{\eta}}\big|\\&\lesssim \frac{\lambda^{\frac52}(\sigma)}{\lambda^{\frac52}(\tau_0)}\int_{C\zeta}^{c\xi}\tilde{\eta}^{\frac12}\frac{\big|F(\frac{\lambda^2(\tau)}{\lambda^2(\sigma)}\zeta, \frac{\lambda^2(\tau)}{\lambda^2(\sigma)}\tilde{\eta})\big|}{\tilde{\eta}}\rho^{\frac12}(\frac{\lambda^2(\tau)}{\lambda^2(\sigma)}\tilde{\eta})\rho^{\frac12}(\frac{\lambda^2(\tau)}{\lambda^2(\tau_0)}\tilde{\eta})\big|x_0(\frac{\lambda^2(\tau)}{\lambda^2(\tau_0)}\tilde{\eta})\big|\,d\tilde{\eta}\\
&\lesssim \frac{\lambda^{\frac52}(\sigma)}{\lambda^{\frac52}(\tau_0)}\int_{0}^{\min\{c\xi, \frac{\lambda^2(\tau_0)}{\lambda^2(\tau)}\}}\frac{\lambda^2(\tau)}{\lambda^2(\sigma)}(\frac{\lambda(\sigma)}{\lambda(\tau)})^{\frac12}(\frac{\lambda(\tau_0)}{\lambda(\tau)})^{\frac12}\big|x_0(\frac{\lambda^2(\tau)}{\lambda^2(\tau_0)}\tilde{\eta})\big|\,d\tilde{\eta}\\
& + \frac{\lambda^{\frac52}(\sigma)}{\lambda^{\frac52}(\tau_0)}\int_{\min\{c\xi, \frac{\lambda^2(\tau_0)}{\lambda^2(\tau)}\}}^{c\xi}\frac{\lambda^2(\tau)}{\lambda^2(\sigma)}(\frac{\lambda(\sigma)}{\lambda(\tau)})^{\frac12}(\frac{\lambda(\tau)}{\lambda(\tau_0)})^{\frac12}\tilde{\eta}^{\frac12}\big|x_0(\frac{\lambda^2(\tau)}{\lambda^2(\tau_0)}\tilde{\eta})\big|\,d\tilde{\eta}\\
&\lesssim \frac{\lambda(\sigma)}{\lambda(\tau)}(\frac{\tau}{\tau_0})^{\kappa}\big\|x_0\big\|_{\tilde{S}_1}. 
\end{align*}
The preceding estimate, combined with a simple orthogonality argument, then yields (keeping in mind that we restrict to $\xi\gg \tilde{\eta}$)
\begin{align*}
&\big|\int_0^\infty\chi_{\xi\gg\zeta}\frac{F(\frac{\lambda^2(\tau)}{\lambda^2(\sigma)}\xi, \frac{\lambda^2(\tau)}{\lambda^2(\sigma)}\zeta)\rho( \frac{\lambda^2(\tau)}{\lambda^2(\sigma)}\zeta)}{\xi - \zeta}f_{\zeta\ll\tilde{\eta}}\,d\zeta\big|\\&
\lesssim \frac{\lambda(\sigma)}{\lambda(\tau)}(\frac{\tau}{\tau_0})^{\kappa}\big\|x_0\big\|_{\tilde{S}_1}.
\end{align*}
It follows that the corresponding contribution to \eqref{eq:app:finalmess} is bounded by $\lesssim (\frac{\tau}{\tau_0})^{\kappa}\big\|x_0\big\|_{\tilde{S}_1}$. 
\\
Next, consider the contribution of $f_{\zeta\sim \tilde{\eta}}$, whence now $\zeta\ll\xi$. Here we have 
\begin{align*}
&\big\|\zeta^{-(0+)}f_{\zeta\sim\tilde{\eta}}\big\|_{L^2_{d\zeta}}\\& = \big\|\zeta^{-(0+)}\int_0^\infty\chi_{\zeta\sim\tilde{\eta}}\frac{F(\frac{\lambda^2(\tau)}{\lambda^2(\sigma)}\zeta, \frac{\lambda^2(\tau)}{\lambda^2(\sigma)}\tilde{\eta})}{\zeta - \tilde{\eta}}\rho( \frac{\lambda^2(\tau)}{\lambda^2(\sigma)}\tilde{\eta})x^{(0)}( \frac{\lambda^2(\tau)}{\lambda^2(\sigma)}\tilde{\eta})\,d\tilde{\eta}\big\|_{L^2_{d\zeta}}\\
&\lesssim \frac{\lambda^{\frac52}(\sigma)}{\lambda^{\frac52}(\tau_0)}\big\|\tilde{\eta}^{-(0+)}(\frac{\lambda^2(\tau)}{\lambda^2(\sigma)}\tilde{\eta})^{\frac34}\chi_{\frac{\lambda^2(\tau)}{\lambda^2(\tau_0)}\tilde{\eta}<1}(\frac{\lambda(\tau_0)}{\lambda(\tau)})^{\frac12}\tilde{\eta}^{-\frac14}x_0(\frac{\lambda^2(\tau)}{\lambda^2(\tau_0)}\tilde{\eta})\,d\tilde{\eta}\big\|_{L^2_{d\zeta}}\\
& + \frac{\lambda^{\frac52}(\sigma)}{\lambda^{\frac52}(\tau_0)}\big\|\tilde{\eta}^{-(0+)}(\frac{\lambda^2(\tau)}{\lambda^2(\sigma)}\tilde{\eta})^{\frac34}\chi_{\frac{\lambda^2(\tau)}{\lambda^2(\tau_0)}\tilde{\eta}\geq 1}(\frac{\lambda(\tau)}{\lambda(\tau_0)})^{\frac12}\tilde{\eta}^{\frac14}x_0(\frac{\lambda^2(\tau)}{\lambda^2(\tau_0)}\tilde{\eta})\,d\tilde{\eta}\big\|_{L^2_{d\zeta}}\\
\end{align*}
Both of the last expressions are directly seen to be bounded by $\lesssim (\frac{\tau}{\tau_0})^{\kappa}\frac{\lambda(\sigma)}{\lambda(\tau)}\big\|x_0\big\|_{\tilde{S}_1}$. Then we can infer that 
\begin{align*}
&\big|\int_0^\infty\chi_{\xi\gg\zeta}\frac{F(\frac{\lambda^2(\tau)}{\lambda^2(\sigma)}\xi, \frac{\lambda^2(\tau)}{\lambda^2(\sigma)}\zeta)\rho( \frac{\lambda^2(\tau)}{\lambda^2(\sigma)}\zeta)}{\xi - \zeta}f_{\zeta\sim\tilde{\eta}}\,d\zeta\big|\\&
\lesssim \big\|\zeta^{-(0+)}f_{\zeta\sim\tilde{\eta}}\big\|_{L^2_{d\zeta}}\lesssim \frac{\lambda(\sigma)}{\lambda(\tau)}(\frac{\tau}{\tau_0})^{\kappa}\big\|x_0\big\|_{\tilde{S}_1}, 
\end{align*}
which again suffices to bound the corresponding contribution to \eqref{eq:app:finalmess}. 
\\
Finally, we consider the contribution of $f_{\zeta\gg\tilde{\eta}}$. Here we can replace $f_{\zeta\gg\tilde{\eta}}$ by 
\begin{align*}
&\frac{\lambda^2(\tau)}{\lambda^2(\sigma)}\int_0^{c\xi}\rho^{\frac12}(\frac{\lambda^2(\tau)}{\lambda^2(\sigma)}\tilde{\eta})\frac{\lambda^{\frac52}(\sigma)}{\lambda^{\frac52}(\tau_0)}\rho^{\frac12}(\frac{\lambda^2(\tau)}{\lambda^2(\tau_0)}\tilde{\eta})\cos[\lambda(\tau)\tilde{\eta}^{\frac12}\int_{\tau_0}^{\sigma}\lambda^{-1}(u)\,du]x_0(\frac{\lambda^{2}(\tau)}{\lambda^2(\tau_0)}\tilde{\eta})\,d\tilde{\eta}\\
& = -\frac{\lambda^2(\tau)}{\lambda^2(\sigma)}\int_{c\xi}^\infty\rho^{\frac12}(\frac{\lambda^2(\tau)}{\lambda^2(\sigma)}\tilde{\eta})\frac{\lambda^{\frac52}(\sigma)}{\lambda^{\frac52}(\tau_0)}\rho^{\frac12}(\frac{\lambda^2(\tau)}{\lambda^2(\tau_0)}\tilde{\eta})\cos[\lambda(\tau)\tilde{\eta}^{\frac12}\int_{\tau_0}^{\sigma}\lambda^{-1}(u)\,du]x_0(\frac{\lambda^{2}(\tau)}{\lambda^2(\tau_0)}\tilde{\eta})\,d\tilde{\eta}\\
& + \frac{\lambda^2(\tau)}{\lambda^2(\sigma)}\int_0^{\infty}\rho^{\frac12}(\frac{\lambda^2(\tau)}{\lambda^2(\sigma)}\tilde{\eta})\frac{\lambda^{\frac52}(\sigma)}{\lambda^{\frac52}(\tau_0)}\rho^{\frac12}(\frac{\lambda^2(\tau)}{\lambda^2(\tau_0)}\tilde{\eta})\cos[\lambda(\tau)\tilde{\eta}^{\frac12}\int_{\tau_0}^{\sigma}\lambda^{-1}(u)\,du]x_0(\frac{\lambda^{2}(\tau)}{\lambda^2(\tau_0)}\tilde{\eta})\,d\tilde{\eta}\\
\end{align*}
Here in the first integral expression on the right, trading a power $\xi^{\frac12-}$, which will come from the $L^2_{d\xi}$-norm of the output, for a power $\tilde{\eta}^{\frac12-}$, one easily finds (recall $\xi<1$)
\begin{align*}
&\big|\frac{\lambda^2(\tau)}{\lambda^2(\sigma)}\int_{c\xi}^\infty\tilde{\eta}^{\frac12-}\rho^{\frac12}(\frac{\lambda^2(\tau)}{\lambda^2(\sigma)}\tilde{\eta})\frac{\lambda^{\frac52}(\sigma)}{\lambda^{\frac52}(\tau_0)}\rho^{\frac12}(\frac{\lambda^2(\tau)}{\lambda^2(\tau_0)}\tilde{\eta})\cos[\lambda(\tau)\tilde{\eta}^{\frac12}\int_{\tau_0}^{\sigma}\lambda^{-1}(u)\,du]x_0(\frac{\lambda^{2}(\tau)}{\lambda^2(\tau_0)}\tilde{\eta})\,d\tilde{\eta}\big|\\
&\lesssim \frac{\lambda(\sigma)}{\lambda(\tau)}(\frac{\tau}{\tau_0})^{\kappa}\big\|x_0\big\|_{\tilde{S}_1}, 
\end{align*}
and from here one easily infers the desired bound for the corresponding contribution to \eqref{eq:app:finalmess}. One may thus replace the expression $f_{\zeta\gg\tilde{\eta}}$ by 
\[
\frac{\lambda^2(\tau)}{\lambda^2(\sigma)}\int_0^{\infty}\rho^{\frac12}(\frac{\lambda^2(\tau)}{\lambda^2(\sigma)}\tilde{\eta})\frac{\lambda^{\frac52}(\sigma)}{\lambda^{\frac52}(\tau_0)}\rho^{\frac12}(\frac{\lambda^2(\tau)}{\lambda^2(\tau_0)}\tilde{\eta})\cos[\lambda(\tau)\tilde{\eta}^{\frac12}\int_{\tau_0}^{\sigma}\lambda^{-1}(u)\,du]x_0(\frac{\lambda^{2}(\tau)}{\lambda^2(\tau_0)}\tilde{\eta})\,d\tilde{\eta}.
\]
This expression we then reduce to the 
\[
\frac{\lambda^2(\tau)}{\lambda^2(\sigma)}\int_0^{\infty}\chi_{\frac{\lambda^2(\tau)}{\lambda^2(\sigma)}\tilde{\eta}<1}\rho^{\frac12}(\frac{\lambda^2(\tau)}{\lambda^2(\sigma)}\tilde{\eta})\frac{\lambda^{\frac52}(\sigma)}{\lambda^{\frac52}(\tau_0)}\rho^{\frac12}(\frac{\lambda^2(\tau)}{\lambda^2(\tau_0)}\tilde{\eta})\cos[\lambda(\tau)\tilde{\eta}^{\frac12}\int_{\tau_0}^{\sigma}\lambda^{-1}(u)\,du]x_0(\frac{\lambda^{2}(\tau)}{\lambda^2(\tau_0)}\tilde{\eta})\,d\tilde{\eta}
\]
up to a term which in terms of $L^\infty_{d\zeta}$ is bounded by $\lesssim (\frac{\tau}{\tau_0})^{\kappa}\big\|x_0\big\|_{\tilde{S}_1}$, which then furnishes an acceptable contribution to \eqref{eq:app:finalmess} as is easily verified. 
In turn, one reduces the preceding expression to 
\[
\frac{\lambda^2(\tau)}{\lambda^2(\sigma)}\int_0^{\infty}\chi_{\frac{\lambda^2(\tau)}{\lambda^2(\sigma)}\tilde{\eta}<1}[\frac{\lambda^2(\tau)}{\lambda^2(\sigma)}\tilde{\eta}]^{-\frac14}\frac{\lambda^{\frac52}(\sigma)}{\lambda^{\frac52}(\tau_0)}\rho^{\frac12}(\frac{\lambda^2(\tau)}{\lambda^2(\tau_0)}\tilde{\eta})\cos[\lambda(\tau)\tilde{\eta}^{\frac12}\int_{\tau_0}^{\sigma}\lambda^{-1}(u)\,du]x_0(\frac{\lambda^{2}(\tau)}{\lambda^2(\tau_0)}\tilde{\eta})\,d\tilde{\eta}
\]
up to a negligible error, and finally, one reduces to 
\begin{align*}
&\frac{\lambda^2(\tau)}{\lambda^2(\sigma)}\int_0^{\infty}[\frac{\lambda^2(\tau)}{\lambda^2(\sigma)}\tilde{\eta}]^{-\frac14}\frac{\lambda^{\frac52}(\sigma)}{\lambda^{\frac52}(\tau_0)}\rho^{\frac12}(\frac{\lambda^2(\tau)}{\lambda^2(\tau_0)}\tilde{\eta})\cos[\lambda(\tau)\tilde{\eta}^{\frac12}\int_{\tau_0}^{\sigma}\lambda^{-1}(u)\,du]x_0(\frac{\lambda^{2}(\tau)}{\lambda^2(\tau_0)}\tilde{\eta})\,d\tilde{\eta}\\
\end{align*}
But then arguing as in {\it{(2)}} of the proof of Proposition~\ref{prop:lingrowthcond}, the preceding equals 
\begin{equation}\label{eq:app:againthisone}
\frac{\lambda(\sigma)}{\lambda(\tau_0)}\int_0^\infty \eta^{-\frac14}\cos[\lambda(\tau_0)\eta^{\frac12}\int_{\tau_0}^\sigma\lambda^{-1}(u)\,du]\rho^{\frac12}(\eta)x_0(\eta)\,d\eta
\end{equation}
which is bounded by $\sigma\big\|x_0\big\|_{\tilde{S}_1}$ thanks to the vanishing condition satisfied by $x_0$. Inserting the preceding expression instead of $f_{\zeta\gg\tilde{\eta}}$ into the formula for $\mathcal{K}_{cc}^2x^{(0)}$ then results almost in a good bound for the corresponding contribution for \eqref{eq:app:finalmess}, except that the singular weight $\xi^{-(0+)}$ used in the norm for the output just barely causes a divergence which we have to avoid by using an orthogonality type argument. In fact, note that we can simply bound (with the preceding expression representing $f_{\zeta\gg\tilde{\eta}}$)
\begin{align*}
\big|\int_{C\xi}^\infty\frac{F(\frac{\lambda^2(\tau)}{\lambda^2(\sigma)}\xi, \frac{\lambda^2(\tau)}{\lambda^2(\sigma)}\zeta)\rho( \frac{\lambda^2(\tau)}{\lambda^2(\sigma)}\zeta)}{\xi - \zeta}f_{\zeta\gg\tilde{\eta}}\,d\zeta\big|
&\lesssim\big\|f_{\zeta\gg\tilde{\eta}}\big\|_{L^\infty_{d\zeta}}\\
&\lesssim\sigma \big\|x_0\big\|_{\tilde{S}_1}, 
\end{align*}
and we 'almost have' (recall the notation introduced earlier above)
\begin{align*}
\big\|\chi_{\xi<1}\xi^{-(0+)}\Phi(\tau, \sigma, \xi)\big\|_{L^2_{d\xi}}\lesssim \frac{\lambda(\tau)}{\lambda(\sigma)}\beta_{\nu}^2(\sigma), 
\end{align*}
whence we just miss the estimate 
\[
\big\|x_0\big\|_{\tilde{S}_1}\int_{\tau}^\infty\frac{\lambda(\tau)}{\lambda(\sigma)}\beta_{\nu}^2(\sigma)\cdot\sigma\,d\sigma\sim \big\|x_0\big\|_{\tilde{S}_1}.
\]
for this contribution to \eqref{eq:app:finalmess}. 
\\
To deal with the divergence, we note that  with $\eta$ the integration variable in \eqref{eq:app:againthisone}, if $\lambda^2(\tau)\xi\gtrsim \lambda^2(\tau_0)\eta$, we can absorb a factor $\xi^{-(0+)}$ into the integral \eqref{eq:app:againthisone} while gaining a factor 
\[
[\frac{\lambda^2(\tau_0)\eta}{\lambda^2(\tau)\xi}]^{0+}[\frac{\tau}{\tau_0}]^{\kappa}
\]
and then a simple orthogonality argument, using localisations to dyadic values of $\xi, \eta$, furnishes the desired sharp bound (with a loss $(\frac{\tau}{\tau_0})^\kappa$, which is admissible in light of the shape estimate \eqref{eq:firstiterlow} we are striving to prove).
On the other hand, if $\lambda^2(\tau)\xi\ll\lambda^2(\tau_0)\eta$, we are performing an integration by parts in the $\sigma$-integral \eqref{eq:app:finalmess}, which in  the worst case trades a factor $\frac{\lambda(\tau)}{\lambda(\sigma)}\xi^{\frac12}$ for a factor 
$\frac{\lambda(\tau_0)}{\lambda(\sigma)}\eta^{\frac12}$ (coming from \eqref{eq:app:againthisone}). Then performing the $L^2_{d\xi}$-integral on \eqref{eq:app:finalmess} with the weight $\xi^{-(0+)}$, one obtains (after dyadic localisation, say) a factor 
\[
\frac{\lambda(\tau)\xi^{\frac12-}}{\lambda(\tau_0)\eta^{\frac12}} = [\frac{\lambda(\tau)\xi^{\frac12}}{\lambda(\tau_0)\eta^{\frac12}}]^{1-}(\frac{\tau}{\tau_0})^{\kappa}\eta^{-(0+)},
\]
and the weight $\eta^{-(0+)}$ can again be absorbed into \eqref{eq:app:againthisone} in light of the definition of $\big\|\cdot\big\|_{\tilde{S}_1}$. Then another orthogonality argument furnishes the desired sharp bound \eqref{eq:firstiterlow} for this contribution to \eqref{eq:app:finalmess} as well. 
This finally concludes the contribution of $f_{\zeta\gg\tilde{\eta}}$ to \eqref{eq:app:finalmess} in case {\it{(2)}}, and thereby case {\it{(2)}}.
\\

{\it{{\bf{(3)}}: Balanced frequencies $\xi\sim \tilde{\eta}$}}. This is again similar to case {\it{(1)}} and we omit it (one trades a factor $\xi^{\frac12}$ for a factor $\tilde{\eta}^{\frac12}$ to five boundedness of the $\tilde{\eta}$-integral.

\centerline{\scshape Joachim Krieger }
\medskip
{\footnotesize
 \centerline{B\^{a}timent des Math\'ematiques, EPFL}
\centerline{Station 8, 
CH-1015 Lausanne, 
  Switzerland}
  \centerline{\email{joachim.krieger@epfl.ch}}
} 

\end{document}